\documentclass[11pt,fleqn]{article}

\usepackage{latexsym}
\usepackage{amssymb}
\usepackage{stmaryrd}
\usepackage{phonetic}

\newcommand{\be}{\begin{enumerate}}
\newcommand{\ee}{\end{enumerate}}
\newcommand{\bi}{\begin{itemize}}
\newcommand{\ei}{\end{itemize}}
\newcommand{\bc}{\begin{center}}
\newcommand{\ec}{\end{center}}
\newcommand{\bsp}{\begin{sloppypar}}
\newcommand{\esp}{\end{sloppypar}}

\newcommand{\sglsp}{\ }
\newcommand{\dblsp}{\ \ }

\newcommand{\sC}{\mbox{$\cal C$}}

\newcommand{\sE}{\mbox{$\cal E$}}
\newcommand{\sF}{\mbox{$\cal F$}}

\newcommand{\sK}{\mbox{$\cal K$}}

\newcommand{\sO}{\mbox{$\cal O$}}

\newcommand{\sS}{\mbox{$\cal S$}}
\newcommand{\sT}{\mbox{$\cal T$}}

\newcommand{\sV}{\mbox{$\cal V$}}

\renewcommand{\phi}{\varphi}
\newcommand{\seq}[1]{{\langle #1 \rangle}}
\newcommand{\set}[1]{{\{ #1 \}}}

\newcommand{\mlist}[1]{{[ #1 ]}}
\newcommand{\sembrack}[1]{\llbracket#1\rrbracket}
\newcommand{\synbrack}[1]{\ulcorner#1\urcorner}
\newcommand{\commabrack}[1]{\lfloor#1\rfloor}

\newcommand{\mname}[1]{\mbox{\sf #1}}
\newcommand{\mcolon}{\mathrel:}
\newcommand{\mdot}{\mathrel.}

\newcommand{\proves}[2]{#1 \vdash #2}

\newcommand{\parrow}{\rightharpoonup}
\newcommand{\tarrow}{\rightarrow}

\newcommand{\StarApp}{\star\,}

\newcommand{\iotaApp}{\iota\,}

\newcommand{\epsilonApp}{\epsilon\,}

\newcommand{\Neg}{\neg} 
\newcommand{\And}{\wedge}

\newcommand{\Or}{\vee}
\newcommand{\Implies}{\supset}

\newcommand{\Iff}{\equiv}

\newcommand{\Forall}{\forall}
\newcommand{\ForallApp}{\forall\,}
\newcommand{\Forsome}{\exists}
\newcommand{\ForsomeApp}{\exists\,}
\newcommand{\ForsomeUniqueApp}{\exists\,!\,}
\newcommand{\IsDef}{\downarrow}
\newcommand{\IsUndef}{\uparrow}
\newcommand{\Equal}{=}
\newcommand{\QuasiEqual}{\simeq}
\newcommand{\Undefined}{\bot}
\newcommand{\If}{\mname{if}}
\newcommand{\IsDefApp}{\!\IsDef}
\newcommand{\IsUndefApp}{\!\IsUndef}
\newcommand{\TRUE}{\mbox{{\sc t}}}
\newcommand{\FALSE}{\mbox{{\sc f}}}

\newcommand{\LambdaApp}{\lambda\,}
\newcommand{\LAMBDAapp}{\Lambda\,}

\newcommand{\ClassAbsApp}{\mbox{\sc C}\,}

\newcommand{\DepTypeProdApp}{\otimes\,}

\newcommand{\funapp}{\mathrel@}

\newcommand{\fol}{\mbox{\sc fol}}

\newcommand{\zf}{\mbox{\sc zf}}
\newcommand{\nbg}{\mbox{\sc nbg}}

\newcommand{\stmm}{\mbox{\sc stmm}}

\newtheorem{axschemas}{Axiom Schemas}
\newtheorem{infrule}{Rule}

\newenvironment{proof}{\par\noindent{\bf Proof\dblsp}}{$\Box$}

\newtheorem{cthm}{Theorem}[subsection]
\newtheorem{ccor}[cthm]{Corollary}
\newtheorem{clem}[cthm]{Lemma}
\newtheorem{cprop}[cthm]{Proposition}
\newtheorem{crem}[cthm]{Remark}
\newtheorem{ceg}[cthm]{Example}

\newcommand{\cexpr}[2]{\textbf{expr}_{#1}\textbf{[}#2\textbf{]}}
\newcommand{\cmexpr}[2]{\textbf{m-expr}_{#1}\textbf{[}#2\textbf{]}}
\newcommand{\cpexpr}[3]{\textbf{p-expr}_{#1}\textbf{[}#2:#3\textbf{]}}
\newcommand{\cop}[2]{\textbf{operator}_{#1}\textbf{[}#2\textbf{]}}
\newcommand{\ctype}[2]{\textbf{type}_{#1}\textbf{[}#2\textbf{]}}
\newcommand{\cterm}[2]{\textbf{term}_{#1}\textbf{[}#2\textbf{]}}
\newcommand{\cterma}[3]{\textbf{term}_{#1}\textbf{[}#2:#3\textbf{]}}
\newcommand{\cform}[2]{\textbf{formula}_{#1}\textbf{[}#2\textbf{]}}
\newcommand{\ckind}[2]{\textbf{kind}_{#1}\textbf{[}#2\textbf{]}}
\newcommand{\crel}[1]{\overline{#1}}

\newcommand{\Dv}{D_{\rm v}}
\newcommand{\Dc}{D_{\rm c}}
\newcommand{\Ds}{D_{\rm s}}
\newcommand{\Df}{D_{\rm f}}
\newcommand{\Do}{D_{\rm o}}
\newcommand{\De}{D_{\rm e}}
\newcommand{\Don}{D_{\rm on}}
\newcommand{\Dsy}{D_{\rm sy}}
\newcommand{\Dop}{D_{\rm op}}
\newcommand{\Dty}{D_{\rm ty}}
\newcommand{\Dte}{D_{\rm te}}
\newcommand{\Dfo}{D_{\rm fo}}
\newcommand{\Dea}{D_{\rm e,1}}
\newcommand{\Deb}{D_{\rm e,2}}

\newcommand{\TypeLE}{\ll}
\newcommand{\TypeEqual}{=_{\rm ty}}

\newcommand{\cnbg}{\mbox{\sc cnbg}}
\newcommand{\subvaluation}{\sqsubseteq}
\newcommand{\qmname}[1]{\synbrack{\mname{#1}}}
\newcommand{\provesast}[2]{#1 \vdash^\ast #2}

\title{\bf Chiron: A Set Theory with Types, Undefinedness, Quotation,
  and Evaluation\thanks{Published as SQRL Report No.~38, McMaster
    University, 2007 (revised 2012).  This research was supported by
    NSERC.}}

\author{William M. Farmer\thanks{Address: Department of Computing and
    Software, McMaster University, 1280 Main Street West, Hamilton,
    Ontario L8S 4K1, Canada.  E-mail:
    \texttt{wmfarmer@mcmaster.ca}.}\\ McMaster University}

\date{28 December 2012}

\begin{document}

\maketitle

\begin{abstract}\noindent
Chiron is a derivative of von-Neumann-Bernays-G\"odel ({\nbg}) set
theory that is intended to be a practical, general-purpose logic for
mechanizing mathematics.  Unlike traditional set theories such as
Zermelo-Fraenkel ({\zf}) and {\nbg}, Chiron is equipped with a type
system, lambda notation, and definite and indefinite description.  The
type system includes a universal type, dependent types, dependent
function types, subtypes, and possibly empty types.  Unlike
traditional logics such as first-order logic and simple type theory,
Chiron admits undefined terms that result, for example, from a
function applied to an argument outside its domain or from an improper
definite or indefinite description.  The most noteworthy part of
Chiron is its facility for reasoning about the syntax of expressions.
\emph{Quotation} is used to refer to a set called a construction that
represents the syntactic structure of an expression, and
\emph{evaluation} is used to refer to the value of the expression that
a construction represents.  Using quotation and evaluation, syntactic
side conditions, schemas, syntactic transformations used in deduction
and computation rules, and other such things can be directly expressed
in Chiron.  This paper presents the syntax and semantics of Chiron,
some definitions and simple examples illustrating its use, a proof
system for Chiron, and a notion of an interpretation of one theory of
Chiron in another.
\end{abstract}

\newpage

\tableofcontents
\listoftables

\newpage

\section{Introduction}

The usefulness of a logic is often measured by its expressivity: the
more that can be expressed in the logic, the more useful the logic is.
By a \emph{logic}, we mean a language (or a family of languages) that
has a formal syntax and a precise semantics with a notion of logical
consequence.  (A logic may also have, but is not required to have, a
proof system.)  By this definition, a theory in a logic---such as
Zermelo-Fraenkel ({\zf}) set theory in first-order order---is itself a
logic.  But what do we mean by \emph{expressivity}?  There are
actually two notions of expressivity.  The \emph{theoretical
expressivity} of a logic is the measure of what ideas can be expressed
in the logic without regard to how the ideas are expressed.  The
\emph{practical expressivity} of a logic is the measure of how readily
ideas can be expressed in the logic.

To illustrate the difference between these two notions, let us compare
two logics, standard first-order logic ({\fol}) and first-order logic
without function symbols (${\fol}^-$).  Since functions can be
represented using either predicate symbols or function symbols, {\fol}
and ${\fol}^-$ clearly have exactly the same theoretical expressivity.
For example, if three functions are represented as unary function
symbols $f,g,h$ in {\fol}, these functions can be represented as
binary predicate symbols $p_f,p_g,p_h$ in ${\fol}^-$.  The statement
that the third function is the composition of the first two functions
is expressed in {\fol} by the formula \[\ForallApp x \mdot h(x) =
f(g(x)), \] while it is expressed in ${\fol}^-$ by the more verbose
formula \[\ForallApp x,z \mdot p_h(x,z) \Iff \ForsomeApp y \mdot
p_g(x,y) \And p_f(y,z).\] The verbosity that comes from using
predicate symbols to represent functions progressively increases as
the complexity of statements about functions increases.  Hence,
${\fol}^-$ has a significantly lower level of practical expressivity
than {\fol} does.

Traditional general-purpose logics---such as predicate logics like
first-order logic and simple type theory and set theories like {\zf}
and von-Neumann-Bernays-G\"odel ({\nbg}) set theory---are primarily
intended to be theoretical tools.  They are designed to be used
\emph{in theory}, not \emph{in practice}.  They are thus very
expressive theoretically, but not very expressive practically.  For
example, in the languages of {\zf} and {\nbg}, there is no vocabulary
for forming a term $f(a)$ that denotes the application of a set $f$
representing a function to a set $a$ representing an argument to $f$.
Moreover, even if such an application operator were added to {\zf} or
{\nbg}, there is no special mechanism for handling ``undefined''
applications.  As a result, statements involving functions and
undefinedness are much more verbose and indirect than they need to be,
and reasoning about functions and undefinedness is usually performed
in the metalogic instead of in the logic itself.

\emph{Chiron} is a set theory that has a much higher level of
practical expressivity than traditional set theories.  It is intended
to be a general-purpose logic that, unlike traditional logics, is
designed to be used in practice.  It integrates {\nbg} set theory,
elements of type theory, a scheme for handling undefinedness, and a
facility for reasoning about the syntax of expressions.  This paper
presents the syntax and semantics of Chiron, some definitions and
simple examples illustrating its use, a proof system for Chiron, and a
notion of an interpretation of one theory of Chiron in another.  A
quicker, more informal presentation of the syntax and semantics of
Chiron is found in~\cite{Farmer07a}.

The following is the outline of the paper.  Section~\ref{sec:overview}
gives an informal overview of Chiron.  Section~\ref{sec:syntax}
presents Chiron's official syntax and an unofficial compact notation
for Chiron.  The semantics of Chiron is given in
section~\ref{sec:semantics}.  We also show in
section~\ref{sec:semantics} that there is a faithful semantic
interpretation of {\nbg} in Chiron.  A large group of useful operators
are defined in sections~\ref{sec:op-defs} and \ref{sec:sub} including
the operators needed for the substitution of a term for the
occurrences of a free variable.  Some of the practical expressivity of
Chiron is illustrated by examples in section~\ref{sec:examples}.
Section~\ref{sec:ps} presents a proof system for Chiron and proves
that it is sound and also complete in a restricted sense.
Section~\ref{sec:interp} defines the notion of a semantic
interpretation of one theory of Chiron in another.  The paper
concludes in section~\ref{sec:conclusion} with a brief summary and a
list of future tasks.  There are two appendices.  The first presents
two alternate semantics for Chiron based on value gaps, and the second
gives an expanded definition of a proper expression.

\section{Overview}\label{sec:overview}

This section gives an informal overview of Chiron.  A formal
definition of the syntax and semantics of Chiron is presented in
subsequent sections.

\subsection{NBG Set Theory}

{\nbg} set theory is closely related to the more well-known {\zf} set
theory.  The underlying logic of both {\nbg} and {\zf} is first-order
logic, and {\nbg} and {\zf} both share the same intuitive model of the
iterated hierarchy of sets.  However, in contrast to {\zf}, variables
in {\nbg} range over both sets and proper classes.  Thus, the universe
of sets $V$ and total functions from $V$ to $V$ like the cardinality
function can be represented as terms in {\nbg} even though they are
proper classes.  There is a faithful interpretation of {\zf} in
{\nbg}~\cite{Novak50,RosserWang50,Shoenfield54}.  This means that
reasoning in {\zf} can be reduced to reasoning in {\nbg} in a
meaning-preserving way and that {\nbg} is consistent iff {\zf} is
consistent. A good introduction to {\nbg} is found in \cite{Goedel40}
or \cite{Mendelson97}.

Chiron is a derivative of {\nbg}.  It is an enhanced version of
{\stmm}~\cite{Farmer01}, a version of {\nbg} with types and
undefinedness.  Chiron has a much richer syntax and more complex
semantics than {\nbg}, but the models for Chiron contain exactly the
same values (i.e., classes) as the models for {\nbg}.  Moreover, there
is a faithful semantic interpretation of {\nbg} in Chiron---which
means that there is a meaning-preserving embedding of {\nbg} in Chiron
such that Chiron is a conservative extension of the image of {\nbg}
under the embedding.  That is, Chiron adds new vocabulary and
assumptions to {\nbg} without compromising the underlying semantics of
{\nbg}, and hence, Chiron is satisfiable iff {\nbg} is satisfiable.

\subsection{Values}

A \emph{value} is a set, class, superclass, truth value, undefined
value, or operation.  A \emph{class} is an element of a model of
{\nbg} set theory.  Each class is a collection of classes.  A
\emph{set} is a class that is a member of a class.  A class is thus a
collection of sets.  A class is \emph{proper} if it not a set.
Intuitively, sets are ``small'' classes and proper classes are ``big''
classes.  A \emph{superclass} is a collection of classes that need not
be a class itself.  Summarizing, the domain $\Dv$ of sets is a proper
subdomain of the domain $\Dc$ of classes, and $\Dc$ is a proper
subdomain of the domain $\Ds$ of superclasses.  $\Dv$ is the universal
class (the class of all sets), and $\Dc$ is the universal superclass
(the superclass of all classes).

There are two truth values, $\TRUE$ representing \emph{true} and
$\FALSE$ representing \emph{false}.  The truth values are not members
of $\Ds$.  There is also an \emph{undefined value} $\Undefined$ which
serves as the value of various sorts of undefined terms such as
undefined function applications and improper definite or indefinite
descriptions.  $\Undefined$ is not a member of $D_s \cup
\set{\TRUE,\FALSE}$.

An \emph{operation} is a mapping over superclasses, the truth values,
and the undefined value.  More precisely, for $n \ge 0$, an
\emph{$n$-ary operation} is a total mapping \[\sigma: D_1 \times
\cdots \times D_n \tarrow D_{n+1}\] where $D_i$ is $\Ds$, $\Dc \cup
\set{\Undefined}$, or $\set{\TRUE,\FALSE}$ for all $i$ with $1 \le i
\le n+1$.  An operation is not a member of $D_s \cup
\set{\TRUE,\FALSE,\Undefined}$.  A \emph{function} is a class of
ordered pairs that represents a (possibly partial) mapping \[f: \Dv
\parrow \Dv.\] Operations are not classes, but many operations can be
represented by functions (which are classes).

\subsection{Expressions}

An \emph{expression} is a tree whose leaves are \emph{symbols}.  There
are four special sorts of expressions: \emph{operators}, \emph{types},
\emph{terms}, and \emph{formulas}.  An expression is \emph{proper} is
it is one of these special sorts of expressions, and an expression is
\emph{improper} if it is not proper.  Proper expressions denote
values, while improper expressions are nondenoting (i.e., they do not
denote anything).

Operators denote operations.  Many sorts of syntactic entities can be
formalized in Chiron as operators.  Examples include logical
connectives; individual constants, function symbols, and predicate
symbols from first-order logic; type constants and type constructors
including dependent type constructors; and definedness operators.
Like a function or predicate symbol in first-order logic, an operator
in Chiron is not useful unless it is applied.

Types are used to restrict the values of operators and variables and
to classify terms by their values.  They denote superclasses.  Terms
are used to describe classes.  They denote classes or the undefined
value $\Undefined$.  A term is \emph{defined} if it denotes a class
and is \emph{undefined} if it denotes $\Undefined$.  Every term is
assigned a type.  Suppose a term $a$ is assigned a type $\alpha$ and
$\alpha$ denotes a superclass $\Sigma_\alpha$.  If $a$ is defined,
i.e., $a$ denotes a class $x$, then $x$ is in $\Sigma_\alpha$.
Formulas are used to make assertions.  They denote truth values.

The proper expressions are categorized according to their first
(leftmost) symbols:

\be

  \item Operator and operator applications
  (\mname{op}, \mname{op-app}).

  \item Variables (\mname{var}).

  \item \bsp Type applications and dependent function types
  (\mname{type-app}, \mname{dep-fun-type}).\esp

  \item Function applications and abstractions (\mname{fun-app},
  \mname{fun-abs}).

  \item Conditional terms (\mname{if}).

  \item Existential quantifications (\mname{exists}). 

  \item Definite and indefinite descriptions (\mname{def-des},
  \mname{indef-des}).

  \item Quotations and evaluations (\mname{quote},
  \mname{eval}).

\ee

\subsection{Dependent Function Types}

A \emph{dependent function type} is a type of the form
\[\gamma = (\mname{dep-fun-type},(\mname{var},x,\alpha),\beta)\] 
where $\alpha$ and $\beta$ are types.  (Dependent function types are
commonly known as \emph{dependent product types}.)  The type $\gamma$
denotes a superclass of possibly partial functions.  A function
abstraction of the form
\[(\mname{fun-abs},(\mname{var},x,\alpha),b),\] where $b$ is a term of
type $\beta$, is of type $\gamma$.

The dependent function type $\gamma$ is a generalization of the more
common function type $\alpha \tarrow \beta$.  If $f$ is a term of type
$\alpha \tarrow \beta$ and $a$ is a term of type $\alpha$, then the
application $f(a)$ is of type $\beta$---which does not depend on the
value of $a$.  In Chiron, however, if $f$ is a term of type $\gamma$
and $a$ is a term of type $\alpha$, then the term
\[(\mname{fun-app},f,a),\] the application of $f$ to $a$, is of the
type \[(\mname{type-app}, \gamma,a),\] the type formed by applying
the type $\gamma$ to $a$---which generally depends on the value of $a$.

\subsection{Undefinedness}

An expression is \emph{undefined} if it has no prescribed meaning or
if it denotes a value that does not exist.  There are several sources
of undefined expressions in Chiron:

\bi

  \item Nondenoting operator, type, and function applications.

  \item Nonexistent function abstractions.

  \item Improper definite and indefinite descriptions.

  \item Out of range variables and evaluations.

\ei
Undefined expressions are handled in Chiron according to the
\emph{traditional approach to undefinedness}~\cite{Farmer04}.  The
value of an undefined term is the undefined value $\Undefined$, but
the value of an undefined type or formula is $\Dc$ (the universal
superclass) or $\FALSE$, respectively.  That is, the values for
undefined types, terms, and formulas are $\Dc$, $\Undefined$, and
$\FALSE$, respectively.  Commonly used in mathematical practice, the
traditional approach to undefinedness enables statements involving
partial functions and definite and indefinite descriptions to be
expressed very concisely~\cite{Farmer04}.

\subsection{Quotation and Evaluation}

A \emph{construction} is a set that represents the syntactic structure
of an expression.  A term of the form $(\mname{quote}, e)$, where $e$
is an expression, denotes the construction that represents $e$.  Thus
a proper expression $e$ has two different meanings:

\be

  \item The \emph{semantic meaning} of $e$ is the value denoted by $e$
  itself.

  \item The \emph{syntactic meaning} of $e$ is the construction
  denoted by $(\mname{quote}, e)$.

\ee

There are two ways to refer to a semantic meaning $v$.  The first is
to directly form a proper expression $e$ not beginning with
\mname{eval} that denotes $v$.  The second is to form a term $a$ that
denotes the construction that represents a proper expression $e$ that
denotes $v$ and then form the type $(\mname{eval}, a, \mname{type})$,
term $(\mname{eval}, a, \alpha)$, or formula $(\mname{eval}, a,
\mname{formula})$ (depending on whether $e$ is a type, a term assigned
the type $\alpha$, or a formula) which denotes $v$.

Likewise there are two ways to refer to a syntactic meaning $c$.  The
first is to directly form a term $a$ not beginning with \mname{quote}
that denotes $c$.  The second is to form an expression $e$ such that
the construction $c$ represents the syntactic structure of $e$ and
then form the expression $(\mname{quote}, e)$ which denotes $c$.

For an expression $e$, the term $(\mname{quote}, e)$ denotes the
syntactic meaning of $e$ and is thus always defined (even when $e$ is
an undefined term or an improper (nondenoting) expression).  However, a
term $(\mname{eval}, a, \alpha)$, where $\alpha$ is a type, may be
undefined.

\newpage

\section{Syntax}\label{sec:syntax}

This section presents the syntax of Chiron which is inspired by the
S-expression syntax of the Lisp family of programming languages.

\subsection{Expressions}\label{subsec:expressions}

Let $\sS$ be a fixed countably infinite set of symbols and $\sK$ be
the set of the 16 symbols given in Table~\ref{tab:keywords}.  Assume
$\sK \subseteq \sS$.  The members of $\sS$ are the \emph{symbols} of
Chiron and the members of $\sK$ are the \emph{key words} of Chiron.
The key words are used to classify expressions and identify different
categories of expressions.

\begin{table}[t]
\bc
\begin{tabular}{|l|l|l|l|}\hline
 
\mname{op} & \mname{type} & \mname{term} & \mname{formula}  \\ \hline

\mname{op-app} & \mname{var} & \mname{type-app} & \mname{dep-fun-type} \\ \hline

\mname{fun-app} & \mname{fun-abs} & \mname{if} & \mname{exists} \\ \hline

\mname{def-des} & \mname{indef-des} & \mname{quote} & \mname{eval} \\ \hline

\end{tabular}
\ec
\caption{The Key Words of Chiron.\label{tab:keywords}}
\end{table}

Let a \emph{signature form} be a sequence $s_1,\ldots,s_{n+1}$ of
symbols where $n \ge 0$ and each $s_i$ is the symbol \mname{type},
\mname{term}, or \mname{formula}.  A \emph{language} of Chiron is a
pair $L = (\sO,\theta)$ where:

\be

  \item $\sO$ is a countable set of symbols such that (1) $\sO$ and
    $\sS$ are disjoint and (2)~$\sO_0 \subseteq \sO$ where
    $\sO_0$ is the set of the 18 symbols given in
    Table~\ref{tab:op-names}.  The members of $\sO$ are called
    \emph{operator names} and the members of $\sO_0$ are the
    \emph{built-in operator names} of Chiron.

  \item $\theta$ maps each $o \in \sO$ to a signature form such that,
    for each $o \in \sO_0$, $\theta(o)$ is the signature form
    assigned to $o$ in Table~\ref{tab:op-names}.

\ee
Throughout this paper let $L = (\sO,\theta)$ be a language of Chiron.

Let $L_i = (\sO_i,\theta_i)$ be a language of Chiron for $i=1,2$.
$L_1$ is a \emph{sublanguage} of $L_2$ (and $L_2$ is an
\emph{extension} of $L_1$), written $L_1 \le L_2$, if $\sO_1 \subseteq
\sO_2$ and $\theta_1$ is a subfunction of $\theta_2$.

\begin{table}[t]
\bc
\begin{tabular}{|rll|}\hline
   & \textbf{Operator Name} & \textbf{Signature Form} \\

1. & \mname{set}             & \mname{type} \\

2. & \mname{class}           & \mname{type} \\

3. & \mname{op-names}        & \mname{term} \\

4. & \mname{lang}            & \mname{type}\\
 
5. & \mname{expr-sym}        & \mname{type} \\

6. & \mname{expr-op-name}    & \mname{term}, \mname{type} \\

7. & \mname{expr}            & \mname{term}, \mname{type} \\

8. & \mname{expr-op}         & \mname{term}, \mname{type} \\

9. & \mname{expr-type}       & \mname{term}, \mname{type} \\

10. & \mname{expr-term}      & \mname{term}, \mname{type} \\

11. & \mname{expr-term-type} & \mname{term}, \mname{term}, \mname{type} \\

12. & \mname{expr-formula}   & \mname{term}, \mname{type} \\

13. & \mname{in}             & \mname{term}, \mname{term}, \mname{formula} \\

14. & \mname{type-equal}     & \mname{type}, \mname{type}, \mname{formula} \\

15. & \mname{term-equal}     & \mname{term}, \mname{term}, \mname{type}, \mname{formula} \\

16. & \mname{formula-equal}  & \mname{formula}, \mname{formula}, \mname{formula} \\

17. & \mname{not}            & \mname{formula}, \mname{formula} \\

18. & \mname{or}             & \mname{formula}, \mname{formula}, \mname{formula} \\ \hline

\end{tabular}
\ec
\caption{The Built-In Operator Names of Chiron.\label{tab:op-names}}
\end{table}

The two formation rules below inductively define the notion of an
\emph{expression} of $L$.  $\cexpr{L}{e}$ asserts that $e$ is an
expression of $L$.

\bi

  \item[]\textbf{Expr-1 (Atomic expression)}\vspace*{-2mm} \[\frac{s
  \in \sS \cup \sO} {\cexpr{L}{s}}\]

  \item[]\textbf{Expr-2 (Compound expression)}\vspace*{-2mm}
  \[\frac{\cexpr{L}{e_1},\ldots,\cexpr{L}{e_n}}
  {\cexpr{L}{(e_1,\ldots,e_n)}}\]

  where $n \ge 0$.

\ei
Hence, an expression is an S-expression (with commas in place of
spaces) that exhibits the structure of a tree whose leaves are symbols
in $\sS \cup \sO$.  Let $\sE_L$ denote the set of expressions of $L$.

A \emph{proper expression} of $L$ is an expression of $L$ defined by
the set of 13 formation rules below.  A proper expression denotes a
class, a truth value, the undefined value, or an operation.  Each
proper expression of $L$ is assigned an expression.
$\cpexpr{L}{e}{e'}$ asserts that $e \in \sE_L$ is a proper expression
of $L$ to which the expression $e' \in \sE_L$ is assigned.  An
\emph{improper expression} of $L$ is an expression of $L$ that is not
a proper expression of $L$.  Improper expressions are nondenoting.

There are four sorts of proper expressions.  An \emph{operator} of $L$
is a proper expression of $L$ to which the expression \mname{op} is
assigned.  A \emph{type} of $L$ is a proper expression of $L$ to which
the expression \mname{type} is assigned.  A \emph{term} of $L$ is a
proper expression of $L$ to which a type of $L$ is assigned.  And
a \emph{formula} of $L$ is a proper expression of $L$ to which the
expression \mname{formula} is assigned.  When $a$ is a term of $L$,
$\alpha$ is a type of $L$, and $\cpexpr{L}{a}{\alpha}$ holds, $a$ is
said to be a \emph{term of type} $\alpha$.  As we mentioned earlier,
operators denote operations, types denote superclasses, terms denote
classes or the undefined value $\Undefined$, and formulas denote the
truth values $\TRUE$ and $\FALSE$.

\bsp $\cop{L}{O}$ means $\cpexpr{L}{O}{\mname{op}}$, $\ctype{L}{\alpha}$
means $\cpexpr{L}{\alpha}{\mname{type}}$, $\cterm{L}{a}$ means
$\cpexpr{L}{a}{\alpha}$ for some type $\alpha$ of $L$, and
$\cform{L}{A}$ means $\cpexpr{L}{A}{\mname{formula}}$.
$\cterma{L}{a}{\alpha}$ means $\cpexpr{L}{a}{\alpha}$ and
$\ctype{L}{\alpha}$, i.e., $a$ is a term of type $\alpha$.  An
expression $k$ is a \emph{kind} of $L$, written $\ckind{L}{k}$, if
$k= \mname{type}$, $\ctype{L}{k}$, or $k=\mname{formula}$.  Thus kinds
are the expressions assigned to types, terms, and formulas.  A proper
expression $e$ of $L$ is said to be an \emph{expression of kind} $k$
if $k= \mname{type}$ and $e$ is a type, $\ctype{L}{k}$ and $e$ is a
term of type $k$, or $k=\mname{formula}$ and $e$ is a formula. \esp

The following formation rules define the 13 proper expression
categories of Chiron:

\bi

  \item[]\textbf{P-Expr-1 (Operator)}\vspace*{-2mm}
  \[\frac{o \in \sO, \ckind{L}{k_1},\ldots,\ckind{L}{k_{n+1}}}
  {\cop{L}{(\mname{op}, o, k_1,\ldots,k_{n+1})}}\] where $n \ge 0$;
  $\theta(o) = s_1,\ldots,s_{n+1}$; and $k_i=s_i=\mname{type}$,
  $\ctype{L}{k_i}$ and $s_i = \mname{term}$, or $k_i=
  s_i=\mname{formula}$ for all $i$ with $1 \le i \le n+1$.

  \item[]\textbf{P-Expr-2 (Operator application)}\vspace*{-2mm}
  \[\frac{\cop{L}{(\mname{op}, o, k_1,\ldots,k_{n+1})},
  \cexpr{L}{e_1},\ldots,\cexpr{L}{e_n}}
  {\cpexpr{L}{(\mname{op-app}, (\mname{op}, o, k_1,\ldots,k_{n+1}), 
  e_1,\ldots,e_n)}{k_{n+1}}}\] 
  \bsp where $n \ge 0$ and ($k_i=\mname{type}$ and $\ctype{L}{e_i}$), 
  ($\ctype{L}{k_i}$ and $\cterm{L}{e_i}$), or
  ($k_i=\mname{formula}$ and $\cform{L}{e_i}$)
  for all $i$ with $1 \le i \le n$. \esp

  \item[]\textbf{P-Expr-3 (Variable)} \vspace*{-2mm} 
  \[\frac{x \in \sS, \ctype{L}{\alpha}}
  {\cterma{L}{(\mname{var}, x, \alpha)}{\alpha}}\]

  \item[]\textbf{P-Expr-4 (Type application)}\vspace*{-2mm}
  \[\frac{\ctype{L}{\alpha}, \cterm{L}{a}} 
  {\ctype{L}{(\mname{type-app},\alpha,a)}}\]

  \item[]\textbf{P-Expr-5 (Dependent function type)}\vspace*{-2mm}
  \[\frac{\cterm{L}{(\mname{var},x,\alpha)},\ctype{L}{\beta}}
  {\ctype{L}{(\mname{dep-fun-type},(\mname{var},x,\alpha),\beta)}}\]

  \item[]\textbf{P-Expr-6 (Function application)}\vspace*{-2mm}
  \[\frac{\cterma{L}{f}{\alpha}, \cterm{L}{a}} 
  {\cterma{L}{(\mname{fun-app},f,a)}{(\mname{type-app},\alpha,a)}}\]

  \item[]\textbf{P-Expr-7 (Function abstraction)}\vspace*{-2mm}
  \[\frac{\cterm{L}{(\mname{var},x,\alpha)},\cterma{L}{b}{\beta}}
  {\cterma{L}{(\mname{fun-abs},(\mname{var},x,\alpha),b)}
  {(\mname{dep-fun-type},(\mname{var},x,\alpha),\beta)}}\]

  \item[]\textbf{P-Expr-8 (Conditional term)}\vspace*{-2mm}
  \[\frac{\cform{L}{A},\cterma{L}{b}{\beta},\cterma{L}{c}{\gamma}}
  {\cterma{L}{(\mname{if},A,b,c)} {\delta}}\]
  where $\delta =\left\{\begin{array}{ll}
                        \beta & \mbox{if }\beta = \gamma\\
                        (\mname{op-app},
                          (\mname{op}, \mname{class},\mname{type})) &
                        \mbox{otherwise}
                        \end{array}
                 \right.$

  \item[]\textbf{P-Expr-9 (Existential quantification)}\vspace*{-2mm}
  \[\frac{\cterm{L}{(\mname{var}, x, \alpha)}, \cform{L}{B}}
  {\cform{L}{(\mname{exists},(\mname{var}, x, \alpha),B)}}\]

  \item[]\textbf{P-Expr-10 (Definite description)}\vspace*{-2mm}
  \[\frac{\cterm{L}{(\mname{var}, x, \alpha)}, \cform{L}{B}}
  {\cterma{L}{(\mname{def-des},(\mname{var}, x, \alpha),B)} {\alpha}}\]

  \item[]\textbf{P-Expr-11 (Indefinite description)}\vspace*{-2mm}
  \[\frac{\cterm{L}{(\mname{var}, x, \alpha)},  \cform{L}{B}}
  {\cterma{L}{(\mname{indef-des},(\mname{var}, x, \alpha),B)} {\alpha}}\]

  \item[]\textbf{P-Expr-12 (Quotation)}\vspace*{-2mm}
  \[\frac{\cexpr{L}{e}}
  {\cterma{L}{(\mname{quote}, e)}{\mname{E}}}\] 
  where $\mname{E} = (\mname{op-app}, (\mname{op}, \mname{expr},
  \mname{L}, \mname{type}),\ell)$ and $\mname{L}$ and $\ell$
  are defined as in Table~\ref{tab:operators}.

  \item[]\textbf{P-Expr-13 (Evaluation)}\vspace*{-2mm}
  \[\frac{\cterm{L}{a},\ckind{L}{k}}
  {\cpexpr{L}{(\mname{eval}, a, k)}{k}}\]

\ei
Note: An expanded definition of a proper expression with 25 proper
expression categories is given in appendix B.

\begin{cprop}
\bsp The formation rules assign a unique expression of $L$ to each
proper expression of $L$.\esp
\end{cprop}

Unless stated otherwise, an \emph{operator name}, \emph{operator},
etc.\ is an \emph{operator name}, \emph{operator}, etc.\ of $L$.  We
will use $s,t,u,v,w,x,y,z,\ldots$ to denote symbols; $o,o',\ldots$ to
denote operator names; $O,O',\ldots$ to denote operators;
$\alpha, \beta, \gamma, \ldots$ to denote types; $a,b,c,\ldots$ to
denote terms; $A,B,C, \ldots$ to denote formulas; and $k,k',\ldots$ to
denote kinds.

The \emph{length} of an expression $e \in \sE_L$, written $|e|$, is
defined recursively by the following statements:

\be

  \item If $s \in \sS \cup \sO$, $|s| = 1$.

  \item If $e_1,\ldots,e_n \in \sE_L$, $|(e_1,\ldots,e_n)| = 1 + |e_1| +
  \cdots |e_n|$.

\ee
Notice that $|(\,)| = 1$ and, in general, $|e|$ equals the number of
symbols and parenthesis pairs occurring in $e$.  The \emph{complexity}
of an expression $e \in \sE_L$, written $c(e)$, is the pair $(m,n) \in
\textbf{N} \times \textbf{N}$ of natural numbers such that: \be

  \item $m$ is the number of occurrences of the symbol \mname{eval} in
    $e$ that are not within a quotation.

  \item $n$ is the length of $e$.

\ee
For $c(e_1) = (m_1,n_1)$ and $c(e_2) = (m_2,n_2)$, $c(e_1) < c(e_2)$
means either $m_1 < m_2$ or ($m_1 = m_2$ and $n_1 < n_2$).  

Let $O = (\mname{op}, o, k_1,\ldots,k_{n+1})$ be an operator.  The
operator name $o$ is called the \emph{name} of $O$, and the list
$k_1,\ldots,k_{n+1}$ of kinds is called the \emph{signature} of $O$.
$O$ is said to be an \emph{$n$-ary} operator because it is applied to
$n$ arguments in an operator application.  $O$ is a \emph{type
  operator}, \emph{term operator}, or \emph{formula operator} if
$k_{n+1}= \mname{type}$, $\ctype{L}{k_{n+1}}$, or $k_{n+1}=
\mname{formula}$, respectively.  A \emph{type constant} is a 0-ary
type operator application of the form
$(\mname{op-app},(\mname{op},o,\mname{type}))$.  A \emph{term
  constant} (or simply \emph{constant}) of type $\alpha$ is a 0-ary
term operator application of the form
$(\mname{op-app},(\mname{op},o,\alpha))$.  A \emph{formula constant}
is a 0-ary formula operator application of the form
$(\mname{op-app},(\mname{op},o,\mname{formula}))$.  Two operators are
\emph{similar} if their names are the same.

Let $o \in \sO$ and $F=\theta(o)$.  If $F$ contains the symbol
\mname{term}, there will be many operators with the name $o$ that have
different signatures, i.e., there will be many operators with the name
$o$ that are similar to each other.  However, each operator name will
usually be assigned a \emph{preferred signature}.  If $F$ does not
contain the symbol \mname{term}, $o$ is assigned $F$ as its preferred
signature.  Each built-in operator name is assigned a preferred
signature.  An operator formed from a built-in operator name and its
preferred signature is called a \emph{built-in operator} of Chiron.
The operators in Table~\ref{tab:operators} are the built-in operators
of Chiron.

\begin{table}[t]
\bc
\begin{tabular}{|rl|}\hline
   & \textbf{Operator} \\
 
1. & $(\mname{op}, \mname{set}, \mname{type})$ \\

2. & $(\mname{op}, \mname{class}, \mname{type})$ \\

3. & $(\mname{op}, \mname{op-names}, \mname{V})$ \\

4. & $(\mname{op}, \mname{lang}, \mname{type})$\\

5. & $(\mname{op}, \mname{expr-sym}, \mname{type})$ \\

6. & $(\mname{op}, \mname{expr-op-name}, \mname{L}, \mname{type})$ \\

7. & $(\mname{op}, \mname{expr}, \mname{L}, \mname{type})$ \\

8. & $(\mname{op}, \mname{expr-op}, \mname{L}, \mname{type})$ \\

9. & $(\mname{op}, \mname{expr-type}, \mname{L}, \mname{type})$ \\

10. & $(\mname{op}, \mname{expr-term}, \mname{L}, \mname{type})$ \\

11. & $(\mname{op}, \mname{expr-term-type}, \mname{L}, \mname{E}_{\rm ty}, 
      \mname{type})$ \\

12. & $(\mname{op}, \mname{expr-formula}, \mname{L}, \mname{type})$ \\

13. & $(\mname{op}, \mname{in}, \mname{V}, \mname{C}, \mname{formula})$ \\

14. & $(\mname{op}, \mname{type-equal}, \mname{type}, \mname{type},
      \mname{formula})$ \\

15. & $(\mname{op}, \mname{term-equal}, \mname{C}, \mname{C},
      \mname{type}, \mname{formula})$ \\

16. & $(\mname{op}, \mname{formula-equal}, \mname{formula}, 
  \mname{formula}, \mname{formula})$ \\

17. & $(\mname{op}, \mname{not},\mname{formula},\mname{formula})$ \\

18. & $(\mname{op}, \mname{or},\mname{formula},\mname{formula},
  \mname{formula})$ \\ 

\multicolumn{2}{|l|}{where:} \\
& $\mname{V} = (\mname{op-app}, (\mname{op}, \mname{set},\mname{type}))$ \\
& $\mname{C} = (\mname{op-app}, (\mname{op}, \mname{class},\mname{type}))$ \\
& $\ell = (\mname{op-app}, (\mname{op}, \mname{opnames},\mname{V}))$ \\
& $\mname{L} = (\mname{op-app}, (\mname{op}, \mname{lang}, \mname{type}))$ \\ 
& $\mname{E}_{\rm ty} = (\mname{op-app}, (\mname{op}, \mname{expr-type}, \mname{L}, \mname{type}),\ell)$ \\
\hline
\end{tabular}
\ec
\caption{The Built-In Operators of Chiron.\label{tab:operators}}
\end{table}

\begin{crem} \em
A language $L=(\sO,\theta)$ can be presented as a set $L'$ of
operators such that, for each symbol $o$, there is at most one
operator in $L'$ whose name is $o$.  $\sO$ is the set of operator
names $o$ such that $o \in \sO$ iff $o$ is the name of some operator
in $L'$, and $\theta$ is the function from $\sO$ to signature forms
such that, for all $(o :: k_1,\ldots,k_{n+1}) \in L'$, $\theta(o)$ is
the signature form corresponding to $k_1,\ldots,k_{n+1}$.  In
addition, for each $(o :: k_1,\ldots,k_{n+1}) \in L'$,
$k_1,\ldots,k_{n+1}$ is the preferred signature of $o$.
\end{crem}

Let $a = (\mname{var}, x, \alpha)$ be a variable.  $x$ is called the
\emph{name} of $a$, and $\alpha$ is called the \emph{type} of $a$.
Two variables are similar if their names are the same.

An expression $e$ is \emph{eval-free} if all occurrences of the symbol
$\mname{eval}$ in $e$ are within a quotation.  Notice that, if $e$ is
eval-free, then the complexity of $e$ is $c(e) = (0,|e|)$.  We will
use $\alpha^{\rm ef}, \beta^{\rm ef}, \gamma^{\rm ef}, \ldots$;
$a^{\rm ef}, b^{\rm ef}, c^{\rm ef},\ldots$; and $A^{\rm ef}, B^{\rm
ef}, C^{\rm ef}, \ldots$ to denote eval-free types, terms, and
formulas, respectively.

A \emph{subexpression} of an expression is defined inductively as
follows:
\be

  \item If $e$ is a proper expression, then $e$ is a subexpression of
  itself.

  \item If $e = (s,e_1,\ldots,e_n)$ is a proper expression such that
    $s$ is not $\mname{quote}$, then $e_i$ is a subexpression of $e$
    for each proper expression $e_i$ with $1 \le i \le n$.

  \item If $e$ is a subexpression of $e'$ and $e'$ is a subexpression
  of $e''$, then $e$ is a subexpression of $e''$.

\ee
$e$ is a \emph{proper subexpression} of $e'$ if $e$ is a subexpression
of $e'$ and $e \not= e'$.  Notice that (1) a subexpression is always a
proper expression, (2) an improper expression has no subexpressions,
(3) a quotation has no proper subexpressions, and (4) if $e_1$ is a
proper subexpression of $e_2$, then $|e_1| < |e_2|$ and $c(e_1) <
c(e_2)$, i.e., the length and the complexity of a proper subexpression
of an expression is strictly less than length and the complexity of
the expression itself.

\subsection{Compact Notation}\label{subsec:compact}

We introduce in this subsection a compact notation for proper
expressions---which we will use in the rest of the paper whenever it
is convenient.  The first group of notational definitions in
Table~\ref{tab:compacta} defines the compact notation for each of the
13 proper expression categories.

\begin{table}[t]
\bc
\begin{tabular}{|ll|}
\hline
\textbf{Compact Notation} & \textbf{Official Notation} \\
$(o::k_1,\ldots,k_{n+1})$ & $(\mname{op},o, k_1,\ldots,k_{n+1})$ \\
$(o::k_1,\ldots,k_{n+1})(e_1,\ldots,e_n)$ &
  $(\mname{op-app}, (\mname{op},o, k_1,\ldots,k_{n+1}), e_1,\ldots,e_n)$ \\
$(x \mcolon \alpha)$ & $(\mname{var}, x, \alpha)$ \\
$\alpha(a)$ & $(\mname{type-app}, \alpha, a)$ \\
$(\LAMBDAapp x \mcolon \alpha \mdot \beta)$ &
  $(\mname{dep-fun-type},(\mname{var},x,\alpha),\beta)$ \\
$f(a)$ & $(\mname{fun-app},f,a)$ \\
$(\LambdaApp x \mcolon \alpha \mdot b)$ &
  $(\mname{fun-abs},(\mname{var},x,\alpha),b)$ \\
$\mname{if}(A,b,c)$ & $(\mname{if},A,b,c)$ \\
$(\ForsomeApp x \mcolon \alpha \mdot B)$ &
  $(\mname{exists},(\mname{var},x,\alpha),B)$ \\
$(\iotaApp x \mcolon \alpha \mdot B)$ &
  $(\mname{def-des},(\mname{var},x,\alpha),B)$ \\
$(\epsilonApp x \mcolon \alpha \mdot B)$ &
  $(\mname{indef-des},(\mname{var},x,\alpha),B)$ \\
$\synbrack{e}$ & $(\mname{quote},e)$ \\
$\sembrack{a}_k$ & $(\mname{eval},a,k)$\\
$\sembrack{a}_{\rm ty}$ & $(\mname{eval},a,\mname{type})$\\
$\sembrack{a}_{\rm te}$ & 
  $(\mname{eval},a,
  (\mname{op-app}, (\mname{op}, \mname{class},\mname{type})))$\\
$\sembrack{a}_{\rm fo}$ & $(\mname{eval},a,\mname{formula})$\\
\hline
\end{tabular}
\ec
\caption{Compact Notation}\label{tab:compacta}
\end{table}

The next group of notational definitions in Table~\ref{tab:compactb}
defines additional compact notation for the built-in operators and the
universal quantifier.

\begin{table}
\bc
\begin{tabular}{|ll|}
\hline
\textbf{Compact Notation} & \textbf{Defining Expression} \\
$\mname{V}$ & $(\mname{set} :: \mname{type})(\,)$ \\
$\mname{C}$ & $(\mname{class} :: \mname{type})(\,)$ \\
$\ell$ & $(\mname{op-names} :: \mname{term})(\,)$ \\
$\mname{L}$ & $(\mname{lang} :: \mname{type})(\,)$ \\
$\mname{E}_{\rm sy}$ & $(\mname{expr-sym} :: \mname{type})(\,)$ \\
$\mname{E}_{{\rm on},a}$ & $(\mname{expr-op-name} :: \mname{L}, 
  \mname{type})(a)$ \\
$\mname{E}_{\rm on}$ & $(\mname{expr-op-name} :: \mname{L}, 
  \mname{type})(\ell)$ \\
$\mname{E}_a$ & $(\mname{expr} :: \mname{L}, \mname{type})(a)$ \\
$\mname{E}$ & $(\mname{expr} :: \mname{L}, \mname{type})(\ell)$ \\
$\mname{E}_{{\rm op},a}$ & $(\mname{expr-op} :: \mname{L}, 
  \mname{type})(a)$ \\
$\mname{E}_{\rm op}$ & $(\mname{expr-op} :: \mname{L},
  \mname{type})(\ell)$ \\
$\mname{E}_{{\rm ty},a}$ & $(\mname{expr-type} :: \mname{L},
  \mname{type})(a)$ \\
$\mname{E}_{\rm ty}$ & $(\mname{expr-type} :: \mname{L},
  \mname{type})(\ell)$ \\
$\mname{E}_{{\rm te},a}$ & $(\mname{expr-term} :: \mname{L},
  \mname{type})(a)$ \\
$\mname{E}_{\rm te}$ & $(\mname{expr-term} :: \mname{L},
  \mname{type})(\ell)$ \\
$\mname{E}_{{\rm te},a}^{b}$ & $(\mname{expr-term-type} :: \mname{L},
  \mname{E}_{\rm ty}, \mname{type})(a,b)$ \\
$\mname{E}_{\rm te}^{b}$ & $(\mname{expr-term-type} :: \mname{L}, 
  \mname{E}_{\rm ty}, \mname{type})(\ell,b)$ \\
$\mname{E}_{{\rm fo},a}$ & $(\mname{expr-formula} :: \mname{L},
  \mname{type})(a)$ \\
$\mname{E}_{\rm fo}$ & $(\mname{expr-formula} :: \mname{L},
  \mname{type})(\ell)$ \\
$(a \in b)$ & $(\mname{in} :: 
  \mname{V},\mname{C},\mname{formula})(a,b)$\\
$(\alpha \TypeEqual \beta)$ & $(\mname{type-equal} :: 
  \mname{type},\mname{type},\mname{formula})(\alpha,\beta)$\\
$(a \Equal_\alpha b)$ & $(\mname{term-equal} :: 
  \mname{C},\mname{C},\mname{type},\mname{formula})(a,b,\alpha)$\\
$(a \Equal b)$ & $(a \Equal_{\sf C} b)$\\
$(A \Iff B)$ & $(\mname{formula-equal} :: 
  \mname{formula},\mname{formula},\mname{formula})(A,B)$\\
$(\neg A)$ & $(\mname{not} :: \mname{formula},\mname{formula})(A)$\\
$(a \not\in b)$ & $(\Neg (a \in b))$ \\
$(a \not\Equal b)$ & $(\Neg (a \Equal b))$ \\
$(A \Or B)$ & $(\mname{or} :: 
  \mname{formula},\mname{formula},\mname{formula})(A,B)$\\
$(\ForallApp x\mcolon\alpha\mdot A)$ &
  $(\Neg(\ForsomeApp x\mcolon\alpha\mdot(\Neg A)))$\\
\hline
\end{tabular}
\ec
\caption{Additional Compact Notation}\label{tab:compactb}
\end{table}

We will often employ the following abbreviation rules when using the
compact notation:
\be

  \item A matching pair of parentheses in an expression may be dropped
  if there is no resulting ambiguity.

  \item A variable $(x\mcolon\alpha)$ occurring in the body $e$ of
  $(\StarApp {x\mcolon\alpha} \mdot e)$, where $\star$ is $\Lambda$,
  $\lambda$, $\Forsome$, $\Forall$, $\iota$, or $\epsilon$ may be
  written as $x$ if there is no resulting ambiguity.

  \item $(\StarApp x_1\mcolon \alpha_1 \ldots (\StarApp x_n\mcolon
  \alpha_n \mdot e) \cdots)$, where $\star$ is $\Lambda$, $\lambda$,
  $\Forsome$, or $\Forall$, may be written as \[(\StarApp
  x_1\mcolon\alpha_1, \ldots, x_n\mcolon \alpha_n \mdot e).\]
  Similarly, $(\StarApp x_1\mcolon \alpha \ldots (\StarApp x_n\mcolon
  \alpha \mdot e) \cdots)$, where $\star$ is $\Lambda$, $\lambda$,
  $\Forsome$, or $\Forall$, may be written as \[(\StarApp
  x_1,\ldots,x_n\mcolon \alpha \mdot e).\]

  \item If we assign a fixed type, say $\alpha$, to a symbol $x$ to be
    used as a variable name, then an expression of the form
    $(\StarApp{x} \mcolon \alpha \mdot e)$, where $\star$ is
    $\Lambda$, $\lambda$, $\Forsome$, $\Forall$, $\iota$, or
    $\epsilon$, may be written as $(\StarApp{x} \mdot e)$.

  \item \bsp If $k_1,\ldots,k_{n+1}$ is the preferred signature
    assigned to an operator name $o$, then an operator application of
    the form \[(o::k_1,\ldots,k_{n+1})(e_1,\ldots,e_n)\] may be
    written as $o(e_1,\ldots,e_n)$ and an operator application of the
    form $(o::k)(\,)$ may be written as $o$. \esp

  \item $\sembrack{a}_{\rm ty}$, $\sembrack{a}_{\rm te}$,
  $\sembrack{a}_\alpha$, and $\sembrack{a}_{\rm fo}$ may be shortened
  to $\sembrack{a}$ if $a$ is of type $\mname{E}_{\rm ty}$,
  $\mname{E}_{\rm te}$, $\mname{E}_{\rm te}^{\synbrack{\alpha}}$, and
  $\mname{E}_{\rm fo}$, respectively.

\ee

Using the compact notation, expressions can be written in Chiron so
that they look very much like expressions written in mathematics
textbooks and papers.

\subsection{Quasiquotation}\label{subsec:quasiquote}

Quasiquotation is a parameterized form of quotation in which the
parameters serve as holes in a quotation that are filled with the
values of expressions.  It is a very powerful syntactic device for
specifying expressions and defining macros.  Quasiquotation was
introduced by W. Quine in 1940 in the first version of his book
\emph{Mathematical Logic}~\cite{Quine03}.  It has been extensively
employed in the Lisp family of programming
languages~\cite{Bawden99}.\footnote{In Lisp, the standard symbol for
quasiquotation is the backquote (\texttt{`}) symbol, and thus in Lisp,
quasiquotation is usually called \emph{backquote}.}

We will introduce quasiquotation into Chiron as a notational
definition. Unlike quotation, quasiquotation will not be part of the
official Chiron syntax.  The meaning of a quasiquotation will be an
expression that denotes a construction.

The three formation rules below inductively define the notion of a
\emph{marked expression} of $L$.  $\cmexpr{L}{m}$ asserts that $m$ is
a marked expression of $L$.

\bi

  \item[]\textbf{M-Expr-1}\vspace*{-2mm} 
  \[\frac{s \in \sS \cup \sO} {\cmexpr{L}{s}}\]

  \item[]\textbf{M-Expr-2}\vspace*{-2mm} 
  \[\frac{\cterm{L}{a}}{\cmexpr{L}{\commabrack{a}}}\]

  \item[]\textbf{M-Expr-3}\vspace*{-2mm}
  \[\frac{\cmexpr{L}{m_1},\ldots,\cmexpr{L}{m_n}}
   {\cmexpr{L}{(m_1,\ldots,m_n)}}\]

  where $n \ge 0$.

\ei
A marked expression of the form $\commabrack{a}$ is called an
\emph{evaluated component}.

\begin{cprop} 
Every expression is a marked expression of $L$.
\end{cprop}

A \emph{quasiquotation} of $L$ is a syntactic entity of the form
\[(\mname{quasiquote},m)\] where $m$ is a marked 
expression of $L$.\footnote{The evaluated components in a
  quasiquotation are sometimes called
  \emph{antiquotations}.}$\phantom{}^,$\footnote{Quasiquotations
  correspond to backquotes in Lisp as follows.  The symbol
  \mname{quasiquote} corresponds to the backquote symbol (\texttt{`})
  and an evaluated component $\commabrack{a}$ in a quasiquotation
  corresponds to $,a$.  Thus the quasiquotation $(\mname{quasiquote},
  (a,b,\commabrack{c}))$ corresponds to the backquote \texttt{`($a$
    $b$ ,$c$)}.}  A compact notation for quasiquotation can be easily
defined as indicated by the following example: Let
$\synbrack{f(\commabrack{a})}$ be the compact notation for the
quasiquotation
\[(\mname{quasiquote},(\mname{fun-app},f,\commabrack{a})).\] Here we
are using $\synbrack{m}$ when $m$ is an expression to mean the
quotation of $m$ and when $m$ is a marked expression containing
evaluated components to mean the quasiquotation of $m$.

We next define the semantics of quasiquotation.  It assumes a
knowledge of the semantics of Chiron given in
section~\ref{sec:semantics} and the \mname{defined-in} and
\mname{ord-pair} operators defined in section~\ref{sec:op-defs}.  Let
$F$ be the function, mapping marked expressions of $L$ to terms of
$L$, recursively defined by:

\be

  \item If $m$ is a symbol $s \in \sS \cup \sO$, then $F(m) =
    (\mname{quote},s)$.

  \item If $m$ is an evaluated component $\commabrack{a}$, then $F(m)
  = a$.

  \item If $m$ is a marked expression $(m_1,\ldots,m_n)$ where $n \ge
  0$, then \[F(m) = [F(m_1),\ldots,F(m_n)].\]

\ee
Note that, as defined in subsection~\ref{subsec:setthop},
$[a_1,\ldots,a_n]$ denotes an $n$-tuple of sets.  For a quasiquotation
$q= (\mname{quasiquote},m)$, define $G(q) = F(m)$.

\begin{cprop} 
\be

  \item[]

  \item If $e$ is an expression, then $G((\mname{quasiquote},e))$ is a
    term of $L$ such that \[\models G((\mname{quasiquote},e)) =
    (\mname{quote},e).\]

  \item If $q$ is a quasiquotation containing an evaluated component
    $\commabrack{(\mname{quote},e)}$ and $q'$ is the quasiquotation
    that results from replacing $\commabrack{(\mname{quote},e)}$ in
    $q$ with $e$, then \[\models G(q) = G(q').\]

  \item If $\set{\commabrack{a_1},\ldots,\commabrack{a_n}}$ is the set
  of evaluated components occurring in a quasiquotation $q$, then
  \[\models (a_1 \IsDef \mname{E} \And \cdots \And a_n \IsDef \mname{E}) 
  \Implies G(q) \IsDef \mname{E}.\]

\ee
\end{cprop}
Note that, as defined in subsection~\ref{subsec:logop}, $a \IsDef
\alpha$ asserts that the value of the term $a$ is in the value of type
$\alpha$.

From now on, we will consider a quasiquotation $q$ to be an alternate
notation for the term $G(q)$.  Many expressions involving syntax can
be expressed very succinctly using quasiquotation.  See
sections~\ref{sec:sub} and \ref{sec:examples} for examples,
especially the liar paradox example in subsection~\ref{subsec:liar.b}.

\section{Semantics}\label{sec:semantics}

This section presents the official semantics of Chiron which is based
on \emph{standard models}.  Two alternate semantics based on other
kinds of models are given in appendix A.  

\subsection{The Liar Paradox}\label{subsec:liar.a}

Using quotation and evaluation, it is possible to express the
\emph{liar paradox} in Chiron.  That is, it is possible to construct 
a term \mname{LIAR} whose value equals the value of
\[\synbrack{\Neg\sembrack{\mname{LIAR}}_{\rm fo}}.\]  
(See subsection~\ref{subsec:liar.b} for details.)  \mname{LIAR} denotes a
construction representing a formula that says in effect ``I am a
formula that is false''.

If the naive semantics is employed for quotation and evaluation, a
contradiction is immediately obtained:
\begin{eqnarray*}
\sembrack{\mname{LIAR}}_{\rm fo} & = &
\sembrack{\synbrack{\Neg\sembrack{\mname{LIAR}}_{\rm fo}}}_{\rm fo}\\
& = & \Neg\sembrack{\mname{LIAR}}_{\rm fo}
\end{eqnarray*}
$\sembrack{\mname{LIAR}}_{\rm fo}$ is thus \emph{ungrounded} in the
sense that its value depends on itself.  This simple argument is
essentially the proof of A. Tarski's 1933 theorem on the
undefinability of truth~\cite[Theorem I]{Tarski33,Tarski35,Tarski35a}.
The theorem says that $\sembrack{x}_{\rm fo}$ cannot serve as a truth
predicate over all formulas.

Any reasonable semantics for Chiron needs a way to block the liar
paradox and similar ungrounded expressions.  We will briefly describe
three approaches.

The first approach is to remove evaluation (\mname{eval}) from Chiron
but keep quotation (\mname{quote}).  This would eliminate ungrounded
expressions.  However, it would also severely limit Chiron's facility
for reasoning about the syntax of expressions.  It would be possible
using quotation to construct terms that denote constructions, but
without evaluation it would not be possible to employ these terms as
the expressions represented by the constructions.  Some of the power
of evaluation could be replaced by introducing functions that map
types of constructions to types of values.  An example would be a
function that maps numerals to natural numbers.

A second approach is to define a semantics with ``value gaps'' so that
expressions like $\sembrack{\mname{LIAR}}_{\rm fo}$ are not assigned
any value at all.  In his famous paper \emph{Outline of a Theory of
Truth}~\cite{Kripke75}, S.~Kripke presents a framework for defining
semantics with \emph{truth-value gaps} using various evaluation
schemes.  Kripke's approach can be easily generalized to allow
\emph{value gaps} for types and terms as well as for formulas.  In 
appendix A we define two value-gap semantics for Chiron using Kripke's
framework with valuation schemes based on the weak Kleene
logic~\cite{Kleene52} and B.~van~Fraassen's notion of a
\emph{supervaluation}~\cite{vanFraassen66}.  Kripke-style value-gap
semantics are very interesting, if not enlightening, but they are
exceeding difficult to work with in practice.  The main problem is
that there is no mechanism to assert that an expression has no value
(see appendix A for details).

The third approach is to consider any evaluation of a term that
denotes a construction representing a non-eval-free expression to be
undefined.  For example, the value of $\sembrack{\mname{LIAR}}_{\rm
  fo}$ would be $\FALSE$, the undefined value for formulas.  The
origin of this idea is found in Tarski's famous paper on the concept
of truth~\cite[Theorem III]{Tarski33,Tarski35,Tarski35a}.  It is
important to understand that an evaluation of a term $a$ containing
the symbol \mname{eval} can be defined as long as $a$ denotes a
construction representing an expression that does not contain
\mname{eval} outside of a quotation.  Thus the value of an evaluation
like $\sembrack{\sembrack{\synbrack{\synbrack{17}}}_{\rm te}}_{\rm
  te}$ would be the value of 17 because the expression 17 presumably
does not contain \mname{eval} at all.  The official semantics for
Chiron defined in this section employs this third approach to blocking
the liar paradox.

\subsection{Prestructures}

A \emph{prestructure} of Chiron is a pair $(D,\in)$, where $D$ is a
nonempty domain and $\in$ is a membership relation on $D$, that
satisfies the axioms of {\nbg} set theory as given, for example, in
\cite{Goedel40} or \cite{Mendelson97}.

Let $P=(\Dc^{P},\in^P)$ be a prestructure of Chiron.  A \emph{class}
of $P$ is a member of $\Dc^{P}$.  A \emph{set} of $P$ is a member $x$
of $\Dc^{P}$ such that $x \in^P y$ for some member $y$ of $\Dc^{P}$.
That is, a set is a class that is itself a member of a class.  A class
is \emph{proper} if it is not a set.  A \emph{superclass} of $P$ is a
collection of classes in $\Dc^{P}$.  We consider a class, as a
collection of sets, to be a superclass.  Let $\Dv^{P}$ be the domain of
sets of $P$ and $\Ds^{P}$ be the domain of superclasses of $P$.  The
following inclusions hold: $\Dv^{P} \subset \Dc^{P} \subset \Ds^{P}$.

A \emph{function} of $P$ is a member $f$ of $\Dc^{P}$ such that:
\be

  \item Each $p \in^P f$ is an ordered pair $\seq{x,y}$ where $x,y$
  are in $\Dv^{P}$.

  \item For all $\seq{x_1,y_1},\seq{x_2,y_2} \in^P f$, if $x_1 = x_2$,
  then $y_1 = y_2$.

\ee
Notice that a function of $P$ may be partial, i.e., there may not be
an ordered pair $\seq{x,y}$ in a function for each $x$ in $\Dv^{P}$.

Let $\Df^{P} \subset \Dc^{P}$ be the domain of functions of $P$.  For
$f,x$ in $\Dc^{P}$, $f(x)$ denotes the unique $y$ in $\Dv^{P}$ such
that $f$ is in $\Df^{P}$ and $\seq{x,y} \in^P f$.  ($f(x)$ is
undefined if there is no such unique $y$ in $\Dv^{P}$.)  For $\Sigma$
in $\Ds^{P}$ and $x$ in $\Dc^{P}$, $\Sigma[x]$ denotes the unique
class in $\Dc^{P}$ composed of all $y$ in $\Dv^{P}$ such that, for some
$f$ in both $\Sigma$ and $\Df^{P}$, $f(x) = y$.  ($\Sigma[x]$ is
undefined if there is no such unique class in $\Dc^{P}$.)

Let ${\TRUE}^{P}$, ${\FALSE}^{P}$, and $\Undefined^{P}$ be distinct
values not in $\Ds^{P}$.  ${\TRUE}^{P}$ and ${\FALSE}^{P}$ represent the
truth values \emph{true} and \emph{false}, respectively.
$\Undefined^{P}$ is the \emph{undefined value} that represents values
that are undefined.

\bsp For $n \ge 0$, an \emph{$n$-ary operation} of $P$ is a total mapping
from \mbox{$D_1 \times \cdots \times D_n$} to $D_{n+1}$ where $D_i$ is
$\Ds^{P}$, $\Dc^{P} \cup \set{\Undefined^{P}}$, or
$\set{{\TRUE}^{P},{\FALSE}^{P}}$ for each $i$ with \mbox{$1 \le i \le
n+1$}.  Let $\Do^{P}$ be the domain of operations of $P$.  We assume
that $\Ds^{P} \cup \set{{\TRUE}^{P},{\FALSE}^{P},\Undefined^{P}}$ and
$\Do^{P}$ are disjoint.\esp

Choose $\Dea^{P}$ to be any countably infinite subset of $\Dv^{P}$
whose members are neither the empty set $\emptyset$ nor ordered pairs.
Then let $\Deb^{P}$ be the subset of $\Dv^{P}$ defined inductively as
follows:

\be

  \item The empty set $\emptyset$ is a member of $\Deb^{P}$.

  \item If $x \in \Dea^{P} \cup \Deb^{P}$ and $y \in \Deb^{P}$,
  then the ordered pair $\seq{x,y}$ of $x$ and $y$ is a member of
  $\Deb^{P}$.

\ee 
Finally, let $\De^{P} = \Dea^{P} \cup \Deb^{P}$.  The members of
$\De^{P}$ are called \emph{constructions}.  They are sets with the
form of trees whose leaves are members of $\Dea^{P}$.  Their purpose
is to represent the syntactic structure of expressions.  The symbols
in expressions are represented by members of $\Dea^{P}$.  A
construction is a \emph{symbol construction}, an \emph{operator name
  construction}, \emph{operator construction}, \emph{type
  construction}, \emph{term construction}, or \emph{formula
  construction} if it represents a symbol, operator name, operator,
type, term, or formula, respectively.

Choose $G^{P}$ to be any bijective mapping from $\sS \cup \sO$ onto
$\Dea^{P}$.  Let $H^{P}$ be the bijective mapping of $\sE_L$ onto
$\De^{P}$ defined recursively by:

\be

  \item If $e \in \sS \cup \sO$, then $H^{P}(e) = G^{P}(e)$.

  \item If $e = (\,) \in \sE_L$, then $H^{P}(e) = \emptyset$.

  \item If $e = (e_1,\ldots,e_n) \in \sE_L$ where $n \ge 1$, then
  \[H^{P}(e) = \seq{H^{P}(e_1),H^{P}((e_2,\ldots,e_n))}.\]

\ee
$H^{P}(e)$ is called the \emph{construction} of the expression $e$.

\subsection{Structures}

A \emph{structure} for $L$ is a tuple
\[S = 
(\Dv,\Dc,\Ds,\Df,\Do,\De,\in,\TRUE,\FALSE,\Undefined,\xi,H,I)\] where:

\be

  \item $P=(D_{c},\in)$ is a prestructure of Chiron. $\Dv = \Dv^{P}$,
    $\Ds = \Ds^{P}$, $\Df = \Df^{P}$, $\Do = \Do^{P}$, ${\TRUE} =
    {\TRUE}^{P}$, ${\FALSE} = {\FALSE}^{P}$, and $\Undefined =
    \Undefined^{P}$.

  \item $\De = \De^{P}$ where $\Dea^{P}$ is some chosen countably
    infinite subset of $\Dv^{P}$ whose members are neither the empty
    set $\emptyset$ nor ordered pairs.

  \item $\xi$ is a choice function on $\Ds$.  Hence, for all nonempty
  superclasses $\Sigma$ in $\Ds$, $\xi(\Sigma)$ is a class in
  $\Sigma$.

  \item $H= H^{P}$ where $G^{P}$ is some chosen bijective mapping from
    $\sS \cup \sO$ onto $\Dea^{P}$.

  \item For each operator name $o \in \sO$ with $\theta(o) =
    s_1,\ldots,s_{n+1}$, $I(o)$ is an $n$-ary operation in $\Do$ from
    $D_1 \times \cdots \times D_n$ into $D_{n+1}$ where $D_i = \Ds$ if
    $s_i =\mname{type}$, $D_i = \Dc \cup \set{\Undefined}$ if $s_i
    =\mname{term}$, and $D_i = \set{\TRUE,\FALSE}$ if
    $s_i=\mname{formula}$ for each $i$ with $1 \le i \le n +1$ such
    that:

  \be

    \item $I(\mname{set})(\,) = \Dv$, the universal class that
      contains all sets.

    \item $I(\mname{class})(\,) = \Dc$, the universal superclass that
      contains all classes.

    \item $I(\mname{op-names})(\,) = \Don$, the subset of $\De$ whose
      members represent the operator names in $\sO$.

    \item $I(\mname{lang})(\,)$ is the set of all subsets of $\Don$.

    \item $I(\mname{expr-sym})(\,)= \Dsy$, the subset of $\De$ whose
      members represent the symbols in $\sS$.

    \item If $x$ is a member of $\Dc \cup \set{\Undefined}$,
      \[I(\mname{expr-op-name})(x)\] is $x$ if $x \subseteq \Don$,
      $\Dc$ if $x = \Undefined$, and not $\Dc$ otherwise.

    \item If $x$ is a member of $\Dc \cup \set{\Undefined}$,
      \[I(\mname{expr})(x)\] is the subset of $\De$ whose members are
      constructed from only members of $\Dsy \cup x$ if $x \subseteq
      \Don$, $\Dc$ if $x = \Undefined$, and not $\Dc$ otherwise.

    \item Let $\Dop$ be the subset of $\De$ whose members represent
      operators of $L$.  If $x$ is a member of $\Dc \cup
      \set{\Undefined}$,
      \[I(\mname{expr-op})(x)\] is the subset of $\Dop$ whose members
      are constructed from only members of $\Dsy \cup x$ if $x
      \subseteq \Don$, $\Dc$ if $x = \Undefined$, and not $\Dc$ otherwise.

    \item Let $\Dty$ be the subset of $\De$ whose members represent
      types of $L$.  If $x$ is a member of $\Dc \cup \set{\Undefined}$,
      \[I(\mname{expr-type})(x)\] is the subset of $\Dty$ whose members
      are constructed from only members of $\Dsy \cup x$ if $x
      \subseteq \Don$, $\Dc$ if $x = \Undefined$, and not $\Dc$
      otherwise.

    \item Let $\Dte$ be the subset of $\De$ whose members represent
      terms of $L$.  If $x$ is a member of $\Dc \cup
      \set{\Undefined}$,
      \[I(\mname{expr-term})(x)\] is the subset of $\Dte$ whose members
      are constructed from only members of $\Dsy \cup x$ if $x
      \subseteq \Don$, $\Dc$ if $x = \Undefined$, and not $\Dc$ otherwise.

    \item If $x,y$ is a member of $\Dc \cup \set{\Undefined}$, then
      \[I(\mname{expr-term-type})(x,y)\] is the subset of $\Dte$ 
      whose members represent terms of the type represented by $y$ and
      are constructed from only members of $\Dsy \cup x$ if $x
      \subseteq \Don$ and $y \in \Dty$, $\Dc$ if $x = \Undefined$ or
      $y = \Undefined$, and not $\Dc$ otherwise.

    \item Let $\Dfo$ be the subset of $\De$ whose members represent
      formulas of $L$.  If $x$ is a member of $\Dc \cup
      \set{\Undefined}$,
      \[I(\mname{expr-formula})(x)\] is the subset of $\Dfo$ whose
      members represent formulas and are constructed from only members
      of $\Dsy \cup x$ if $x \subseteq \Don$, $\Dc$ if $x =
      \Undefined$, and not $\Dc$ otherwise.

    \item \bsp If $x$ and $y$ are members of $\Dc \cup
      \set{\Undefined}$, then \[I(\mname{in})(x,y)\] is $\TRUE$ if $x$
      is a member of $y$ (and hence $x$ is a member of $\Dv$) and is
      $\FALSE$ otherwise. \esp

    \item If $\Sigma,\Sigma'$ are members of $\Ds$, then
    \[I(\mname{type-equal})(\Sigma,\Sigma')\] is $\TRUE$ if $\Sigma$ and
    $\Sigma'$ are identical and is $\FALSE$ otherwise.

    \item \bsp If $x,y$ are members of $\Dc \cup \set{\Undefined}$ and
      $\Sigma$ is a member of $\Ds$,
      then \[I(\mname{term-equal})(x,y,\Sigma)\] is $\TRUE$ if $x,y$
      are identical members of $\Sigma$ and is $\FALSE$
      otherwise. \esp

    \item \bsp If $t,t'$ are members of $\set{\TRUE,\FALSE}$, then
    \[I(\mname{formula-equal})(t,t')\] is $\TRUE$ if $t$ and
    $t'$ are identical and is $\FALSE$ otherwise. \esp

    \item If $t$ is a member of $\set{\TRUE,\FALSE}$, then
    \[I(\mname{not})(t)\] is $\TRUE$ if $t$ is $\FALSE$ and is
    $\FALSE$ otherwise.

    \item If $t,t'$ are members of $\set{\TRUE,\FALSE}$, then
    \[I(\mname{or})(t,t')\] is $\TRUE$ if at least one of $t$ and
    $t'$ is $\TRUE$ and is $\FALSE$ otherwise.

  \ee

\ee 

Fix a structure 
\[S = 
(\Dv,\Dc,\Ds,\Df,\Do,\De,\in,\TRUE,\FALSE,\Undefined,\xi,H,I)\] for
$L$.  An \emph{assignment} into $S$ is a mapping that assigns a class
in $\Dc$ to each symbol in $\sS$.  Given an assignment $\phi$ into
$S$, a symbol $s \in \sS$, and a class $d \in \Dc$, let $\phi[s
  \mapsto d]$ be the assignment $\phi'$ into $S$ such that $\phi'(s) =
d$ and $\phi'(t) = \phi(t)$ for all symbols $t \in \sS$ different from
$s$.  Let $\mname{assign}(S)$ be the collection of assignments into
$S$.

\subsection{Valuations}\label{subsec:val}

A \emph{valuation} for $S$ is a possibly partial mapping $V$ from
$\sE_L \times \mname{assign}(S)$ into $\Do \cup \Ds \cup \set{\TRUE,
  \FALSE, \Undefined}$ such that, for all $e \in \sE_L$ and $\phi \in
\mname{assign}(S)$, if $V_{\phi}(e)$ is defined, then $V_{\phi}(e) \in
\Do$ if $e$ is an operator, $V_{\phi}(e) \in \Ds$ if $e$ is a type,
$V_{\phi}(e) \in \Dc \cup \set{\Undefined}$ if $e$ is a term, and
$V_{\phi}(e) \in \set{\TRUE,\FALSE}$ if $e$ is a formula.  (We write
$V_{\phi}(e)$ instead of $V(e,\phi)$.)

Let the \emph{standard valuation for $S$} be the valuation $V$ for $S$
defined recursively by the following statements:

\be

  \item Let $e$ be improper.  Then $V_{\phi}(e)$ is undefined.

  \item Let $O = (\mname{op}, o, k_1,\ldots,k_n,k_{n+1})$ be proper.
    Then $I(o)$ is an $n$-ary operation in $\Do$ from $D_1 \times
    \cdots \times D_n$ into $D_{n+1}$.  $V_{\phi}(O)$ is the $n$-ary
    operation in $\Do$ from $D_1 \times \cdots \times D_n$ into
    $D_{n+1}$ defined as follows.  Let $(d_1,\ldots,d_n) \in D_1
    \times \cdots \times D_n$ and $d = I(o)(d_1,\ldots,d_n)$.  If
    $d_i$ is in $V_{\phi}(k_i)$ or $d_i = \Undefined$ for all $i$ such
    that $1 \le i \le n$ and $\ctype{L}{k_i}$ and $d$ is in
    $V_{\phi}(k_{n+1})$ or $d = \Undefined$ when $\ctype{L}{k_{n+1}}$,
    then $V_{\phi}(O)(d_1,\ldots,d_n) = d$.  Otherwise,
    $V_{\phi}(O)(d_1,\ldots,d_n)$ is $\Dc$ if $k_{n+1} =
    \mname{type}$, $\Undefined$ if $\ctype{L}{k_{n+1}}$, and $\FALSE$
    if $k_{n+1} = \mname{formula}$.

  \item Let $e = (\mname{op-app}, O, e_1,\ldots,e_n)$ be proper.
    Then \[V_{\phi}(e) = V_{\phi}(O)(V_{\phi}(e_1),\ldots,V_{\phi}(e_n)).\]

\iffalse

  \item Let $O = (\mname{op}, o, k_1,\ldots,k_n,k_{n+1})$ be proper.
    Then $V_{\phi}(O) = I(o)$.

  \item \bsp Let $e = (\mname{op-app}, O, e_1,\ldots,e_n)$ be proper
    where $O = (\mname{op}, o, k_1,\ldots,k_{n+1})$.  Let \[d =
    V_{\phi}(O)(V_{\phi}(e_1),\ldots,V_{\phi}(e_n)).\] If
    $V_{\phi}(e_i)$ is in $V_{\phi}(k_i)$ or $V_{\phi}(e_i) =
    \Undefined$ for all $i$ such that $1 \le i \le n$ and
    $\ctype{L}{k_i}$ and $d$ is in $V_{\phi}(k_{n+1})$ or $d =
    \Undefined$ when $\ctype{L}{k_{n+1}}$, then $V_{\phi}(e) = d.$
    Otherwise $V_{\phi}(e)$ is $\Dc$ if $k_{n+1} = \mname{type}$,
    $\Undefined$ if $\ctype{L}{k_{n+1}}$, and $\FALSE$ if $k_{n+1} =
    \mname{formula}$. \esp

\fi

  \item Let $a = (\mname{var}, x, \alpha)$ be proper.  If
  $\phi(x)$ is in $V_{\phi}(\alpha)$, then $V_{\phi}(a) = \phi(x)$.
  Otherwise $V_{\phi}(a) = \Undefined$.

  \item Let $\beta = (\mname{type-app},\alpha,a)$ be proper.  If
    $V_{\phi}(a)\not= \Undefined$ and $ V_{\phi}(\alpha)[V_{\phi}(a)]$
    is defined, then $V_{\phi}(\beta) =
    V_{\phi}(\alpha)[V_{\phi}(a)]$.  Otherwise $V_{\phi}(\beta) =
    \Dc$.

  \item Let $\gamma = (\mname{dep-fun-type}, (\mname{var},x,\alpha),
    \beta)$ be proper.  Then $V_{\phi}(\gamma)$ is the superclass of
    all $g$ in $\Df$ such that, for all $d$ in $\Dv$, if $g(d)$ is
    defined, then $d$ is in $V_{\phi}(\alpha)$ and $g(d)$ is in
    $V_{\phi[x \mapsto d]}(\beta)$.

  \item \bsp Let $b = (\mname{fun-app},f,a)$ be proper.  If
  $V_{\phi}(f) \not= \Undefined$, $V_{\phi}(a) \not= \Undefined$, and
  $V_{\phi}(f)(V_{\phi}(a))$ is defined, then $V_{\phi}(b) =
  V_{\phi}(f)(V_{\phi}(a))$.  Otherwise $V_{\phi}(b) =
  \Undefined$. \esp

  \item Let $f = (\mname{fun-abs},(\mname{var},x,\alpha),b)$ be proper
    in $L$.  If \[g = \{\seq{d,d'} \; | \; d \mbox{ is a set in }
    V_{\phi}(\alpha) \mbox{ and } d' = V_{\phi[x \mapsto d]}(b) \mbox{
    is a set}\}\] is in $\Df$, then $V_{\phi}(f) = g$.  Otherwise
    $V_{\phi}(f) = \Undefined$.

  \item Let $a = (\mname{if},A,b,c)$ be proper.  If
  $V_{\phi}(A) = \TRUE$, then $V_{\phi}(a) = V_{\phi}(b)$.  Otherwise
  $V_{\phi}(a) = V_{\phi}(c)$.

  \item Let $A = (\mname{exists},(\mname{var}, x, \alpha),B)$ be
  proper.  If there is some $d$ in $V_{\phi}(\alpha)$ such that
  $V_{\phi[x \mapsto d]}(B) = \TRUE$, then $V_{\phi}(A) = \TRUE$.
  Otherwise, $V_{\phi}(A) = \FALSE$.

  \item Let $a = (\mname{def-des},(\mname{var}, x, \alpha),B)$ be
  proper.  If there is a unique $d$ in $V_{\phi}(\alpha)$ such
  that $V_{\phi[x \mapsto d]}(B) = \TRUE$, then $V_{\phi}(a) = d$.
  Otherwise, $V_{\phi}(a) = \Undefined$.

  \item Let $a = (\mname{indef-des},(\mname{var}, x, \alpha),B)$ be
  proper.  If there is some $d$ in $V_{\phi}(\alpha)$ such that
  $V_{\phi[x \mapsto d]}(B) = \TRUE$, then $V_{\phi}(a) = \xi(\Sigma)$
  where $\Sigma$ is the superclass of all $d$ in $V_{\phi}(\alpha)$
  such that $V_{\phi[x \mapsto d]}(B) = \TRUE$.  Otherwise,
  $V_{\phi}(a) = \Undefined$.

  \item Let $a = (\mname{quote},e)$ be proper.  $V_{\phi}(a) =
  H(e)$.

  \item \bsp Let $e = (\mname{eval},a,k)$ be proper.  If (1)
  $V_{\phi}(a)$ is in $\Dty$ and $k = \mname{type}$, $V_{\phi}(a)$ is
  in $\Dte$ and $\ctype{L}{k}$, or $V_{\phi}(a)$ is in $\Dfo$ and $k =
  \mname{formula}$; (2) $H^{-1}(V_{\phi}(a))$ is eval-free; and (3)
  $V_{\phi} (H^{-1}(V_{\phi}(a)))$ is in $V_{\phi}(k)$ if $\ctype{L}{k}$,
  then $V_{\phi}(e) = V_{\phi}(H^{-1}(V_{\phi}(a)))$.  Otherwise
  $V_{\phi}(e)$ is $\Dc$ if $k = \mname{type}$, $\Undefined$ if
  $\ctype{L}{k}$, and $\FALSE$ if $k = \mname{formula}$. \esp
 
\ee

\begin{clem}\label{lem:v-well-defined}
The standard valuation for $S$ is well defined.
\end{clem}

\begin{proof}
We will show that, for all $e \in \sE_L$ and $\phi \in
\mname{assign}(S)$, $V_{\phi}(e)$ is well defined.  Our proof is by
induction on the complexity of $e$.  There are 14 cases corresponding
to the 14 clauses of the definition of a standard valuation.

\bi

  \item[] Case 1: $e$ is improper.  $V_{\phi}(e)$ is always
  undefined in this case.

  \item[] Cases 2--12: $e$ is a proper expression in $L$ that is not a
    quotation or evaluation.  For each case, $V_{\phi}(e)$ depends on
    well-defined components of $S$ and a collection of values
    $V_{\phi'}(e')$ where $e'$ ranges over a set of subexpressions of
    $e$ and $\phi'$ is $\phi$ or ranges over an infinite subset of
    $\mname{assign}(S)$.  Each such $V_{\phi'}(e')$ is well defined by
    the induction hypothesis because $e'$ is a subexpression of $e$,
    and hence, $c(e') < c(e)$.  Therefore, $V_{\phi}(e)$ is well
    defined.

  \item[] Case 13: $e = (\mname{quote},e')$ is proper.
  $V_{\phi}(e)$ = $H(e')$ is well defined since $H$ is a well-defined
  component of $S$.

  \item[] Case 14: $e = (\mname{eval},a,k)$ is proper.  $V_{\phi}(e)$
    depends on one, two, or three of the values of $V_{\phi}(a)$,
    $V_{\phi}(k)$, and $V_{\phi}(H^{-1}(V_{\phi}(a)))$.  $V_{\phi}(a)$
    and $V_{\phi}(k)$ when $\ctype{L}{k}$ holds are well defined by
    the induction hypothesis because $a$ and $k$ are subexpressions of
    $e$, and hence, $c(a) < c(e)$ and $c(k) < c(e)$.
    $V_{\phi}(H^{-1}(V_{\phi}(a)))$ when $H^{-1}(V_{\phi}(a)) \in
    \sE_L$ is well defined by the induction hypothesis because
    $H^{-1}(V_{\phi}(a))$ is eval-free and $e$ is not, and hence,
    $c(H^{-1}(V_{\phi}(a))) < c(e)$.  Therefore, $V_{\phi}(e)$ is well
    defined.

\ei
\end{proof}

\begin{cthm}
Let $V$ be the standard valuation for $S$.  Then the following
statements hold for all $e \in \sE_L$ and $\phi \in
\mname{assign}(S)$:

\be

  \item $V_{\phi}(e)$ is defined iff $e$ is proper.

  \item If $e$ is an $n$-ary operator of $L$, then
  $V_{\phi}(e)$ is an $n$-ary operation in $\Do$.

  \item If $e$ is a type of $L$, then $V_{\phi}(e)$ is in $\Ds$.

  \item If $e$ is a term of $L$ of type $\alpha$, then $V_{\phi}(e)$
  is in $V_{\phi}(\alpha) \cup \set{\Undefined}$.

  \item If $e$ is a formula of $L$, then $V_{\phi}(e)$ is in
  $\set{\TRUE,\FALSE}$.

\ee

\end{cthm}

\begin{proof}
By Lemma~\ref{lem:v-well-defined}, the standard valuation for $S$ is
well defined.  Parts 1, 2, 3, and 5 of the theorem follow immediately
from the definitions of a structure and the standard valuation for a
structure.

Our proof of part 4 is by induction on the length of the term $e$.

\bi

  \item[] Case 1: $e = (\mname{op-app}, O, e_1,\ldots,e_n)$ where $O =
    (\mname{op}, o, k_1,\ldots,k_n,k_{n+1})$.  Then $e$ is of type
    $k_{n+1}$.  By the definition of $V$, $V_{\phi}(e)$ is in
    $V_{\phi}(k_{n+1}) \cup \set{\Undefined}$.

  \item[] Case 2: $e = (\mname{var}, x, \alpha)$.  Then $e$ is of type
    $\alpha$.  By the definition of $V$, $V_{\phi}(e)$ is in
    $V_{\phi}(\alpha) \cup \set{\Undefined}$.  

  \item[] Case 3: $e = (\mname{fun-app},f,a)$ where $f$ is of type
    $\alpha$.  Then $e$ is of type $\alpha(a)$.  By the induction
    hypothesis, $V_{\phi}(f)$ is in $V_{\phi}(\alpha) \cup
    \set{\Undefined}$.  By the definition of $V$, if $V_{\phi}(e) =
    V_{\phi}(f)(V_{\phi}(a)) \not= \Undefined$, then $V_{\phi}(f)
    \not= \Undefined$ and $V_{\phi}(a) \not= \Undefined$, and so
    $V_{\phi}(e)$ is in $V_{\phi}(\alpha)[V_{\phi}(\alpha)] =
    V_{\phi}(\alpha(a))$.  Therefore, $V_{\phi}(e)$ is in
    $V_{\phi}(\alpha(a)) \cup \set{\Undefined}$.

  \item[] Case 4: $e = (\mname{fun-abs},(\mname{var},x,\alpha),b)$
    where $b$ is of type $\beta$.  Then $e$ is of type $\gamma =
    (\LAMBDAapp x \mcolon \alpha \mdot \beta)$.  By the induction
    hypothesis, $V_{\phi'}(b)$ is in $V_{\phi'}(\beta) \cup
    \set{\Undefined}$ for all $\phi' \in \mname{assign}(S)$. Suppose
    $g = V_{\phi}(e) \not= \Undefined$.  For all sets $d$ in
    $V_{\phi}(\alpha)$, if $g(d)$ is defined, $g(d)= V_{\phi[x \mapsto
        d]}(b)$ is a set in $V_{\phi[x \mapsto d]}(\beta)$.  For all
    sets $d$ not in $V_{\phi}(\alpha)$, $g(d)$ is undefined.
    Therefore, by the definition of $V$, $V_{\phi}(e)$ is in
    $V_{\phi}(\gamma) \cup \set{\Undefined}$.

  \item[] Case 5: $e = (\mname{if},A,b,c)$ where $b$ is of type
    $\beta$ and $c$ is of type $\gamma$.  Then $e$ is of type 
    \[\delta = \left\{\begin{array}{ll}
                     \beta & \mbox{if }\beta = \gamma\\
                     \mname{C} & \mbox{otherwise}
                     \end{array}
              \right.\] 
    By the induction hypothesis, $V_{\phi}(b)$ is in $V_{\phi}(\beta)
    \cup \set{\Undefined}$ and $V_{\phi}(c)$ is in $V_{\phi}(\gamma)
    \cup \set{\Undefined}$. $V_{\phi}(\beta) \subseteq
    V_{\phi}(\mname{C})$ and $V_{\phi}(\gamma) \subseteq
    V_{\phi}(\mname{C})$. Therefore, by the definition of $V$,
    $V_{\phi}(e)$ is in $V_{\phi}(\delta) \cup \set{\Undefined}$.

  \item[] Case 6: $e = (\mname{def-des},(\mname{var}, x, \alpha),B)$.
    Then $e$ is of type $\alpha$.  By the definition of $V$,
    $V_{\phi}(e)$ is in $V_{\phi}(\alpha) \cup \set{\Undefined}$.

  \item[] Case 7: $e = (\mname{indef-des},(\mname{var}, x,
    \alpha),B)$.  Then $e$ is of type $\alpha$.  By the definition of
    $V$, $V_{\phi}(e)$ is in $V_{\phi}(\alpha) \cup \set{\Undefined}$.

  \item[] Case 8: $e = (\mname{quote},e)$.  Then $e$ is of type
    $\mname{E}$.  By the definition of $H$, $I$, and $V$,
    $V_{\phi}(e)$ is in $\De = V_{\phi}(\mname{E})$.

  \item[] Case 9: $e = (\mname{eval},a,k)$. Then $e$ is of type
    $k$. By the definition of $V$, $V_{\phi}(e)$ is in
    $V_{\phi}(k) \cup \set{\Undefined}$.

\ei
\end{proof}

\subsection{Models}

A \emph{model for $L$} is a pair $M=(S,V)$ where $S$ is a structure
for $L$ and $V$ is a valuation for $S$.  Let $M=(S,V)$ be a model for
$L$.  An expression $e$ is \emph{denoting} [\emph{nondenoting}] in $M$
with respect to an assignment $\phi \in \mname{assign}(S)$ if
$V_{\phi}(e)$ is defined [undefined].  A denoting term $a$
is \emph{defined} [\emph{undefined}] in $M$ with respect to $\phi$ if
$V_{\phi}(a)$ is in $\Dc$ [$V_{\phi}(a) = \Undefined$].  If an
expression $e$ is denoting in $M$ with respect to $\phi$, then
its \emph{value} in $M$ with respect to $\phi$ is $V_{\phi}(e)$.

A model $M=(S,V)$ for $L$ is a \emph{standard model} for $L$ if $V$ is
the standard valuation for $S$.  The official semantics of Chiron is
based on standard models.  A formula $A$ is \emph{valid in $M$},
written $M \models A$, if $V_{\phi}(A) = \TRUE$ for all $\phi \in
\mname{assign}(S)$.  $A$ is \emph{valid}, written $\models A$, if $M
\models A$ for all standard models $M$.  A \emph{standard model} of a
set $\Gamma$ of formulas is a standard model $M$ such that $M \models
A$ for all $A \in \Gamma$.  

Let $e$ be a proper expression and $x$ be a symbol.  $e$ is
\emph{semantically closed in $M$} if $V_{\phi}(e)$ does not depend on
$\phi$, i.e., $V_{\phi}(e) = V_{\phi'}(e)$ for all $\phi,\phi' \in
\mname{assign}(S)$.  $e$ is \emph{semantically closed in $M$ with
  respect to $x$} if $V_{\phi}(e)$ does not depend on $\phi(x)$, i.e.,
$V_{\phi}(e) = V_{\phi[x \mapsto d]}(e)$ for all $\phi \in
\mname{assign}(S)$ and $d \in\Dc$.

\subsection{Theories}

A \emph{theory} of Chiron is pair $T=(L,\Gamma)$ where $L$ is a
language of Chiron and $\Gamma$ is a set of formulas of $L$ called
the \emph{axioms} of $T$.  $T$ is said to be \emph{over} $L$.
A \emph{standard model} of $T$ is a standard model for $L$ that is a
standard model of $\Gamma$.  A formula $A$ is \emph{valid in $T$},
written $T \models A$, if $M \models A$ for all standard models $M$ of
$T$.  The \emph{empty theory over $L$} is the theory $(L,\emptyset)$.

Let $e$ be a proper expression of $L$ and $T$ be a theory over $L$.
$e$ is \emph{semantically closed in $T$} if $e$ is semantically closed
in every model of $T$.  $e$~is \emph{semantically closed} (without
reference to a theory) if it is semantically closed in the empty
theory over $L$.  $e$ is \emph{semantically closed in $T$ with respect
  to $x$} if $e$ is semantically closed in every model of $T$ with
respect to $x$.

Let $T_i=(L_i,\Gamma_i)$ be a theory of Chiron for $i = 1,2$.  $T_1$
is a \emph{subtheory} of $T_2$ (and $T_2$ is an \emph{extension} of
$T_1$), written $T_1 \le T_2$, if $L_1 \le L_2$ and $\Gamma_1
\subseteq \Gamma_2$.

\begin{clem} \label{lem:theory-ext}
Let $T_i=(L_i,\Gamma_i)$ be a theory of Chiron for $i = 1,2$ such that
$T_1 \le T_2$.  Assume $L_1 = L_2$. Then \[T_1 \models A \mbox{
  implies } T_2 \models A\] for all formulas $A$ of $L_1$.
\end{clem}

\begin{proof}
Assume the hypotheses of the lemma.  Let $A$ be a formula of $L_1$
such that (a) $T_1 \models A$.  Let (b) $M = (S,V)$ be a standard
model of $T_2$.  We must show that $M \models A$.  (b) implies (c) $M$ is
a standard model of $T_1$ since $T_1 \le T_2$ and $L_1 = L_2$.  (a)
and (c) imply $M \models A$.
\end{proof}

\bigskip

The following example shows that the assumption in
Lemma~\ref{lem:theory-ext} that $L_1 = L_2$ is necessary.

\begin{ceg} \em
Let $L_i = (\sO_i,\theta_i)$ be a language of Chiron and
$T_i=(L_i,\Gamma_i)$ be a theory of Chiron for $i = 1,2$ such that
$\sO_1 = \set{o_1,\ldots,o_n}$, $L_1 < L_2$, and $T_1 \le T_2$.  If
$A$ is $\ell = \set{\synbrack{o_1},\ldots,\synbrack{o_n}}$, then $T_1
\models A$ but $T_2 \not\models A$.
\end{ceg}

\subsection{Notes Concerning the Semantics of Chiron} \label{subsec:notes}

Fix a standard model $M=(S,V)$ for $L$.

\be

  \item An improper expression never has a value, but its quotation
    (as well as the quotation of any proper expression) always has a
    value.  The latter is a set called a construction that represents
    the syntactic structure of the expression as a tree.

  \item The value of an undefined type is $\Dc$ (the universal
    superclass).  The value of an undefined term is the undefined
    value $\Undefined$.  And the value of an undefined formula is
    $\FALSE$.

  \item Proper expressions---particularly variables and
    evaluations---within a quotation are not semantically ``active''.
    They can only become active if the quotation is within an
    evaluation.

  \item Any types in an operator's signature will normally be
    semantically closed.  This is the case for all the built-in
    operators of Chiron as well as all the operators defined in
    sections~\ref{sec:op-defs} and \ref{sec:sub}.

  \item Suppose $O = (\mname{op}, o, k_1,\ldots,k_n,k_{n+1})$ is an
    operator.  The operation assigned to $O$ by a standard valuation
    is a ``restriction'' of the operation $I(o)$.  Roughly speaking,
    the domain and range of $I(o)$ is restricted according to types in
    the signature $k_1,\ldots,k_n,k_{n+1}$.

  \item \bsp Dependent function types, function abstractions,
    existential quantifications, definition descriptions, and
    indefinite descriptions are \emph{variable binding} expressions.
    According to the semantics of Chiron, a variable binding
    $(\StarApp{x} \mcolon \alpha \mdot e)$, where $\star$ is
    $\Lambda$, $\lambda$, $\Forsome$, $\iota$, or $\epsilon$, binds
    all the ``free'' variables occurring in $e$ that are similar to
    $(x \mcolon \alpha)$.  Variables are bound in the traditional,
    naive way.  It might be possible to use other more sophisticated
    variable binding mechanisms such as de Bruijn
    notation~\cite{deBruijn72} or nominal
    data types~\cite{GabbayPitts02,Pitts03}. \esp

  \item The type $\alpha$ of a variable $(\mname{var}, x, \alpha)$
  restricts the value of a variable in two ways.  First, if
  $(\mname{var}, x, \alpha)$ immediately follows $\Lambda$, $\lambda$,
  $\Forsome$, $\iota$, or $\epsilon$ in a variable binding expression,
  then the values assigned to $x$ for the body of the variable binding
  expression are restricted to the values in the superclass denoted by
  $\alpha$.  Second, $V_{\phi}((\mname{var}, x, \alpha)) = \phi(x)$
  iff $\phi(x)$ is in the superclass denoted by $\alpha$.
  ($V_{\phi}((\mname{var}, x, \alpha)) = \Undefined$ if $\phi(x)$ is
  not in the superclass denoted by $\alpha$.)

  \item In a variable $(\mname{var}, x, \alpha)$, the value of
    $\alpha$ usually does not depend on the value assigned to $x$.
    Suppose $d \in V_\phi(\alpha)$.  If $V_\phi(\alpha)$ does not
    depend on $x$, then $V_{\phi[x \mapsto d]}((\mname{var}, x,
    \alpha)) = d$.  If $V_\phi(\alpha)$ does depend on $x$, then
    $V_{\phi[x \mapsto d]}((\mname{var}, x, \alpha))$ might equal
    $\Undefined$.  Hence, in the latter case, it might be necessary to
    use $(\mname{var}, x, \mname{C})$ instead of $(\mname{var}, x,
    \alpha)$ in the body of a variable binding $(\StarApp{x} \mcolon
    \alpha \mdot e)$ where $\star$ is $\Lambda$, $\lambda$,
    $\Forsome$, $\iota$, or $\epsilon$
  
  \item Symbols and operator names are atomic in Chiron.  That is,
    they cannot be ``constructed'' from other expressions.

  \item $\ell$, the set of constructions that represent the operator
    names of $L$, represents $L$ itself.  The members of the type
    $\mname{L}$ represent sublanguages of $L$.  The construction types
    $\mname{E}_{{\rm on},a}$, $\mname{E}_a$, $\mname{E}_{{\rm op},a}$,
    $\mname{E}_{{\rm ty},a}$, $\mname{E}_{{\rm te},a}$,
    $\mname{E}^{b}_{{\rm te},a}$, $\mname{E}_{{\rm fo},a}$, where $a$
    is a sublanguage in $\mname{L}$, are parameterized versions of the
    respective types $\mname{E}_{{\rm on}}$, $\mname{E}$,
    $\mname{E}_{{\rm op}}$, $\mname{E}_{{\rm ty}}$, $\mname{E}_{{\rm
        te}}$, $\mname{E}^{b}_{{\rm te}}$, $\mname{E}_{{\rm fo}}$.
    (The construction type $\mname{E}_{\rm sy}$ does not depend on the
    set of operator names.)

  \item The notions of a free variable, substitution for a variable,
  variable capturing, etc.\ can be formalized in Chiron as defined
  operators (see subsection~\ref{subsec:sub-op}.)  As a result,
  syntactic side conditions can be expressed directly within Chiron
  formulas.  However, these notions are more complicated in Chiron
  than in traditional predicate logic due to the presence of
  evaluation.  For example, when the value of $(e \mcolon
  \mname{E}_{\rm te})$ equals the value of $\synbrack{(y \mcolon
  \mname{C})}$ with $x \not= y$, the variable $(y \mcolon \mname{C})$
  is free in \[\ForallApp x \mcolon \mname{C} \mdot x = \sembrack{(e
  \mcolon \mname{E}_{\rm te})}_{\rm te}\] because this formula is
  semantically equivalent to \[\ForallApp x \mcolon \mname{C} \mdot x
  = (y \mcolon \mname{C}).\]

  \item Since variables denote classes, they can be called \emph{class
  variables}.  There are no operation, superclass, or truth value
  variables in Chiron.  Thus direct quantification over operations,
  superclasses, and truth values is not possible.  Direct
  quantification over the undefined value $\Undefined$ is also not
  possible.  However, indirect quantification over ``definable''
  operations, ``definable'' superclasses, ``definable'' members of
  $\Dc \cup \set{\Undefined}$, or truth values can be done via
  variables of type $\mname{E}_{\rm op}$, $\mname{E}_{\rm ty}$,
  $\mname{E}_{\rm te}$, or $\mname{E}_{\rm fo}$, respectively.  As in
  standard {\nbg} set theory, only classes are first-class values in
  Chiron.

  \item Since sets and classes are superclasses, a type may denote a
  set or a proper class.  In particular, a type may denote the
  \emph{empty set}.  That is, types in Chiron are allowed to be empty.
  Empty types result, for example, from a type application
  $\alpha(a)$ where $a$ denotes a value that is not
  in the domain of any function in the superclass that $\alpha$
  denotes.

  \item $\alpha$ is a \emph{subtype} of $\beta$ in $M$ if
    $V_\phi(\alpha) \subseteq V_\phi(\beta)$ for all $\phi \in
    \mname{assign}(S)$.  For example, $\mname{E}$ is a subtype of
    $\mname{V}$ in every standard model.  \mname{C} is the \emph{universal
      type} by virtue of denoting $D_c$, the universal superclass.
    Every type is a subtype of \mname{C} in every standard model.

  \item An application of a term denoting a function to an undefined
    term is itself undefined.  That is, function application is strict
    with respect to undefinedness.  In contrast, the application of
    term operators is not necessarily strict with respect to
    undefinedness.

  \item Suppose $f(a)$ is a function application where $f$ is a term
    of type $(\LAMBDAapp x \mcolon \alpha \mdot \beta)$.  Then $a$
    could be any term of any type whatsoever.  However, if the value
    of $a$ is not a set in the value of $\alpha$, then $f(a)$ is
    undefined.  Similarly, suppose
    \[(o::k_1, \ldots,k_{i-1}, \alpha,k_{i+1}, 
    \ldots,k_{n+1})(e_1, \ldots, e_{i-1}, a, e_{i+1}, \ldots, e_n)\]
    is an operator application.  Then $a$ could be any term of any
    type whatsoever.  However, if the value of $a$ is a class not in
    the value of $\alpha$, then the value of this operator application
    is $\Dc$ if $k_{n+1} = \mname{type}$, $\Undefined$ if
    $\ctype{L}{k_{n+1}}$, and $\FALSE$ if $k_{n+1} = \mname{formula}$.

  \item Since functions are classes that represent mappings from sets
    to sets, proper classes in the denotations of the types $\alpha$
    and $\beta$ have no effect on the meaning of the dependent
    function type $(\LAMBDAapp x \mcolon \alpha \mdot \beta)$.  As a
    consequence, \[V_\phi((\LAMBDAapp x \mcolon \mname{V} \mdot
    \mname{V})) = V_\phi((\LAMBDAapp x \mcolon \mname{C} \mdot
    \mname{C})) = \Df,\] the domain of functions in $M$, for all $\phi
    \in \mname{assign}(S)$.  For the same reason, if $V_\phi(a) \not=
    \Undefined$, then $V_\phi(\alpha(a))$ is a class (i.e., not a
    proper superclass) for all $\phi \in \mname{assign}(S)$.

  \item Suppose a built-in operator \[(\mname{op}, \mname{lub},
    \mname{type},\mname{type},\mname{type})\] is added to Chiron that
    denotes an operation that, given superclasses $\Sigma_1$ and
    $\Sigma_2$, returns a superclass that is the least upper bound of
    $\Sigma_1$ and $\Sigma_2$.  Then formation rule P-Expr-8 could be
    modified so that a conditional term $(\mname{if},A,b,c)$ is
    assigned the type
    \[(\mname{op-app}, (\mname{op}, \mname{lub},
    \mname{type},\mname{type},\mname{type}), \beta,\gamma)\] where
    $\beta$ and $\gamma$ are the types of $b$ and $c$, respectively.
    See appendix B for details.

  \item A \emph{G\"odel number}~\cite{Goedel31} is a number that
  encodes an expression.  Analogously, a \emph{G\"odel set} is a set
  that encodes an expression.  Employing this terminology, the
  construction that represents an expression $e \in \sE_L$ is the
  \emph{G\"odel set} of $e$ which is denoted by $(\mname{quote},e)$.
  Hence, ``G\"odel numbering'' is built into Chiron.

  \item Quotations cannot be expressed as applications of a built-in
    operator since expressions are not values (i.e., not classes,
    superclasses, truth values, or operations).

  \item The formation rule P-Expr-12 could be sharpened so that the
    type of $(\mname{quote},e)$ is $\mname{E}_{\rm ty}$ when $e$ is a
    type, $\mname{E}_{\rm te}^{\synbrack{\alpha}}$ when $e$ is a term of type
    $\alpha$, etc.  See appendix B for details.

  \item When an evaluation $b =(\mname{eval},a, k)$ is ``semantically
    well-formed'', the value of $b$ is, roughly speaking, \emph{the
      value of the value of $a$}.

  \item \bsp An ``ungrounded expression'' is considered to be an
    undefined expression.  For example, the value of an ungrounded
    formula like $\sembrack{\mname{LIAR}}_{\rm fo}$ is $\FALSE$.\esp

  \item If we restrict evaluations to those of the form
    $(\mname{eval},a, k)$ where $k$ is \mname{type}, \mname{C}, or
    \mname{formula}, evaluations could be expressed as applications of
    the following three built-in operators with appropriate
    definitions:

  \be

    \item $(\mname{eval-type} :: \mname{E}_{\rm ty}, \mname{type})$.

    \item $(\mname{eval-term} :: \mname{E}_{\rm te}, \mname{C})$.

    \item $(\mname{eval-formula} :: \mname{E}_{\rm fo}, \mname{formula})$.

  \ee

  \item The operator $(\mname{eval-formula} :: \mname{E}_{\rm fo},
    \mname{formula})$ can be defined by the following formula
    schema: \[(\mname{eval-formula} :: \mname{E}_{\rm fo},
    \mname{formula})(a) \Iff \sembrack{a}_{\rm fo}.\] This operator is
    a truth predicate.  It satisfies four of the eight criteria ((a),
    (c), (f), (h)) for a truth predicate given by H. Leitgeb
    in~\cite{Leitgeb07}, and it meets all eight criteria for eval-free
    formulas.

  \item Operators of the form $(o :: k)$ are applied to an empty tuple
    of arguments.  The value of $(o :: k)$ is a $0$-ary operation
    $\sigma$ such that $\sigma(\,) = v$ for some value $v$.  We will
    sometimes abuse terminology and say that the value of $(o :: k)$
    is $v$ instead of $o$.

  \item A value $x \in \De$ is a construction that represents an
    expression of $L$.  We will sometimes abuse terminology and say $x
    \in \De$ is an \emph{expression} of $L$.  Similarly, we will
    sometimes say that a value $x \in \De$ that represents a symbol,
    operator name, operator, type, term, or formula of $L$ is a
    \emph{symbol}, \emph{operator name}, \emph{operator}, \emph{type},
    \emph{term}, or \emph{formula} of $L$, respectively.

  \item Let $\sE_{\rm op}$, $\sE_{\rm ty}$, $\sE_{\rm te}$, and
    $\sE_{\rm fo}$ be sets of expressions of $L$ defined inductively
    by the following statements:

  \be

    \item $\sE_{\rm op}$ is the set of built-in operators named
      \mname{set}, \mname{class}, \mname{term-equal}, \mname{in},
      \mname{not}, and \mname{or}.

    \item $\sE_{\rm ty}$ is the set $\set{\mname{V},\mname{C}}$ of
      types.

    \item If $x$ is in $\sS$ and $\alpha$ is in $\sE_{\rm ty}$, then
      $(x : \alpha)$ is in $\sE_{\rm te}$.

    \item If $x$ is in $\sS$, $\alpha$ is in $\sE_{\rm ty}$, $a$ and
      $b$ are in $\sE_{\rm te}$, and $A$ and $B$ are in $\sE_{\rm
      fo}$, then $a = b$, $a \in b$, $\Neg A$, $A \Or B$, and
      $\ForsomeApp x \mcolon \alpha \mdot A$ are in $\sE_{\rm fo}$.

  \ee

  Let {\cnbg} be the restriction of Chiron to the operators, types,
  terms, and formulas in $\sE_{\rm op}$, $\sE_{\rm ty}$, $\sE_{\rm
    te}$, and $\sE_{\rm fo}$, respectively.  {\cnbg} is a version of
  {\nbg} embedded in any language of Chiron (see
  Corollary~\ref{cor:cnbg} below).

\ee

\subsection{Relationship to NBG}\label{subsec:faithful}

We show in this section that there is a faithful semantic
interpretation of {\nbg} set theory in Chiron.  Loosely speaking, this
means Chiron is a conservative extension of {\nbg}.  That is, Chiron
adds new reasoning machinery to {\nbg} without compromising the
underlying semantics of {\nbg}.

Let $L=(\sO,\theta)$ be any language of Chiron.  {\nbg} is usually
formulated as a theory in first-order logic over a language $L_{\rm
nbg}$ containing an infinite set $\sV$ of variables, a unary predicate
symbol $V$, binary predicates symbols $=$ and $\in$, and some complete
set of logical connectives (say $\Neg,\Or,\Forsome$).  Assume $\sV
= \sS$, i.e., variables of $L_{\rm nbg}$ are the symbols of Chiron.
Let $\sF_{\rm nbg}$ denote the set of formulas of $L_{\rm
nbg}$. A \emph{model} of {\nbg} is a structure $N=(D^N,V^N,=^N,\in^N)$
for $L_{\rm nbg}$ that satisfies the axioms of {\nbg}.
An \emph{assignment} into $N$ is a mapping that assigns a member of
$D^N$ to each variable $x \in \sV$.  Let $\mname{assign}(N)$ be the
collection of assignments into $N$.  The \emph{valuation} for $N$ is a
total mapping $W^N$ from $\sF_{\rm nbg} \times
\mname{assign}(N)$ to the set $\set{\TRUE,\FALSE}$ of truth values.  A
formula $A$ of $L_{\rm nbg}$ is \emph{valid}, written $\models_{\rm
  nbg}A$, if $W^{N}_{\phi}(A) = \TRUE$ for all models $N$ of {\nbg}
and all $\phi \in \mname{assign}(N)$.\footnote{As above for a
  valuation $V$, we write $W^{N}(A,\phi)$ as $W^{N}_{\phi}(A)$.}

$\Phi$ is a \emph{translation from} {\nbg} \emph{to} $L$ if $\Phi$
maps the terms (variables) of $L_{\rm nbg}$ to the terms of $L$ and
the formulas of $L_{\rm nbg}$ to formulas of $L$.  $\Phi$ is
a \emph{(semantic) interpretation of} {\nbg} \emph{in} Chiron if
$\Phi$ is a translation from {\nbg} to $L$ and, for all formulas $A$
of $L_{\rm nbg}$, $\models_{\rm nbg} A$ implies $\models
\Phi(A)$.  That is, $\Phi$ is an interpretation of {\nbg} in Chiron if
$\Phi$ is a meaning-preserving translation from {\nbg} to $L$.  $\Phi$
is a \emph{faithful interpretation of} {\nbg} \emph{in} Chiron if
$\Phi$ is an interpretation of {\nbg} in Chiron and, for all formulas
$A$ of $L_{\rm nbg}$, $\models \Phi(A)$ implies $\models_{\rm nbg}A$.
That is, $\Phi$ is a faithful interpretation of {\nbg} in Chiron if
Chiron over $L$ is a conservative extension of the image of {\nbg}
under $\Phi$.

Let $\Phi$ be the translation from {\nbg} to $L$ recursively defined
by:

\be

  \item If $x \in \sV$, then $\Phi(x) = (x \mcolon \mname{C})$.

  \item If $V(x)$ is a formula of $L_{\rm nbg}$, then $\Phi(V(x)) =
    (\Phi(x) =_{\sf V} \Phi(x))$.

  \item If $(x = y)$ is a formula of $L_{\rm nbg}$, then $\Phi((x = y)) =
  (\Phi(x) = \Phi(y))$.

  \item If $(x \in y)$ is a formula of $L_{\rm nbg}$, then $\Phi((x \in
  y)) = (\Phi(x) \in \Phi(y))$.

  \item If $(\Neg A)$ is a formula of $L_{\rm nbg}$, then $\Phi((\Neg
  A)) = (\Neg \Phi(A))$.

  \item If $(A \Or B)$ is a formula of $L_{\rm nbg}$, then $\Phi((A
  \Or B)) = (\Phi(A) \Or \Phi(B))$.

  \item If $(\ForsomeApp x \mdot A)$ is a formula of $L_{\rm nbg}$,
  then $\Phi((\ForsomeApp x \mdot A)) = (\ForsomeApp x \mcolon
  \mname{C} \mdot \Phi(A))$.

\ee  

\begin{clem}\label{lem:faithful}
Let $N=(D^N,V^N,=^N,\in^N)$ be a model of {\nbg} and $M=(S,V)$ be a
standard model for $L$, where \[S =
(\Dv,\Dc,\Ds,\Df,\Do,\De,\in,\TRUE,\FALSE,\Undefined,\xi,H,I),\] such
that $(D^N,\in^N)$ is identical to the prestructure $(\Dc,\in)$.

\be

  \item For all $d \in D^N$, $V^N(d)$ iff $d$ is in $\Dv$.

  \item For all $d,d' \in D^N$, $d =^N d'$ iff  $d = d'$.

  \item For all $d,d' \in D^N$, $d \in^N d'$ iff $d \in d'$.

  \item $\mname{assign}(N) = \mname{assign}(M)$.

  \item For all formulas $A$ of $L_{\rm nbg}$ and $\phi \in
    \mname{assign}(N)$, $W^{N}_{\phi}(A) = V_{\phi}(\Phi(A))$.

\ee
\end{clem}

\begin{proof}
Clauses 1--4 are obvious.  Clause 5 is easily proved by induction on
the structure of the formulas of $L_{\rm nbg}$ since the logical
connectives $\Neg,\Or,\Forsome$ of $L_{\rm nbg}$ have the same meanings
as the negation operator, the disjunction operator, and existential
quantification, respectively, in Chiron.
\end{proof}
  
\begin{cthm}
For all formulas $A$ of $L_{\rm nbg}$, \[\models_{\mbox{\small \sc
nbg}} A \dblsp \mbox{iff} \dblsp \models \Phi(A).\] That is, $\Phi$ is
a faithful interpretation of {\nbg} in Chiron.
\end{cthm}

\begin{proof}
For every model $N=(D^N,V^N,=^N,\in^N)$ of {\nbg}, there is a
standard model $M=(S,V)$ for $L$, where \[S =
(\Dv,\Dc,\Ds,\Df,\Do,\De,\in,\TRUE,\FALSE,\Undefined,\xi,H,I),\] such
that $(D^N,\in^N)$ is identical to the prestructure $(\Dc,\in)$.
Likewise, for every standard model $M=(S,V)$ for $L$, where \[S =
(\Dv,\Dc,\Ds,\Df,\Do,\De,\in,\TRUE,\FALSE,\Undefined,\xi,H,I),\] there
is a model $N=(D^N,V^N,=^N,\in^N)$ of {\nbg} such that the
prestructure $(\Dc,\in)$ is identical to $(D^N,\in^N)$.  The theorem
follows from this observation and clause 5 of
Lemma~\ref{lem:faithful}.
\end{proof}

\bigskip

In the previous subsection we defined a restriction of Chiron named
{\cnbg} that we claimed is a version of {\nbg} embedded in any
language of Chiron.  This claim is established by the following
corollary.

\begin{ccor}\label{cor:cnbg}
$\Phi$ is a faithful interpretation of {\nbg} in {\cnbg}.
\end{ccor}

\section{Operator Definitions}\label{sec:op-defs}

There are two ways of assigning a meaning to an operator.  The first
is to make the operator one of the built-in operators like
$(\mname{op}, \mname{set}, \mname{type})$ ($(\mname{set} ::
\mname{type})$ in compact notation) and then define its meaning as
part of the definition of a standard model.  The second is to
construct one or more formulas that together define the meaning of an
operator.  We will use this latter approach to define several useful
logical, set-theoretic, and syntactic operators in this section and
substitution operators in the next section.

Each operator definition consists of an operator \[O = (o ::
k_1,\ldots,k_{n+1})\] and a set of defining axioms for $O$.  The
definition defines $o$ to be an operator name, $\theta(o)$ to be the
signature form corresponding to $k_1,\ldots,k_{n+1}$, and
$k_1,\ldots,k_{n+1}$ to be the preferred signature of $o$.  Each
defining axiom is usually eval-free formula.  The defining axioms are
presented as individual formulas or as formula schemas.  The set of
defining axioms presented by a formula schema usually depends on the
language that is being considered.  An operator definition may
optionally include compact syntax for applications of the defined
operator.

\subsection{Logical Operators}\label{subsec:logop}

\be
  
  \item \textbf{Truth}\\
  Operator: 
  $(\mname{true}::\mname{formula})$\\
  Defining axioms:
  \begin{eqnarray*}
  (\mname{true}::\mname{formula})(\,) \Iff \mname{C} \TypeEqual \mname{C}.
  \end{eqnarray*}

%\newpage

  Compact notation: 
  \bi

    \item[] $\mname{T}$ means $(\mname{true}::\mname{formula})(\,)$.

  \ei
  
  \item \textbf{Falsehood}\\
  Operator: 
  $(\mname{false}::\mname{formula})$\\
  Defining axioms:
  \begin{eqnarray*}
  (\mname{false}::\mname{formula})(\,) \Iff \mname{V} \TypeEqual \mname{C}.
  \end{eqnarray*}

%\newpage

  Compact notation: 
  \bi

    \item[] $\mname{F}$ means $(\mname{false}::\mname{formula})(\,)$.

  \ei

  \item \textbf{Conjunction}\\
  Operator: 
  $(\mname{and}::\mname{formula},\mname{formula},\mname{formula})$\\
  Defining axioms:
  \begin{eqnarray*}
  (\mname{and}::\mname{formula},\mname{formula},\mname{formula})
     (A^{\rm ef},B^{\rm ef}) \Iff 
     \Neg(\Neg A^{\rm ef} \Or \Neg B^{\rm ef}).
  \end{eqnarray*}

  \iffalse
  \begin{eqnarray*}
  \lefteqn{\ForallApp e_1,e_2 \mcolon \mname{E}_{\rm fo} \mdot 
     (\mname{and}::\mname{formula},\mname{formula},\mname{formula})
     (\sembrack{e_1},\sembrack{e_2}) \Iff} \\ & &
     \Neg(\Neg \sembrack{e_1} \Or \Neg \sembrack{e_2}).
  \end{eqnarray*}
  \fi

%\newpage

  Compact notation: 
  \bi

    \item[] $(A \And B)$ means 
    $(\mname{and}::\mname{formula},\mname{formula},\mname{formula})
    (A,B)$.

  \ei

  Note: The defining axioms are presented as a formula schema.  Recall
  that metavariables like $A^{\rm ef}$ and $B^{\rm ef}$ denote
  eval-free formulas.

  \item \textbf{Implication}\\
  Operator: 
  $(\mname{implies}:: \mname{formula}, \mname{formula}, \mname{formula})$\\
  Defining axioms:
  \begin{eqnarray*}
  (\mname{implies}::\mname{formula},\mname{formula},\mname{formula})
     (A^{\rm ef},B^{\rm ef}) \Iff
     \Neg A^{\rm ef} \Or B^{\rm ef}.
  \end{eqnarray*}

  \iffalse
  \begin{eqnarray*}
  \lefteqn{\ForallApp e_1,e_2 \mcolon \mname{E}_{\rm fo} \mdot 
     (\mname{implies}::\mname{formula},\mname{formula},\mname{formula})
     (\sembrack{e_1},\sembrack{e_2}) \Iff} \\ & &
     \Neg \sembrack{e_1} \Or \sembrack{e_2}.
  \end{eqnarray*}
  \fi

%\newpage

  Compact notation: 
  \bi

    \item[] $(A \Implies B)$ means 
    $(\mname{implies}::\mname{formula},\mname{formula},\mname{formula})
    (A,B)$.

  \ei

  \item \textbf{Definedness in a Type}\\
  Operator: 
  $(\mname{defined-in}::\mname{C},\mname{type},\mname{formula})$\\
  Defining axioms:
  \begin{eqnarray*}
  (\mname{defined-in}::\mname{C},\mname{type},\mname{formula})
     (a^{\rm ef},\alpha^{\rm ef}) \Iff
     (a^{\rm ef} =_{\alpha^{\rm ef}} a^{\rm ef}).
  \end{eqnarray*}

  \iffalse
  \begin{eqnarray*}
  \lefteqn{\ForallApp e_1 \mcolon \mname{E}_{\rm te},  
     e_2 \mcolon \mname{E}_{\rm ty} \mdot 
     (\mname{defined-in}::\mname{C},\mname{type},\mname{formula})
     (\sembrack{e_1},\sembrack{e_2}) \Iff} \\ & &
     \sembrack{e_1} =_{\sembrack{e_2}} \sembrack{e_1}.
  \end{eqnarray*}
  \fi

%\newpage

  Compact notation: 
  \bi

    \item[] $(a \IsDef \alpha)$ means 
    $(\mname{defined-in}::\mname{C},\mname{type},\mname{formula})
    (a,\alpha)$.

    \item[] $(a \IsUndef \alpha)$ means $\Neg(a \IsDef \alpha)$.

    \item[] $(a \IsDefApp)$ means $(a \IsDef \mname{C})$.

    \item[] $(a \IsUndefApp)$ means $\Neg(a \IsDefApp)$.
  
  \ei

  \item \textbf{Quasi-Equality}\\
  Operator: 
  $(\mname{quasi-equal}::\mname{C},\mname{C},\mname{formula})$\\
  Defining axioms:
  \begin{eqnarray*}
  \lefteqn{(\mname{quasi-equal}::\mname{C},\mname{C},\mname{formula})
     (a^{\rm ef},b^{\rm ef}) \Iff} \\ & &
     ({a^{\rm ef} \IsDefApp} \Or {b^{\rm ef} \IsDefApp}) 
     \Implies a^{\rm ef} \Equal b^{\rm ef}.
  \end{eqnarray*}

  Note: The application of \mname{quasi-equal} to two undefined terms
  is true.  Hence \mname{quasi-equal} is not strict with respect to
  undefinedness.  Nearly all the operators defined in this section are
  strict with respect to undefinedness.

  \iffalse
  \begin{eqnarray*}
  \lefteqn{\ForallApp e_1,e_2 \mcolon \mname{E}_{\rm te} \mdot 
     (\mname{quasi-equal}::\mname{C},\mname{C},\mname{formula})
     (\sembrack{e_1},\sembrack{e_2}) \Iff} \\ & &
     ({\sembrack{e_1} \IsDefApp} \Or {\sembrack{e_2} \IsDefApp}) 
     \Implies \sembrack{e_1} \Equal \sembrack{e_2}.
  \end{eqnarray*}
  \fi

%\newpage

  Compact notation: 
  \bi

    \item[] $(a \QuasiEqual b)$ means 
    $(\mname{quasi-equal}::\mname{C},\mname{C},\mname{formula})
    (a,b)$.

    \item[] $(a \not\QuasiEqual b)$ means $\Neg (a \QuasiEqual b)$.
  
  \ei

  \item \textbf{Canonical Undefined Term}\\
  Operator: 
  $(\mname{undefined}::\mname{C})$\\
  Defining axioms:
  \begin{eqnarray*}
  (\mname{undefined}::\mname{C})(\,) \QuasiEqual
     (\iotaApp x\mcolon \mname{C} \mdot x \not= x).
  \end{eqnarray*}
  Compact notation: 
  \bi

    \item[] $\Undefined_{\sf C}$ means 
    $(\mname{undefined}::\mname{C})(\,)$.

  \ei

  \item \textbf{Canonical Empty Type}\\
  Operator: 
  $(\mname{empty-type}::\mname{type})$\\
  Defining axioms:
  \begin{eqnarray*}
  \lefteqn{\ForallApp x \mcolon \mname{C} \mdot 
     x \IsUndef (\mname{empty-type}::\mname{type})(\,).}
  \end{eqnarray*}

%\newpage

  Compact notation: 
  \bi

    \item[] $\nabla$ means
      $(\mname{empty-type}::\mname{type})(\,).$
  
  \ei

  \item \textbf{Type Order}\\
  Operator: 
  $(\mname{type-le}::\mname{type},\mname{type},\mname{formula})$\\
  Defining axioms:
  \begin{eqnarray*}
  \lefteqn{(\Neg\mname{free-in}(\synbrack{x},\synbrack{\alpha^{\rm ef}}) \And
     \Neg\mname{free-in}(\synbrack{x},\synbrack{\beta^{\rm ef}}))
     \Implies} \\ & &
     (\mname{type-le}::\mname{type},\mname{type},\mname{formula})
     (\alpha^{\rm ef}, \beta^{\rm ef}) \Iff
     \ForallApp x \mcolon \alpha^{\rm ef} \mdot 
     x \IsDef \beta^{\rm ef}.
  \end{eqnarray*}

  \iffalse
  \begin{eqnarray*}
  \lefteqn{\ForallApp e_1,e_2 \mcolon \mname{E}_{\rm ty} \mdot 
     (\mname{type-le}::\mname{type},\mname{type},\mname{formula})
     (\sembrack{e_1},\sembrack{e_2}) \Iff} \\ & &
     \ForallApp e \mcolon \mname{E}_{\rm sy} \mdot
     (\Neg\mname{free-in}(e,e_1) \And \Neg\mname{free-in}(e,e_2))
     \Implies  \\ & &
     \hspace{10ex} \sembrack{\synbrack{
     \ForallApp \commabrack{e} \mcolon 
     \commabrack{e_1} \mdot 
     (\commabrack{e} \mcolon \commabrack{e_1})
     \IsDef \commabrack{e_2}}}_{\rm fo}.
  \end{eqnarray*}
  \fi

%\newpage

  Compact notation: 
  \bi

    \item[] $(\alpha \TypeLE \beta)$ means 
    $(\mname{type-le}::\mname{type},\mname{type},\mname{formula})
    (\alpha,\beta)$.

  \ei  

  Note: The formula schema of defining axioms utilizes
  the \mname{free-in} operator defined in
  subsection~\ref{subsec:sub-op}.

  \item \textbf{Conditional Type}\\
  Operator: 
  $(\mname{if-type}::\mname{formula}, \mname{type}, \mname{type},
  \mname{type})$\\
  Defining axioms:
  \begin{eqnarray*}
     \lefteqn{A^{\rm ef} \Implies} \\ & &
     (\mname{if-type}::\mname{formula}, \mname{type},
     \mname{type}, \mname{type}) 
     (A^{\rm ef},\beta^{\rm ef},\gamma^{\rm ef}) \TypeEqual \beta^{\rm ef}.
  \end{eqnarray*}
  \vspace{-5ex}  
  \begin{eqnarray*}
     \lefteqn{\Neg A^{\rm ef} \Implies} \\ & &
     (\mname{if-type}::\mname{formula}, \mname{type},
     \mname{type}, \mname{type}) 
     (A^{\rm ef},\beta^{\rm ef},\gamma^{\rm ef}) \TypeEqual \gamma^{\rm ef}.
  \end{eqnarray*}
  
  \iffalse
  \begin{eqnarray*}
  \lefteqn{\ForallApp e_1 \mcolon \mname{E}_{\rm fo},
     e_2,e_3 \mcolon \mname{E}_{\rm ty} \mdot } \\ & &
     (\sembrack{e_1} \Implies \\ & &
     \hspace{4ex} (\mname{if-type}::\mname{formula}, \mname{type},
     \mname{type}, \mname{formula}) \\ & &
     \hspace{4ex} (\sembrack{e_1},\sembrack{e_2},\sembrack{e_3})
     \TypeEqual \sembrack{e_2}) \And \\ & &
     (\Neg\sembrack{e_1} \Implies \\ & &
     \hspace{4ex} (\mname{if-type}::\mname{formula}, \mname{type},
     \mname{type}, \mname{formula}) \\ & &
     \hspace{4ex} (\sembrack{e_1},\sembrack{e_2},\sembrack{e_3})
     \TypeEqual \sembrack{e_3})
  \end{eqnarray*}
  \fi

%\newpage

  Compact notation: 
  \bi

    \item[] $\If(A,\beta,\gamma)$ means \\ 
    \hspace*{4ex} $(\mname{if-type}::\mname{formula}, \mname{type},\mname{type},
    \mname{type})(A,\beta,\gamma)$.

  \ei

  \item \textbf{Conditional Formula}\\
  Operator: 
  $(\mname{if-formula}::\mname{formula}, \mname{formula}, \mname{formula},
  \mname{formula})$\\
  Defining axioms:
  \begin{eqnarray*}
     \lefteqn{A^{\rm ef} \Implies} \\ & &
     (\mname{if-formula}::\mname{formula}, \mname{formula},
     \mname{formula}, \mname{formula}) \\ & &
     (A^{\rm ef},B^{\rm ef},C^{\rm ef}) \TypeEqual B^{\rm ef}.
  \end{eqnarray*}
  \vspace{-5ex}  
  \begin{eqnarray*}
     \lefteqn{\Neg A^{\rm ef} \Implies} \\ & &
     (\mname{if-formula}::\mname{formula}, \mname{formula},
     \mname{formula}, \mname{formula}) \\ & &
     (A^{\rm ef},B^{\rm ef},C^{\rm ef}) \TypeEqual C^{\rm ef}.
  \end{eqnarray*}

  \iffalse
  \begin{eqnarray*}
  \lefteqn{\ForallApp e_1,e_2,e_3 \mcolon \mname{E}_{\rm fo} \mdot } \\ & &
     (\sembrack{e_1} \Implies \\ & &
     \hspace{4ex} (\mname{if-formula}::\mname{formula}, \mname{formula},
     \mname{formula}, \mname{formula})\\ & &
     \hspace{4ex} (\sembrack{e_1},\sembrack{e_2},\sembrack{e_3})
     \Iff \sembrack{e_2}) \And \\ & &
     (\Neg\sembrack{e_1} \Implies \\ & &
     \hspace{4ex} (\mname{if-formula}::\mname{formula}, \mname{formula},
     \mname{formula}, \mname{formula})\\ & &
     \hspace{4ex} (\sembrack{e_1},\sembrack{e_2},\sembrack{e_3})
     \Iff \sembrack{e_3})
  \end{eqnarray*}
  \fi

%\newpage

  Compact notation: 
  \bi

    \item[] $\If(A,B,C)$ means \\
    \hspace*{4ex} $(\mname{if-formula}::\mname{formula}, \mname{formula},
    \mname{formula}, \mname{formula})(A,B,C)$.

  \ei

  \item \textbf{Simple Function Type}\\
  Operator: 
  $(\mname{sim-fun-type}::\mname{type},\mname{type},\mname{type})$\\
  Defining axioms:
  \begin{eqnarray*}
  \lefteqn{(\mname{sim-fun-type}::\mname{type},\mname{type},\mname{type})
     (\alpha^{\rm ef},\beta^{\rm ef}) \TypeEqual} \\ & &
     \If(\mname{syn-closed}(\synbrack{\beta^{\rm ef}}),
     (\LAMBDAapp x \mcolon \alpha^{\rm ef} \mdot \beta^{\rm ef}),
     \mname{C}).
  \end{eqnarray*}

  \iffalse
  \begin{eqnarray*}
  \lefteqn{\ForallApp e_1,e_2 \mcolon \mname{E}_{\rm ty} \mdot 
     (\mname{sim-fun-type}::\mname{type},\mname{type},\mname{type})
     (\sembrack{e_1},\sembrack{e_2}) \TypeEqual} \\ & &
     \If(\mname{syn-closed}(e_2),
     (\LAMBDAapp x \mcolon \alpha \mdot \beta),
     \mname{C}).
  \end{eqnarray*}
  \fi

%\newpage

  Compact notation: 
  \bi

    \item[] $(\alpha \tarrow \beta)$ means 
    $(\mname{sim-fun-type}::\mname{type},\mname{type},\mname{type})
    (\alpha,\beta)$.

  \ei  

  Note: The formula schema of defining axioms utilizes
  the \mname{syn-closed} operator defined in
  subsection~\ref{subsec:sub-op}.

\ee

\subsection{Set-Theoretic Operators} \label{subsec:setthop}

\be

  \item \textbf{Empty set}\\
  Operator: 
  $(\mname{empty-set}::\mname{V})$\\
  Defining axioms:
  \begin{eqnarray*}
  (\mname{empty-set}::\mname{V})(\,)\QuasiEqual
     (\iotaApp u \mcolon \mname{V} \mdot 
     \ForallApp v \mcolon \mname{V} \mdot v \not\in u).
  \end{eqnarray*}

\newpage

  Compact notation: 
  \bi

    \item[] $\emptyset$ means $(\mname{empty-set}::\mname{V})(\,)$.
  
  \ei

  \item \textbf{Universal class}\\
  Operator: 
  $(\mname{universal-class}::\mname{C})$\\
  Defining axioms:
  \begin{eqnarray*}
  (\mname{univeral-class}::\mname{C})(\,)\QuasiEqual
     \iotaApp x \mcolon \mname{C} \mdot 
     \ForallApp u \mcolon \mname{V} \mdot u \in x.
  \end{eqnarray*}

%\newpage

  Compact notation: 
  \bi

    \item[] $\textrm{U}$ means $(\mname{universal-set}::\mname{C})(\,)$.
  
  \ei

  \item \textbf{Pair}\\
  Operator: 
  $(\mname{pair}::\mname{V},\mname{V},\mname{V})$\\
  Defining axioms:
  \begin{eqnarray*}
  \lefteqn{\ForallApp u,v \mcolon \mname{V} \mdot 
     (\mname{pair}::\mname{V},\mname{V},\mname{V})(u,v) \QuasiEqual} \\ & &
     \iotaApp w\mcolon\mname{V} \mdot
     \ForallApp w' \mcolon \mname{V} \mdot
      w' \in w \Iff (w' = u \Or w' = v).
  \end{eqnarray*}
  \vspace{-5ex}  
  \begin{eqnarray*}
     ({a^{\rm ef} \IsUndefApp} \Or {b^{\rm ef} \IsUndefApp}) 
     \Implies 
     (\mname{pair}::\mname{V},\mname{V},\mname{V})
     (a^{\rm ef}, b^{\rm ef}) \IsUndefApp.
  \end{eqnarray*}

  \iffalse
  \begin{eqnarray*}
  \lefteqn{\ForallApp e_1,e_2 \mcolon \mname{E}_{\rm te} \mdot 
     ({\sembrack{e_1} \IsUndefApp} \Or {\sembrack{e_2} \IsUndefApp}) 
     \Implies {}} \\ & &
     (\mname{pair}::\mname{V},\mname{V},\mname{V})
     (\sembrack{e_1},\sembrack{e_2}) \IsUndefApp.
  \end{eqnarray*}
  \fi

  Note: The formula schema of defining axioms that comes second
  asserts that \mname{pair} is strict with respect to undefinedness.
  The operators named \mname{singleton}, \mname{triple}, etc.\ are
  defined in a similar way to \mname{pair}.

  \medskip

%\newpage

  Compact notation: 
  \bi

    \item[] $\set{\,}$ means $\emptyset$.

    \item[] $\set{a}$ means 
    $(\mname{singleton}::\mname{V},\mname{V})(a)$.

    \item[] $\set{a,b}$ means 
    $(\mname{pair}::\mname{V},\mname{V},\mname{V})(a,b)$.

    \item[] $\set{a,b,c}$ means 
    $(\mname{triple}::\mname{V},\mname{V},\mname{V},\mname{V})(a,b,c)$.

    \item[] \hspace{4ex}$\vdots$
  
  \ei

  \item \textbf{Ordered Pair}\\
  Operator: 
  $(\mname{ord-pair}::\mname{V},\mname{V},\mname{V})$\\
  Defining axioms:
  \begin{eqnarray*}
  \lefteqn{\ForallApp u,v \mcolon \mname{V} \mdot
    (\mname{ord-pair}::\mname{V},\mname{V},\mname{V}) (u,v)
    \QuasiEqual} \\ & &
  \set{\set{u},\set{u,v}}.
  \end{eqnarray*}
  \vspace{-5ex}   
  \begin{eqnarray*}
     ({a^{\rm ef} \IsUndefApp} \Or {b^{\rm ef} \IsUndefApp}) 
     \Implies 
     (\mname{ord-pair}::\mname{V},\mname{V},\mname{V})
     (a^{\rm ef}, b^{\rm ef}) \IsUndefApp.
  \end{eqnarray*}

  \iffalse
  \begin{eqnarray*}
  \lefteqn{\ForallApp e_1,e_2 \mcolon \mname{E}_{\rm te} \mdot 
     ({\sembrack{e_1} \IsUndefApp} \Or 
     {\sembrack{e_2} \IsUndefApp}) \Implies {}} \\ & &
     (\mname{ord-pair}::\mname{V},\mname{V},\mname{V})
     (\sembrack{e_1},\sembrack{e_2}) \IsUndefApp.
  \end{eqnarray*}
  \fi

%\newpage

  Compact notation: 
  \bi

    \item[] $\seq{a,b}$ means 
    $(\mname{ord-pair}::\mname{V},\mname{V},\mname{V})(a,b)$.

    \item[] $\seq{a_1,\ldots,a_n}$ means $\seq{a_1,\seq{a_2,\ldots,a_n}}$
    for $n \ge 3$.

    \item[] $\mlist{\,}$ means $\emptyset$.
 
    \item[] $\mlist{a_1,\ldots,a_n}$ means
    $\seq{a_1,\mlist{a_2,\ldots,a_n}}$ for $n \ge 1$.
  
  \ei

  \item \textbf{Subclass}\\
  Operator: 
  $(\mname{subclass}::\mname{C},\mname{C},\mname{formula})$\\
  Defining axioms:
  \begin{eqnarray*}
  \lefteqn{\ForallApp x,y \mcolon \mname{C} \mdot 
     (\mname{subclass}::\mname{C},\mname{C},\mname{formula})
     (x,y) \Iff} \\ & &
     {\ForallApp u \mcolon \mname{V} \mdot 
     u \in x \Implies u \in y}.
  \end{eqnarray*}
  \vspace{-5ex} 
  \begin{eqnarray*}
     ({a^{\rm ef} \IsUndefApp} \Or 
     {b^{\rm ef} \IsUndefApp}) \Implies 
     (\mname{subclass}::\mname{C},\mname{C},\mname{formula})
     (a^{\rm ef}, b^{\rm ef}) \Iff \mname{F}.
  \end{eqnarray*} 

  \iffalse
  \begin{eqnarray*}
  \lefteqn{\ForallApp e_1,e_2 \mcolon \mname{E}_{\rm te} \mdot 
     ({\sembrack{e_1} \IsUndefApp} \Or 
     {\sembrack{e_2} \IsUndefApp}) \Implies {}} \\ & &
     (\mname{subclass}::\mname{C},\mname{C},\mname{formula})
     (\sembrack{e_1},\sembrack{e_2}) \Iff \mname{F}.
  \end{eqnarray*}
  \fi

%\newpage

  Compact notation: 
  \bi

    \item[] $a \subseteq b$ means 
    $(\mname{subclass}::\mname{C},\mname{C},\mname{formula})(a,b)$.

    \item[] $a \subset b$ means $a \subseteq b \And a \not= b$.
  
  \ei

  \item \textbf{Union}\\
  Operator: 
  $(\mname{union}::\mname{C},\mname{C},\mname{C})$\\
  Defining axioms:
  \begin{eqnarray*}
  \lefteqn{\ForallApp x,y \mcolon \mname{C} \mdot 
     (\mname{union}::\mname{C},\mname{C},\mname{C})
     (x,y) \QuasiEqual} \\ & &
     \iotaApp z \mcolon \mname{C} \mdot
     \ForallApp u \mcolon \mname{V} \mdot u \in z \Iff 
     (u \in x \Or u \in y).
  \end{eqnarray*}
  \vspace{-5ex}   
  \begin{eqnarray*}
     ({a^{\rm ef} \IsUndefApp} \Or {b^{\rm ef} \IsUndefApp}) 
     \Implies 
     (\mname{union}::\mname{C},\mname{C},\mname{C})
     (a^{\rm ef}, b^{\rm ef}) \IsUndefApp.
  \end{eqnarray*}

  \iffalse
  \begin{eqnarray*}
  \lefteqn{\ForallApp e_1,e_2 \mcolon \mname{E}_{\rm te} \mdot 
     ({\sembrack{e_1} \IsUndefApp} \Or 
     {\sembrack{e_2} \IsUndefApp}) \Implies {}} \\ & &
     (\mname{union}::\mname{C},\mname{C},\mname{C})
     (\sembrack{e_1},\sembrack{e_2}) \IsUndefApp.
  \end{eqnarray*}
  \fi

%\newpage

  Compact notation: 
  \bi

    \item[] $a \cup b$ means 
    $(\mname{union}::\mname{C},\mname{C},\mname{C})(a,b)$.
  
  \ei

  \item \textbf{Intersection}\\
  Operator: 
  $(\mname{intersection}::\mname{C},\mname{C},\mname{C})$\\
  Defining axioms:
  \begin{eqnarray*}
  \lefteqn{\ForallApp x,y \mcolon \mname{C} \mdot 
     (\mname{intersection}::\mname{C},\mname{C},\mname{C})
     (x,y) \QuasiEqual} \\ & &
     \iotaApp z \mcolon \mname{C} \mdot
     \ForallApp u \mcolon \mname{V} \mdot u \in z \Iff 
     (u \in x \And u \in y).
  \end{eqnarray*}
  \vspace{-5ex}   
  \begin{eqnarray*}
     ({a^{\rm ef} \IsUndefApp} \Or {b^{\rm ef} \IsUndefApp}) 
     \Implies 
     (\mname{intersection}::\mname{C},\mname{C},\mname{C})
     (a^{\rm ef}, b^{\rm ef}) \IsUndefApp.
  \end{eqnarray*}

  \iffalse
  \begin{eqnarray*}
  \lefteqn{\ForallApp e_1,e_2 \mcolon \mname{E}_{\rm te} \mdot 
     ({\sembrack{e_1} \IsUndefApp} \Or 
     {\sembrack{e_2} \IsUndefApp}) \Implies {}} \\ & &
     (\mname{intersection}::\mname{C},\mname{C},\mname{C})
     (\sembrack{e_1},\sembrack{e_2}) \IsUndefApp.
  \end{eqnarray*}
  \fi

%\newpage

  Compact notation: 
  \bi

    \item[] $a \cap b$ means 
    $(\mname{intersection}::\mname{C},\mname{C},\mname{C})(a,b)$.
  
  \ei

  \item \textbf{Complement}\\
  Operator: 
  $(\mname{complement}::\mname{C},\mname{C})$\\
  Defining axioms:
  \begin{eqnarray*}
  \lefteqn{\ForallApp x \mcolon \mname{C} \mdot 
     (\mname{complement}::\mname{C},\mname{C})
     (x) \QuasiEqual} \\ & &
     \iotaApp y \mcolon \mname{C} \mdot
     \ForallApp u \mcolon \mname{V} \mdot 
     u \in y \Iff u \not\in x.
  \end{eqnarray*}
  \vspace{-5ex}   
  \begin{eqnarray*}
     {a^{\rm ef} \IsUndefApp}
     \Implies 
     (\mname{complement}::\mname{C},\mname{C})
     (a^{\rm ef}) \IsUndefApp.
  \end{eqnarray*}

  \iffalse
  \begin{eqnarray*}
  \lefteqn{\ForallApp e \mcolon \mname{E}_{\rm te} \mdot 
    {\sembrack{e} \IsUndefApp} \Implies {}} \\ & &
    (\mname{complement}::\mname{C},,\mname{C})
     (\sembrack{e}) \IsUndefApp.
  \end{eqnarray*}
  \fi

%\newpage

  Compact notation: 
  \bi

    \item[] $\overline{a}$ means 
    $(\mname{complement}::\mname{C},\mname{C})(a)$.
  
  \ei

  \item \textbf{Head}\\
  Operator: 
  $(\mname{head}::\mname{V},\mname{V})$\\
  Defining axioms:
  \begin{eqnarray*}
  \lefteqn{\ForallApp w \mcolon \mname{V} \mdot 
     (\mname{head}::\mname{V},\mname{V})
     (w) \QuasiEqual} \\ & &
     \iotaApp u \mcolon\mname{V} \mdot
     \ForsomeApp v \mcolon\mname{V} \mdot
     w = \seq{u,v}.
  \end{eqnarray*}
  \vspace{-5ex} 
  \begin{eqnarray*}
    {a^{\rm ef} \IsUndefApp} \Implies 
    (\mname{head}::\mname{V},,\mname{V})(a^{\rm ef}) \IsUndefApp.
  \end{eqnarray*}

  \iffalse
  \begin{eqnarray*}
  \lefteqn{\ForallApp e \mcolon \mname{E}_{\rm te} \mdot 
    {\sembrack{e} \IsUndefApp} \Implies {}} \\ & &
    (\mname{head}::\mname{V},,\mname{V})
     (\sembrack{e}) \IsUndefApp.
  \end{eqnarray*}
  \fi

%\newpage

  Compact notation: 
  \bi

    \item[] $\mname{hd}(a)$ means 
    $(\mname{head}::\mname{V},\mname{V})(a)$.
  
  \ei  

  \item \textbf{Tail}\\
  Operator: 
  $(\mname{tail}::\mname{V},\mname{V})$\\
  Defining axioms:
  \begin{eqnarray*}
  \lefteqn{\ForallApp w \mcolon \mname{V} \mdot 
     (\mname{tail}::\mname{V},\mname{V})
     (w) \QuasiEqual} \\ & &
     \iotaApp v \mcolon\mname{V} \mdot
     \ForsomeApp u \mcolon\mname{V} \mdot
     w = \seq{u,v}.
  \end{eqnarray*}
  \vspace{-5ex}
  \begin{eqnarray*}
    {a^{\rm ef} \IsUndefApp} \Implies 
    (\mname{tail}::\mname{V},,\mname{V})(a^{\rm ef}) \IsUndefApp.
  \end{eqnarray*}

  \iffalse 
  \begin{eqnarray*}
  \lefteqn{\ForallApp e \mcolon \mname{E}_{\rm te} \mdot 
    {\sembrack{e} \IsUndefApp} \Implies {}} \\ & &
    (\mname{tail}::\mname{V},,\mname{V})
     (\sembrack{e}) \IsUndefApp.
  \end{eqnarray*}
  \fi

%\newpage

  Compact notation: 
  \bi

    \item[] $\mname{tl}(a)$ means 
    $(\mname{tail}::\mname{V},\mname{V})(a)$.
  
  \ei 

  \item \textbf{Append}\\
  Operator: 
  $(\mname{append}::\mname{V},\mname{V},\mname{V})$\\
  Defining axioms:
  \begin{eqnarray*}
  \lefteqn{\ForallApp x,y \mcolon \mname{V} \mdot 
     (\mname{append}::\mname{V},\mname{V},\mname{V})
     (x,y) \QuasiEqual} \\ & & 
     \mname{if}(x = \emptyset, 
     y,
     \seq{\mname{hd}(x),
     \mname{append}(\mname{tl}(x),y)}).     
  \end{eqnarray*}
  \vspace{-5ex}   
  \begin{eqnarray*}
     ({a^{\rm ef} \IsUndefApp} \Or {b^{\rm ef} \IsUndefApp}) 
     \Implies 
     (\mname{append}::\mname{V},\mname{V},\mname{V})
     (a^{\rm ef}, b^{\rm ef}) \IsUndefApp.
  \end{eqnarray*}

  \iffalse 
  \begin{eqnarray*}
  \lefteqn{\ForallApp e_1,e_2 \mcolon \mname{E}_{\rm te} \mdot 
     ({\sembrack{e_1} \IsUndefApp} \Or {\sembrack{e_2} \IsUndefApp}) 
     \Implies {}} \\ & &
     (\mname{append}::\mname{V},\mname{V},\mname{V})
     (\sembrack{e_1},\sembrack{e_2}) \IsUndefApp.
  \end{eqnarray*}
  \fi

%\newpage

  Compact notation: 
  \bi

    \item[] $a {\verb+^+} b$ means 
    $(\mname{append}::\mname{V},\mname{V},\mname{V})(a,b)$.
  
  \ei 

  \item \textbf{In List}\\
  Operator: 
  $(\mname{in-list}::\mname{V},\mname{V},\mname{formula})$\\
  Defining axioms:
  \begin{eqnarray*}
  \lefteqn{\ForallApp x,y \mcolon \mname{V} \mdot 
     (\mname{in-list}::\mname{V},\mname{V},\mname{formula})
     (x,y) \Iff} \\ & & 
     x = \mname{hd}(y) \Or \mname{in-list}(x,\mname{tl}(y)).
  \end{eqnarray*}
  \vspace{-5ex}   
  \begin{eqnarray*}
     ({a^{\rm ef} \IsUndefApp} \Or {b^{\rm ef} \IsUndefApp}) 
     \Implies 
     (\mname{in-list}::\mname{V},\mname{V},\mname{formula})
     (a^{\rm ef}, b^{\rm ef}) \Iff \mname{F}.
  \end{eqnarray*}

  \iffalse 
  \begin{eqnarray*}
  \lefteqn{\ForallApp e_1,e_2 \mcolon \mname{E}_{\rm te} \mdot 
     ({\sembrack{e_1} \IsUndefApp} \Or {\sembrack{e_2} \IsUndefApp}) 
     \Implies {}} \\ & &
     (\mname{append}::\mname{V},\mname{V},\mname{V})
     (\sembrack{e_1},\sembrack{e_2}) \IsUndefApp.
  \end{eqnarray*}
  \fi

%\newpage

  Compact notation: 
  \bi

    \item[] $a \in_{\rm li} b$ means 
    $(\mname{in-list}::\mname{V},\mname{V},\mname{formula})(a,b)$.
  
  \ei 

  \item \textbf{List Type}\\
  Operator: 
  $(\mname{list-type}::\mname{type},\mname{type})$\\
  Defining axioms:
  \begin{eqnarray*}
     \lefteqn{\Neg\mname{free-in}(\synbrack{x},
     \synbrack{\alpha^{\rm ef}}) \Implies {}}  \\ & &
     \ForallApp x \mcolon \mname{C} \mdot 
     x \IsDef (\mname{list-type}::\mname{type},\mname{type})
     (\alpha^{\rm ef}) \Iff \\ & &
     \hspace{4ex} x = [\,] \Or
     \ForsomeApp y \mcolon \alpha^{\rm ef}, 
     z \mcolon \mname{list-type}(\alpha^{\rm ef}) \mdot
     x = \seq{y,z}.
  \end{eqnarray*}

  \item \textbf{Type Union}\\
  Operator: 
  $(\mname{type-union}::\mname{type},\mname{type},\mname{type})$\\
  Defining axioms:
  \begin{eqnarray*}
     \lefteqn{(\Neg\mname{free-in}(\synbrack{x},
     \synbrack{\alpha^{\rm ef}}) \And 
     \Neg\mname{free-in}(\synbrack{x},\synbrack{\beta^{\rm ef}}))
     \Implies {}}  \\ & &
     \ForallApp x \mcolon \mname{C} \mdot 
     x \IsDef (\mname{type-union}::\mname{type},\mname{type},\mname{type})
     (\alpha^{\rm ef}, \beta^{\rm ef}) \Iff \\ & &
     \hspace{4ex}{x \IsDef \alpha^{\rm ef}} \Or {x \IsDef \beta^{\rm ef}}.
  \end{eqnarray*}

  Compact notation: 
  \bi

    \item[] $(\alpha \cup \beta)$ means 
    $(\mname{type-union}::\mname{type},\mname{type},\mname{type})
    (\alpha,\beta)$.

  \ei 

  \item \textbf{Type Intersection}\\
  Operator: 
  $(\mname{type-intersection}::\mname{type},\mname{type},\mname{type})$\\
  Defining axioms:
  \begin{eqnarray*}
     \lefteqn{(\Neg\mname{free-in}(\synbrack{x},
     \synbrack{\alpha^{\rm ef}}) \And 
     \Neg\mname{free-in}(\synbrack{x},\synbrack{\beta^{\rm ef}}))
     \Implies {}}  \\ & &
     \ForallApp x \mcolon \mname{C} \mdot 
     x \IsDef (\mname{type-intersection}::
     \mname{type},\mname{type},\mname{type})
     (\alpha^{\rm ef}, \beta^{\rm ef}) \Iff \\ & &
     \hspace{4ex}{x \IsDef \alpha^{\rm ef}} \And {x \IsDef \beta^{\rm ef}}.
  \end{eqnarray*}

%\newpage

  Compact notation: 
  \bi

    \item[] $(\alpha \cap \beta)$ means 
    $(\mname{type-intersection}::\mname{type},\mname{type},\mname{type})
    (\alpha,\beta)$.

  \ei 

  \item \textbf{Type Complement}\\
  Operator: 
  $(\mname{type-complement}::\mname{type},\mname{type})$\\
  Defining axioms:
  \begin{eqnarray*}
     \lefteqn{\Neg\mname{free-in}(\synbrack{x},
     \synbrack{\alpha^{\rm ef}}) \Implies {}}  \\ & &
     \ForallApp x \mcolon \mname{C} \mdot 
     x \IsDef (\mname{type-complement}::\mname{type},\mname{type})
     (\alpha^{\rm ef}) \Iff \\ & &
     \hspace{4ex}{x \IsUndef \alpha^{\rm ef}}.
  \end{eqnarray*}

\newpage

  Compact notation: 
  \bi

    \item[] $\overline{\alpha}$ means 
    $(\mname{type-complement}::\mname{type},\mname{type})(\alpha)$.

  \ei 

  \item \textbf{Type Product}\\
  Operator: 
  $(\mname{type-prod}::\mname{type},\mname{type},\mname{type})$\\
  Defining axioms:
  \begin{eqnarray*}
     \lefteqn{(\Neg\mname{free-in}(\synbrack{x},
     \synbrack{\alpha^{\rm ef}}) \And 
     \Neg\mname{free-in}(\synbrack{x},\synbrack{\beta^{\rm ef}}))
     \Implies {}}  \\ & &
     \ForallApp x \mcolon \mname{C} \mdot 
     x \IsDef (\mname{type-prod}::\mname{type},\mname{type},\mname{type})
     (\alpha^{\rm ef}, \beta^{\rm ef}) \Iff \\ & &
     \hspace{4ex}{\mname{hd}(x) \IsDef \alpha^{\rm ef}} \And
     {\mname{tl}(x) \IsDef \beta^{\rm ef}}.
  \end{eqnarray*}

  \iffalse
  \begin{eqnarray*}
  \lefteqn{\ForallApp e_1,e_2 \mcolon \mname{E}_{\rm ty}, 
     \ForallApp e \mcolon \mname{E}_{\rm sy} \mdot} \\ & &
     (\Neg\mname{free-in}(e,e_1) \And \Neg\mname{free-in}(e,e_2))
     \Implies  \\ & &
     \sembrack{\synbrack{
     \ForallApp \commabrack{e} \mcolon \mname{C} \mdot \\ & &
     \hspace{4ex}(\commabrack{e} \mcolon \mname{C})
     \IsDef (\mname{type-prod}::\mname{type},\mname{type},\mname{type})
     (\commabrack{e_1},\commabrack{e_2}) \Iff \\ & &
     \hspace{4ex}\mname{hd}(\commabrack{e} \mcolon \mname{C}) 
     \IsDef \commabrack{e_1} \And
     \mname{tl}(\commabrack{e} \mcolon \mname{C}) 
     \IsDef \commabrack{e_2}}}_{\rm fo}.
  \end{eqnarray*}
  \fi

%\newpage

  Compact notation: 
  \bi

    \item[] $(\alpha \times \beta)$ means 
    $(\mname{type-prod}::\mname{type},\mname{type},\mname{type})
    (\alpha,\beta)$.

  \ei 

  \item \textbf{Type Sum}\\
  Operator: 
  $(\mname{type-sum}::\mname{type},\mname{type},\mname{type})$\\
  Defining axioms:
  \begin{eqnarray*}
     \lefteqn{(\Neg\mname{free-in}(\synbrack{x},
     \synbrack{\alpha^{\rm ef}}) \And 
     \Neg\mname{free-in}(\synbrack{x},\synbrack{\beta^{\rm ef}}))
     \Implies {}}  \\ & &
     \ForallApp x \mcolon \mname{C} \mdot 
     x \IsDef (\mname{type-sum}::\mname{type},\mname{type},\mname{type})
     (\alpha^{\rm ef}, \beta^{\rm ef}) \Iff \\ & &
     \hspace{4ex}({\mname{hd}(x) \IsDef \alpha^{\rm ef}} \And
     {\mname{tl}(x) = \emptyset}) \Or
     ({\mname{hd}(x) \IsDef \beta^{\rm ef}} \And
     {\mname{tl}(x) = \set{\emptyset}}).
  \end{eqnarray*}

  \iffalse
  \begin{eqnarray*}
  \lefteqn{\ForallApp e_1,e_2 \mcolon \mname{E}_{\rm ty}, 
     \ForallApp e \mcolon \mname{E}_{\rm sy} \mdot} \\ & &
     (\Neg\mname{free-in}(e,e_1) \And \Neg\mname{free-in}(e,e_2))
     \Implies  \\ & &
     \sembrack{\synbrack{
     \ForallApp \commabrack{e} \mcolon \mname{C} \mdot \\ & &
     \hspace{4ex}(\commabrack{e} \mcolon \mname{C}) 
     \IsDef (\mname{type-sum}::\mname{type},\mname{type},\mname{type})
     (\commabrack{e_1},\commabrack{e_2}) \Iff \\ & &
     \hspace{4ex}(\mname{hd}(\commabrack{e} \mcolon \mname{C}) 
     \IsDef \commabrack{e_1} \And
     \mname{tl}(\commabrack{e} \mcolon \mname{C}) 
     = \emptyset) \Or \\ & & 
     \hspace{4ex}(\mname{hd}(\commabrack{e} \mcolon \mname{C}) 
     \IsDef \commabrack{e_2} \And
     \mname{tl}(\commabrack{e} \mcolon \mname{C}) 
     = \set{\emptyset})}}_{\rm fo}.
  \end{eqnarray*}
  \fi

%\newpage

  Compact notation: 
  \bi

    \item[] $(\alpha + \beta)$ means 
    $(\mname{type-sum}::\mname{type},\mname{type},\mname{type})
    (\alpha,\beta)$.

  \ei 

  \item \textbf{Binary Relation}\\
  Operator: 
  $(\mname{bin-rel}::\mname{C},\mname{formula})$\\
  Defining axioms:
  \begin{eqnarray*}
  \lefteqn{\ForallApp x \mcolon \mname{C} \mdot 
     (\mname{bin-rel}::\mname{C},\mname{formula})
     (x) \Iff} \\ & &
     {\ForallApp w \mcolon \mname{V} \mdot 
     w \in x \Implies 
     \ForsomeApp u,v \mcolon \mname{V} \mdot w = \seq{u,v}}.
  \end{eqnarray*}
  \vspace{-5ex} 
  \begin{eqnarray*}
     {a^{\rm ef} \IsUndefApp} \Implies {}
     (\mname{bin-rel}::\mname{C},\mname{formula})
     (a^{\rm ef}) \Iff \mname{F}.
  \end{eqnarray*}

  \iffalse
  \begin{eqnarray*}
  \lefteqn{\ForallApp e\mcolon \mname{E}_{\rm te} \mdot 
     {\sembrack{e} \IsUndefApp} \Implies {}} \\ & &
     (\mname{bin-rel}::\mname{C},\mname{formula})
     (\sembrack{e}) \Iff \mname{F}.
  \end{eqnarray*}
  \fi

  \item \textbf{Univocal}\\
  Operator: 
  $(\mname{univocal}::\mname{C},\mname{formula})$\\
  Defining axioms:
  \begin{eqnarray*}
  \lefteqn{\ForallApp x \mcolon \mname{C} \mdot 
     (\mname{univocal}::\mname{C},\mname{formula})(x) \Iff} \\ & &
     {\ForallApp u,v,v' \mcolon \mname{V} \mdot 
     (\seq{u,v} \in x \And \seq{u,v'} \in x) \Implies v = v'}.
  \end{eqnarray*}
  \vspace{-5ex}  
  \begin{eqnarray*}
     {a^{\rm ef} \IsUndefApp} \Implies {}
     (\mname{univocal}::\mname{C},\mname{formula})
     (a^{\rm ef}) \Iff \mname{F}.
  \end{eqnarray*}

  \iffalse
  \begin{eqnarray*}
  \lefteqn{\ForallApp e\mcolon \mname{E}_{\rm te} \mdot 
     {\sembrack{e} \IsUndefApp} \Implies {}} \\ & &
     (\mname{univocal}::\mname{C},\mname{formula})
     (\sembrack{e}) \Iff \mname{F}.
  \end{eqnarray*}
  \fi

  \item \textbf{Function}\\
  Operator: 
  $(\mname{fun}::\mname{C},\mname{formula})$\\
  Defining axioms:
  \begin{eqnarray*}
  \lefteqn{\ForallApp x \mcolon \mname{C} \mdot 
     (\mname{fun}::\mname{C},\mname{formula})(x) \Iff} \\ & &
     \mname{bin-rel}(x) \And \mname{univocal}(x).
  \end{eqnarray*}
  \vspace{-5ex}  
  \begin{eqnarray*}
     {a^{\rm ef} \IsUndefApp} \Implies {}
     (\mname{fun}::\mname{C},\mname{formula})
     (a^{\rm ef}) \Iff \mname{F}.
  \end{eqnarray*}

  \iffalse
  \begin{eqnarray*}
  \lefteqn{\ForallApp e\mcolon \mname{E}_{\rm te} \mdot 
     {\sembrack{e} \IsUndefApp} \Implies {}} \\ & &
     (\mname{fun}::\mname{C},\mname{formula})
     (\sembrack{e}) \Iff \mname{F}.
  \end{eqnarray*}
  \fi

  \item \textbf{Domain of a Class}\\
  Operator: 
  $(\mname{dom}::\mname{C},\mname{C})$\\
  Defining axioms:
  \begin{eqnarray*}
  \lefteqn{\ForallApp x \mcolon \mname{C} \mdot 
     (\mname{dom}::\mname{C},\mname{C})(x) \QuasiEqual} \\ & &
     \iotaApp y \mcolon \mname{C} \mdot
     \ForallApp u \mcolon \mname{V} \mdot 
     u \in y \Iff (\ForsomeApp v \mcolon \mname{V} \mdot 
     \seq{u,v} \in x).
  \end{eqnarray*}
  \vspace{-5ex}  
  \begin{eqnarray*}
     {a^{\rm ef} \IsUndefApp} \Implies {}
     (\mname{dom}::\mname{C},\mname{C})
     (a^{\rm ef}) \IsUndefApp.
  \end{eqnarray*}

  \iffalse
  \begin{eqnarray*}
  \lefteqn{\ForallApp e \mcolon \mname{E}_{\rm te} \mdot 
    {\sembrack{e} \IsUndefApp} \Implies {}} \\ & &
    (\mname{dom}::\mname{C},\mname{C})
     (\sembrack{e}) \IsUndefApp.
  \end{eqnarray*}
  \fi

  \item \textbf{Range of a Class}\\
  Operator: 
  $(\mname{ran}::\mname{C},\mname{C})$\\
  Defining axioms:
  \begin{eqnarray*}
  \lefteqn{\ForallApp x \mcolon \mname{C} \mdot 
     (\mname{ran}::\mname{C},\mname{C})(x) \QuasiEqual} \\ & &
     \iotaApp y \mcolon \mname{C} \mdot
     \ForallApp u \mcolon \mname{V} \mdot 
     u \in y \Iff (\ForsomeApp v \mcolon \mname{V} \mdot 
     \seq{v,u} \in x).
  \end{eqnarray*}
  \vspace{-5ex} 
  \begin{eqnarray*}
     {a^{\rm ef} \IsUndefApp} \Implies {}
     (\mname{ran}::\mname{C},\mname{C})
     (a^{\rm ef}) \IsUndefApp.
  \end{eqnarray*}

  \iffalse
  \begin{eqnarray*}
  \lefteqn{\ForallApp e \mcolon \mname{E}_{\rm te} \mdot 
    {\sembrack{e} \IsUndefApp} \Implies {}} \\ & &
    (\mname{ran}::\mname{C},\mname{C})
     (\sembrack{e}) \IsUndefApp.
  \end{eqnarray*}
  \fi

  \item \textbf{Total Function on a Class}\\
  Operator: 
  $(\mname{total}::\mname{C},\mname{C},\mname{formula})$\\
  Defining axioms:
  \begin{eqnarray*}
  \lefteqn{\ForallApp f,x \mcolon \mname{C} \mdot 
     (\mname{total}::\mname{C},\mname{C}, \mname{formula})(f,x) \Iff} \\ & &
     \mname{fun}(f) \And \mname{dom}(f) = x.
  \end{eqnarray*}
  \vspace{-5ex}  
  \begin{eqnarray*}
     ({a^{\rm ef} \IsUndefApp} \And 
     {b^{\rm ef} \IsUndefApp}) \Implies {}
     (\mname{total}::\mname{C},\mname{C},\mname{formula})
     (a^{\rm ef}, b^{\rm ef}) \Iff \mname{F}.
  \end{eqnarray*}

  \iffalse
  \begin{eqnarray*}
  \lefteqn{\ForallApp e_1,e_2\mcolon \mname{E}_{\rm te} \mdot 
     {({\sembrack{e_1} \IsUndefApp} \And {\sembrack{e_2} \IsUndefApp})}
     \Implies {}} \\ & &
     (\mname{total}::\mname{C},\mname{C},\mname{formula})
     (\sembrack{e_1},\sembrack{e_2}) \Iff \mname{F}.
  \end{eqnarray*}
  \fi

  \item \textbf{Surjective Function on a Class}\\
  Operator: 
  $(\mname{surjective}::\mname{C},\mname{C},\mname{formula})$\\
  Defining axioms:
  \begin{eqnarray*}
  \lefteqn{\ForallApp f,y \mcolon \mname{C} \mdot 
     (\mname{surjective}::\mname{C},\mname{C},
     \mname{formula})(f,y) \Iff} \\ & &
     \mname{fun}(f) \And \mname{ran}(f) = y.
  \end{eqnarray*}
  \vspace{-5ex}  
  \begin{eqnarray*}
     ({a^{\rm ef} \IsUndefApp} \And 
     {b^{\rm ef} \IsUndefApp}) \Implies {}
     (\mname{surjective}::\mname{C},\mname{C},\mname{formula})
     (a^{\rm ef}, b^{\rm ef}) \Iff \mname{F}.
  \end{eqnarray*}

  \iffalse
  \begin{eqnarray*}
  \lefteqn{\ForallApp e_1,e_2\mcolon \mname{E}_{\rm te} \mdot 
     {({\sembrack{e_1} \IsUndefApp} \And {\sembrack{e_2} \IsUndefApp})}
     \Implies {}} \\ & &
     (\mname{surjective}::\mname{C},\mname{C},\mname{formula})
     (\sembrack{e_1},\sembrack{e_2}) \Iff \mname{F}.
  \end{eqnarray*}
  \fi

  \item \textbf{Injective Function on a Class}\\
  Operator: 
  $(\mname{injective}::\mname{C},\mname{C},\mname{formula})$\\
  Defining axioms:
  \begin{eqnarray*}
  \lefteqn{\ForallApp f,x \mcolon \mname{C} \mdot 
     (\mname{injective}::\mname{C},\mname{C},
     \mname{formula})(f,x) \Iff} \\ & &
     \mname{fun}(f) \And 
     \Forall u,v \mcolon \mname{V} \mdot 
     (u \in x \And v \in x \And f(u) = f(v)) \Implies u = v.
  \end{eqnarray*}
  \vspace{-5ex}  
  \begin{eqnarray*}
     ({a^{\rm ef} \IsUndefApp} \And 
     {b^{\rm ef} \IsUndefApp}) \Implies {}
     (\mname{injective}::\mname{C},\mname{C},\mname{formula})
     (a^{\rm ef}, b^{\rm ef}) \Iff \mname{F}.
  \end{eqnarray*}
 
  \iffalse
  \begin{eqnarray*}
  \lefteqn{\ForallApp e_1,e_2\mcolon \mname{E}_{\rm te} \mdot 
     {({\sembrack{e_1} \IsUndefApp} \And {\sembrack{e_2} \IsUndefApp})}
     \Implies {}} \\ & &
     (\mname{injective}::\mname{C},\mname{C},\mname{formula})
     (\sembrack{e_1},\sembrack{e_2}) \Iff \mname{F}.
  \end{eqnarray*}
  \fi

  \item \textbf{Bijective Function from a Class to a Class}\\
  Operator: 
  $(\mname{bijective}::\mname{C},\mname{C},\mname{C},\mname{formula})$\\
  Defining axioms:
  \begin{eqnarray*}
  \lefteqn{\ForallApp f,x,y \mcolon \mname{C} \mdot 
     (\mname{bijective}::\mname{C},\mname{C},\mname{C},
     \mname{formula})(f,x,y) \Iff} \\ & &
     \mname{fun}(f) \And 
     \mname{total}(f,x) \And
     \mname{surjective}(f,y) \And
     \mname{injective}(f,x).
  \end{eqnarray*}
  \vspace{-5ex} 
  \begin{eqnarray*}
  \lefteqn{({a^{\rm ef} \IsUndefApp} \And 
     {b^{\rm ef} \IsUndefApp} \And 
     {c^{\rm ef} \IsUndefApp}) \Implies {}} \\ & &
     (\mname{bijective}::\mname{C},\mname{C},\mname{C},\mname{formula})
     (a^{\rm ef}, b^{\rm ef}, c^{\rm ef}) \Iff \mname{F}.
  \end{eqnarray*}

  \iffalse 
  \begin{eqnarray*}
  \lefteqn{\ForallApp e_1,e_2,e_3\mcolon \mname{E}_{\rm te} \mdot 
     {({\sembrack{e_1} \IsUndefApp} \And {\sembrack{e_2} \IsUndefApp}
     \And {\sembrack{e_3} \IsUndefApp})}
     \Implies {}} \\ & &
     (\mname{bijective}::\mname{C},\mname{C},\mname{C},\mname{formula})
     (\sembrack{e_1},\sembrack{e_2},\sembrack{e_3}) \Iff \mname{F}.
  \end{eqnarray*}
  \fi

  \item \textbf{Infinite Class}\\
  Operator: 
  $(\mname{infinite}::\mname{C},\mname{formula})$\\
  Defining axioms:
  \begin{eqnarray*}
  \lefteqn{\ForallApp x \mcolon \mname{C} \mdot 
     (\mname{infinite}::\mname{C},\mname{formula})(x) \Iff} \\ & &
     \ForsomeApp f,y \mcolon \mname{C} \mdot 
     y \subset x \And \mname{bijective}(f,x,y).
  \end{eqnarray*}
  \vspace{-5ex}  
  \begin{eqnarray*}
     {a^{\rm ef} \IsUndefApp} \Implies {}
     (\mname{infinite}::\mname{C},\mname{formula})
     (a^{\rm ef}) \Iff \mname{F}.
  \end{eqnarray*}

  \iffalse
  \begin{eqnarray*}
  \lefteqn{\ForallApp e\mcolon \mname{E}_{\rm te} \mdot 
     {\sembrack{e} \IsUndefApp}
     \Implies {}} \\ & &
     (\mname{infinite}::\mname{C},\mname{formula})
     (\sembrack{e}) \Iff \mname{F}.
  \end{eqnarray*}
  \fi

  \item \textbf{Countably Infinite Class}\\
  Operator: 
  $(\mname{countably-infinite}::\mname{C},\mname{formula})$\\
  Defining axioms:
  \begin{eqnarray*}
  \lefteqn{\ForallApp x \mcolon \mname{C} \mdot 
     (\mname{countably-infinite}::\mname{C},\mname{formula})(x) \Iff} \\ & &
     \mname{infinite}(x) \And {}(\ForallApp y \mcolon \mname{C} \mdot 
     \mname{infinite}(y) \Implies  \\ & &
     \hspace{4ex}
     (\Forsome y' \mcolon \mname{C} \mdot y' \subseteq y
     \Implies \mname{bijective}(f,x,y'))).
  \end{eqnarray*}
  \vspace{-5ex}  
  \begin{eqnarray*}
     {a^{\rm ef} \IsUndefApp} \Implies {}
     (\mname{countably-infinite}::\mname{C},\mname{formula})
     (a^{\rm ef}) \Iff \mname{F}.
  \end{eqnarray*}

  \iffalse
  \begin{eqnarray*}
  \lefteqn{\ForallApp e\mcolon \mname{E}_{\rm te} \mdot 
     {\sembrack{e} \IsUndefApp}
     \Implies {}} \\ & &
     (\mname{countably-infinite}::\mname{C},\mname{formula})
     (\sembrack{e}) \Iff \mname{F}.
  \end{eqnarray*}
  \fi

  \item \textbf{Sum Class}\\
  Operator: 
  $(\mname{sum}::\mname{C},\mname{C})$\\
  Defining axioms:
  \begin{eqnarray*}
  \lefteqn{\ForallApp x \mcolon \mname{C} \mdot 
     (\mname{sum}::\mname{C},\mname{C})(x) \QuasiEqual} \\ & &
     \iotaApp y \mcolon \mname{C} \mdot
     \ForallApp u \mcolon \mname{V} \mdot 
     u \in y \Iff (\ForsomeApp v \mcolon \mname{V} \mdot 
     u \in v \And v \in x).
  \end{eqnarray*}
  \vspace{-5ex} 
   \begin{eqnarray*}
     {a^{\rm ef} \IsUndefApp} \Implies {}
     (\mname{sum}::\mname{C},\mname{C})
     (a^{\rm ef}) \IsUndefApp.
  \end{eqnarray*}

  \iffalse
  \begin{eqnarray*}
  \lefteqn{\ForallApp e \mcolon \mname{E}_{\rm te} \mdot 
    {\sembrack{e} \IsUndefApp} \Implies {}} \\ & &
    (\mname{sum}::\mname{C},\mname{C})
     (\sembrack{e}) \IsUndefApp.
  \end{eqnarray*}
  \fi

  \item \textbf{Power Class}\\
  Operator: 
  $(\mname{power}::\mname{C},\mname{C})$\\
  Defining axioms:
  \begin{eqnarray*}
  \lefteqn{\ForallApp x \mcolon \mname{C} \mdot 
     (\mname{power}::\mname{C},\mname{C})(x) \QuasiEqual} \\ & &
     \iotaApp y \mcolon \mname{C} \mdot
     \ForallApp u \mcolon \mname{V} \mdot 
     u \in y \Iff u \subseteq x.
  \end{eqnarray*}
  \vspace{-5ex} 
   \begin{eqnarray*}
     {a^{\rm ef} \IsUndefApp} \Implies {}
     (\mname{power}::\mname{C},\mname{C})
     (a^{\rm ef}) \IsUndefApp.
  \end{eqnarray*}

  \iffalse
  \begin{eqnarray*}
  \lefteqn{\ForallApp e \mcolon \mname{E}_{\rm te} \mdot 
    {\sembrack{e} \IsUndefApp} \Implies {}} \\ & &
    (\mname{power}::\mname{C},\mname{C})
     (\sembrack{e}) \IsUndefApp.
  \end{eqnarray*}
  \fi

  \item \textbf{Type is Class Checker}\\
  Operator: 
  $(\mname{type-is-class}::\mname{type},\mname{formula})$\\
  Defining axioms:
  \begin{eqnarray*}
  \lefteqn{(\Neg\mname{free-in}(\synbrack{x},\synbrack{\alpha^{\rm ef}}) \And 
     \Neg\mname{free-in}(\synbrack{y},\synbrack{\alpha^{\rm ef}}))
     \Implies {}} \\ & &
     (\mname{type-is-class}::\mname{type},
     \mname{formula})(\alpha^{\rm ef}) \Iff \\ & &
     \hspace{4ex}
     \ForsomeApp x \mcolon \mname{C} \mdot
     \ForallApp y \mcolon \mname{C} \mdot 
     y \IsDef \alpha^{\rm ef} \Iff y \in x.
  \end{eqnarray*}

  \iffalse
  \begin{eqnarray*}
  \lefteqn{\ForallApp e_1 \mcolon \mname{E}_{\rm ty} \mdot
     (\mname{type-is-class}::\mname{type},\mname{formula})(\sembrack{e_1})
     \Iff} \\ & &
     \ForallApp e, e' \mcolon \mname{E}_{\rm sy} \mdot 
     (\Neg\mname{free-in}(e,e_1) \And \Neg\mname{free-in}(e',e_1))
     \Implies  \\ & &
     \hspace{4ex}\sembrack{\synbrack{
     \ForsomeApp \commabrack{e} \mcolon \mname{C} \mdot
     \ForallApp \commabrack{e'} \mcolon \mname{C} \mdot \\ & & 
     \hspace{6ex}
     (\commabrack{e'} \mcolon \mname{C})\IsDef \commabrack{e_1} \Iff
     (\commabrack{e'} \mcolon \mname{C}) 
     \in (\commabrack{e} \mcolon \mname{C})}}_{\rm fo}.
  \end{eqnarray*}
  \fi

  \item \textbf{Type is Set Checker}\\
  Operator: 
  $(\mname{type-is-set}::\mname{type},\mname{formula})$\\
  Defining axioms:
  \begin{eqnarray*}
  \lefteqn{(\Neg\mname{free-in}(\synbrack{u},\synbrack{\alpha^{\rm ef}}) \And 
     \Neg\mname{free-in}(\synbrack{y},\synbrack{\alpha^{\rm ef}}))
     \Implies {}} \\ & &
     (\mname{type-is-set}::\mname{type},
     \mname{formula})(\alpha^{\rm ef}) \Iff \\ & &
     \hspace{4ex}
     \ForsomeApp u \mcolon \mname{V} \mdot
     \ForallApp y \mcolon \mname{C} \mdot 
     y \IsDef \alpha^{\rm ef} \Iff y \in u.
  \end{eqnarray*}

  \iffalse
  \begin{eqnarray*}
  \lefteqn{\ForallApp e_1 \mcolon \mname{E}_{\rm ty} \mdot
     (\mname{type-is-set}::\mname{type},\mname{formula})(\sembrack{e_1})
     \Iff} \\ & &
     \ForallApp e, e' \mcolon \mname{E}_{\rm sy} \mdot 
     (\Neg\mname{free-in}(e,e_1) \And \Neg\mname{free-in}(e',e_1))
     \Implies  \\ & &
     \hspace{4ex}\sembrack{\synbrack{
     \ForsomeApp \commabrack{e} \mcolon \mname{V} \mdot
     \ForallApp \commabrack{e'} \mcolon \mname{C} \mdot  \\ & & 
     \hspace{6ex}
     (\commabrack{e'} \mcolon \mname{C})\IsDef \commabrack{e_1} \Iff
     (\commabrack{e'} \mcolon \mname{C}) 
     \in (\commabrack{e} \mcolon \mname{V})}}_{\rm fo}.
  \end{eqnarray*}
  \fi

  \item \textbf{Type to Term}\\
  Operator: 
  $(\mname{type-to-term}::\mname{type},\mname{C})$\\
  Defining axioms:
  \begin{eqnarray*}
  \lefteqn{(\Neg\mname{free-in}(\synbrack{x},\synbrack{\alpha^{\rm ef}}) 
     \And 
     \Neg\mname{free-in}(\synbrack{y},\synbrack{\alpha^{\rm ef}}))
     \Implies}  \\ & &
     (\mname{type-to-term}::\mname{type},\mname{C})
     (\alpha^{\rm ef}) \QuasiEqual \\ & &
     \hspace{4ex}\iotaApp x \mcolon \mname{C} \mdot
     \ForallApp y \mcolon \mname{C} \mdot
     y \IsDef \alpha^{\rm ef} \Iff y \in x.
  \end{eqnarray*}

  Compact notation: 
  \bi

    \item[] $\mname{term}(\alpha)$ means 
    $(\mname{type-to-term}::\mname{type},\mname{C})(\alpha)$.

  \ei 

  \iffalse
  \begin{eqnarray*}
  \lefteqn{\ForallApp e_1 \mcolon \mname{E}_{\rm ty}, 
     \ForallApp e, e' \mcolon \mname{E}_{\rm sy} \mdot} \\ & &
     (\Neg\mname{free-in}(e,e_1) \And \Neg\mname{free-in}(e',e_1))
     \Implies  \\ & &
     \sembrack{\synbrack{
     (\mname{type-to-term}::\mname{type},\mname{C})
     (\commabrack{e_1}) \QuasiEqual \\ & &
     \hspace{4ex}\iotaApp \commabrack{e} \mcolon \mname{C} \mdot
     \ForallApp \commabrack{e'} \mcolon \mname{C} \mdot  \\ & & 
     \hspace{6ex}
     (\commabrack{e'} \mcolon \mname{C})\IsDef \commabrack{e_1} \Iff
     (\commabrack{e'} \mcolon \mname{C})
     \in (\commabrack{e} \mcolon \mname{C})}}_{\rm fo}.
  \end{eqnarray*}
  \fi

  \item \textbf{Term to Type}\\
  Operator: 
  $(\mname{term-to-type}::\mname{C},\mname{type})$\\
  Defining axioms:
  \begin{eqnarray*}
     \ForallApp x \mcolon \mname{C} \mdot
     \ForallApp y \mcolon \mname{C} \mdot 
     y \IsDef (\mname{term-to-type}::\mname{C},\mname{type})(x) \Iff 
     y \in x.
  \end{eqnarray*}
  \vspace{-5ex}
  \begin{eqnarray*}
     {a^{\rm ef} \IsUndefApp} \Implies {}
     (\mname{term-to-type}::\mname{C},\mname{type})
     (a^{\rm ef}) \TypeEqual \mname{C}.
  \end{eqnarray*}

  Compact notation: 
  \bi

    \item[] $\mname{type}(a)$ means 
    $(\mname{term-to-term}::\mname{C},\mname{type})(a)$.

  \ei 

  \iffalse 
  \begin{eqnarray*}
  \lefteqn{\ForallApp e \mcolon \mname{E}_{\rm te} \mdot 
    {\sembrack{e} \IsUndefApp} \Implies {}} \\ & &
    (\mname{term-to-type}::\mname{C},\mname{type})
    (\sembrack{e}) \TypeEqual \mname{C}.
  \end{eqnarray*}
  \fi

  \item \textbf{Power Type}\\
  Operator: 
  $(\mname{power-type}::\mname{type},\mname{type})$\\
  Defining axioms:
  \begin{eqnarray*}
     (\mname{power-type}::\mname{type},\mname{type})(\alpha^{\rm ef})
     \TypeEqual \mname{type}(\mname{power}(\mname{term}(\alpha^{\rm ef}))).
  \end{eqnarray*}

  \iffalse
  \begin{eqnarray*}
  \lefteqn{\ForallApp e \mcolon \mname{E}_{\rm ty} \mdot
     (\mname{power-type}::\mname{type},\mname{type})(\sembrack{e})
     \TypeEqual} \\ & &
     \mname{type}(\mname{power}(\mname{term}(\sembrack{e}))).
  \end{eqnarray*}
  \fi

\ee

\subsection{Syntactic Operators} \label{subsec:synop}

\be 

  \item \textbf{Proper Expression Checker}\\
  Operator: 
  $(\mname{is-p-expr}::\mname{E},\mname{formula})$\\
  Defining axioms:
  \begin{eqnarray*}
  \lefteqn{\ForallApp e \mcolon \mname{E} \mdot 
     (\mname{is-p-expr}::\mname{E},\mname{formula})(e) \Iff} \\ & &
     e \IsDef \mname{E}_{\rm op} \Or
     e \IsDef \mname{E}_{\rm ty} \Or
     e \IsDef \mname{E}_{\rm te} \Or
     e \IsDef \mname{E}_{\rm fo}.
  \end{eqnarray*}
  \vspace{-5ex}   
  \begin{eqnarray*}
     {a^{\rm ef} \IsUndefApp} \Implies {}
     (\mname{is-p-expr}::\mname{E},\mname{formula})
     (a^{\rm ef}) \Iff \mname{F}.
  \end{eqnarray*}

  \iffalse
  \begin{eqnarray*}
  \lefteqn{\ForallApp e \mcolon \mname{E}_{\rm te} \mdot 
    {\sembrack{e} \IsUndefApp} \Implies {}} \\ & &
    (\mname{is-p-expr}::\mname{E},\mname{formula})
    (\sembrack{e}) \Iff \mname{F}.
  \end{eqnarray*}
  \fi

  Note: Checkers for the different sorts of proper expressions are
  defined in a similar way: $\mname{is-op}$, $\mname{is-type}$,
  $\mname{is-term}$, $\mname{is-term-of-type}$, and
  $\mname{is-formula}$.

  \item \textbf{First Component Selector for a Proper Expression}\\
  Operator: 
  $(\mname{1st-comp}::\mname{E},\mname{E})$\\
  Defining axioms:
  \begin{eqnarray*}
  \lefteqn{\ForallApp e \mcolon \mname{E} \mdot 
     (\mname{1st-comp}::\mname{E},\mname{E})(e) \QuasiEqual} \\ & &
     \mname{if}(\mname{is-p-expr}(e),
     \mname{hd}(\mname{tl}(e)),
     \Undefined_{\sf C}).
  \end{eqnarray*}
  \vspace{-5ex}   
  \begin{eqnarray*}
     {a^{\rm ef} \IsUndefApp} \Implies {}
     (\mname{1st-comp}::\mname{E},\mname{E})
     (a^{\rm ef}) \IsUndefApp.
  \end{eqnarray*}

  \iffalse
  \begin{eqnarray*}
  \lefteqn{\ForallApp e \mcolon \mname{E}_{\rm te} \mdot 
    {\sembrack{e} \IsUndefApp} \Implies {}} \\ & &
    (\mname{1st-comp}::\mname{E},\mname{E})
    (\sembrack{e}) \IsUndefApp.
  \end{eqnarray*}
  \fi

  \bsp Note: The second, third, fourth, \ldots\ component selectors
  for proper expressions are defined in a similar way:
  $\mname{2nd-comp}$, $\mname{3rd-comp}$,
  $\mname{4th-comp}$,~\ldots. \esp

  \item \textbf{Operator Application Checker}\\
  Operator: 
  $(\mname{is-op-app}::\mname{E},\mname{formula})$\\
  Defining axioms:
  \begin{eqnarray*}
  \lefteqn{\ForallApp e \mcolon \mname{E} \mdot 
     (\mname{is-op-app}::\mname{E},\mname{formula})(e) \Iff} \\ & &
     \mname{is-p-expr}(e) \And \mname{hd}(e) = \synbrack{\mname{op-app}}.
  \end{eqnarray*}
  \vspace{-5ex}   
  \begin{eqnarray*}
     {a^{\rm ef} \IsUndefApp} \Implies {}
     (\mname{is-op-app}::\mname{E},\mname{formula})
     (a^{\rm ef}) \Iff \mname{F}.
  \end{eqnarray*}

  \iffalse
  \begin{eqnarray*}
  \lefteqn{\ForallApp e \mcolon \mname{E}_{\rm te} \mdot 
    {\sembrack{e} \IsUndefApp} \Implies {}} \\ & &
    (\mname{is-op-app}::\mname{E},\mname{formula})
    (\sembrack{e}) \Iff \mname{F}.
  \end{eqnarray*}
  \fi

  \bsp Note: Checkers for the other remaining 11 proper expression
  categories are defined in a similar way: \mname{is-var},
  \mname{is-type-app}, \mname{is-dep-fun-type}, \mname{is-fun-app},
  \mname{is-fun-abs}, \mname{is-if}, \mname{is-exists},
  \mname{is-def-des}, \mname{is-indef-des}, \mname{is-quote},
  \mname{is-eval}. \esp

  \item \textbf{Disjunction Checker}\\
  Operator: 
  $(\mname{is-or}::\mname{E},\mname{formula})$\\
  Defining axioms:
  \begin{eqnarray*}
  \lefteqn{\ForallApp e \mcolon \mname{E} \mdot 
     (\mname{is-or}::\mname{E},\mname{formula})(e) \Iff} \\ & &
     \mname{is-op-app}(e) \And 
     \mname{1st-comp}(\mname{1st-comp}(e)) = \synbrack{\mname{or}}.
  \end{eqnarray*}
  \vspace{-5ex}   
  \begin{eqnarray*}
     {a^{\rm ef} \IsUndefApp} \Implies {}
     (\mname{is-or}::\mname{E},\mname{formula})
     (a^{\rm ef}) \Iff \mname{F}.
  \end{eqnarray*}

  \iffalse
  \begin{eqnarray*}
  \lefteqn{\ForallApp e \mcolon \mname{E}_{\rm te} \mdot 
    {\sembrack{e} \IsUndefApp} \Implies {}} \\ & &
    (\mname{is-disj}::\mname{E},\mname{formula})
    (\sembrack{e}) \Iff \mname{F}.
  \end{eqnarray*}
  \fi

  Note: Checkers for other kinds of operator applications are defined
  in a similar way: \mname{is-in}, \mname{is-type-equal},
  \mname{is-union}, etc.

  \item \textbf{First Argument Selector for an Operator Application}\\
  Operator: 
  $(\mname{1st-arg}::\mname{E},\mname{E})$\\
  Defining axioms:
  \begin{eqnarray*}
  \lefteqn{\ForallApp e \mcolon \mname{E} \mdot 
     (\mname{1st-arg}::\mname{E},\mname{E})(e) \QuasiEqual} \\ & &
     \mname{if}(\mname{is-op-app}(e),
     \mname{2nd-comp}(e), 
     \Undefined_{\sf C}).
  \end{eqnarray*}
  \vspace{-5ex}   
  \begin{eqnarray*}
     {a^{\rm ef} \IsUndefApp} \Implies {}
     (\mname{1st-arg}::\mname{E},\mname{E})
     (a^{\rm ef}) \IsUndefApp.
  \end{eqnarray*}

  \iffalse
  \begin{eqnarray*}
  \lefteqn{\ForallApp e \mcolon \mname{E}_{\rm te} \mdot 
    {\sembrack{e} \IsUndefApp} \Implies {}} \\ & &
    (\mname{1st-arg}::\mname{E},\mname{E})
    (\sembrack{e}) \IsUndefApp.
  \end{eqnarray*}
  \fi

%\newpage

  Note: The second, third, fourth, $\ldots$ argument selectors for
  operator application are defined in a similar way: $\mname{2nd-arg},
  \mname{3rd-arg}, \mname{4th-arg}, \ldots.$

  \item \textbf{Variable Binder Checker}\\
  Operator: 
  $(\mname{is-binder}::\mname{E},\mname{formula})$\\
  Defining axioms:
  \begin{eqnarray*}
  \lefteqn{\ForallApp e \mcolon \mname{E} \mdot 
     (\mname{is-binder}::\mname{E},\mname{formula})(e) \Iff} \\ & &
     \mname{is-dep-fun-type}(e) \Or
     \mname{is-fun-abs}(e) \Or
     \mname{is-exists}(e) \Or \\ & &
     \mname{is-def-des}(e) \Or
     \mname{is-indef-des}(e).
  \end{eqnarray*}
  \vspace{-5ex}   
  \begin{eqnarray*}
     {a^{\rm ef} \IsUndefApp} \Implies {}
     (\mname{is-binder}::\mname{E},\mname{formula})
     (a^{\rm ef}) \Iff \mname{F}.
  \end{eqnarray*}

  \iffalse
  \begin{eqnarray*}
  \lefteqn{\ForallApp e \mcolon \mname{E}_{\rm te} \mdot 
    {\sembrack{e} \IsUndefApp} \Implies {}} \\ & &
    (\mname{is-binder}::\mname{E},\mname{formula})
    (\sembrack{e}) \Iff \mname{F}.
  \end{eqnarray*}
  \fi

  \item \textbf{Variable Selector for Variable Binder}\\
  Operator: 
  $(\mname{binder-var}::\mname{E},\mname{E})$\\
  Defining axioms:
  \begin{eqnarray*}
  \lefteqn{\ForallApp e \mcolon \mname{E} \mdot 
     (\mname{binder-var}::\mname{E},\mname{E})(e) \QuasiEqual} \\ & &
     \mname{if}(\mname{is-binder}(e),
     \mname{1st-comp}(e),
     \Undefined_{\sf C}).
  \end{eqnarray*}
  \vspace{-5ex}   
  \begin{eqnarray*}
     {a^{\rm ef} \IsUndefApp} \Implies {}
     (\mname{binder-var}::\mname{E},\mname{E})
     (a^{\rm ef}) \IsUndefApp.
  \end{eqnarray*}

  \iffalse
  \begin{eqnarray*}
  \lefteqn{\ForallApp e \mcolon \mname{E}_{\rm te} \mdot 
    {\sembrack{e} \IsUndefApp} \Implies {}} \\ & &
    (\mname{binder-var}::\mname{E},\mname{E})
    (\sembrack{e}) \IsUndefApp.
  \end{eqnarray*}
  \fi

  \bsp Note: Selectors for a binder name and a binder body are defined
  in a similar way: \mname{binder-name} and \mname{binder-body}. \esp

  \item \textbf{Function Redex Checker}\\
  Operator: 
  $(\mname{is-fun-redex}::\mname{E},\mname{formula})$\\
  Defining axioms:
  \begin{eqnarray*}
  \lefteqn{\ForallApp e \mcolon \mname{E} \mdot 
     (\mname{is-fun-redex}::\mname{E},\mname{formula})(e) \Iff} \\ & &
     \mname{is-fun-app}(e) \And \mname{is-fun-abs}(\mname{1st-comp}(e)).
  \end{eqnarray*}
  \vspace{-5ex}  
  \begin{eqnarray*}
     {a^{\rm ef} \IsUndefApp} \Implies {}
     (\mname{is-fun-redex}::\mname{E},\mname{formula})
     (a^{\rm ef}) \Iff \mname{F}.
  \end{eqnarray*}

  \iffalse 
  \begin{eqnarray*}
  \lefteqn{\ForallApp e \mcolon \mname{E}_{\rm te} \mdot 
    {\sembrack{e} \IsUndefApp} \Implies {}} \\ & &
    (\mname{is-fun-redex}::\mname{E},\mname{formula})
    (\sembrack{e}) \Iff \mname{F}.
  \end{eqnarray*}
  \fi

%\newpage

  \bsp Note: A dependent function type redex checker named
  \mname{is-dep-fun-type-redex} is defined in a similar way.\esp

  \item \textbf{Redex Checker}\\
  Operator: 
  $(\mname{is-redex}::\mname{E},\mname{formula})$\\
  Defining axioms:
  \begin{eqnarray*}
  \lefteqn{\ForallApp e \mcolon \mname{E} \mdot 
     (\mname{is-redex}::\mname{E},\mname{formula})(e) \Iff} \\ & &
     \mname{is-dep-fun-type-redex}(e) \Or \mname{is-fun-redex}(e).
  \end{eqnarray*}
  \vspace{-5ex}   
  \begin{eqnarray*}
     {a^{\rm ef} \IsUndefApp} \Implies {}
     (\mname{is-redex}::\mname{E},\mname{formula})
     (a^{\rm ef}) \Iff \mname{F}.
  \end{eqnarray*}

  \iffalse
  \begin{eqnarray*}
  \lefteqn{\ForallApp e \mcolon \mname{E}_{\rm te} \mdot 
    {\sembrack{e} \IsUndefApp} \Implies {}} \\ & &
    (\mname{is-redex}::\mname{E},\mname{formula})
    (\sembrack{e}) \Iff \mname{F}.
  \end{eqnarray*}
  \fi

  \item \textbf{Variable Selector for a Redex}\\
  Operator: 
  $(\mname{redex-var}::\mname{E},\mname{E})$\\
  Defining axioms:
  \begin{eqnarray*}
  \lefteqn{\ForallApp e \mcolon \mname{E} \mdot 
     (\mname{redex-var}::\mname{E},\mname{E})(e) \QuasiEqual} \\ & &
     \mname{if}(\mname{is-redex}(e), 
     \mname{binder-var}(\mname{1st-comp}(e)),
     \Undefined_{\sf C}).
  \end{eqnarray*}
  \vspace{-5ex}  
  \begin{eqnarray*}
     {a^{\rm ef} \IsUndefApp} \Implies {}
     (\mname{redex-var}::\mname{E},\mname{E})
     (a^{\rm ef}) \IsUndefApp.
  \end{eqnarray*}

  \iffalse 
  \begin{eqnarray*}
  \lefteqn{\ForallApp e \mcolon \mname{E}_{\rm te} \mdot 
    {\sembrack{e} \IsUndefApp} \Implies {}} \\ & &
    (\mname{redex-var}::\mname{E},\mname{E})
    (\sembrack{e}) \IsUndefApp.
  \end{eqnarray*}
  \fi

  Note: Body and argument selectors for a redex named
  \mname{redex-body} and \mname{redex-arg} are defined in a similar
  way.

  \item \textbf{Variable Similarity}\\
  Operator: 
  $(\mname{var-sim}::\mname{E},\mname{E},\mname{formula})$\\
  Defining axioms:
  \begin{eqnarray*}
  \lefteqn{\ForallApp e_1,e_2 \mcolon \mname{E} \mdot 
     (\mname{var-sim}::\mname{E},\mname{E},\mname{formula})(e_1,e_2) 
     \Iff} \\ & &
     \mname{is-var}(e_1) \And \mname{is-var}(e_2) \And
     \mname{1st-comp}(e_1) = \mname{1st-comp}(e_2).
  \end{eqnarray*}
  \vspace{-5ex}   
  \begin{eqnarray*}
     ({a^{\rm ef} \IsUndefApp} \And
     {b^{\rm ef} \IsUndefApp}) \Implies 
     (\mname{var-sim}::\mname{E},\mname{E},\mname{formula})
     (a^{\rm ef}, b^{\rm ef}) \Iff \mname{F}.
  \end{eqnarray*}

  \iffalse
  \begin{eqnarray*}
  \lefteqn{\ForallApp e_1,e_2 \mcolon \mname{E}_{\rm te} \mdot 
    ({\sembrack{e_1} \IsUndefApp} \Or {\sembrack{e_2} \IsUndefApp}) 
    \Implies {}} \\ & &
    (\mname{var-sim}::\mname{E},\mname{E},\mname{formula})
    (\sembrack{e_1},\sembrack{e_2}) \Iff \mname{F}.
  \end{eqnarray*}
  \fi

  Compact notation: 
  \bi

    \item[] $e_1 \sim e_2$ means
    $(\mname{var-sim}::\mname{E},\mname{E},\mname{formula})(e_1,e_2)$.

    \item[] $e_1 \not\sim e_2$ means $\Neg(e_1 \sim e_2)$.
  
  \ei

\iffalse
  \item \textbf{Symbol Occurs Checker}\\
  Operator: 
  $(\mname{symbol-occurs}::\mname{E},\mname{formula})$\\
  Defining axioms:
  \begin{eqnarray*}
  \lefteqn{\ForallApp e \mcolon \mname{E} \mdot 
     (\mname{symbol-occurs}::\mname{E},\mname{formula})(e) \Iff} \\ & &
     e \IsDef \mname{E}_{\rm sy} \Or 
     \mname{symbol-occurs}(\mname{hd}(e)) \Or
     \mname{symbol-occurs}(\mname{tl}(e)).
  \end{eqnarray*}
  \vspace{-5ex} 
  \begin{eqnarray*}
  \lefteqn{\ForallApp e \mcolon \mname{E}_{\rm te} \mdot 
    {\sembrack{e} \IsUndefApp} \Implies {}} \\ & &
    (\mname{symbol-occurs}::\mname{E},\mname{formula})
    (\sembrack{e}) \Iff \mname{F}.
  \end{eqnarray*}
\fi

  \item \textbf{Eval-Free Checker}\\
  Operator: 
  $(\mname{is-eval-free}::\mname{E},\mname{formula})$\\
  Defining axioms:
  \begin{eqnarray*}
  \lefteqn{\ForallApp e \mcolon \mname{E} \mdot 
     (\mname{is-eval-free}::\mname{E},\mname{formula})(e) \Iff} \\ & &
     \mname{is-quote}(e) \Or \\ & &
     (e \IsDef \mname{E}_{\rm sy} \And 
      e \not= \synbrack{\mname{eval}}) \Or \\ & &
     e = \mlist{\,} \Or \\ & &
     (\mname{is-eval-free}(\mname{hd}(e)) \And
     \mname{is-eval-free}(\mname{tl}(e))).
  \end{eqnarray*}
  \vspace{-5ex}  
  \begin{eqnarray*}
     {a^{\rm ef} \IsUndefApp} \Implies {}
     (\mname{is-eval-free}::\mname{E},\mname{formula})
     (a^{\rm ef}) \Iff \mname{F}.
  \end{eqnarray*}

  \iffalse 
  \begin{eqnarray*}
  \lefteqn{\ForallApp e \mcolon \mname{E}_{\rm te} \mdot 
    {\sembrack{e} \IsUndefApp} \Implies {}} \\ & &
    (\mname{is-eval-free}::\mname{E},\mname{formula})
    (\sembrack{e}) \Iff \mname{F}.
  \end{eqnarray*}
  \fi

  \item \textbf{Coercion to a Type Construction}\\
  Operator: 
  $(\mname{coerce-to-type}::\mname{E},\mname{E}_{\rm ty})$\\
  Defining axioms:
  \begin{eqnarray*}
  \lefteqn{\ForallApp e_1 \mcolon \mname{E} \mdot 
     (\mname{coerce-to-type}::\mname{E},\mname{E}_{\rm ty})(e_1) 
     \Iff} \\ & &
     \iotaApp e_2 \mcolon \mname{E}_{\rm ty} \mdot e_1 \TypeEqual e_2.
  \end{eqnarray*}
  \vspace{-5ex}  
  \begin{eqnarray*}
     {a^{\rm ef} \IsUndefApp} \Implies {}
     (\mname{coerce-to-type}::\mname{E},\mname{E}_{\rm ty})
     (a^{\rm ef}) \IsUndefApp.
  \end{eqnarray*}

  Note: The operators \mname{coerce-to-term} and
  \mname{coerce-to-formula} are defined in a
  similar way.

  \item \textbf{Nominal Type of a Term}\\
  Operator: 
  $(\mname{nominal-type}::\mname{E}_{\rm te},\mname{E}_{\rm ty})$\\
  Defining axioms:
  \begin{eqnarray*}
  \lefteqn{\ForallApp e_1 \mcolon \mname{E}_{\rm te} \mdot 
     (\mname{nominal-type}::\mname{E}_{\rm te},\mname{E}_{\rm ty})(e_1) 
     \Iff} \\ & &
     \iotaApp e_2 \mcolon \mname{E}_{\rm ty} \mdot 
     e_1 \IsDef \mname{E}_{\rm te}^{e_2}.
  \end{eqnarray*}
  \vspace{-5ex}  
  \begin{eqnarray*}
     {a^{\rm ef} \IsUndefApp} \Implies {}
     (\mname{nominal-type}::\mname{E}_{\rm te},\mname{E}_{\rm ty})
     (a^{\rm ef}) \IsUndefApp.
  \end{eqnarray*}

  \item \textbf{Operator Name Strictness Checker}\\
  Operator: 
  $(\mname{is-strict-op-name}::\mname{E}_{\rm on},\mname{formula})$\\
  Defining axioms:
  \begin{eqnarray*}
  \lefteqn{\ForallApp o \mcolon \mname{E}_{\rm on} \mdot 
     (\mname{is-strict-op-name}::\mname{E}_{\rm on},
     \mname{formula})(o) \Iff} \\ & &
     \ForallApp e \mcolon \mname{E} \mdot \\ & &
     \hspace{4ex}(\mname{is-op-app}(e) \And \\ & &
     \hspace{5ex} o = \mname{1st-comp}(\mname{1st-comp}(e)) \And \\ & &
     \hspace{5ex} \synbrack{\bot_{\cal C}} \in_{\rm li} 
        \mname{tl}(\mname{tl}(e))) \\ & &
     \hspace{8ex}\Implies \\ & &
     \hspace{4ex}((\mname{is-type}(e) \Implies 
        \sembrack{e}_{\rm ty} = \mname{C}) \And {}\\ & &
     \hspace{5ex} (\mname{is-term}(e) \Implies 
        \sembrack{e}_{\rm te}\IsUndefApp) \And {}\\ & &
     \hspace{5ex} (\mname{is-formula}(e) \Implies 
        \sembrack{e}_{\rm fo} = \mname{F})).
  \end{eqnarray*}
  \vspace{-5ex}  
  \begin{eqnarray*}
     {a^{\rm ef} \IsUndefApp} \Implies {}
     (\mname{is-strict-op-name}::\mname{E}_{\rm on},\mname{formula})
     (a^{\rm ef}) \Iff \mname{F}.
  \end{eqnarray*}

  Compact notation: 

  \bi

    \item[] $\mname{strict}(o)$ means 
    $(\mname{is-strict-op-name}::\mname{E}_{\rm on},\mname{formula})(o)$.
  
  \ei

\ee

\begin{crem} \em
All the defining axioms given in this section are eval-free.
\end{crem}

\subsection{Another Notational Definition} \label{subsec:rel-eval}

Using some of the operators defined in this section, we give a
notational definition for evaluation relativized to a language:

\be

  \item[] \textbf{Relativized Evaluation}

  $(\mname{eval}, a, k, b)$ 
  means
  \[\If(a \IsDefApp (\mname{E}_{{\rm ty},b} \cup \mname{E}_{{\rm te},b} \cup
  \mname{E}_{{\rm fo},b}), \sembrack{a}_k, \sembrack{\Undefined_{\sf C}}_k).\]

  Compact notation: 
  \bi

    \item[] $\sembrack{a}_{k,b}$ means $(\mname{eval}, a, k, b)$.

  \ei

\ee
Notice that $\sembrack{a}_{k,\ell}$ and $\sembrack{a}_k$ are
logically equivalent.

\section{Substitution} \label{sec:sub}

In this section we will define the operators needed for substitution
of a term for the free occurrences of a variable and prove that they
have their intended meanings.  Substitution on eval-free expressions
is very similar to substitution on first-order formulas, but
substitution on non-eval-free expressions is tricky.  There are two
reasons for this.  

We will illustrate the first reason with an example.  Consider the
formula $a = \sembrack{(e \mcolon \mname{E})}$.  The variable $(e
\mcolon \mname{E})$ is certainly free in the formula.  However, if the
value of the variable is $\synbrack{(y \mcolon \mname{C})}$, then the
formula is equivalent to $a = (y \mcolon \mname{C})$, and so the $(y
\mcolon \mname{C})$ can also be said to be free in the formula.  This
example shows that some free occurrences of a variable in an
expression may not be syntactically visible.  This significantly
complicates substitution.

The second reason has to do with the evaluation of a substitution.  A
substitution maps a quoted expression $\synbrack{e}$ to a new quoted
expression $\synbrack{e'}$.  The latter will usually be the argument
to an evaluation $\sembrack{\synbrack{e'}}_k$ so that the value of
$e'$ can be expressed.  If $e'$ is not eval-free, then the evaluation
will be undefined and, as a result, substitution will not work as
desired on non-eval-free expressions.  To avoid this problem,
evaluations must be ``cleansed'' from the result of a substitution.

\subsection{Substitution Operators} \label{subsec:sub-op}

We define now the operators related to the substitution of a term for
the free occurrences of a variable.  Each operator will be defined
only for quotations.  As a result, the defining axioms will be an
infinite set organized into a finite number of formula schemas.

%\newpage

\be

  \item \textbf{Good Evaluation Arguments}\\
  Operator: 
  $(\mname{gea}::\mname{E},\mname{E},\mname{formula})$\\
  Defining axioms:
  \begin{eqnarray*}
  \lefteqn{\Neg(\mname{gea}::\mname{E},\mname{E},\mname{formula})
    (\synbrack{e_1},\synbrack{e_2})} \\[1ex] & & 
  \parbox[t]{60ex}{where $e_1$ is a non-eval-free expression,
  $e_1$ is a type and $e_2$ is not \mname{type}, 
  $e_1$ is a term and $e_2$ is not a type, 
  or $e_1$ is a formula and $e_2$ is not \mname{formula}.}
  \end{eqnarray*}
  \vspace{-3ex} 
  \begin{eqnarray*}
  (\mname{gea}::\mname{E},\mname{E},\mname{formula})
  (\synbrack{\alpha^{\rm ef}},\synbrack{\mname{type}}).
  \end{eqnarray*}
  \vspace{-5ex} 
  \begin{eqnarray*}
  (\mname{gea}::\mname{E},\mname{E},\mname{formula})
  (\synbrack{a^{\rm ef}},\synbrack{\alpha}).
  \end{eqnarray*}
  \vspace{-5ex} 
  \begin{eqnarray*}
  (\mname{gea}::\mname{E},\mname{E},\mname{formula})
  (\synbrack{A^{\rm ef}},\synbrack{\mname{formula}}).
  \end{eqnarray*}

  \item \textbf{Free Variable Occurrence in an Expression}\\
  Operator: 
  $(\mname{free-in}::\mname{E},\mname{E},\mname{formula})$\\
  Defining axioms:
  \begin{eqnarray*}
  \lefteqn{\Neg(\mname{free-in}::\mname{E},\mname{E},\mname{formula})(\synbrack{e_1},\synbrack{e_2})}
  \\[1ex] & &
  \mbox{where $e_1$ is not a symbol or $e_2$ is an improper expression.} 
  \end{eqnarray*}
  \vspace{-5ex}
  \begin{eqnarray*}
  \lefteqn{(\mname{free-in}::\mname{E},\mname{E},\mname{formula})
  (\synbrack{x},\synbrack{(o::k_1,\ldots,k_{n+1})}) \Iff} \\ & &
  \mname{free-in}(\synbrack{x}, \synbrack{k_1}) \Or
  \cdots \Or \mname{free-in}(\synbrack{x}, \synbrack{k_{n+1}})
  \\[1ex] & &
  \mbox{where $(o::k_1,\ldots,k_{n+1})$ is proper and $n \ge 0$.}
  \end{eqnarray*}
  \vspace{-5ex}
  \begin{eqnarray*}
  \lefteqn{(\mname{free-in}::\mname{E},\mname{E},\mname{formula})(\synbrack{x},
  \synbrack{O(e_1,\ldots,e_n)}) \Iff} \\ & &
  \mname{free-in}(\synbrack{x}, \synbrack{O}) \Or \\ &&
  \mname{free-in}(\synbrack{x}, \synbrack{e_1}) \Or \cdots \Or
  \mname{free-in}(\synbrack{x}, \synbrack{e_n})
  \\[1ex] & &
  \mbox{where $O(e_1,\ldots,e_n)$ is proper and $n \ge 0$.}
  \end{eqnarray*}
  \vspace{-5ex}
  \begin{eqnarray*}
  \lefteqn{(\mname{free-in}::\mname{E},\mname{E},\mname{formula})(\synbrack{x},\synbrack{(x \mcolon \alpha)}).}
  \end{eqnarray*}
  \vspace{-5ex}
  \begin{eqnarray*}
  \lefteqn{(\mname{free-in}::\mname{E},\mname{E},\mname{formula})(\synbrack{x}, \synbrack{(y \mcolon \alpha)}) \Iff} \\ & &
  \mname{free-in}(\synbrack{x},\synbrack{\alpha})
  \\[1ex] & &
  \mbox{where $x \not=y$.}
  \end{eqnarray*}
  \vspace{-5ex}
  \begin{eqnarray*}
  \lefteqn{(\mname{free-in}::\mname{E},\mname{E},\mname{formula})(\synbrack{x},
  \synbrack{(\StarApp {x\mcolon\alpha} \mdot e)}) \Iff} \\ & &
  \mname{free-in}(\synbrack{x}, \synbrack{\alpha})
  \\[1ex] & &
  \mbox{where $(\StarApp {x\mcolon\alpha} \mdot e)$ is proper and
  $\star$ is $\Lambda$, $\lambda$, $\Forsome$, $\iota$, or $\epsilon$.}
  \end{eqnarray*}
  \vspace{-5ex}
  \begin{eqnarray*}
  \lefteqn{(\mname{free-in}::\mname{E},\mname{E},\mname{formula})(\synbrack{x},
  \synbrack{(\StarApp {y\mcolon\alpha} \mdot e)}) \Iff} \\ & &
  \mname{free-in}(\synbrack{x},\synbrack{\alpha}) \Or 
  \mname{free-in}(\synbrack{x},\synbrack{e})
  \\[1ex] & &
  \mbox{where $x \not= y$, $(\StarApp {x\mcolon\alpha} \mdot e)$ is proper, 
  and $\star$ is $\Lambda$, $\lambda$, $\Forsome$, $\iota$, or $\epsilon$.}
  \end{eqnarray*}
  \vspace{-5ex}
  \begin{eqnarray*}
  \lefteqn{(\mname{free-in}::\mname{E},\mname{E},\mname{formula})(\synbrack{x}, \synbrack{\alpha(a)}) \Iff}\\ & &
  \mname{free-in}(\synbrack{x}, \synbrack{\alpha}) \Or
  \mname{free-in}(\synbrack{x}, \synbrack{a}).
  \end{eqnarray*}
  \vspace{-5ex}
  \begin{eqnarray*}
  \lefteqn{(\mname{free-in}::\mname{E},\mname{E},\mname{formula})(\synbrack{x}, \synbrack{f(a)}) \Iff}\\ & &
  \mname{free-in}(\synbrack{x}, \synbrack{f}) \Or
  \mname{free-in}(\synbrack{x}, \synbrack{a}).
  \end{eqnarray*}
  \vspace{-5ex}
  \begin{eqnarray*}
  \lefteqn{(\mname{free-in}::\mname{E},\mname{E},\mname{formula})(\synbrack{x},
  \synbrack{\mname{if}(A,b,c)}) \Iff} \\ & &
  \mname{free-in}(\synbrack{x}, \synbrack{A}) \Or
  \mname{free-in}(\synbrack{x}, \synbrack{b}) \Or
  \mname{free-in}(\synbrack{x}, \synbrack{c}).
  \end{eqnarray*}
  \vspace{-5ex}
  \begin{eqnarray*}
  \lefteqn{\Neg(\mname{free-in}::\mname{E},\mname{E},\mname{formula})(\synbrack{x},
  \synbrack{\synbrack{e}})} 
  \\[1ex] & &
  \mbox{where $e$ is any expression.}
  \end{eqnarray*}
  \vspace{-5ex}
  \begin{eqnarray*}
  \lefteqn{(\mname{free-in}::\mname{E},\mname{E},
  \mname{formula})(\synbrack{x},
  \synbrack{\sembrack{a}_k}) \Iff} \\ & &
  \mname{free-in}(\synbrack{x}, \synbrack{a}) \Or \\ &&
  \mname{free-in}(\synbrack{x}, \synbrack{k}) \Or \\ &&
  \mname{free-in}(\synbrack{x}, a).
  \end{eqnarray*}
  \vspace{-5ex}  
  \begin{eqnarray*}
     ({a^{\rm ef} \IsUndefApp} \Or {b^{\rm ef} \IsUndefApp}) 
     \Implies 
     \Neg (\mname{free-in}::\mname{E},\mname{E},
     \mname{formula})(a^{\rm ef}, b^{\rm ef}).
  \end{eqnarray*}

\iffalse
  \begin{eqnarray*}
  \lefteqn{(\mname{free-in}::\mname{E},\mname{E},
  \mname{formula})(\synbrack{x},
  \synbrack{\sembrack{a}_k}) \Iff} \\ & &
  \mname{free-in}(\synbrack{x}, \synbrack{a}) \Or \\ &&
  \mname{free-in}(\synbrack{x}, \synbrack{k}) \Or \\ &&
  (\mname{gea}(a,\synbrack{k}) \And
  \mname{free-in}(\synbrack{x}, a)).
  \end{eqnarray*}
\fi

  Note: For an evaluation $(\mname{eval}, a, k)$, a variable may be
  free in the term $a$, in the kind $k$ when $k$ is a type, and in the
  expression that the evaluation represents.

  \item \textbf{Syntactically Closed Proper Expression}\\
  Operator: 
  $(\mname{syn-closed}::\mname{E},\mname{formula})$\\
  Defining axioms:
  \begin{eqnarray*}
  \lefteqn{\ForallApp e \mcolon \mname{E} \mdot 
     (\mname{syn-closed}::\mname{E},\mname{formula})(e) 
     \Iff} \\ & & 
     \mname{is-p-expr}(e) \And 
     {\ForallApp e' \mcolon \mname{E}_{\rm sy}} \mdot
     \Neg\mname{free-in}(e',e).
  \end{eqnarray*}

  \item \textbf{Cleanse Eval Symbols from an Expression}\\
  Operator: 
  $(\mname{cleanse}:: \mname{E},\mname{E})$\\
  Defining axioms:
  \begin{eqnarray*}
  \lefteqn{(\mname{cleanse}::\mname{E},\mname{E})  
  (\synbrack{e}) =} \\ & &
  \synbrack{e} \\[1ex] & &
  \mbox{where $e$ is an improper expression.} 
  \end{eqnarray*}
  \vspace{-5ex}
  \begin{eqnarray*}
  \lefteqn{(\mname{cleanse}::\mname{E},\mname{E})
  (\synbrack{(o::k_1,\ldots,k_{n+1})}) =} \\ & &
  \synbrack{(o::\commabrack{\widehat{k}_1}, \ldots,
  \commabrack{\widehat{k}_{n+1}})} \\[1ex] & &
  \parbox[t]{60ex}{where $(o::k_1,\ldots,k_{n+1})$ is proper, $n \ge 0$, 
  and \\$\widehat{k}_i = 
  \mname{cleanse}(\synbrack{k_i})$ for all $i$ with $1 \le i \le n + 1$.}
  \end{eqnarray*}
  \vspace{-3ex}
  \begin{eqnarray*}
  \lefteqn{(\mname{cleanse}::\mname{E},\mname{E})
  (\synbrack{O(e_1,\ldots,e_n)}) =} \\ & &
  \synbrack{\commabrack{\widehat{O}}
  (\commabrack{\widehat{e}_1}, \ldots, \commabrack{\widehat{e}_n})}
  \\[1ex] & &
  \parbox[t]{60ex}{where $O(e_1,\ldots,e_n)$ is proper, $n \ge 0$,\\
  $\widehat{O} = \mname{cleanse}(\synbrack{O})$, and \\ 
  $\widehat{e}_i = 
  \mname{cleanse}(\synbrack{e_i})$ for all $i$ with $1 \le i \le n$.}
  \end{eqnarray*}
  \vspace{-3ex}
  \begin{eqnarray*}
  \lefteqn{(\mname{cleanse}::\mname{E},\mname{E})
  (\synbrack{(x \mcolon \alpha)}) =} \\ & &
  \synbrack{(x \mcolon \commabrack{\widehat{\alpha}})} \\[1ex] & &
  \mbox{where $\widehat{\alpha} = \mname{cleanse}(\synbrack{\alpha})$.}
  \end{eqnarray*}
  \vspace{-5ex}
  \begin{eqnarray*}
  \lefteqn{(\mname{cleanse}::\mname{E},\mname{E})
  (\synbrack{(\StarApp {x\mcolon\alpha} \mdot e)}) =} \\ & &
  \synbrack{(\StarApp {x \mcolon \commabrack{\widehat{\alpha}}} \mdot 
  \commabrack{\widehat{e}})} \\[1ex] & &
  \parbox[t]{60ex}{where $(\StarApp {x\mcolon\alpha} \mdot e)$ is proper;
  $\star$ is $\Lambda$, $\lambda$, $\Forsome$, $\iota$, or $\epsilon$; \\
  $\widehat{\alpha} = \mname{cleanse}(\synbrack{\alpha})$; and
  $\widehat{e} = \mname{cleanse}(\synbrack{e})$.}
  \end{eqnarray*}
  \vspace{-3ex}
  \begin{eqnarray*}
  \lefteqn{(\mname{cleanse}::\mname{E},\mname{E})
  (\synbrack{\alpha(b)}) =} \\ & &
  \synbrack{\commabrack{\widehat{\alpha}}(\commabrack{\widehat{b}})}
   \\[1ex] & &
  \mbox{where $\widehat{\alpha} = 
  \mname{cleanse}(\synbrack{\alpha})$ and $\widehat{b} =
  \mname{cleanse}(\synbrack{b})$.}
  \end{eqnarray*}
  \vspace{-5ex}
  \begin{eqnarray*}
  \lefteqn{(\mname{cleanse}::\mname{E},\mname{E})
  (\synbrack{f(b)}) =} \\ & &
  \synbrack{\commabrack{\widehat{f}}(\commabrack{\widehat{b}})}
   \\[1ex] & &
  \mbox{where $\widehat{f} = 
  \mname{cleanse}(\synbrack{f})$ and $\widehat{b} =
  \mname{cleanse}(\synbrack{b})$.}
  \end{eqnarray*}
  \vspace{-5ex}
  \begin{eqnarray*}
  \lefteqn{(\mname{cleanse}::\mname{E},\mname{E})
  (\synbrack{\mname{if}(A,b,c)}) =} \\ & &
  \synbrack{\mname{if}(\commabrack{\widehat{A}},
  \commabrack{\widehat{b}}, \commabrack{\widehat{c}})} \\[1ex] & &
  \mbox{where $\widehat{A} = \mname{cleanse}(\synbrack{A})$,} \\ &&
  \widehat{b} = \mname{cleanse}(\synbrack{b}), \mbox{ and} \\ &&
  \widehat{c} = \mname{cleanse}(\synbrack{c}).
  \end{eqnarray*}
  \vspace{-5ex}
  \begin{eqnarray*}
  \lefteqn{(\mname{cleanse}::\mname{E},\mname{E})
  (\synbrack{\synbrack{e}}) =}\\ & &
  \synbrack{\synbrack{e}} \\[1ex] & & 
  \mbox{where $e$ is any expression.}
  \end{eqnarray*}
  \vspace{-5ex}
  \begin{eqnarray*}
  \lefteqn{(\mname{cleanse}::\mname{E},\mname{E})
  (\synbrack{\sembrack{a}_{\rm ty}}) =} \\ & &
  \If(\mname{syn-closed}(\widehat{a}),w,\Undefined_{\sf C})\\[1ex] & &
  \mbox{where $\widehat{a} = \mname{cleanse}(\synbrack{a})$,} \\ & &
  \widehat{a}' = 
  \mname{coerce-to-type}(\sembrack{\widehat{a}}_{\rm te}), \mbox{ and} \\ &&
  w = \synbrack{\mname{if}(\mname{gea}(a,\synbrack{\mname{type}}), 
  \commabrack{\widehat{a}'}, \mname{C})}. 
  \end{eqnarray*}
  \vspace{-5ex}
  \begin{eqnarray*}
  \lefteqn{(\mname{cleanse}::\mname{E},\mname{E})
  (\synbrack{\sembrack{a}_\alpha}) =} \\ & &
  \If(\mname{syn-closed}(\widehat{a}),w,\Undefined_{\sf C})\\[1ex] & &
  \mbox{where $\widehat{a} = \mname{cleanse}(\synbrack{a})$,} \\ & &
  \widehat{a}' = 
  \mname{coerce-to-term}(\sembrack{\widehat{a}}_{\rm te}), \\ & &
  \widehat{\alpha} = \mname{cleanse}(\synbrack{\alpha}), \mbox{ and} \\ & &
  w = \synbrack{\mname{if}(\mname{gea}(a,\synbrack{\alpha}) \And
  \commabrack{\widehat{a}'} \IsDef \commabrack{\widehat{\alpha}}, 
  \commabrack{\widehat{a}'}, \Undefined_{\sf C})}.
  \end{eqnarray*}
  \vspace{-5ex}
  \begin{eqnarray*}
  \lefteqn{(\mname{cleanse}::\mname{E},\mname{E})
  (\synbrack{\sembrack{a}_{\rm fo}}) =} \\ & &
  \If(\mname{syn-closed}(\widehat{a}),w,\Undefined_{\sf C})\\[1ex] & &
  \mbox{where $\widehat{a} = \mname{cleanse}(\synbrack{a}),$} \\ & &
  \widehat{a}' = 
  \mname{coerce-to-formula}(\sembrack{\widehat{a}}_{\rm te}), \mbox{ and} \\ &&
  w = \synbrack{\mname{if}(\mname{gea}(a,\synbrack{\mname{formula}}), 
  \commabrack{\widehat{a}'}, \mname{F})}.
  \end{eqnarray*}
  \vspace{-5ex}  
  \begin{eqnarray*}
     {a^{\rm ef} \IsUndefApp} 
     \Implies 
     (\mname{cleanse}::\mname{E},\mname{E})(a^{\rm ef}) \IsUndefApp.
  \end{eqnarray*}

  \item \textbf{Substitution for a Variable in an Expression}\\
  Operator: 
  $(\mname{sub}:: \mname{E},\mname{E},\mname{E},\mname{E})$\\
  Defining axioms:
  \begin{eqnarray*}
  \lefteqn{(\mname{sub}::\mname{E},\mname{E},\mname{E},
  \mname{E})  
  (\synbrack{e_1},\synbrack{e_2},\synbrack{e_3}) =} \\ & &
  \synbrack{e_3} \\[1ex] & &
  \parbox[t]{60ex}{where $e_1$ is not a term, $e_2$ is not a symbol, or 
  $e_3$ is an improper expression.} 
  \end{eqnarray*}
  \vspace{-3ex}
  \begin{eqnarray*}
  \lefteqn{(\mname{sub}::\mname{E},\mname{E},\mname{E},\mname{E})
  (\synbrack{a}, \synbrack{x},\synbrack{(o::k_1,\ldots,k_{n+1})}) =} \\ & &
  \synbrack{(o::\commabrack{\widehat{k}_1}, \ldots,
  \commabrack{\widehat{k}_{n+1}})} \\[1ex] & &
  \parbox[t]{60ex}{where $(o::k_1,\ldots,k_{n+1})$ is proper, $n \ge 0$, 
  and $\widehat{k}_i = \mname{sub}(\synbrack{a},
  \synbrack{x}, \synbrack{k_i})$ for all $i$ with $1 \le i \le n +
  1$.}
  \end{eqnarray*}
  \vspace{-3ex}
  \begin{eqnarray*}
  \lefteqn{(\mname{sub}::\mname{E},\mname{E},\mname{E},\mname{E})
  (\synbrack{a}, \synbrack{x},\synbrack{O(e_1,\ldots,e_n)}) =} \\ & &
  \synbrack{\commabrack{\widehat{O}}
  (\commabrack{\widehat{e}_1}, \ldots, \commabrack{\widehat{e}_n})}
  \\[1ex] & &
  \parbox[t]{60ex}{where $O(e_1,\ldots,e_n)$ is proper, $n \ge 0$,\\
  $\widehat{O} = \mname{sub}(\synbrack{a}, \synbrack{x},
  \synbrack{O})$, and \\ $\widehat{e}_i = \mname{sub}(\synbrack{a},
  \synbrack{x}, \synbrack{e_i})$ for all $i$ with $1 \le i \le n$.}
  \end{eqnarray*}
  \vspace{-3ex}
  \begin{eqnarray*}
  \lefteqn{(\mname{sub}::\mname{E},\mname{E},\mname{E},\mname{E})
    (\synbrack{a}, \synbrack{x},\synbrack{(x \mcolon \alpha)}) =}
  \\ & & \synbrack{\mname{if}(a \IsDef \commabrack{\widehat{\alpha}},
  \commabrack{\widehat{a}}, \Undefined_{\sf C})} \\[1ex] & &
  \mbox{where $\widehat{\alpha} = \mname{sub}(\synbrack{a},
  \synbrack{x}, \synbrack{\alpha})$ and $\widehat{a} =
  \mname{cleanse}(\synbrack{a})$.}
  \end{eqnarray*}
  \vspace{-5ex}
  \begin{eqnarray*}
  \lefteqn{(\mname{sub}::\mname{E},\mname{E},\mname{E},\mname{E})
  (\synbrack{a}, \synbrack{x},\synbrack{(y \mcolon \alpha)}) =} \\ & &
  \synbrack{(y \mcolon \commabrack{\widehat{\alpha}})} \\[1ex] & &
  \mbox{where $x \not=y$ and
  $\widehat{\alpha} = \mname{sub}(\synbrack{a}, \synbrack{x},
  \synbrack{\alpha})$.}
  \end{eqnarray*}
  \vspace{-5ex}
  \begin{eqnarray*}
  \lefteqn{(\mname{sub}::\mname{E},\mname{E},\mname{E},\mname{E})
  (\synbrack{a}, \synbrack{x},
  \synbrack{(\StarApp {x\mcolon\alpha} \mdot e)}) =} \\ & &
  \synbrack{(\StarApp {x \mcolon \commabrack{\widehat{\alpha}}} \mdot 
  \commabrack{\widehat{e}})} \\[1ex] & &
  \parbox[t]{60ex}{where $(\StarApp
  {x\mcolon\alpha} \mdot e)$ is proper; $\star$ is $\Lambda$,
  $\lambda$, $\Forsome$, $\iota$, or $\epsilon$; \\ 
  $\widehat{\alpha} = 
  \mname{sub}(\synbrack{a}, \synbrack{x}, \synbrack{\alpha})$; and 
  $\widehat{e} = \mname{cleanse}(\synbrack{e})$.}
  \end{eqnarray*}
  \vspace{-3ex}
  \begin{eqnarray*}
  \lefteqn{(\mname{sub}::\mname{E},\mname{E},\mname{E},\mname{E})
  (\synbrack{a}, \synbrack{x},
  \synbrack{(\StarApp {y\mcolon\alpha} \mdot e)}) =} \\ & &
  \synbrack{(\StarApp {y \mcolon \commabrack{\widehat{\alpha}}} \mdot 
  \commabrack{\widehat{e}})} \\[1ex] & &
  \parbox[t]{60ex}{where $x \not= y$; $(\StarApp {y\mcolon\alpha} \mdot e)$
  is proper; $\star$ is $\Lambda$, $\lambda$, $\Forsome$, $\iota$, or
  $\epsilon$; $\widehat{\alpha} = \mname{sub}(\synbrack{a},
  \synbrack{x}, \synbrack{\alpha})$; and $\widehat{e} =
  \mname{sub}(\synbrack{a}, \synbrack{x}, \synbrack{e})$.}
  \end{eqnarray*}
  \vspace{-3ex}
  \begin{eqnarray*}
  \lefteqn{(\mname{sub}::\mname{E},\mname{E},\mname{E},\mname{E})
  (\synbrack{a}, \synbrack{x},\synbrack{\alpha(b)}) =} \\ & &
  \synbrack{\commabrack{\widehat{\alpha}}(\commabrack{\widehat{b}})}
   \\[1ex] & &
  \parbox[t]{60ex}{where $\widehat{\alpha} = \mname{sub}(\synbrack{a},
  \synbrack{x}, \synbrack{\alpha})$ and $\widehat{b} =
  \mname{sub}(\synbrack{a}, \synbrack{x}, \synbrack{b})$.}
  \end{eqnarray*}
  \vspace{-5ex}
  \begin{eqnarray*}
  \lefteqn{(\mname{sub}::\mname{E},\mname{E},\mname{E},\mname{E})
  (\synbrack{a}, \synbrack{x},\synbrack{f(b)}) =} \\ & &
  \synbrack{\commabrack{\widehat{f}}(\commabrack{\widehat{b}})}
   \\[1ex] & &
  \parbox[t]{60ex}{where $\widehat{f} = \mname{sub}(\synbrack{a},
  \synbrack{x}, \synbrack{f})$ and $\widehat{b} =
  \mname{sub}(\synbrack{a}, \synbrack{x}, \synbrack{b})$.}
  \end{eqnarray*}
  \vspace{-5ex}
  \begin{eqnarray*}
  \lefteqn{(\mname{sub}::\mname{E},\mname{E},\mname{E},\mname{E})
  (\synbrack{a}, \synbrack{x},\synbrack{\mname{if}(A,b,c)}) =} \\ & &
  \synbrack{\mname{if}(\commabrack{\widehat{A}},
  \commabrack{\widehat{b}}, \commabrack{\widehat{c}})} \\[1ex] & &
  \parbox[t]{60ex}{where $\widehat{A} = 
  \mname{sub}(\synbrack{a}, \synbrack{x},
  \synbrack{A})$, $\widehat{b} = \mname{sub}(\synbrack{a},
  \synbrack{x}, \synbrack{b})$, and $\widehat{c} =
  \mname{sub}(\synbrack{a}, \synbrack{x}, \synbrack{c})$.}
  \end{eqnarray*}
  \vspace{-3ex}
  \begin{eqnarray*}
  \lefteqn{(\mname{sub}::\mname{E},\mname{E},\mname{E},\mname{E})
  (\synbrack{a}, \synbrack{x},\synbrack{\synbrack{e}}) =} \\ & &
  \synbrack{\synbrack{e}} \\[1ex] & & 
  \mbox{where $e$ is any expression.}
  \end{eqnarray*}
  \vspace{-5ex}
  \begin{eqnarray*}
  \lefteqn{(\mname{sub}::\mname{E},\mname{E},\mname{E},\mname{E})
  (\synbrack{a}, \synbrack{x},\synbrack{\sembrack{b}_{\rm ty}}) =} \\ & &
  \If(\mname{syn-closed}(\widehat{b}),w,\Undefined_{\sf C})\\[1ex] & &
  \mbox{where $\widehat{b} = 
  \mname{sub}(\synbrack{a}, \synbrack{x}, \synbrack{b})$,} \\ & &
  \widehat{b}' = \mname{coerce-to-type}
  (\mname{sub}(\synbrack{a}, \synbrack{x}, \sembrack{\widehat{b}}_{\rm te})),
  \\ & &
  w =  \synbrack{\mname{if}(\mname{gea}(\commabrack{\widehat{b}},
  \synbrack{\mname{type}}), \commabrack{\widehat{b}'}, \mname{C})}.
  \end{eqnarray*}
  \vspace{-5ex}
  \begin{eqnarray*}
  \lefteqn{(\mname{sub}::\mname{E},\mname{E},\mname{E},\mname{E})
  (\synbrack{a}, \synbrack{x},\synbrack{\sembrack{b}_\alpha}) =} \\ & &
  \If(\mname{syn-closed}(\widehat{b}),w,\Undefined_{\sf C}) \\[1ex] & &
  \mbox{where $\widehat{b} = 
  \mname{sub}(\synbrack{a}, \synbrack{x}, \synbrack{b})$,} \\ & &
  \widehat{b}' = \mname{coerce-to-term}
  (\mname{sub}(\synbrack{a}, \synbrack{x}, \sembrack{\widehat{b}}_{\rm te})),
  \\ & &
  \widehat{\alpha} = 
  \mname{sub}(\synbrack{a}, \synbrack{x}, \synbrack{\alpha}), 
  \mbox{ and} \\ & &
  w =   \synbrack{\mname{if}(\mname{gea}(\commabrack{\widehat{b}}, 
  \synbrack{\alpha}) \And
  \commabrack{\widehat{b}'} \IsDef \commabrack{\widehat{\alpha}}, 
  \commabrack{\widehat{b}'}, \Undefined_{\sf C})}.
  \end{eqnarray*}
  \vspace{-5ex}
  \begin{eqnarray*}
  \lefteqn{(\mname{sub}::\mname{E},\mname{E},\mname{E},\mname{E})
  (\synbrack{a}, \synbrack{x},\synbrack{\sembrack{b}_{\rm fo}}) =} \\ & &
  \If(\mname{syn-closed}(\widehat{b}),w,\Undefined_{\sf C}) \\[1ex] & &
  \mbox{where $\widehat{b} = 
  \mname{sub}(\synbrack{a}, \synbrack{x}, \synbrack{b})$,} \\ & &
  \widehat{b}' = \mname{coerce-to-formula}
  (\mname{sub}(\synbrack{a}, \synbrack{x}, \sembrack{\widehat{b}}_{\rm te})),
  \\ & &
  w =   \synbrack{\mname{if}(\mname{gea}(\commabrack{\widehat{b}},
  \synbrack{\mname{formula}}), 
  \commabrack{\widehat{b}'}, \mname{F})}.
  \end{eqnarray*}
  \vspace{-5ex}  
  \begin{eqnarray*}
     ({a^{\rm ef} \IsUndefApp} \Or {b^{\rm ef} \IsUndefApp} \Or 
      {c^{\rm ef} \IsUndefApp}) 
     \Implies 
     (\mname{sub}::\mname{E},\mname{E},\mname{E},\mname{E})
     (a^{\rm ef}, b^{\rm ef}, c^{\rm ef}) \IsUndefApp.
  \end{eqnarray*}

  Note: For an evaluation $(\mname{eval}, a, k)$, a substitution for a
  variable is performed on $a$ to obtain $a'$ and on $k$ when $k$ is a
  type to obtain $k'$ and then on the expression represented by the
  new evaluation resulting from these substitutions.

  \item \textbf{Free for a Variable in an Expression}\\
  Operator: 
  $(\mname{free-for}:: \mname{E},\mname{E},\mname{E},\mname{formula})$\\
  Defining axioms:
  \begin{eqnarray*}
  \lefteqn{(\mname{free-for}::\mname{E},\mname{E},\mname{E},
  \mname{formula})  
  (\synbrack{e_1},\synbrack{e_2},\synbrack{e_3})} \\[1ex] & &
  \parbox[t]{60ex}{where $e_1$ is not a term, $e_2$ is not a symbol, or 
  $e_3$ is an improper expression.} 
  \end{eqnarray*}
  \vspace{-3ex}
  \begin{eqnarray*}
  \lefteqn{(\mname{free-for}::\mname{E},\mname{E},\mname{E},\mname{formula})
  (\synbrack{a}, \synbrack{x},\synbrack{(o::k_1,\ldots,k_{n+1})})
  \Iff} \\ & & \mname{free-for}(\synbrack{a}, \synbrack{x},
  \synbrack{k_1}) \And \cdots \And \mname{free-for}(\synbrack{a},
  \synbrack{x}, \synbrack{k_{n+1}}) \\[1ex] & & 
  \mbox{where $(o::k_1,\ldots,k_{n+1})$ is proper and $n \ge 0$.}
  \end{eqnarray*}
  \vspace{-5ex}
  \begin{eqnarray*}
  \lefteqn{(\mname{free-for}::\mname{E},\mname{E},\mname{E},\mname{formula})
  (\synbrack{a}, \synbrack{x},\synbrack{O(e_1,\ldots,e_n)}) \Iff} \\ & &
  \mname{free-for}(\synbrack{a}, \synbrack{x},\synbrack{O}) \And \\ & &
  \mname{free-for}(\synbrack{a}, \synbrack{x},
  \synbrack{e_1}) \And \cdots \And \mname{free-for}(\synbrack{a},
  \synbrack{x}, \synbrack{e_n}) \\[1ex] & &
  \mbox{where $O(e_1,\ldots,e_n)$ is proper and $n \ge 0$.}
  \end{eqnarray*}
  \vspace{-5ex}
  \begin{eqnarray*}
  \lefteqn{(\mname{free-for}::\mname{E},\mname{E},\mname{E},\mname{formula})
  (\synbrack{a}, \synbrack{x},\synbrack{(y \mcolon \alpha)}) \Iff} \\ & &
  \mname{free-for}(\synbrack{a},\synbrack{x}, \synbrack{\alpha}). 
  \end{eqnarray*}
  \vspace{-5ex}
  \begin{eqnarray*}
  \lefteqn{(\mname{free-for}::\mname{E},\mname{E},\mname{E},\mname{formula})
  (\synbrack{a}, \synbrack{x},
  \synbrack{(\StarApp {x\mcolon\alpha} \mdot e)}) \Iff} \\ & &
  \mname{free-for}(\synbrack{a},\synbrack{x},\synbrack{\alpha})
  \\[1ex] & &
  \mbox{where $(\StarApp {x\mcolon\alpha} \mdot e)$ is proper and 
  $\star$ is $\Lambda$, $\lambda$, $\Forsome$, $\iota$, or $\epsilon$.}
  \end{eqnarray*}
  \vspace{-5ex}
  \begin{eqnarray*}
  \lefteqn{(\mname{free-for}::\mname{E},\mname{E},\mname{E},\mname{formula})
  (\synbrack{a}, \synbrack{x},
  \synbrack{\StarApp {y\mcolon\alpha} \mdot e}) \Iff} \\ & &
  \mname{free-for}(\synbrack{a},\synbrack{x},\synbrack{\alpha}) \And 
  \\ & &
  %% (\Neg\mname{special-free-in}(\synbrack{a},\synbrack{x}, 
  %% \synbrack{e}) \Or \\ & &
  (\Neg\mname{free-in}(\synbrack{x}, \synbrack{e}) \Or 
  \Neg\mname{free-in}(\synbrack{y},\synbrack{a})) \And {} \\ & &
  \mname{free-for}(\synbrack{a}, \synbrack{x},
  \synbrack{e})\\[1ex] & &
  \mbox{where $x \not= y$, $(\StarApp {y\mcolon\alpha} \mdot e)$ is proper, 
  and $\star$ is $\Lambda$, $\lambda$, $\Forsome$, $\iota$, or $\epsilon$.}
  \end{eqnarray*}
  \vspace{-5ex}
  \begin{eqnarray*}
  \lefteqn{(\mname{free-for}::\mname{E},\mname{E},\mname{E},\mname{formula})
  (\synbrack{a}, \synbrack{x},\synbrack{\alpha(b)}) \Iff} \\ & &
  \mname{free-for}(\synbrack{a}, \synbrack{x}, \synbrack{\alpha}) \And 
  \mname{free-for}(\synbrack{a}, \synbrack{x}, \synbrack{b}).
  \end{eqnarray*}
  \vspace{-5ex}
  \begin{eqnarray*}
  \lefteqn{(\mname{free-for}::\mname{E},\mname{E},\mname{E},\mname{formula})
  (\synbrack{a}, \synbrack{x},\synbrack{f(b)}) \Iff} \\ & &
  \mname{free-for}(\synbrack{a}, \synbrack{x}, \synbrack{f}) \And 
  \mname{free-for}(\synbrack{a}, \synbrack{x}, \synbrack{b}).
  \end{eqnarray*}
  \vspace{-5ex}
  \begin{eqnarray*}
  \lefteqn{(\mname{free-for}::\mname{E},\mname{E},\mname{E},\mname{formula})
  (\synbrack{a}, \synbrack{x},\synbrack{\mname{if}(A,b,c)}) \Iff} \\ & &
  \mname{free-for}(\synbrack{a}, \synbrack{x}, \synbrack{A}) \And \\ & &
  \mname{free-for}(\synbrack{a}, \synbrack{x}, \synbrack{b}) \And \\ & &
  \mname{free-for}(\synbrack{a}, \synbrack{x}, \synbrack{c}).
  \end{eqnarray*}
  \vspace{-5ex}
  \begin{eqnarray*}
  \lefteqn{(\mname{free-for}::\mname{E},\mname{E},\mname{E},\mname{formula})
  (\synbrack{a}, \synbrack{x},\synbrack{\synbrack{e}})} \\[1ex] & &
  \mbox{where $e$ is any expression.}
  \end{eqnarray*}
  \vspace{-5ex}
  \begin{eqnarray*}
  \lefteqn{(\mname{free-for}::\mname{E},\mname{E},\mname{E},\mname{formula})
  (\synbrack{a}, \synbrack{x},\synbrack{\sembrack{b}_k}) \Iff} \\ & &
  \mname{free-for}(\synbrack{a}, \synbrack{x}, \synbrack{b}) \And  \\ & &
  \mname{free-for}(\synbrack{a}, \synbrack{x}, \synbrack{k}) \And  \\ & &
  \mname{free-for}(\synbrack{a}, \synbrack{x}, 
  \sembrack{\widehat{b}}_{\rm te}). \\[1ex] & &
  \mbox{where $\widehat{b} = 
  \mname{sub}(\synbrack{a}, \synbrack{x}, \synbrack{b})$}.
  \end{eqnarray*}
  \vspace{-5ex}  
  \begin{eqnarray*}
     \lefteqn{({a^{\rm ef} \IsUndefApp} \Or {b^{\rm ef} \IsUndefApp} \Or 
      {c^{\rm ef} \IsUndefApp}) 
     \Implies {}} \\ & &
     \Neg (\mname{free-for}::\mname{E},\mname{E},\mname{E},
     \mname{formula})(a^{\rm ef}, b^{\rm ef}, c^{\rm ef}).
  \end{eqnarray*}

\ee

\begin{crem} \em
All the defining axioms given in this subsection are eval-free except
for those involving the operators \mname{free-in} and \mname{free-for}
applied to quoted evaluations.
\end{crem}

\subsection{Kernel, Normal, and Eval-Free Normal Theories}

The \emph{kernel language} of Chiron is the language $L_{\rm ker} =
(\sO,\theta)$ where: \be

  \item $o \in \sO$ iff $o$ is a built-in operator name of Chiron or a
  name of an operator defined in section~\ref{sec:op-defs} or in the
  previous subsection.

  \item For all $o \in \sO$, $\theta(o)$ is the signature form
  assigned to $o$.

\ee
A language $L$ is \emph{normal} if $L_{\rm ker} \le L$.

Let $L$ be a normal language.  The \emph{kernel theory over $L$},
written $T_{\rm ker}^{L}$, is the theory $(L,\Gamma_{\rm ker}^{L})$
where $A \in \Gamma_{\rm ker}^{L}$ iff $A$ is a one of the defining
axioms (with respect to the $L$) for an operator defined in
section~\ref{sec:op-defs} or the previous section.

Let $T = (L,\Gamma)$ be a theory.  $T$ is \emph{normal} if $T_{\rm
  ker}^{L} \le T$.  The \emph{eval-free subtheory} of $T$, written
$\mname{eval-free}(T)$, is the theory $T = (L,\Gamma')$ where $\Gamma'
= \set{A \in \Gamma \;|\; A \mbox{ is eval-free}}$.  $T$ is
\emph{eval-free normal} if $T = \mname{eval-free}(T')$ for some normal
theory $T'$.

Let $a$ be a term, $x$ be a symbol, and $e$ be a proper expression of
$L$.  We say that $x$ is \emph{free in} $e$ if
$\mname{free-in}(\synbrack{x},\synbrack{e})$ is valid in $T_{\rm
  ker}^{L}$, and $a$ is \emph{free for} $x$ \emph{in} $e$ if
$\mname{free-for}(\synbrack{a},\synbrack{x},\synbrack{e})$ is valid
in $T_{\rm ker}^{L}$.  Similarly, we say $x$ is \emph{not free in} $e$
if $\Neg \mname{free-in}(\synbrack{x},\synbrack{e})$ is valid in
$T_{\rm ker}^{L}$, and $a$ is \emph{not free for} $x$ \emph{in} $e$ if
$\Neg \mname{free-for}(\synbrack{a},\synbrack{x},\synbrack{e})$ is
valid in $T_{\rm ker}^{L}$.  We also say $e$ is \emph{syntactically
  closed} if $\mname{syn-closed}(\synbrack{e})$ is valid in $T_{\rm
  ker}^{L}$.

\begin{cprop} \label{prop:quotation}
Every quotation is syntactically closed.
\end{cprop}

\subsection{Evaluation and Quasiquotation Lemmas}

The lemmas in this subsection are facts about evaluations that are
needed for the substitution lemmas given in the next subsection.  Let
$=_k$ be $\TypeEqual$ if $k$ is a $\mname{type}$, $\QuasiEqual$ if $k$
is a type, and $\Iff$ if $k$ is a $\mname{formula}$.

\begin{clem} [Good Evaluation Arguments] \label{lem:gea}
Let $T = (L,\Gamma)$ be a normal theory, $\alpha$ be a type, $a$ and
$b$ be terms, and $A$ be a formula of $L$.
\be

  \item $T \models \mname{gea}(\synbrack{\alpha},
    \synbrack{\mname{type}}) \Implies
    \sembrack{\synbrack{\alpha}}_{\rm ty} \TypeEqual \alpha$.

  \item $T \models \Neg \mname{gea}(b,\synbrack{\mname{type}})
    \Implies \sembrack{b}_{\rm ty} \TypeEqual \mname{C}$.

  \item $T \models \mname{gea}(\synbrack{a}, \synbrack{\alpha})
    \Implies \sembrack{\synbrack{a}}_\alpha \QuasiEqual \If(a \IsDef
    \alpha, a, \Undefined_{\sf C})$.

  \item $T \models \Neg \mname{gea}(b,\synbrack{\alpha}) \Implies
    \sembrack{b}_\alpha \QuasiEqual \Undefined_{\sf C}$.

  \item $T \models \mname{gea}(\synbrack{A},
    \synbrack{\mname{formula}}) \Implies
    \sembrack{\synbrack{A}}_{\rm fo} \Iff A$.

  \item $T \models \Neg \mname{gea}(b,\synbrack{\mname{formula}})
    \Implies \sembrack{b}_{\rm fo} \Iff \mname{F}$.

\ee
\end{clem}

\begin{proof} 
Follows from the definition of the standard valuation on evaluations.
\end{proof}

\begin{clem} [Evaluation of Quasiquotations] \label{lem:eval-qq}
Let $M$ be a standard model of a normal theory $T = (L,\Gamma)$ and
$a$ be a term of type \mname{E} of $L$ that denotes a type, term, or
formula of kind $k = \mname{type}$, \mname{C}, or \mname{formula},
respectively, such that $M \models \mname{gea}(a,\synbrack{k})$.

\be

  \item Let $a =
    \synbrack{(o::e_1,\ldots,e_{n+1})
      (\commabrack{c_1},\ldots,\commabrack{c_n})}$ denote an operator
    application where:
    \be

      \item $e_i$ is \mname{type}, \mname{formula}, or
        $\commabrack{b_i}$ for all $i$ with $1 \le i \le n+1$.

      \item $c_i$ denotes a type, term, or formula of kind $k_i =
        \mname{type}$, \mname{C}, or \mname{formula}, respectively,
        for all $i$ with $1 \le i \le n$.

    \ee 

    Then \[M \models \sembrack{a}_k =_k
    (o::\overline{e_1},\ldots,\overline{e_{n+1}})
    (\sembrack{c_1}_{k_i},\ldots,\sembrack{c_n}_{k_n})\]

    where \[\overline{e_i} = 
    \left\{
      \begin{array}{ll}
      \sembrack{b_i}_{\rm ty} & \mbox{if } e_i = \commabrack{b_i} \\
      e_i & \mbox{if }e_i=\mname{type}\mbox{ or } \mname{formula}.
      \end{array}
    \right.\]

  \item If $a = \synbrack{(x:\commabrack{b})}$ denotes a variable,
    then \[M \models \sembrack{a}_{\rm te} \QuasiEqual
    (x:\sembrack{b}_{\rm ty}).\]

  \item If $a = \synbrack{(\StarApp {x\mcolon \commabrack{b_1}} \mdot
    \commabrack{b_2})}$ denotes a variable binder where $\star$ is
    $\Lambda$, $\lambda$, $\Forsome$, $\iota$, or $\epsilon$ and $b_2$
    is semantically closed in $M$ with respect to $x$, then \[M
    \models \sembrack{a}_k =_k (\StarApp {x\mcolon \sembrack{b_1}_{\rm
        ty}} \mdot \sembrack{b_2}_{\rm fo}).\]

  \item If $a = \synbrack{\commabrack{b_1}(\commabrack{b_2})}$ denotes
    a type application, then \[M \models \sembrack{a}_{\rm ty}
    \TypeEqual \sembrack{b_1}_{\rm ty}(\sembrack{b_2}_{\rm te}).\]

  \item If $a = \synbrack{\commabrack{b_1}(\commabrack{b_2})}$ denotes
    a function application, then \[M \models \sembrack{a}_{\rm te}
    \QuasiEqual \sembrack{b_1}_{\rm te}(\sembrack{b_2}_{\rm te}).\]

  \item If $a = \synbrack{\If(\commabrack{b_1},\commabrack{b_2},
    \commabrack{b_3})}$ denotes a conditional term, then \[M \models
    \sembrack{a}_{\rm te} \QuasiEqual \If(\sembrack{b_1}_{\rm
      fo},\sembrack{b_2}_{\rm te}, \sembrack{b_3}_{\rm te}).\]

  \item If $a = \synbrack{\sembrack{\commabrack{b}}_{\rm ty}}$ denotes
    a type evaluation, then \[M \models \sembrack{a}_{\rm ty}
    \TypeEqual \sembrack{\sembrack{b}_{\rm te}}_{\rm ty}.\]

  \item If $a =
    \synbrack{\sembrack{\commabrack{b_1}}_{\commabrack{b_2}}}$ denotes
    a term evaluation, then \[M \models \sembrack{a}_{\rm te}
    \QuasiEqual \sembrack{\sembrack{b_1}_{\rm
        te}}_{\sembrack{b_2}_{ty}}.\]

  \item If $a = \synbrack{\sembrack{\commabrack{b}}_{\rm fo}}$ denotes
    a formula evaluation, then \[M \models \sembrack{a}_{\rm fo} \Iff
    \sembrack{\sembrack{b}_{\rm te}}_{\rm fo}.\]

\ee
\end{clem}

\begin{proof}
Let $M=(S,V)$ be a standard model of $T$ and $\phi \in
\mname{assign}(S)$.

\bigskip

\noindent \textbf{Parts 1, 4--9} \sglsp Similar to part 2.

\bigskip

\noindent \textbf{Part 2} \sglsp We must show that
$V_\phi(\sembrack{a}_{\rm te}) = V_\phi((x:\sembrack{b}_{\rm ty}))$.
  \begin{eqnarray*}
  & & V_\phi(\sembrack{a}_{\rm te}) \\
  & = & V_\phi(\sembrack{\synbrack{(x:\commabrack{b})}}_{\rm te}) \\
  & = & V_\phi(H^{-1}(V_\phi(\synbrack{(x:\commabrack{b})}))) \\
  & = & V_\phi(H^{-1}(V_\phi([\synbrack{\mname{var}},\synbrack{x},b]))) \\
  & = & V_\phi((H^{-1}(V_\phi(\synbrack{\mname{var}})),
               H^{-1}(V_\phi(\synbrack{x})),
               H^{-1}(V_\phi(b)))) \\
  & = & V_\phi((H^{-1}(H(\mname{var})),
               H^{-1}(H(x)),
               H^{-1}(V_\phi(b)))) \\
  & = & V_\phi((\mname{\mname{var}},x,\sembrack{b}_{\rm ty})) \\
  & = & V_\phi((x:\sembrack{b}_{\rm ty}))
  \end{eqnarray*}
The third line is by the definition of $V$ on evaluations and $M
\models \mname{gea}(a,\synbrack{\mname{C}})$; the fourth is by the
definition of a quasiquotation; the fifth is by the definitions of $H$
and $[a_1,\ldots,a_n]$; the sixth is by the definition of $V$ on
quotations; and the seventh is by the definition of $V$ on evaluations
and the fact that $M \models \mname{gea}(a,\synbrack{\mname{C}})$
implies $M \models \mname{gea}(b,\synbrack{\mname{type}})$.

\bigskip

\noindent \textbf{Part 3} \sglsp We must show that
$V_\phi(\sembrack{a}_k) = V_\phi((\StarApp {x\mcolon
  \sembrack{b_1}_{\rm ty}} \mdot \sembrack{b_2}_{\rm fo}))$ assuming
$b_2$ is semantically closed in $M$ with respect to $x$.
  \begin{eqnarray*}
  & & V_\phi(\sembrack{a}_k) \\
  & = & V_\phi(\sembrack{\synbrack{(\StarApp {x\mcolon \commabrack{b_1}} \mdot 
  \commabrack{b_2})}}_k) \\
  & = & V_\phi(H^{-1}(V_\phi(\synbrack{(\StarApp {x\mcolon \commabrack{b_1}} 
  \mdot \commabrack{b_2})}))) \\
  & = & V_\phi(H^{-1}(V_\phi(
  [\synbrack{\star},[\synbrack{\mname{var}},\synbrack{x},b_1],b_2]))) \\
  & = & V_\phi(
  (H^{-1}(V_\phi(\synbrack{\star})), \\
  & & \hspace{3.5ex} (H^{-1}(V_\phi(\synbrack{\mname{var}})),
  H^{-1}(V_\phi(\synbrack{x})),
  H^{-1}(V_\phi(b_1))),\\
  & & \hspace{3.5ex} H^{-1}(V_\phi(b_2)))) \\
  & = & V_\phi(
  (H^{-1}(H(\star)), \\
  & & \hspace{3.5ex} (H^{-1}(H(\mname{var})),
  H^{-1}(H(x)),
  H^{-1}(V_\phi(b_1))),\\
  & & \hspace{3.5ex} H^{-1}(V_\phi(b_2)))) \\
  & = & V_\phi((\star,
  (\mname{var}, x, \sembrack{b_1}_{\rm ty}),
  \sembrack{b_2}_{\rm fo})) \\
  & = & V_\phi((\StarApp {x\mcolon \sembrack{b_1}_{\rm ty}} \mdot
    \sembrack{b_2}_{\rm fo}))
  \end{eqnarray*}
The third line is by the definition of $V$ on evaluations and $M
\models \mname{gea}(a,\synbrack{k})$; the fourth is by the definition
of a quasiquotation; the fifth is by the definitions of $H$ and
$[a_1,\ldots,a_n]$; the sixth is by the definition of $V$ on
quotations; and the seventh is by the definition of $V$ on
evaluations, the fact that $M \models \mname{gea}(a,\synbrack{k})$
implies $M \models \mname{gea}(b_1,\synbrack{\mname{type}})$ and $M
\models \mname{gea}(b_2,\synbrack{\mname{formula}})$, and the fact
$b_2$ is semantically closed in $M$ with respect to $x$.
\end{proof}

\begin{clem} \label{lem:qq-sc} 
Let $M=(S,V)$ be a standard model of a normal theory $T = (L,\Gamma)$
and $q$ be a quasiquotation of $L$ whose set of evaluated components
is $\set{\commabrack{a_1},\ldots,\commabrack{a_n}}$.  If $a_i$ is
semantically closed in $M$ for all $i$ with $1 \le i \le n$, then $q$
is semantically closed in $M$.
\end{clem}

\begin{proof}
Let $\phi, \phi' \in \mname{assign}(S)$.  We must show that $V_\phi(q)
= V_{\phi'}(q)$.  Our proof is by induction on the length of $q$.  If
$q$ is a quotation, then $q$ is obviously semantically closed, and so
we may assume $q$ is not a quotation.  Let $q =
\synbrack{(m_1,\ldots,m_k)}$ where $k \ge 1$.
\begin{eqnarray*}
  & & V_\phi(q) \\
& = & V_\phi(\synbrack{(m_1,\ldots,m_k)}) \\
& = & V_\phi([\synbrack{m_1},\ldots,\synbrack{m_k}]) \\
& = & [V_\phi(\synbrack{m_1}),\ldots,V_\phi(\synbrack{m_k})] \\
& = & [V_{\phi'}(\synbrack{m_1}),\ldots,V_{\phi'}(\synbrack{m_k})] \\
& = & V_{\phi'}([\synbrack{m_1},\ldots,\synbrack{m_k}]) \\
& = & V_{\phi'}(\synbrack{(m_1,\ldots,m_k)}) \\
& = & V_{\phi'}(q)
\end{eqnarray*}
The third line is by the definition of a quasiquotation; the fourth is
by the definition of $[a_1,\ldots,a_n]$; and the fifth is by the
following argument.  If $\synbrack{m_i}$ is a quotation, then
$\synbrack{m_i}$ is semantically closed, and so
$V_\phi(\synbrack{m_i}) = V_{\phi'}(\synbrack{m_i})$.  If $m_i$ is
an evaluated component $\commabrack{a_j}$, then
$V_\phi(\synbrack{m_i}) = V_\phi(a_j) = V_{\phi'}(a_j) =
V_{\phi'}(\synbrack{m_i})$ since $a_j$ is semantically closed in $T$
by hypothesis.  If $\synbrack{m_i}$ is not a quotation and $m_i$ is
not an evaluated component, then $V_\phi(\synbrack{m_i}) =
V_{\phi'}(\synbrack{m_i})$ by the induction hypothesis.
\end{proof}

\subsection{Substitution Lemmas}

The next several lemmas show that the operators specified in the
previous subsection have their intended meanings.  Let $k[e]$ be
$\mname{type}$ if $e$ is a type, $\mname{C}$ if $e$ is a term, and
$\mname{formula}$ if $e$ is a formula.

\begin{clem} [Eval-Free] \label{lem:eval-free}
Let $L$ be a normal language and $a$ be an eval-free term, $x$ be a
symbol, and $e$ be an eval-free proper expression of $L$.  Let $T =
\mname{eval-free}(T_{\rm ker}^{L})$.
\be

  \item Either $T \models \mname{free-in}(\synbrack{x},\synbrack{e})$
    or $T \models \Neg \mname{free-in} (\synbrack{x},\synbrack{e})$.
    (That is, either $x$ is free in $e$ or $x$ is not free in $e$.)

  \item Either $T \models \mname{free-for} (\synbrack{a},
    \synbrack{x}, \synbrack{e})$ or $T \models \Neg \mname{free-for}
    (\synbrack{a}, \synbrack{x}, \synbrack{e})$.  That is, either
    $a$ is free for $x$ in $e$ or $a$ is not free for $x$ in $e$.

  \item $T \models \mname{free-in} (\synbrack{x},\synbrack{e})$ for
    at most finitely many symbols $x$.

  \item $T \models \mname{cleanse}(\synbrack{e}) = \synbrack{e}$ and
    thus $T \models \mname{cleanse}(\synbrack{e}) \IsDefApp$.

  \item $T \models \mname{sub} (\synbrack{a},
    \synbrack{x},\synbrack{e}) = \synbrack{e'}$ for some eval-free
    proper expression $e'$ and thus $T \models \mname{sub}
    (\synbrack{a}, \synbrack{x},\synbrack{e}) \IsDefApp$.

  \item If $T \models \Neg
    \mname{free-in}(\synbrack{x},\synbrack{e})$, then $T \models
    \mname{sub} (\synbrack{a}, \synbrack{x}, \synbrack{e}) =
    \synbrack{e}$.

\ee
\end{clem}

\begin{proof}

\bigskip

\noindent \textbf{Parts 1--2, 4--5} \sglsp Follows immediately by
induction on the length of $e$.

\bigskip

\noindent \textbf{Part 3} \sglsp Follows from the fact that $x$ is
free in $e$ iff $e$ contains a subexpression of the form $(x \mcolon
\alpha)$.

\bigskip

\noindent \textbf{Part 6} \sglsp Let $S(\synbrack{e'})$ mean
$\mname{sub} (\synbrack{a}, \synbrack{x}, \synbrack{e'})$.  Assume
\[T \models \Neg \mname{free-in}(\synbrack{x},\synbrack{e}) \dblsp 
\mbox{[designated $H(\synbrack{e})$]}.\] We must show that \[T
\models S(\synbrack{e}) = \synbrack{e} \dblsp \mbox{[designated
    $C(\synbrack{e})$]}.\] Our proof is by induction on the length of
$e$.  There are 10 cases corresponding to the 10 formula schemas used
to define $S(\synbrack{e})$ when $e$ is eval-free proper expression.

\be

  \item[] \textbf{Case 1}. $e = (o::k_1,\ldots,k_{n+1})$.
    $H(\synbrack{e})$ implies $H(\synbrack{k_i})$ for all $i$ with
    $1 \le i \le n + 1$ and $\ctype{L}{k_i}$.  From these hypotheses,
    $C(\synbrack{k_i})$ follows for all $i$ with $1 \le i \le n + 1$
    and $\ctype{L}{k_i}$ from the induction hypothesis.
    $C(\synbrack{k_i})$ also holds for all $i$ with $1 \le i \le n +
    1$ and $k_i = \mname{type}$ or $\mname{formula}$.  From these
    conclusions, $C(\synbrack{e})$ follows by the definition of
    \mname{sub}.

  \item[] \textbf{Cases 2, 4, 6--9}. Similar to case 1.

  \item[] \textbf{case 3}. $e = (x \mcolon \alpha)$.  The hypothesis
    $H(\synbrack{e})$ is false in this case.

  \item[] \textbf{case 5}. $e = (\StarApp {x\mcolon\alpha} \mdot e')$
    where $\star$ is $\Lambda$, $\lambda$, $\Forsome$, $\iota$, or
    $\epsilon$.  $H(\synbrack{e})$ implies $H(\synbrack{\alpha})$.
    From this hypothesis, $C(\synbrack{\alpha})$ follows from the
    induction hypothesis.  Then $C(\synbrack{e})$ follows from
    $C(\synbrack{\alpha})$ by part~4 of this lemma and the definition
    of \mname{sub}.

\iffalse
  \item[] \textbf{case 6}. $e = (\StarApp {y\mcolon\alpha} \mdot e')$
    where $x \not= y$ and $\star$ is $\Lambda$, $\lambda$, $\Forsome$,
    $\iota$, or $\epsilon$.  $H(\synbrack{e})$ implies
    $H(\synbrack{\alpha})$ and $H(\synbrack{e'})$.  From these
    hypotheses, $C(\synbrack{\alpha})$ and $C(\synbrack{e'})$ follow
    from the induction hypothesis.  Then $C(\synbrack{e})$ follows
    from $C(\synbrack{\alpha})$, $C(\synbrack{e'})$, and
    $H(\synbrack{\synbrack{e'}})$ by the definition of \mname{sub}.
\fi

  \item[] \textbf{case 10}. $e = \synbrack{e'}$.  $C(\synbrack{e})$
    is always true in this case by the definition of \mname{sub}.

\ee
\end{proof}

\begin{crem} \em
Lemma~\ref{lem:eval-free}, with $T_{\rm ker}^{L}$ used in place of
$\mname{eval-free}(T_{\rm ker}^{L})$, does not hold for all
non-eval-free expressions.  For instance, the example in
subsection~\ref{subsec:inf-dep-1} exhibits a non-eval free expression
for which part 1 of this lemma does not hold.
\end{crem}

A \emph{universal closure} of a formula $A$ is a formula \[\ForallApp
x_1,\ldots,x_n \mcolon \mname{C} \mdot A\] where $n \ge 0$ such that
$x$ is free in $A$ iff $x \in \set{x_1,\ldots,x_n}$ and $x$ is not
free in $A$ iff $x \not\in \set{x_1,\ldots,x_n}$.

\begin{clem} \label{lem:univ-closure}
Every universal closure is syntactically closed.
\end{clem}

\begin{proof}
By the definitions of \mname{free-in}, syntactically closed, and
universal closure.
\end{proof}

\begin{crem} \em
By virtue of parts~1 and~3 of Lemma~\ref{lem:eval-free}, universal
closures always exist for eval-free formulas, and by the example in
subsection~\ref{subsec:inf-dep-1}, may not exist for non-eval-free
formulas.
\end{crem}

%\newpage

\begin{clem} [Free Variable] \label{lem:fv}
Let $M=(S,V)$ be a standard model of a normal theory $T = (L,\Gamma)$,
$x$ be a symbol, and $e$ be a proper expression of $L$.  \be

  \item \bsp If $M \models \Neg
    \mname{free-in}(\synbrack{x},\synbrack{e})$, then $V_\phi(e) =
    V_{\phi[x \mapsto d]}(e)$ for all $\phi \in \mname{assign}(S)$ and
    $d \in\Dc$. (That is, if $M \models \Neg
    \mname{free-in}(\synbrack{x},\synbrack{e})$, $e$ is semantically
    closed in $M$ with respect to $x$.)\esp

  \item Let $X = \set{x \in \sS \;|\; M \models
    \Neg\mname{free-in}(\synbrack{x},\synbrack{e})}$.  Then
    $V_\phi(e) = V_{\phi'}(e)$ for all $\phi,\phi' \in
    \mname{assign}(S)$ such that $\phi(x) = \phi'(x)$ whenever $x
    \not\in X$.

\ee
\end{clem}

\begin{proof}  

\bigskip

\noindent \textbf{Part 1} \sglsp  We will show that,
if \[M \models \Neg \mname{free-in}(\synbrack{x},\synbrack{e}) \dblsp
\mbox{[designated $H(\synbrack{e})$]},\] then 

\bi

  \item[] $V_\phi(e) = V_{\phi[x \mapsto d]}(e)$ for all
    $\phi \in \mname{assign}(S)$ and $d \in\Dc$ \dblsp 
    {[designated $C(\synbrack{e})$]}.

\ei
Our proof is by induction on the complexity of $e$.  There are 11
cases corresponding to the 11 formula schemas used to define
$\mname{free-in}(\synbrack{e_1},\synbrack{e_2}))$ when $e_1$ is a
symbol and $e_2$ is a proper expression.

\be

  \item[] \textbf{case 1}: $e = (o::k_1,\ldots,k_{n+1})$.  By the
    definition of $\mname{free-in}$, $H(\synbrack{e})$ implies
    $H(\synbrack{k_i})$ for all $i$ with $1 \le i \le n + 1$ and
    $\ctype{L}{k_i}$.  By the induction hypothesis, this implies
    $C(\synbrack{k_i})$ for all $i$ with $1 \le i \le n + 1$ and
    $\ctype{L}{k_i}$.  Therefore, $C(\synbrack{e})$ holds since
    $V_\phi(e)$ is defined in terms of $V_\phi(k_i)$ for all $i$ with
    $1 \le i \le n + 1$ and $\ctype{L}{k_i}$.

  \item[] \textbf{case 2}: $e = O(e_1,\ldots,e_n)$.  By the definition
    of $\mname{free-in}$, $H(\synbrack{e})$ implies $H(\synbrack{O})$
    and $H(\synbrack{e_i})$ for all $i$ with $1 \le i \le n$.  By the
    induction hypothesis, this implies $C(\synbrack{O})$ and
    $C(\synbrack{e_i})$ for all $i$ with $1 \le i \le n$.  Therefore,
    $C(\synbrack{e})$ holds since $V_\phi(e)$ is defined in terms of
    $V_\phi(O)$ and $V_\phi(e_i)$ for all $i$ with $1 \le i \le n$.

\iffalse

  \item[] \textbf{case 1}: $e = (o::k_1,\ldots,k_{n+1})$.
    $C(\synbrack{e})$ follows immediately since $V_\phi(e)$ does not
    depend on $\phi$.

  \item[] \textbf{case 2}: $e = O(e_1,\ldots,e_n)$ and $O =
    (o::k_1,\ldots,k_{n+1})$.  By the definition of $\mname{free-in}$,
    $H(\synbrack{e})$ implies (1) $H(\synbrack{k_i})$ for all $i$
    with $1 \le i \le n + 1$ and $\ctype{L}{k_i}$ and (2)
    $H(\synbrack{e_i})$ for all $i$ with $1 \le i \le n$.  By the
    induction hypothesis, this implies (1) $C(\synbrack{k_i})$ for
    all $i$ with $1 \le i \le n + 1$ and $\ctype{L}{k_i}$ and (2)
    $C(\synbrack{e_i})$ for all $i$ with $1 \le i \le
    n$. $C(\synbrack{O})$ holds by case 1.  Therefore,
    $C(\synbrack{e})$ holds since $V_\phi(e)$ is defined in terms of
    $V_\phi(O)$, $V_\phi(k_i)$ for all $i$ with $1 \le i \le n + 1$
    and $\ctype{L}{k_i}$, and $V_\phi(e_i)$ for all $i$ with $1 \le i \le
    n$.

\fi

  \item[] \textbf{case 3}: $e = (x \mcolon \alpha)$.  By the
    definition of $\mname{free-in}$, $H(\synbrack{e})$ is false.
    Therefore, the lemma is true for this case.

  \item[] \textbf{case 4}: $e = (y \mcolon \alpha)$ where $x \not=
    y$. By the definition of $\mname{free-in}$, $H(\synbrack{e})$
    implies $H(\synbrack{\alpha})$.  By the induction hypothesis,
    this implies $C(\synbrack{\alpha})$.  Therefore,
    $C(\synbrack{e})$ holds since $V_\phi(e)$ is defined in terms of
    $\phi(y)$ and $V_\phi(\alpha)$.

  \item[] \textbf{case 5}: $e = (\StarApp x \mcolon \alpha \mdot e')$
    and $\star$ is $\Lambda$, $\lambda$, $\Forsome$, $\iota$, or
    $\epsilon$.  By the definition of $\mname{free-in}$,
    $H(\synbrack{e})$ implies $H(\synbrack{\alpha})$.  By the
    induction hypothesis, this implies $C(\synbrack{\alpha})$.  It is
    obvious that \[V_{\phi[x \mapsto d']}(e') = V_{\phi[x \mapsto
    d][x \mapsto d']}(e')\] for all $\phi \in \mname{assign}(S)$ and
    $d,d' \in\Dc$.  Therefore, $C(\synbrack{e})$ holds since
    $V_\phi(e)$ is defined in terms of $V_\phi(\alpha)$ and
    $V_{\phi[x \mapsto d']}(e')$ where $d' \in\Dc$.

  \item[] \bsp \textbf{case 6}: $e = (\StarApp y \mcolon \alpha \mdot
    e')$, $x \not= y$, and $\star$ is $\Lambda$, $\lambda$,
    $\Forsome$, $\iota$, or $\epsilon$.  By the definition of
    $\mname{free-in}$, $H(\synbrack{e})$ implies
    $H(\synbrack{\alpha})$ and $H(\synbrack{e'})$.  By the induction
    hypothesis, this implies $C(\synbrack{\alpha})$ and
    $C(\synbrack{e'})$.  The latter implies \[V_{\phi[y \mapsto
    d']}(e') = V_{\phi[y \mapsto d'][x \mapsto d]}(e') =
    V_{\phi[x \mapsto d][y \mapsto d']}(e')\] for all
    $\phi \in \mname{assign}(S)$ and $d,d' \in\Dc$.  (This is the
    place in the proof where the full strength of the induction
    hypothesis is needed.)  Therefore, $C(\synbrack{e})$ holds since
    $V_\phi(e)$ is defined in terms of $V_\phi(\alpha)$ and
    $V_{\phi[y \mapsto d']}(e')$ where $d' \in\Dc$. \esp

  \item[] \textbf{case 7}: $e = \alpha(a)$.  Similar to case 4.

  \item[] \textbf{Case 8}: $e = f(a)$.  Similar to case 4.

  \item[] \textbf{Case 9}: $e = \mname{if}(A,b,c)$.  Similar to case 4.

  \item[] \textbf{Case 10}: $e = \synbrack{e'}$.  $C(\synbrack{e})$
    follows immediately since $V_\phi(e)$ does not depend on $\phi$.

  \item[] \bsp \textbf{Case 11}: $e = \sembrack{a}_k$.  By the
    definition of $\mname{free-in}$, $H(\synbrack{e})$ implies
    (1)~$H(\synbrack{a})$, (2) $H(\synbrack{k})$ when
    $\ctype{L}{k}$, and (3) $H(a)$.  By the induction hypothesis, (1)
    and (2) imply $C(\synbrack{a})$ and $C(\synbrack{k})$ when
    $\ctype{L}{k}$.  Suppose $V_\phi(\mname{gea}(a,\synbrack{k})) =
    \TRUE$ for some $\phi \in \mname{assign}(S)$.  Then $V_\phi(a) =
    V_\phi(\synbrack{e'})$ for some eval-free expression $e'$ and
    (3)~implies $V_\phi(\Neg\mname{free-in}(\synbrack{x},
    \synbrack{e'})) = \TRUE$.  By Lemma~\ref{lem:eval-free}, this
    implies $H(\synbrack{e'})$, and so by the induction hypothesis,
    $C(\synbrack{e'})$ holds.  Therefore, $C(\synbrack{e})$ holds
    since $V_\phi(e)$ is defined in terms of $V_\phi(a)$, $V_\phi(k)$
    when $\ctype{L}{k}$, and $V_\phi(e')$.  Now suppose that
    $V_\phi(\mname{gea}(a,\synbrack{k})) = \FALSE$ for all $\phi \in
    \mname{assign}(S)$.  Then, by Lemma~\ref{lem:gea}, $V_\phi(e)$ is
    the undefined value for kind $k$ for all $\phi \in
    \mname{assign}(S)$.  $C(\synbrack{e})$ follows immediately since
    $V_\phi(e)$ does not depend on $\phi$. \esp

\ee

\noindent \textbf{Part 2} \sglsp Similar to the proof of part 1.
\end{proof}

\begin{clem}[Syntactically Closed] \label{lem:syn-closed} \bsp
Let $M$ be a standard model of a normal theory $T = (L,\Gamma)$ and
$e$ be an expression of $L$.  If $M \models
\mname{syn-closed}(\synbrack{e})$, then $e$ is semantically closed in
$M$. \esp
\end{clem}

\bsp
\begin{proof} 
Let $M = (S,V)$ be a standard model of $T$.  Assume $M \models
\mname{syn-closed}(\synbrack{e})$.  Then, by the definition of
\mname{syn-closed}, $e$ is a proper expression and $M \models \Neg
\mname{free-in}(\synbrack{x},\synbrack{e})$ for all $x \in \sS$.
Hence, by part 2 of Lemma~\ref{lem:fv}, $V_{\phi}(e) = V_{\phi'}(e)$
for all $\phi,\phi' \in \mname{assign}(S)$.  Therefore, $e$ is
semantically closed in $M$.
\end{proof}
\esp

\bigskip

The next lemma shows that part 1 of Lemma~\ref{lem:fv} is not
sufficient to prove the previous lemma.

\begin{clem}
Let $M=(S,V)$ be a standard model of a normal theory $T = (L,\Gamma)$
and $e$ be a proper expression of $L$.  Suppose $V_\phi(e) = V_{\phi[x
    \mapsto d]}(e)$ for all $\phi \in \mname{assign}(S)$, $x \in \sS$,
and $d \in\Dc$.  Then it is not necessary that $e$ be syntactically
closed.
\end{clem}

\begin{proof}
See the example in subsection~\ref{subsec:inf-dep-2}.
\end{proof}

\begin{clem} [Cleanse] \label{lem:cleanse} 
Let $M=(S,V)$ be a standard model for a normal theory $T=(L,\Gamma)$
and $e$ be an expression of $L$.

\be

  \item If $e$ is an operator, type, term, or formula and
    $V_\phi(\mname{cleanse}(\synbrack{e})) \not= \Undefined$,
    then \[H^{-1}(V_\phi(\mname{cleanse}(\synbrack{e})))\] is an
    operator similar to $e$, type, term, or formula, respectively,
    that is eval-free for all $\phi \in \mname{assign}(S)$.

  \item $\mname{cleanse}(\synbrack{e})$ is semantically closed in $M$.

  \item If $e$ is an operator, type, term, or formula that contains an
    evaluation $\sembrack{e'}_k$ not in a quotation and $M \models
    \mname{free-in}(\synbrack{x}, \synbrack{e'})$ for some $x \in
    \sS$, then \[M \models \mname{cleanse}(\synbrack{e})
    \IsUndefApp.\]

  \item If $e$ is a type, term, or formula and $M \models
    \mname{cleanse}(\synbrack{e}) \IsDefApp$, then \[M \models
    \sembrack{\mname{cleanse}(\synbrack{e})}_{k[e]} = e.\]

\ee
\end{clem}

\begin{proof} 
Let $A(\synbrack{e})$ mean $\mname{cleanse}(\synbrack{e})$.

\bigskip

\noindent \textbf{Part 1} \sglsp Follows straightforwardly by
induction on the complexity of $e$.

\bigskip

\noindent \textbf{Part 2} \sglsp Our proof is by induction on the
complexity of $e$.  We must show that $A(\synbrack{e})$ is
semantically closed in $M$.  There are 12 cases corresponding to the
12 formula schemas used to defined $\mname{cleanse}(\synbrack{e})$
when $e$ is an expression.

\be

  \item[] \textbf{Case 1}: $e$ is improper.  By the definition of
    \mname{cleanse}, $A(\synbrack{e}) = \synbrack{e}$.  The last
    expression is a quotation and is thus semantically closed in $M$
    by Proposition~\ref{prop:quotation} and
    Lemma~\ref{lem:syn-closed}.  Therefore, $A(\synbrack{e})$ is
    semantically closed in $M$.

  \item[] \textbf{Cases 2--8}: $A(\synbrack{e})$ is
    semantically closed in $M$ by the induction hypothesis, the
    definition of \mname{cleanse}, and Lemma~\ref{lem:qq-sc}.

\iffalse
  \item[] \textbf{Case 5}: $e = (\StarApp {x\mcolon\alpha} \mdot e')$
    where $\star$ is $\Lambda$, $\lambda$, $\Forsome$, $\iota$, or
    $\epsilon$.  By the
    definition of \mname{cleanse},
    \begin{eqnarray*}
    & & \mname{cleanse}((\StarApp {x\mcolon\alpha} \mdot e')) \\
    & = & \synbrack{(\StarApp 
      {x \mcolon \commabrack{\mname{cleanse}(\synbrack{\alpha})}} \mdot 
      \commabrack{\If(\Neg\mname{free-in}(\synbrack{x},\synbrack{\widehat{e'}}),
      \widehat{e'},
      \Undefined_{\sf C})})}
    \end{eqnarray*}
    in $M$ where $\widehat{e'} = \mname{cleanse}(\synbrack{e'})$.
    $\mname{cleanse}(\synbrack{\alpha})$ and
    \[\If(\Neg\mname{free-in}(\synbrack{x},\synbrack{\widehat{e'}}),
    \widehat{e'}, \Undefined_{\sf C})\] are semantically closed in $M$
    by the induction hypothesis.  Therefore,
    $\mname{cleanse}(\synbrack{e})$ is semantically closed in $M$ by
    Lemma~\ref{lem:qq-sc}.
\fi

  \item[] \textbf{Case 9}: $e = \synbrack{e'}$.  By the definition of
    \mname{cleanse}, \[A(\synbrack{e}) = A(\synbrack{\synbrack{e'}}) =
    \synbrack{\synbrack{e'}}\] in $M$.  The last expression is a
    quotation and is thus semantically closed in $M$ by
    Proposition~\ref{prop:quotation} and Lemma~\ref{lem:syn-closed}.
    Therefore, $A(\synbrack{e})$ is semantically closed in $M$.

  \item[] \textbf{Case 10}: $e = \sembrack{a}_{\rm ty}$.  By the
    definition of \mname{cleanse}, \[A(\synbrack{e}) =
    \If(\mname{syn-closed}(\widehat{a}),w,\Undefined_{\sf C})\] \bsp
    in $M$ where $\widehat{a} = A(\synbrack{a})$, $\widehat{a}' =
    \mname{coerce-to-type}(\sembrack{\widehat{a}}_{\rm te})$, and $w =
    \synbrack{\mname{if}(\mname{gea}(a,\synbrack{\mname{type}}),
      \commabrack{\widehat{a}'}, \mname{C})}$. If $M \models
    A(\synbrack{e})\IsUndefApp$, then $A(\synbrack{e})$ is
    semantically closed in $M$.  Therefore, it suffices to show that
    $w$ is semantically closed in $M$ under the assumption $M \models
    \mname{syn-closed}(\widehat{a})$.  Then $\widehat{a}'$ is
    semantically closed in $M$ by Lemma~\ref{lem:syn-closed}. This
    implies $w$ is semantically closed in $M$ by
    Lemma~\ref{lem:qq-sc}.\esp

  \item[] \textbf{Case 11}: $e = \sembrack{a}_\alpha$. Similar to
    case~10.

  \item[] \textbf{Case 12}: $e = \sembrack{a}_{\rm fo}$.  Similar to
    case~10.

\ee

\noindent \textbf{Part 3} \sglsp Follows immediately from the
definition of \mname{cleanse}.

\bigskip

\noindent \textbf{Part 4} \sglsp Let $e$ be a type, term, or formula.
Assume \[M \models A(\synbrack{e})\IsDefApp \dblsp \mbox{[designated
    $H(\synbrack{e})$]}.\] We must show that \[M \models
\sembrack{A(\synbrack{e})}_{k[e]} = e \dblsp \mbox{[designated
    $C(\synbrack{e})$]}.\] Our proof is by induction on the complexity
of $e$.  There are 10 cases corresponding to the 10 formula schemas
used to defined $\mname{cleanse}(\synbrack{e})$ when $e$ is a type,
term, or formula.

\be

  \item[] \textbf{Case 1}: $e = O(e_1,\ldots,e_n)$ and $O =
    (o::k_1,\ldots,k_{n+1})$.  By the definition of \mname{cleanse},
    $H(\synbrack{e})$ implies (1)~$H(\synbrack{k_i})$ for all $i$ with
    $1 \le i \le n + 1$ and $\ctype{L}{k_i}$ and (2)
    $H(\synbrack{e_i})$ for all $i$ with $1 \le i \le n$.  By the
    induction hypothesis, this implies (1)~$C(\synbrack{k_i})$ for all
    $i$ with $1 \le i \le n + 1$ and $\ctype{L}{k_i}$ and (2)
    $C(\synbrack{e_i})$ for all $i$ with $1 \le i \le n$.  Let $\phi
    \in \mname{assign}(S)$.  Then
    \begin{eqnarray*}
    & & V_\phi(\sembrack{A(\synbrack{e})}_{k[e]}) \\
    & = & V_\phi(\sembrack{A(\synbrack{O(e_1,\ldots,e_n)})}_{k[e]}) \\
    & = & V_\phi(\sembrack{\synbrack{\commabrack{A(\synbrack{O})}
    (\commabrack{A(\synbrack{e_1})}, \ldots, 
    \commabrack{A(\synbrack{e_n})})}}_{k[e]}) \\
    & = & V_\phi(\overline{A(\synbrack{O})}
    (\sembrack{A(\synbrack{e_1})}_{k[e_1]}, \ldots, 
    \sembrack{A(\synbrack{e_n})}_{k[e_n]})) \\
    & = & V_{\phi}(O(e_1,\ldots,e_n)) \\
    & = & V_{\phi}(e)    
    \end{eqnarray*}
    where \[\overline{A(\synbrack{O})} =
    (o::\overline{A(\synbrack{k_1})},\ldots,
    \overline{A(\synbrack{k_{n+1}})})\] and
    \[\overline{A(\synbrack{k_i})} = 
    \left\{\begin{array}{ll}
             \sembrack{A(\synbrack{k_i})}_{\rm ty} & \mbox{if }\ctype{L}{k_i} \\
             k_i & \mbox{if }k_i=\mname{type}\mbox{ or } \mname{formula}.
            \end{array}
    \right.\] \bsp The third line is by the definition of
    \mname{cleanse}; the fourth is by Lemma~\ref{lem:eval-qq} and part
    1 of this lemma; and the fifth holds since (1) $O$ and
    $\overline{A(\synbrack{O})}$ are similar by part 1 of this lemma,
    (2) $C(\synbrack{k_i})$ holds for all $i$ with $1 \le i \le n + 1$
    and $\ctype{L}{k_i}$, and (3) $C(\synbrack{e_i})$ holds for all $i$
    with $1 \le i \le n$.  Therefore, $C(\synbrack{e})$ holds. \esp

  \item[] \textbf{Cases 2, 4--6}: Similar to case 1.

  \item[] \textbf{Case 3}: $e = (\StarApp {x\mcolon\alpha} \mdot e')$
    where $\star$ is $\Lambda$, $\lambda$, $\Forsome$, $\iota$, or
    $\epsilon$.  By the definition of \mname{cleanse},
    $H(\synbrack{e})$ implies $H(\synbrack{\alpha})$ and
    $H(\synbrack{e'})$.  By the induction hypothesis, this implies
    $C(\synbrack{\alpha})$ and $C(\synbrack{e'})$.  Then
    \begin{eqnarray*}
    & & V_\phi(\sembrack{A(\synbrack{e})}_{k[e]}) \\
    & = & V_\phi(\sembrack{A(\synbrack{(\StarApp 
      {x\mcolon\alpha} \mdot e')})}_{k[e]}) \\
    & = & V_\phi(\sembrack{\synbrack{(\StarApp 
      {x \mcolon \commabrack{A(\synbrack{\alpha})}} \mdot 
      \commabrack{A(\synbrack{e'})})}}_{k[e]})\\
    & = & V_\phi((\StarApp 
      {x \mcolon \sembrack{A(\synbrack{\alpha})}_{\rm ty}} \mdot 
      \sembrack{A(\synbrack{e'})}_{k[e']})) \\
    & = & V_\phi((\StarApp {x \mcolon \alpha} \mdot e')) \\
    & = & V_\phi(e)
    \end{eqnarray*}
    \bsp The third line is by the definition of \mname{cleanse}; the
    fourth is by Lemma~\ref{lem:eval-qq} and parts 1 and 2 of this
    lemma; and the fifth is by $C(\synbrack{\alpha})$ and
    $C(\synbrack{e'})$.  Therefore, $C(\synbrack{e})$ holds.\esp

  \item[] \textbf{Case 7}: $e = \synbrack{e'}$.  Let $\phi \in
    \mname{assign}(S)$.  Then
    \begin{eqnarray*}
    & & V_\phi(\sembrack{A(\synbrack{e})}_{k[e]}) \\
    & = & V_\phi(\sembrack{A(\synbrack{\synbrack{e'}})}_{\rm te}) \\
    & = & V_\phi(\sembrack{\synbrack{\synbrack{e'}}}_{\rm te}) \\
    & = & V_\phi(\synbrack{e'}) \\
    & = & V_\phi(e).
    \end{eqnarray*}
    The third line is by the definition of \mname{cleanse} and the
    fourth is by Lemma~\ref{lem:gea}.  Therefore, $C(\synbrack{e})$
    holds.

  \item[] \bsp \textbf{Case 8}: $e = \sembrack{a}_{\rm ty}$.  Let
    $\phi \in \mname{assign}(S)$.  Suppose $V_\phi(\mname{gea}(a,
    \synbrack{\mname{type}})) = \TRUE$.  By the definition of
    \mname{cleanse}, $H(\synbrack{e})$ implies $H(\synbrack{a})$.  By
    the induction hypothesis, this implies $C(\synbrack{a})$.  Let $u
    = \mname{coerce-to-type}(\sembrack{A(\synbrack{a})}_{\rm te})$.
    Then
    \begin{eqnarray*}
    & & V_\phi(\sembrack{A(\synbrack{e})}_{k[e]}) \\
    & = & V_\phi(\sembrack{A(\synbrack{\sembrack{a}_{\rm ty}})}_{\rm ty}) \\
    & = & V_\phi(\sembrack{\If(\mname{syn-closed}(A(\synbrack{a})), \\ & &
      \hspace{2ex}\synbrack{\If(\mname{gea}(a,\synbrack{\mname{type}}),
      \commabrack{u}, 
      \mname{C})}, \\ & &
      \hspace{2ex}\Undefined_{\sf C})}_{\rm ty}) \\
    & = & V_\phi(\sembrack{
      \synbrack{\If(\mname{gea}(a,\synbrack{\mname{type}}),
      \commabrack{u}, 
      \mname{C})}}_{\rm ty}) \\
    & = & V_\phi(\mname{if}(\mname{gea}(a,\synbrack{\mname{type}}),
      \sembrack{u}_{\rm ty},
      \mname{C})) \\
    & = & V_\phi(\sembrack{\mname{coerce-to-type}
      (\sembrack{A(\synbrack{a})}_{\rm te})}_{\rm ty}) \\
    & = & V_\phi(\sembrack{\mname{coerce-to-type}(a)}_{\rm ty}) \\
    & = & V_\phi(\sembrack{a}_{\rm ty}) \\
    & = & V_\phi(e).
    \end{eqnarray*}
    The third line is by the definition of \mname{cleanse}; the fourth
    is by $H(\synbrack{e})$; the fifth is by Lemma~\ref{lem:eval-qq}
    and part 1 of this lemma; the sixth by the hypothesis that
    $V_\phi(\mname{gea}(a, \synbrack{\mname{type}})) = \TRUE$ and the
    definition of $u$; the seventh is by $C(\synbrack{a})$; and the
    eighth is by $V_\phi(\mname{gea}(a, \synbrack{\mname{type}})) =
    \TRUE$.  Therefore, $C(\synbrack{e})$ holds.

    Now suppose $V_\phi(\mname{gea}(a, \synbrack{\mname{type}})) =
    \FALSE$. By a similar derivation to the one above,
    $C(\synbrack{e})$ holds. \esp

  \item[] \bsp \textbf{Case 9}: $e = \sembrack{a}_\alpha$.  Let $\phi
    \in \mname{assign}(S)$.  Suppose $V_\phi(\mname{gea}(a,
    \synbrack{\alpha}) = \TRUE$ and $V_\phi(\sembrack{a}_{\rm te}
    \IsDef \alpha) = \TRUE$.  By the definition of \mname{cleanse},
    $H(\synbrack{e})$ implies $H(\synbrack{\alpha})$ and
    $H(\synbrack{a})$.  By the induction hypothesis, this implies
    $C(\synbrack{\alpha})$ and $C(\synbrack{a})$.  Let $u =
    \mname{coerce-to-term}(\sembrack{A(\synbrack{a})}_{\rm te})$.
    Then
    \begin{eqnarray*}
    & & V_\phi(\sembrack{A(\synbrack{e})}_{k[e]}) \\
    & = & V_\phi(\sembrack{A(\synbrack{\sembrack{a}_\alpha})}_{\rm te}) \\
    & = & V_\phi(\sembrack{\If(\mname{syn-closed}(A(\synbrack{a})), \\ & &
      \hspace{2ex}\synbrack{\If(\mname{gea}(a,\synbrack{\alpha})
      \And \commabrack{u} \IsDef \commabrack{A(\synbrack{\alpha})},
      \commabrack{u}, 
      \Undefined_{\sf C})}, \\ & &
      \hspace{2ex}\Undefined_{\sf C})}_{\rm te}) \\
    & = & V_\phi(\sembrack{\synbrack{\If(\mname{gea}(a,\synbrack{\alpha})
      \And \commabrack{u} \IsDef \commabrack{A(\synbrack{\alpha})},
      \commabrack{u}, 
      \Undefined_{\sf C})}}_{\rm te}) \\
    & = & V_\phi(\If(\mname{gea}(a,\synbrack{\alpha})
      \And \sembrack{u}_{\rm te} \IsDef 
      \sembrack{A(\synbrack{\alpha})}_{\rm ty},
      \sembrack{u}_{\rm te}, 
      \Undefined_{\sf C})) \\
    & = & V_\phi(
      \If(\sembrack{\mname{coerce-to-term}
      (\sembrack{A(\synbrack{a})}_{\rm te})}_{\rm te} \IsDef 
      \sembrack{A(\synbrack{\alpha})}_{\rm ty}, \\ & &
      \hspace{2ex}\sembrack{\mname{coerce-to-term}
      (\sembrack{A(\synbrack{a})}_{\rm te})}_{\rm te}, \\ & &
      \hspace{2ex}\Undefined_{\sf C})) \\
    & = & V_\phi(
      \If(\sembrack{\mname{coerce-to-term}(a)}_{\rm te} \IsDef \alpha, \\ & &
      \hspace{2ex}\sembrack{\mname{coerce-to-term}(a)}_{\rm te}, \\ & &
      \hspace{2ex}\Undefined_{\sf C})) \\
    & = & V_\phi(
      \If(\sembrack{a}_{\rm te} \IsDef \alpha, 
      \sembrack{a}_{\rm te}, 
      \Undefined_{\sf C})) \\
    & = & V_\phi(\sembrack{a}_\alpha) \\
    & = & V_\phi(e).
    \end{eqnarray*}
    The third line is by the definition of \mname{cleanse}; the fourth
    is by $H(\synbrack{e})$; the fifth is by Lemma~\ref{lem:eval-qq}
    and part 1 of this lemma; the sixth by the hypothesis that
    $V_\phi(\mname{gea}(a, \synbrack{\mname{type}})) = \TRUE$ and the
    definition of $u$; the seventh is by $C(\synbrack{a})$ and
    $C(\synbrack{\alpha})$; the eighth is by $V_\phi(\mname{gea}(a,
    \synbrack{\mname{type}})) = \TRUE$; and the ninth is by
    $V_\phi(\sembrack{a}_{\rm te} \IsDef \alpha) = \TRUE$.  Therefore,
    $C(\synbrack{e})$ holds.

    Now suppose $V_\phi(\mname{gea}(a, \synbrack{\mname{type}}))=
    \FALSE$ or $V_\phi(\sembrack{a}_{\rm te} \IsDef \alpha) = \FALSE$.
    By a similar derivation to the one above, $C(\synbrack{e})$
    holds. \esp

  \item[] \textbf{Case 10}: $e = \sembrack{a}_{\rm fo}$.
    Similar to case~8.

\ee
\end{proof}

%\newpage

\begin{clem} [Substitution A] \label{lem:sub-a} \bsp
Let $M=(S,V)$ be a standard model for a normal theory $T=(L,\Gamma)$
and $a$ be a term, $x$ be a symbol, and $e$ be an expression of $L$.

\be

  \item If $e$ is an operator, type, term, or formula and
    $V_\phi(\mname{sub}(\synbrack{a}, \synbrack{x}, \synbrack{e}))
    \not= \Undefined$, then \[H^{-1}(V_\phi(\mname{sub}(\synbrack{a},
    \synbrack{x}, \synbrack{e})))\] is an operator similar to $e$,
    type, term, or formula, respectively, that is eval-free for all
    $\phi \in \mname{assign}(S)$.

  \item $\mname{sub} (\synbrack{a}, \synbrack{x}, \synbrack{e})$ is
    semantically closed in $M$.

  \item If $e$ is an operator, type, term, or formula that contains an
    evaluation $\sembrack{e'}_k$ not in a quotation and $M \models
    \mname{free-in}(\synbrack{y}, \mname{sub} (\synbrack{a},
    \synbrack{x},\synbrack{e'}))$ for some $y \in \sS$, then \[M
    \models \mname{sub} (\synbrack{a}, \synbrack{x}, \synbrack{e})
    \IsUndefApp.\]

  \item If $e$ is a type, term, or formula, $M \models \mname{sub}
    (\synbrack{a}, \synbrack{x}, \synbrack{e}) \IsDefApp$, and $M
    \models\Neg \mname{free-in}(\synbrack{x},\synbrack{e})$,
    then \[M \models \sembrack{\mname{sub} (\synbrack{a},
      \synbrack{x}, \synbrack{e})}_{k[e]} = e.\]

\ee \esp
\end{clem}

\begin{proof} 
Let $S(\synbrack{e})$ mean \[\mname{sub}
(\synbrack{a}, \synbrack{x}, \synbrack{e})\] and
$E(\synbrack{e})$ mean \[H^{-1}(V_\phi(S(\synbrack{e}))).\]

\bigskip

\noindent \textbf{Part 1a} \sglsp Fix $\phi \in \mname{assign}(S)$.
Assume $e$ is an operator, type, term, or formula and
$V_\phi(S(\synbrack{e})) \not= \Undefined$.  We must show
$E(\synbrack{e})$ is an operator similar to $e$, type, term, or
formula, respectively.  Our proof is by induction on the complexity of
$e$.  There are 13 cases corresponding to the 13 formula schemas used
to define $\mname{sub}(\synbrack{e_1}, \synbrack{e_2},
\synbrack{e_3})$ when $e_1$ is a term, $e_2$ is a symbol, and $e_3$
is a proper expression.

\be

  \item[] \textbf{Case 1}: $e = (o::k_1,\ldots,k_{n+1})$.  By the
    first formula schema of the definition of $\mname{sub}$,
    $V_\phi(S(\synbrack{\mname{type}})) =
    V_\phi(\synbrack{\mname{type}})$ and
    $V_\phi(S(\synbrack{\mname{formula}})) =
    V_\phi(\synbrack{\mname{formula}})$, and so
    $E(\synbrack{\mname{type}}) = \mname{type}$ and
    $E(\synbrack{\mname{formula}}) = \mname{formula}$.  By the
    second formula schema of the definition of $\mname{sub}$,
    \begin{eqnarray*}
    & & E(\synbrack{e}) \\
    & = & H^{-1}(V_\phi(S(\synbrack{e}))) \\
    & = & H^{-1}(V_\phi(\synbrack{(o::\commabrack{S(\synbrack{k_1})},
      \ldots, \commabrack{S(\synbrack{k_{n+1}})})})) \\
    & = & H^{-1}(V_\phi(\synbrack{(o::E(\synbrack{k_1}),
      \ldots,E(\synbrack{k_{n+1}}))})) \\
    & = & (o::E(\synbrack{k_1}),\ldots,E(\synbrack{k_{n+1}})).
    \end{eqnarray*}
    By the induction hypothesis, for all $i$ with $1 \le i \le n + 1$
    and $\ctype{L}{k_i}$, $E(\synbrack{k_i})$ is a type. Therefore, if
    $e$ is an operator, then $E(\synbrack{e})$ is an operator similar
    to $e$.

  \item[] \textbf{Case 2}: $e = O(e_1,\ldots,e_n)$ and $O =
    (o::k_1,\ldots,k_{n+1})$.  By the third formula schema of the
    definition of $\mname{sub}$,
    \begin{eqnarray*}
    & & E(\synbrack{e}) \\
    & = & H^{-1}(V_\phi(S(\synbrack{e})))\\
    & = & H^{-1}(V_\phi(\synbrack{\commabrack{S(\synbrack{O})}
    (\commabrack{S(\synbrack{e_1})}, \ldots, 
    \commabrack{S(\synbrack{e_n})})})) \\
    & = & H^{-1}(V_\phi(\synbrack{E(\synbrack{O})(E(\synbrack{e_1}), 
       \ldots, E(\synbrack{e_n}))})) \\
    & = & E(\synbrack{O})(E(\synbrack{e_1}), \ldots, E(\synbrack{e_n})).
    \end{eqnarray*}
    By the induction hypothesis, (1) $E(\synbrack{O})$ is an operator
    similar to $O$ and (2) for all $i$ with $1 \le i \le n$, if $e_i$
    is a type, term, or formula, then $E(\synbrack{e_i})$ is a type,
    term, or formula, respectively.  Therefore, if $e$ is a type,
    term, or formula, then $E(\synbrack{e})$ is a type, term, or a
    formula, respectively.

  \item[] \bsp \textbf{Cases 3--13}. Similar to case 2.  cases 3 and 5
  use part 1 of Lemma~\ref{lem:cleanse}.\esp

\ee

\noindent \textbf{Part 1b} \sglsp Fix $\phi \in \mname{assign}(S)$.
Assume $e$ is an operator, type, term, or formula and
$V_\phi(\mname{sub}(\synbrack{a}, \synbrack{x}, \synbrack{e}))
\not= \Undefined$.  We must show $E(\synbrack{e})$ is eval-free.  Our
proof is by induction on the complexity of $e$.  There are 13 cases as
for part 1a.

\be

  \item[] \textbf{Cases 1--2, 4, 6--9}. Each \mname{eval} symbol
    occurring in $e$ that is not in a quotation is removed by the
    definition of \mname{sub}.  Hence any \mname{eval} symbol
    occurring in $E(\synbrack{e})$ that is not in a quotation must
    have been introduced explicitly in $S(\synbrack{e})$ or
    implicitly in $S(\synbrack{e})$ via a subcomponent
    $S(\synbrack{e'})$ of $S(\synbrack{e})$ where $e'$ is a
    component of $e$.  In these cases, no \mname{eval} symbols are
    explicitly introduced in $S(\synbrack{e})$ by the definition
    of \mname{sub}, and no \mname{eval} symbols are implicitly
    introduced in $S(\synbrack{e})$ by the induction hypothesis.
    Therefore, $E(\synbrack{e})$ is eval-free.

  \item[] \textbf{Case 3}. $e = (x \mcolon \alpha)$.
    $E(\synbrack{e})$ is either
    $H^{-1}(V_\phi(\mname{cleanse}(\synbrack{a})))$ or
    $\Undefined_{\sf C}$.  Therefore, $E(\synbrack{e})$ is eval-free
    since $H^{-1}(V_\phi(\mname{cleanse}(\synbrack{a})))$ is
    eval-free by part 1 of Lemma~\ref{lem:cleanse} and
    $\Undefined_{\sf C}$ is obviously eval-free.

  \item[] \textbf{Case 5}. $e = (\StarApp {x\mcolon\alpha} \mdot e')$
    where $\star$ is $\Lambda$, $\lambda$, $\Forsome$, $\iota$, or
    $\epsilon$.  $E(\synbrack{\alpha})$ is eval-free by the induction
    hypothesis, and $H^{-1}(V_\phi(\mname{cleanse}(\synbrack{e'})))$
    is eval-free by part 1 of Lemma~\ref{lem:cleanse}.  Therefore,
    $E(\synbrack{e})$ is eval-free by the definition of \mname{sub}.

  \item[] \textbf{Case 10}. $e = \synbrack{e'}$.  $E(\synbrack{e}) =
    e$.  Therefore, $E(\synbrack{e})$ is eval-free since $e$ is a
    quotation and hence eval-free.

  \item[] \textbf{Cases 11--13}. $e = \sembrack{b}_k$.  By the
    induction hypothesis, $E(S(\synbrack{b}))$ is eval-free,
    $E(S(\sembrack{S(\synbrack{b})}_{\rm te}))$ is eval-free, and
    $E(S(\synbrack{k}))$ is eval-free if $\ctype{L}{k}$.  Therefore,
    $E(\synbrack{e})$ is eval-free by the definition of \mname{sub}.

\iffalse

  \item[] \textbf{Cases 11--13}. $e = \sembrack{b}_k$ and either (1)
    $H^{-1}(V_\phi(\sembrack{S(\synbrack{b})}_{\rm te}))$ is
    eval-free and $E(\synbrack{e})$ is
    $E(S(\sembrack{S(\synbrack{b})}_{\rm te}))$ or (2)
    $E(\synbrack{e})$ is $\mname{C}$, $\Undefined_{\sf C}$, or
    $\mname{F}$. If (1), $E(S(\sembrack{S(\synbrack{b})}_{\rm te}))$
    is eval-free by the induction hypothesis, and if (2), $\mname{C}$,
    $\Undefined_{\sf C}$, and $\mname{F}$ are obviously eval-free.
    Therefore, $E(\synbrack{e})$ is eval-free.

\fi

\ee

\noindent \textbf{Part 2} \sglsp Similar to the proof of part 2 of
Lemma~\ref{lem:cleanse}.

\bigskip

\noindent \textbf{Part 3} \sglsp Follows immediately from the
definition of \mname{sub}.

\bigskip

\noindent \textbf{Part 4} \sglsp Let $e$ is a type, term, or formula.
Assume \[M \models S(\synbrack{e}) \IsDefApp \dblsp \mbox{[designated
    $H_1(\synbrack{e})$]}\] and
\[M \models \Neg \mname{free-in}(\synbrack{x},\synbrack{e}) \dblsp
\mbox{[designated $H_2(\synbrack{e})$]}.\] We must show that \[M
\models \sembrack{S(\synbrack{e})}_{k[e]} = e \dblsp \mbox{[designated
    $C(\synbrack{e})$]}.\] Our proof is by induction on the complexity
of $e$.  There are 12 cases corresponding to the 12 formula schemas
used to define $\mname{sub}(\synbrack{e_1}, \synbrack{e_2},
\synbrack{e_3})$ when $e_1$ is a term, $e_2$ is a symbol, and $e_3$ is
a type, term, or formula.

\be

  \item[] \textbf{Case 1}: $e = O(e_1,\ldots,e_n)$ and $O =
    (o::k_1,\ldots,k_{n+1})$.  By the definitions of $\mname{sub}$ and
    $\mname{free-in}$, $H_1(\synbrack{e})$ and $H_2(\synbrack{e})$
    imply (1) $H_1(\synbrack{k_i})$ and $H_2(\synbrack{k_i})$ for
    all $i$ with $1 \le i \le n + 1$ and $\ctype{L}{k_i}$ and
    (2)~$H_1(\synbrack{e_i})$ and $H_2(\synbrack{e_i})$ for all $i$
    with $1 \le i \le n$.  By the induction hypothesis, this implies
    (1)~$C(\synbrack{k_i})$ for all $i$ with $1 \le i \le n + 1$ and
    $\ctype{L}{k_i}$ and (2) $C(\synbrack{e_i})$ for all $i$ with $1
    \le i \le n$.  Let $\phi \in \mname{assign}(S)$.  Then
    \begin{eqnarray*}
    & & V_\phi(\sembrack{S(\synbrack{e})}_{k[e]}) \\
    & = & V_\phi(\sembrack{S(\synbrack{O(e_1,\ldots,e_n)})}_{k[e]}) \\
    & = & V_\phi(\sembrack{\synbrack{\commabrack{S(\synbrack{O})}
    (\commabrack{S(\synbrack{e_1})}, \ldots, 
    \commabrack{S(\synbrack{e_n})})}}_{k[e]}) \\
    & = & V_\phi(\overline{S(\synbrack{O})}
    (\sembrack{S(\synbrack{e_1})}_{k[e_1]}, \ldots, 
    \sembrack{S(\synbrack{e_n})}_{k[e_n]})) \\
    & = & V_{\phi}(O(e_1,\ldots,e_n)) \\
    & = & V_{\phi}(e)    
    \end{eqnarray*}
    where \[\overline{S(\synbrack{O})} =
    (o::\overline{S(\synbrack{k_1})},\ldots,
    \overline{S(\synbrack{k_{n+1}})})\] and
    \[\overline{S(\synbrack{k_i})} = 
    \left\{\begin{array}{ll}
             \sembrack{S(\synbrack{k_i})}_{\rm ty} & \mbox{if }\ctype{L}{k_i} \\
             k_i & \mbox{if }k_i=\mname{type}\mbox{ or } \mname{formula}.
            \end{array}
    \right.\] The third line is by the definition of \mname{sub}; the
    fourth is by Lemma~\ref{lem:eval-qq} and part 1 of this lemma; and
    the fifth holds since (1) $O$ and $\overline{S(\synbrack{O})}$ are
    similar by part 1 of this lemma, (2) $C(\synbrack{k_i})$ holds for
    all $i$ with $1 \le i \le n + 1$ and $\ctype{L}{k_i}$, and (3)
    $C(\synbrack{e_i})$ holds for all $i$ with $1 \le i \le n$.
    Therefore, $C(\synbrack{e})$ holds.

  \item[] \textbf{Case 2}: $e = (x \mcolon \alpha)$.  The hypothesis
    $H_2(\synbrack{e})$ is false in this case.

  \item[] \textbf{Case 3}: $e = (y \mcolon \alpha)$ where $x \not= y$.
    By the definitions of $\mname{sub}$ and $\mname{free-in}$,
    $H_1(\synbrack{e})$ and $H_2(\synbrack{e})$ imply
    $H_1(\synbrack{\alpha})$ and $H_2(\synbrack{\alpha})$.  By the
    induction hypothesis, this implies $C(\synbrack{\alpha})$.  Let
    $\phi \in \mname{assign}(S)$.  Then
    \begin{eqnarray*}
    & & V_\phi(\sembrack{S(\synbrack{e})}_{k[e]}) \\
    & = & V_\phi(\sembrack{S(\synbrack{(y \mcolon \alpha)})}_{\rm te}) \\
    & = & V_\phi(\sembrack{\synbrack{(y 
    \mcolon \commabrack{S(\alpha)})}}_{\rm te}) \\
    & = & V_\phi((y \mcolon 
    \sembrack{S(\synbrack{\alpha})}_{\rm ty})) \\
    & = & V_{\phi}((y \mcolon \alpha)) \\
    & = & V_{\phi}(e).
    \end{eqnarray*}
    The third line is by the definition of \mname{sub}; the fourth is
    by Lemma~\ref{lem:eval-qq} and part 1 of this lemma; and the fifth
    is by $C(\synbrack{\alpha})$.  Therefore, $C(\synbrack{e})$ holds.

  \item[] \textbf{Case 4}: $e = (\StarApp {x\mcolon\alpha} \mdot e')$
    where $\star$ is $\Lambda$, $\lambda$, $\Forsome$, $\iota$, or
    $\epsilon$.  By the definitions of $\mname{sub}$ and
    $\mname{free-in}$, $H_1(\synbrack{e})$ and $H_2(\synbrack{e})$
    imply $H_1(\synbrack{\alpha})$, $H_2(\synbrack{\alpha})$, and $M
    \models \mname{cleanse}(\synbrack{e'}) \IsDefApp$.  By the
    induction hypothesis, $H_1(\synbrack{\alpha})$ and
    $H_2(\synbrack{\alpha})$ implies $C(\synbrack{\alpha})$.  Let
    $\phi \in \mname{assign}(S)$.  Then
    \begin{eqnarray*}
    & & V_\phi(\sembrack{S(\synbrack{e})}_{k[e]}) \\
    & = & V_\phi(\sembrack{S(\synbrack{(\StarApp 
      {x\mcolon\alpha} \mdot e')})}_{k[e]}) \\
    & = & V_\phi(\sembrack{\synbrack{(\StarApp 
      {x \mcolon \commabrack{S(\synbrack{\alpha})}} \mdot 
      \commabrack{\mname{cleanse}(\synbrack{e'})})}}_{k[e]})\\
    & = & V_\phi((\StarApp 
      {x \mcolon \sembrack{S(\synbrack{\alpha})}_{\rm ty}} \mdot 
      \sembrack{\mname{cleanse}(\synbrack{e'})}_{k[e']})) \\
    & = & V_\phi((\StarApp {x \mcolon \alpha} \mdot e')) \\
    & = & V_\phi(e)
    \end{eqnarray*} 
    The third line is by the definition of \mname{sub}; the fourth is
    by Lemma~\ref{lem:eval-qq}, parts 1 and 2 of this lemma, and parts
    1 and 2 of Lemma~\ref{lem:cleanse}; and the fifth is by
    $C(\synbrack{\alpha})$, $M \models \mname{cleanse}(\synbrack{e'})
    \IsDefApp$, and part 4 of Lemma~\ref{lem:cleanse}.

  \item[] \textbf{Case 5--8}: Similar to case 3.

  \item[] \textbf{Case 9}.  $e = \synbrack{e'}$.  $C(\synbrack{e})$
    is always true by the definition of \mname{sub}.

  \item[] \bsp \textbf{Case 10}: $e = \sembrack{b}_{\rm ty}$.  Let
    $\phi \in \mname{assign}(S)$.  Suppose $V_\phi(\mname{gea}(b,
    \synbrack{\mname{type}})) = \TRUE$.  Then there is an eval-free
    expression $e'$ such that $V_\phi(b) = V_\phi(\synbrack{e'})$.  By
    the definitions of $\mname{sub}$ and $\mname{free-in}$,
    $H_1(\synbrack{e})$ and $H_2(\synbrack{e})$ imply (1)
    $H_1(\synbrack{b})$ and $H_2(\synbrack{b})$ and
    (2)~$H_1(\synbrack{e'})$ and $H_2(\synbrack{e'})$.  By the
    induction hypothesis, this implies (1)~$C(\synbrack{b})$ and (2)
    $C(\synbrack{e'})$.  Let $u = \mname{coerce-to-type}
    (S(\sembrack{S(\synbrack{b})}_{\rm te}))$. Then
    \begin{eqnarray*}
    & & V_\phi(\sembrack{S(\synbrack{e})}_{k[e]}) \\
    & = & V_\phi(\sembrack{S(\synbrack{\sembrack{b}_{\rm ty}})}_{\rm ty}) \\
    & = & V_\phi(\sembrack{\If(\mname{syn-closed}
      (S(\synbrack{b})), \\ & &
      \hspace{2ex}\synbrack{\If(\mname{gea}
      (\commabrack{S(\synbrack{b})},\synbrack{\mname{type}}),
      \commabrack{u}, 
      \mname{C})}, \\ & &
      \hspace{2ex}\Undefined_{\sf C})}_{\rm ty}) \\
    & = & V_\phi(\sembrack{
      \synbrack{\If(\mname{gea}
      (\commabrack{S(\synbrack{b})},\synbrack{\mname{type}}),
      \commabrack{u}, 
      \mname{C})}}_{\rm ty}) \\
    & = & V_\phi(\mname{if}(\mname{gea}
      (\sembrack{S(\synbrack{b})}_{\rm te},\synbrack{\mname{type}}),
      \sembrack{u}_{\rm ty},
      \mname{C})) \\
    & = & V_\phi(\mname{if}(\mname{gea}(b,\synbrack{\mname{type}}),
      \sembrack{u}_{\rm ty},
      \mname{C})) \\
    & = & V_\phi(\sembrack{\mname{coerce-to-type}
      (S(\sembrack{S(\synbrack{b})}_{\rm te}))}_{\rm ty}) \\
    & = & V_\phi(\sembrack{\mname{coerce-to-type}(S(b))}_{\rm ty}) \\
    & = & V_\phi(\sembrack{S(b)}_{\rm ty}) \\
    & = & V_\phi(\sembrack{S(\synbrack{e'})}_{\rm ty}) \\
    & = & V_\phi(e') \\
    & = & V_{\phi}(\sembrack{\synbrack{e'}}_{\rm ty}) \\
    & = & V_{\phi}(\sembrack{b}_{\rm ty}) \\
    & = & V_\phi(e).
    \end{eqnarray*}
    The third line is by the definition of \mname{sub}; the fourth is
    by $H_1(\synbrack{e})$; the fifth is by Lemma~\ref{lem:eval-qq}
    and part 1 of this lemma; the sixth is by $C(\synbrack{b})$; the
    seventh is by the hypothesis that $V_\phi(\mname{gea}(b,
    \synbrack{\mname{type}})) = \TRUE$ and the definition of $u$; the
    eight is by $C(\synbrack{b})$; the ninth is by
    $V_\phi(\mname{gea}(b, \synbrack{\mname{type}})) = \TRUE$ and part
    1 of this lemma; the tenth is by substitution of $\synbrack{e'}$
    for $b$; the eleventh is by $C(\synbrack{e'})$; the twelveth is by
    the semantics of evaluation; and the thirteenth is by substitution
    of $b$ for $\synbrack{e'}$.  Therefore, $C(\synbrack{e})$ holds.

    Now suppose $V_\phi(\mname{gea}(b, \synbrack{\mname{type}})) =
    \FALSE$.  By a similar derivation to the one above,
    $C(\synbrack{e})$ holds. \esp

  \item[] \bsp \textbf{Case 11}: $e = \sembrack{b}_\alpha$.  Let $\phi
    \in \mname{assign}(S)$.  Suppose $V_\phi(\mname{gea}(b,
    \synbrack{\mname{type}})= \TRUE$ and $V_\phi(\sembrack{b}_{\rm te}
    \IsDef \alpha) = \TRUE$.  Then there is an eval-free expression
    $e'$ such that $V_\phi(b) = V_\phi(\synbrack{e'})$.  By the
    definitions of $\mname{sub}$ and $\mname{free-in}$,
    $H_1(\synbrack{e})$ and $H_2(\synbrack{e})$ imply (1)
    $H_1(\synbrack{b})$ and $H_2(\synbrack{b})$, (2)
    $H_1(\synbrack{\alpha})$ and $H_2(\synbrack{\alpha})$, and
    (3)~$H_1(\synbrack{e'})$ and $H_2(\synbrack{e'})$.  By the
    induction hypothesis, this implies (1) $C(\synbrack{b})$, (2)
    $C(\synbrack{\alpha})$, and (3)~$C(\synbrack{e'})$.  Let $u =
    \mname{coerce-to-term} (S(\sembrack{S(\synbrack{b})}_{\rm
      te}))$. Then
    \begin{eqnarray*}
    & & V_\phi(\sembrack{S(\synbrack{e})}_{k[e]}) \\
    & = & V_\phi(\sembrack{S(\synbrack{\sembrack{b}_\alpha})}_{\rm te}) \\
    & = & V_\phi(\sembrack{\If(\mname{syn-closed}
      (S(\synbrack{b})), \\ & &
      \hspace{2ex}\synbrack{\If(\mname{gea}
      (\commabrack{S(\synbrack{b})},\synbrack{\alpha}) \And 
      \commabrack{u} \IsDef \commabrack{S(\synbrack{\alpha})},
      \commabrack{u}, 
      \Undefined_{\sf C})}, \\ & &
      \hspace{2ex}\Undefined_{\sf C})}_{\rm te}) \\
    & = & V_\phi(\sembrack{
      \synbrack{\If(\mname{gea}
      (\commabrack{S(\synbrack{b})},\synbrack{\alpha}) \And 
      \commabrack{u} \IsDef \commabrack{S(\synbrack{\alpha})},
      \commabrack{u}, 
      \Undefined_{\sf C})}}_{\rm te}) \\
    & = & V_\phi(\mname{if}(\mname{gea}
      (\sembrack{S(\synbrack{b})}_{\rm te},\synbrack{\alpha}) \And 
      \sembrack{u}_{\rm te} \IsDef \sembrack{S(\synbrack{\alpha})}_{\rm ty},
      \sembrack{u}_{\rm te},
      \Undefined_{\sf C})) \\
    & = & V_\phi(\mname{if}(\mname{gea}(b,\synbrack{\alpha}) \And 
      \sembrack{u}_{\rm te} \IsDef \alpha,
      \sembrack{u}_{\rm te},
      \Undefined_{\sf C})) \\
    & = & V_\phi(\mname{if}(\sembrack{\mname{coerce-to-term}
      (S(\sembrack{S(\synbrack{b})}_{\rm te}))}_{\rm te} \IsDef \alpha, \\ & &
      \hspace{2ex}\sembrack{\mname{coerce-to-term}
      (S(\sembrack{S(\synbrack{b})}_{\rm te}))}_{\rm te}, \\ & &
      \hspace{2ex}\Undefined_{\sf C})) \\
    & = & V_\phi(\mname{if}(\sembrack{\mname{coerce-to-term}
      (S(b))}_{\rm te} \IsDef \alpha, \\ & &
      \hspace{2ex}\sembrack{\mname{coerce-to-term}
      (S(b))}_{\rm te}, \\ & &
      \hspace{2ex}\Undefined_{\sf C})) \\
    & = & V_\phi(\mname{if}(\sembrack{S(b)}_{\rm te} \IsDef \alpha, 
      \sembrack{S(b)}_{\rm te}, 
      \Undefined_{\sf C})) \\
    & = & V_\phi(\mname{if}(\sembrack{S(\synbrack{e'})}_{\rm te} \IsDef \alpha, 
      \sembrack{S(\synbrack{e'})}_{\rm te}, 
      \Undefined_{\sf C})) \\
    & = & V_\phi(\mname{if}(e' \IsDef \alpha, 
      e', 
      \Undefined_{\sf C})) \\
    & = & V_\phi(\mname{if}(\sembrack{\synbrack{e'}}_{\rm te} \IsDef \alpha, 
      \sembrack{\synbrack{e'}}_{\rm te}, 
      \Undefined_{\sf C})) \\
    & = & V_\phi(\mname{if}(\sembrack{b}_{\rm te} \IsDef \alpha, 
      \sembrack{b}_{\rm te}, 
      \Undefined_{\sf C})) \\
    & = & V_{\phi}(\sembrack{b}_\alpha) \\
    & = & V_\phi(e).
    \end{eqnarray*}
    The third line is by the definition of \mname{sub}; the fourth is
    by $H_1(\synbrack{e})$; the fifth is by Lemma~\ref{lem:eval-qq}
    and part 1 of this lemma; the sixth is by $C(\synbrack{b})$ and
    $C(\synbrack{\alpha})$; the seventh is by the hypothesis that
    $V_\phi(\mname{gea}(b, \synbrack{\mname{type}})) = \TRUE$ and the
    definition of $u$; the eight is by $C(\synbrack{b})$; the ninth is
    by $V_\phi(\mname{gea}(b, \synbrack{\mname{type}})) = \TRUE$ and
    part 1 of this lemma; the tenth is by substitution of
    $\synbrack{e'}$ for $b$; the eleventh is by $C(\synbrack{e'})$;
    the twelveth is by the semantics of evaluation; the thirteenth is
    by substitution of $b$ for $\synbrack{e'}$; and the fourteenth in
    by $V_\phi(\sembrack{b}_{\rm te} \IsDef \alpha) = \TRUE$.
    Therefore, $C(\synbrack{e})$ holds.

    Now suppose $V_\phi(\mname{gea}(b, \synbrack{\mname{type}})) =
    \FALSE$ or $V_\phi(\sembrack{b}_{\rm te} \IsDef \alpha) = \FALSE$
    By a similar derivation to the one above, $C(\synbrack{e})$
    holds. \esp

  \item[] \textbf{Case 12}: $e = \sembrack{b}_{\rm fo}$.  Similar to
    case 10.

\ee
\end{proof}

\begin{clem} [Substitution B] \label{lem:sub-b}
Let $M=(S,V)$ be a standard model for a normal theory $T=(L,\Gamma)$
and $a$ be a term, $x$ be a symbol, and $e$ be a type, term, or
formula of $L$.  If $M \models \mname{sub} (\synbrack{a},
\synbrack{x}, \synbrack{e}) \IsDefApp$ and $M \models
\mname{free-for}(\synbrack{a}, \synbrack{x}, \synbrack{e})$, then
\[V_\phi(\sembrack{\mname{sub}(\synbrack{a},
\synbrack{x},\synbrack{e})}_{k[e]}) = V_{\phi[x \mapsto
    V_\phi(a)]}(e)\] for all $\phi \in \mname{assign}(S)$ such that
$V_\phi(a) \not= \Undefined$.
\end{clem}

\begin{proof}  Let $S(\synbrack{e})$ mean
$\mname{sub}(\synbrack{a}, \synbrack{x}, \synbrack{e})$.  Assume
\[M \models S(\synbrack{e}) \IsDefApp \dblsp
  \mbox{[designated $H_1(\synbrack{e})$]},\] and \[M \models
  \mname{free-for}(\synbrack{a}, \synbrack{x}, \synbrack{e}) \dblsp
  \mbox{[designated $H_2(\synbrack{e})$]}.\] We must show that \bi

  \item[] \bsp $V_\phi(\sembrack{S(\synbrack{e})}_{k[e]}) = V_{\phi[x
      \mapsto V_\phi(a)]}(e)$ for all $\phi \in \mname{assign}(S)$
    such that $V_\phi(a) \not= \Undefined$ \dblsp {[designated
        $C(\synbrack{e})$]}. \esp

\ei
Our proof is by induction on the complexity of $e$.  There are 12 cases
corresponding to the 12 formula schemas used to define
$\mname{sub}(\synbrack{e_1},\synbrack{e_2},\synbrack{e_3}))$ when
$e_1$ is a term, $e_2$ is a symbol, $e_3$ is a type, term, or formula.

\be

  \item[] \textbf{Case 1}: $e = O(e_1,\ldots,e_n)$ and $O =
    (o::k_1,\ldots,k_{n+1})$.  By the definitions of $\mname{sub}$ and
    $\mname{free-for}$, $H_1(\synbrack{e})$ and $H_2(\synbrack{e})$
    imply (1) $H_1(\synbrack{k_i})$ and $H_2(\synbrack{k_i})$ for
    all $i$ with $1 \le i \le n + 1$ and $\ctype{L}{k_i}$ and (2)
    $H_1(\synbrack{e_i})$ and $H_2(\synbrack{e_i})$ for all $i$ with
    $1 \le i \le n$.  By the induction hypothesis, this implies (1)
    $C(\synbrack{k_i})$ for all $i$ with $1 \le i \le n + 1$ and
    $\ctype{L}{k_i}$ and (2) $C(\synbrack{e_i})$ for all $i$ with $1
    \le i \le n$.  Let $\phi \in \mname{assign}(S)$ such that
    $V_\phi(a) \not= \Undefined$.  Then
    \begin{eqnarray*}
    & & V_\phi(\sembrack{S(\synbrack{e})}_{k[e]}) \\
    & = & V_\phi(\sembrack{S(\synbrack{O(e_1,\ldots,e_n)})}_{k[e]}) \\
    & = & V_\phi(\sembrack{\synbrack{\commabrack{S(\synbrack{O})}
    (\commabrack{S(\synbrack{e_1})}, \ldots, 
    \commabrack{S(\synbrack{e_n})})}}_{k[e]}) \\
    & = & V_\phi(\overline{S(\synbrack{O})}
    (\sembrack{S(\synbrack{e_1})}_{k[e_1]}, \ldots, 
    \sembrack{S(\synbrack{e_n})}_{k[e_n]})) \\
    & = & V_{\phi[x \mapsto V_\phi(a)]}(O(e_1,\ldots,e_n)) \\
    & = & V_{\phi[x \mapsto V_\phi(a)]}(e)    
    \end{eqnarray*}
    where \[\overline{S(\synbrack{O})} =
    (o::\overline{S(\synbrack{k_1})},\ldots,
    \overline{S(\synbrack{k_{n+1}})})\] and
    \[\overline{S(\synbrack{k_i})} = 
    \left\{\begin{array}{ll}
             \sembrack{S(\synbrack{k_i})}_{\rm ty} & \mbox{if }\ctype{L}{k_i} \\
             k_i & \mbox{if }k_i=\mname{type}\mbox{ or }\mname{formula}.
            \end{array}
    \right.\] The third line is by the definition of \mname{sub}; the
    fourth is by Lemma~\ref{lem:eval-qq} and part 1 of
    Lemma~\ref{lem:sub-a}; and the fifth holds since (1) $O$ and
    $\overline{S(\synbrack{O})}$ are similar by part 1 of
    Lemma~\ref{lem:sub-a}, (2) $C(\synbrack{k_i})$ holds for all $i$
    with $1 \le i \le n + 1$ and $\ctype{L}{k_i}$, and (3)
    $C(\synbrack{e_i})$ holds for all $i$ with $1 \le i \le n$.
    Therefore, $C(\synbrack{e})$ holds.

  \item[] \textbf{Case 2}: $e = (x \mcolon \alpha)$.  By the
    definitions of $\mname{free-for}$ and $\mname{free-for}$,
    $H_1(\synbrack{e})$ and $H_2(\synbrack{e})$ imply
    $H_1(\synbrack{\alpha})$, $H_2(\synbrack{\alpha})$, and $M \models
    \mname{cleanse}(\synbrack{e'}) \IsDefApp$.  By the induction
    hypothesis, this implies $C(\synbrack{\alpha})$.  Let $\phi \in
    \mname{assign}(S)$ such that $V_\phi(a) \not= \Undefined$.  Then
    \begin{eqnarray*}
    & & V_\phi(\sembrack{S(\synbrack{e})}_{k[e]}) \\
    & = & V_\phi(\sembrack{S(\synbrack{(x \mcolon \alpha)})}_{\rm te}) \\
    & = & V_\phi(\sembrack{\If(a \IsDef \sembrack{S(\alpha)}_{\rm ty}, 
    \mname{cleanse}(\synbrack{a}),
    \synbrack{\Undefined_{\sf C}})}_{\rm te}) \\
    & = & V_\phi(\If(a \IsDef \sembrack{S(\alpha)}_{\rm ty},
    \sembrack{\mname{cleanse}(\synbrack{a})}_{\rm te}, 
    \Undefined_{\sf C})) \\
    & = & V_\phi(\If(a \IsDef \sembrack{S(\alpha)}_{\rm ty},
    a,
    \Undefined_{\sf C})) \\
    & = & V_{\phi[x \mapsto V_\phi(a)]}
    (\If((x \mcolon \mname{C}) \IsDef \alpha, 
    (x \mcolon \mname{C}),
    \Undefined_{\sf C})) \\
    & = & V_{\phi[x \mapsto V_\phi(a)]}((x \mcolon \alpha))  \\  
    & = & V_{\phi[x \mapsto V_\phi(a)]}(e).
    \end{eqnarray*}
    \bsp The third line is by the definition of \mname{sub}; the
    fourth is by Lemma~\ref{lem:eval-qq} and part 1 of
    Lemma~\ref{lem:cleanse}; the fifth is by $M \models
    \mname{cleanse}(\synbrack{e'}) \IsDefApp$ and part 4 of
    Lemma~\ref{lem:cleanse}; the sixth is by $C(\synbrack{\alpha})$;
    and the seventh is by the definition of the standard valuation on
    variables.  Therefore, $C(\synbrack{e})$ holds. \esp

  \item[] \textbf{Case 3}: $e = (y \mcolon \alpha)$ where $x \not= y$.
    By the definitions of $\mname{sub}$ and $\mname{free-for}$,
    $H_1(\synbrack{e})$ and $H_2(\synbrack{e})$ imply
    $H_1(\synbrack{\alpha})$ and $H_2(\synbrack{\alpha})$.  By the
    induction hypothesis, this implies $C(\synbrack{\alpha})$.  Let
    $\phi \in \mname{assign}(S)$ such that $V_\phi(a) \not=
    \Undefined$.  Then
    \begin{eqnarray*}
    & & V_\phi(\sembrack{S(\synbrack{e})}_{k[e]}) \\
    & = & V_\phi(\sembrack{S(\synbrack{(y \mcolon \alpha)})}_{\rm te}) \\
    & = & V_\phi(\sembrack{\synbrack{(y 
    \mcolon \commabrack{S(\alpha)})}}_{\rm te}) \\
    & = & V_\phi((y \mcolon 
    \sembrack{S(\synbrack{\alpha})}_{\rm ty})) \\
    & = & V_{\phi[x \mapsto V_\phi(a)]}((y \mcolon \alpha)) \\
    & = & V_{\phi[x \mapsto V_\phi(a)]}(e).
    \end{eqnarray*}
    The third line is by the definition of \mname{sub}; the fourth is
    by Lemma~\ref{lem:eval-qq} and part 1 of Lemma~\ref{lem:sub-a};
    and the fifth is by $C(\synbrack{\alpha})$.  Therefore,
    $C(\synbrack{e})$ holds.

  \item[] \textbf{Case 4}: $e = (\StarApp x \mcolon \alpha \mdot e')$
    and $\star$ is $\Lambda$, $\lambda$, $\Forsome$, $\iota$, or
    $\epsilon$.  By the definitions of $\mname{sub}$ and
    $\mname{free-for}$, $H_1(\synbrack{e})$ and $H_2(\synbrack{e})$
    imply $H_1(\synbrack{\alpha})$, $H_2(\synbrack{\alpha})$, and $M
    \models \mname{cleanse}(\synbrack{e'}) \IsDefApp$.  By the
    induction hypothesis, this implies $C(\synbrack{\alpha})$.  Let
    $\phi \in \mname{assign}(S)$ such that $V_\phi(a) \not=
    \Undefined$.  Then
    \begin{eqnarray*}
    & & V_\phi(\sembrack{S(\synbrack{e})}_{k[e]}) \\
    & = & V_\phi(\sembrack{S(\synbrack{(\StarApp x \mcolon 
    \alpha \mdot e')})}_{k[e]}) \\
    & = & V_\phi(\sembrack{\synbrack{(\StarApp x \mcolon 
    \commabrack{S(\synbrack{\alpha})} \mdot 
    \commabrack{\mname{cleanse}(\synbrack{e'})})}}_{k[e]}) \\
    & = & V_\phi((\StarApp x \mcolon 
    \sembrack{S(\synbrack{\alpha})}_{\rm ty} \mdot 
    \sembrack{\mname{cleanse}(\synbrack{e'})}_{k[e']})) \\
    & = & V_\phi((\StarApp x \mcolon 
    \sembrack{S(\synbrack{\alpha})}_{\rm ty} \mdot e')) \\
    & = & V_{\phi[x \mapsto V_\phi(a)]}
    ((\StarApp x \mcolon \alpha \mdot e')) \\
    & = & V_{\phi[x \mapsto V_\phi(a)]}(e).
    \end{eqnarray*}
    The third line is by the definition of \mname{sub}; the fourth is
    by Lemma~\ref{lem:eval-qq}, part 1 of Lemma~\ref{lem:sub-a}, and
    parts 1 and 2 of Lemma~\ref{lem:cleanse}; the fifth is by $M
    \models \mname{cleanse}(\synbrack{e'}) \IsDefApp$ and part 4 of
    Lemma~\ref{lem:cleanse}; and the sixth is by
    $C(\synbrack{\alpha})$ and the fact that\[V_{\phi[x \mapsto
        d]}(e') = V_{\phi[x \mapsto V_\phi(a)][x \mapsto d]}(e')\] for
    all $d\in\Dc$.  Therefore, $C(\synbrack{e})$ holds.

  \item[] \bsp \textbf{Case 5}: $e = (\StarApp y \mcolon \alpha \mdot
    e')$, $x \not= y$, and $\star$ is $\Lambda$, $\lambda$,
    $\Forsome$, $\iota$, or $\epsilon$.  By the definitions of
    $\mname{sub}$ and $\mname{free-for}$, $H_1(\synbrack{e})$ and
    $H_2(\synbrack{e})$ imply (1) $H_1(\synbrack{\alpha})$ and
    $H_2(\synbrack{\alpha})$, (2) either $(\ast)$ $M \models
    \Neg\mname{free-in}(\synbrack{x}, \synbrack{e'})$ or $(\ast\ast)$
    $M \models \Neg\mname{free-in}(\synbrack{y},\synbrack{a})$, and
    (3) $H_1(\synbrack{e'})$ and $H_2(\synbrack{e'})$.  By the
    induction hypothesis, this implies $C(\synbrack{\alpha})$ and
    $C(\synbrack{e'})$.  Let $\phi \in \mname{assign}(S)$ such that
    $V_\phi(a) \not= \Undefined$.  Then
    \begin{eqnarray*} & &
      V_\phi(\sembrack{S(\synbrack{e})}_{k[e]}) \\ & = &
      V_\phi(\sembrack{S(\synbrack{(\StarApp y \mcolon \alpha \mdot
          e')})}_{k[e]}) \\ & = &
      V_\phi(\sembrack{\synbrack{(\StarApp y \mcolon
          \commabrack{S(\synbrack{\alpha})} \mdot
          \commabrack{S(\synbrack{e'})}}}_{k[e]}))\\ & = &
      V_\phi((\StarApp y \mcolon \sembrack{S(\synbrack{\alpha})}_{\rm
        ty} \mdot \sembrack{S(\synbrack{e'})}_{k[e']})) \\ & = &
      V_{\phi[x \mapsto V_\phi(a)]} ((\StarApp y \mcolon \alpha \mdot
      e')) \\ & = & V_{\phi[x \mapsto V_\phi(a)]}(e).  
    \end{eqnarray*}
    The third line is by the definition of \mname{sub}; the fourth is
    by Lemma~\ref{lem:eval-qq} and parts 1 and 2 of
    Lemma~\ref{lem:sub-a}; and the fifth is by $C(\synbrack{\alpha})$
    and separate arguments for the two cases $(\ast)$ and
    $(\ast\ast)$.  In case $(\ast)$, \begin{eqnarray*} & & V_{\phi[y
          \mapsto d]}(\sembrack{S(\synbrack{e'})}_{k[e]}) \\ & = &
      V_{\phi[y \mapsto d]}(e') \\ & = & V_{\phi[y \mapsto d][x
          \mapsto V_\phi(a)]}(e') \\ & = & V_{\phi[x \mapsto
          V_\phi(a)][y \mapsto d]}(e') \end{eqnarray*} for all $d \in
    \Dc$. The second line is by $(\ast)$, $H_1(\synbrack{e'})$, and
    part 4 of Lemma~\ref{lem:sub-a}, and the third is by $(\ast)$ and
    part 1 of Lemma~\ref{lem:fv}.  In case
    $(\ast\ast)$, \begin{eqnarray*} & & V_{\phi[y \mapsto
          d]}(\sembrack{S(\synbrack{e'})}_{k[e]}) \\ & = & V_{\phi[y
          \mapsto d][x \mapsto V_{\phi[y \mapsto d]}(a)]}(e') \\ & = &
      V_{\phi[y \mapsto d][x \mapsto V_\phi(a)]}(e') \\ & = &
      V_{\phi[x \mapsto V_\phi(a)][y \mapsto d]}(e') \end{eqnarray*}
    for all $d \in \Dc$.  The second line is by $C(\synbrack{e'})$,
    and the third is by $(\ast\ast)$ and part 1 of Lemma~\ref{lem:fv}.
    Therefore, $C(\synbrack{e})$ holds. \esp

  \item[] \textbf{Case 6}: $e = \alpha(a)$.  Similar to case 3.

  \item[] \textbf{Case 7}: $e = f(a)$.  Similar to case 3.

  \item[] \textbf{Case 8}: $e = \mname{if}(A,b,c)$.  Similar to case 3.

  \item[] \textbf{Case 9}: $e = \synbrack{e'}$.  Let $\phi \in
    \mname{assign}(S)$ such that $V_\phi(a) \not= \Undefined$.  Then
    \begin{eqnarray*}
    & & V_\phi(\sembrack{S(\synbrack{e})}_{k[e]}) \\
    & = & V_\phi(\sembrack{S(\synbrack{\synbrack{e'}})}_{\rm te}) \\
    & = & V_\phi(\sembrack{\synbrack{\synbrack{e'}}}_{\rm te}) \\
    & = & V_\phi(\synbrack{e'}) \\
    & = & V_{\phi[x \mapsto V_\phi(a)]}(\synbrack{e'}) \\
    & = & V_{\phi[x \mapsto V_\phi(a)]}(e)
    \end{eqnarray*}
    The third line is by the definition of \mname{sub}; the fourth is
    by Lemma~\ref{lem:gea}; and the fifth is by the fact that
    $V_\phi(e)$ does not depend on $\phi$.  Therefore,
    $C(\synbrack{e})$ holds.

  \item[] \bsp \textbf{Case 10}: $e = \sembrack{b}_{\rm ty}$.  Let
    $\phi \in \mname{assign}(S)$ such that $V_\phi(a) \not=
    \Undefined$.  Suppose
    $V_\phi(\mname{gea}(\synbrack{H^{-1}(V_{\phi[x \mapsto
          V_\phi(a)]}(b))}, \synbrack{\mname{type}})) = \TRUE$.  By
    the definitions of $\mname{sub}$ and $\mname{free-for}$,
    $H_1(\synbrack{e})$ and $H_2(\synbrack{e})$ imply (1)
    $H_1(\synbrack{b})$ and $H_2(\synbrack{b})$ and (2)
    $H_1(\sembrack{S(\synbrack{b})}_{\rm te})$ and
    $H_2(\sembrack{S(\synbrack{b})}_{\rm te})$.  By the induction
    hypothesis, this implies (1)~$C(\synbrack{b})$ and (2)
    $C(\sembrack{S(\synbrack{b})}_{\rm te})$.  Let $u =
    \mname{coerce-to-type} (S(\sembrack{S(\synbrack{b})}_{\rm
      te}))$. Then
    \begin{eqnarray*}
    & & V_\phi(\sembrack{S(\synbrack{e})}_{k[e]}) \\
    & = & V_\phi(\sembrack{S(\synbrack{\sembrack{b}_{\rm ty}})}_{\rm ty}) \\
    & = & V_\phi(\sembrack{\If(\mname{syn-closed}
      (S(\synbrack{b})), \\ & &
      \hspace{2ex}\synbrack{\If(\mname{gea}
      (\commabrack{S(\synbrack{b})},\synbrack{\mname{type}}),
      \commabrack{u}, 
      \mname{C})}, \\ & &
      \hspace{2ex}\Undefined_{\sf C})}_{\rm ty}) \\
    & = & V_\phi(\sembrack{
      \synbrack{\If(\mname{gea}
      (\commabrack{S(\synbrack{b})},\synbrack{\mname{type}}),
      \commabrack{u}, 
      \mname{C})}}_{\rm ty}) \\
    & = & V_\phi(\mname{if}(\mname{gea}
      (\sembrack{S(\synbrack{b})}_{\rm te},\synbrack{\mname{type}}),
      \sembrack{u}_{\rm ty},
      \mname{C})) \\
    & = & V_\phi(\mname{if}(\mname{gea}(\synbrack{H^{-1}(V_{\phi[x \mapsto
        V_\phi(a)]}(b))},\synbrack{\mname{type}}),
      \sembrack{u}_{\rm ty},
      \mname{C})) \\
    & = & V_\phi(\sembrack{\mname{coerce-to-type}
      (S(\sembrack{S(\synbrack{b})}_{\rm te}))}_{\rm ty}) \\
    & = & V_\phi(\sembrack{\mname{coerce-to-type}
      (S(\synbrack{H^{-1}(V_{\phi[x \mapsto V_\phi(a)]}(b))}))}_{\rm ty}) \\
    & = & V_\phi(\sembrack{S(\synbrack{H^{-1}(V_{\phi[x \mapsto
        V_\phi(a)]}(b))})}_{\rm ty}) \\
    & = & V_{\phi[x \mapsto V_\phi(a)]}(H^{-1}(V_{\phi[x \mapsto V_\phi(a)]}(b))) \\
    & = & V_{\phi[x \mapsto V_\phi(a)]}(\sembrack{
      \synbrack{H^{-1}(V_{\phi[x \mapsto V_\phi(a)]}(b))}}_{\rm ty}) \\
    & = & V_{\phi[x \mapsto V_\phi(a)]}(\sembrack{b}_{\rm ty}) \\    
    & = & V_{\phi[x \mapsto V_\phi(a)]}(e).
    \end{eqnarray*}
    The third line is by the definition of \mname{sub}; the fourth is
    by $H_1(\synbrack{e})$; the fifth is by Lemma~\ref{lem:eval-qq}
    and part 1 of this lemma; the sixth is by $C(\synbrack{b})$; the
    seventh is by the hypothesis that
    $V_\phi(\mname{gea}(\synbrack{H^{-1}(V_{\phi[x \mapsto
          V_\phi(a)]}(b))}, \synbrack{\mname{type}})) = \TRUE$ and the
    definition of $u$; the eight is by $C(\synbrack{b})$; the ninth is
    by $V_\phi(\mname{gea}(\synbrack{H^{-1}(V_{\phi[x \mapsto
          V_\phi(a)]}(b))}, \synbrack{\mname{type}})) = \TRUE$ and
    part 1 of this lemma; the tenth is by
    $C(\sembrack{S(\synbrack{b})}_{\rm te})$ since
    \[V_\phi(\sembrack{S(\synbrack{b})}_{\rm te}) = 
    V_{\phi[x \mapsto V_\phi(a)]}(b) = 
    V_\phi(\synbrack{H^{-1}(V_{\phi[x \mapsto V_\phi(a)]}(b))})\] by
    $C(\synbrack{b})$; the eleventh is by the semantics of evaluation;
    and the twelveth is by \[V_{\phi[x \mapsto
        V_\phi(a)]}(\sembrack{\synbrack{H^{-1}(V_{\phi[x \mapsto
            V_\phi(a)]}(b))}}_{\rm ty}) = V_{\phi[x \mapsto
        V_\phi(a)]}(b).\] Therefore, $C(\synbrack{e})$ holds.

    Now suppose $V_\phi(\mname{gea}(\synbrack{H^{-1}(V_{\phi[x \mapsto
          V_\phi(a)]}(b))}, \synbrack{\mname{type}})) = \FALSE$.  By a
    similar derivation to the one above, $C(\synbrack{e})$ holds. \esp

  \item[] \bsp \textbf{Case 11}: $e = \sembrack{b}_\alpha$.  Let $\phi
    \in \mname{assign}(S)$ such that $V_\phi(a) \not= \Undefined$.
    Suppose $V_\phi(\mname{gea}(\synbrack{H^{-1}(V_{\phi[x \mapsto
          V_\phi(a)]}(b))}, \synbrack{\mname{type}})) = \TRUE$ and
    $V_\phi(\sembrack{b}_{\rm te} \IsDef \alpha) = \TRUE$.  By the
    definitions of $\mname{sub}$ and $\mname{free-for}$,
    $H_1(\synbrack{e})$ and $H_2(\synbrack{e})$ imply (1)
    $H_1(\synbrack{b})$ and $H_2(\synbrack{b})$, (2)
    $H_1(\synbrack{\alpha})$ and $H_2(\synbrack{\alpha})$, and
    (3)~$H_1(\synbrack{e'})$ and $H_2(\synbrack{e'})$.  By the
    induction hypothesis, this implies (1) $C(\synbrack{b})$, (2)
    $C(\synbrack{\alpha})$, and (3)~$C(\sembrack{S(\synbrack{b})}_{\rm
      te})$.  Let $u = \mname{coerce-to-term}
    (S(\sembrack{S(\synbrack{b})}_{\rm te}))$. Then
    \begin{eqnarray*}
    & & V_\phi(\sembrack{S(\synbrack{e})}_{k[e]}) \\
    & = & V_\phi(\sembrack{S(\synbrack{\sembrack{b}_\alpha})}_{\rm te}) \\
    & = & V_\phi(\sembrack{\If(\mname{syn-closed}
      (S(\synbrack{b})), \\ & &
      \hspace{2ex}\synbrack{\If(\mname{gea}
      (\commabrack{S(\synbrack{b})},\synbrack{\alpha}) \And 
      \commabrack{u} \IsDef \commabrack{S(\synbrack{\alpha})},
      \commabrack{u}, 
      \Undefined_{\sf C})}, \\ & &
      \hspace{2ex}\Undefined_{\sf C})}_{\rm te}) \\
    & = & V_\phi(\sembrack{
      \synbrack{\If(\mname{gea}
      (\commabrack{S(\synbrack{b})},\synbrack{\alpha}) \And 
      \commabrack{u} \IsDef \commabrack{S(\synbrack{\alpha})},
      \commabrack{u}, 
      \Undefined_{\sf C})}}_{\rm te}) \\
    & = & V_\phi(\mname{if}(\mname{gea}
      (\sembrack{S(\synbrack{b})}_{\rm te},\synbrack{\alpha}) \And 
      \sembrack{u}_{\rm te} \IsDef \sembrack{S(\synbrack{\alpha})}_{\rm ty},
      \sembrack{u}_{\rm te},
      \Undefined_{\sf C})) \\
    & = & V_\phi(\mname{if}(\mname{gea}(\synbrack{H^{-1}(V_{\phi[x \mapsto
          V_\phi(a)]}(b))},\synbrack{\alpha}) \And {} \\ & &
      \hspace{2ex}\sembrack{u}_{\rm te} \IsDef 
      \sembrack{S(\synbrack{\alpha})}_{\rm ty}, \\ & &
      \hspace{2ex}\sembrack{u}_{\rm te}, \\ & &
      \hspace{2ex}\Undefined_{\sf C})) \\
    & = & V_\phi(\mname{if}(\sembrack{\mname{coerce-to-term}
      (S(\sembrack{S(\synbrack{b})}_{\rm te}))}_{\rm te} \IsDef 
      \sembrack{S(\synbrack{\alpha})}_{\rm ty}, \\ & &
      \hspace{2ex}\sembrack{\mname{coerce-to-term}
      (S(\sembrack{S(\synbrack{b})}_{\rm te}))}_{\rm te}, \\ & &
      \hspace{2ex}\Undefined_{\sf C})) \\
    & = & V_\phi(\mname{if}(\sembrack{\mname{coerce-to-term}
      (S(\synbrack{H^{-1}(V_{\phi[x \mapsto V_\phi(a)]}(b))}))}_{\rm te} 
      \IsDef {} \\ & &
      \hspace{2ex}\sembrack{S(\synbrack{\alpha})}_{\rm ty}, \\ & &
      \hspace{2ex}\sembrack{\mname{coerce-to-term}
      (S(\synbrack{H^{-1}(V_{\phi[x \mapsto V_\phi(a)]}(b))}))}_{\rm te}, \\ & &
      \hspace{2ex}\Undefined_{\sf C})) \\
    & = & V_\phi(\mname{if}(\sembrack{
      S(\synbrack{H^{-1}(V_{\phi[x \mapsto V_\phi(a)]}(b))})}_{\rm te} \IsDef 
      \sembrack{S(\synbrack{\alpha})}_{\rm ty}, \\ & &
      \hspace{2ex}\sembrack{
      S(\synbrack{H^{-1}(V_{\phi[x \mapsto V_\phi(a)]}(b))})}_{\rm te}, \\ & &
      \hspace{2ex}\Undefined_{\sf C})) \\
    & = & V_{\phi[x \mapsto V_\phi(a)]}(\mname{if}
      (H^{-1}(V_{\phi[x \mapsto V_\phi(a)]}(b))
      \IsDef \alpha, \\ & &
      \hspace{2ex}H^{-1}(V_{\phi[x \mapsto V_\phi(a)]}(b)), \\ & &
      \hspace{2ex}\Undefined_{\sf C})) \\ & = & V_{\phi[x \mapsto
          V_\phi(a)]}(\mname{if} (\sembrack{\synbrack{H^{-1}(V_{\phi[x
              \mapsto V_\phi(a)]}(b))}}_{\rm te} \IsDef \alpha, \\ & &
      \hspace{2ex}\sembrack{\synbrack{H^{-1}
      (V_{\phi[x \mapsto V_\phi(a)]}(b))}}_{\rm te}, \\ & &
      \hspace{2ex}\Undefined_{\sf C})) \\
    & = & V_{\phi[x \mapsto V_\phi(a)]}(\mname{if}(\sembrack{b}_{\rm te} 
      \IsDef \alpha, 
      \sembrack{b}_{\rm te}, 
      \Undefined_{\sf C})) \\
    & = & V_{\phi[x \mapsto V_\phi(a)]}(\sembrack{b}_\alpha) \\
    & = & V_{\phi[x \mapsto V_\phi(a)]}(e).
    \end{eqnarray*}
    The third line is by the definition of \mname{sub}; the fourth is
    by $H_1(\synbrack{e})$; the fifth is by Lemma~\ref{lem:eval-qq}
    and part 1 of this lemma; the sixth is by $C(\synbrack{b})$; the
    seventh is by the hypothesis that
    $V_\phi(\mname{gea}(\synbrack{H^{-1}(V_{\phi[x \mapsto
          V_\phi(a)]}(b))}, \synbrack{\mname{type}})) = \TRUE$ and the
    definition of $u$; the eight is by $C(\synbrack{b})$; the ninth is
    by $V_\phi(\mname{gea}(\synbrack{H^{-1}(V_{\phi[x \mapsto
          V_\phi(a)]}(b))}, \synbrack{\mname{type}})) = \TRUE$ and
    part 1 of this lemma; the tenth is by $C(\synbrack{\alpha})$ and
    $C(\sembrack{S(\synbrack{b})}_{\rm te})$ since
    \[V_\phi(\sembrack{S(\synbrack{b})}_{\rm te}) = 
    V_{\phi[x \mapsto V_\phi(a)]}(b) = 
    V_\phi(\synbrack{H^{-1}(V_{\phi[x \mapsto V_\phi(a)]}(b))})\] by
    $C(\synbrack{b})$; the eleventh is by the semantics of evaluation;
    and the twelveth is by \[V_{\phi[x \mapsto
        V_\phi(a)]}(\sembrack{\synbrack{H^{-1}(V_{\phi[x \mapsto
            V_\phi(a)]}(b))}}_{\rm te}) = V_{\phi[x \mapsto
        V_\phi(a)]}(b).\] Therefore, $C(\synbrack{e})$ holds.

    Now suppose $V_\phi(\mname{gea}(\synbrack{H^{-1}(V_{\phi[x \mapsto
          V_\phi(a)]}(b))}, \synbrack{\mname{type}})) = \FALSE$ or
    $V_\phi(\sembrack{b}_{\rm te} \IsDef \alpha) = \FALSE$.  By a
    similar derivation to the one above, $C(\synbrack{e})$ holds. \esp

  \item[] \textbf{Case 12}: $e = \sembrack{b}_{\rm fo}$.  Similar to
    case 10.

\ee
\end{proof}

\subsection{More Notational Definitions}

Using some of the operators defined in this section and the previous
section, we give some notational definitions for additional variable
binders.  Let $L$ be a normal language.

\be

  \item \textbf{Class Abstraction}

  $(\ClassAbsApp u \mcolon \alpha \mdot A)$ 
  means
  \[(\iotaApp x \mcolon \mname{power-type}(\alpha) \mdot 
  (\ForallApp u \mcolon \alpha \mdot u \in x \Iff A))\]
  where $x$ is some member of $\sS$ such that \[T_{\rm ker}^{L} \models 
  \Neg\mname{free-in}(\synbrack{x},\synbrack{\alpha})\] and 
  \[T_{\rm ker}^{L} \models 
  \Neg\mname{free-in}(\synbrack{x},\synbrack{A}).\]

Compact notation: 
  \bi

    \item[] $\set{a_1,\ldots,a_n}$ means 
    $(\ClassAbsApp u \mcolon \mname{V} \mdot u = a_1 \Or \cdots \Or u = a_n)$
    for $n \ge 0$.
  \ei 

  \item \textbf{Dependent Type Product}

  $(\DepTypeProdApp x \mcolon \alpha \mdot \beta)$ 
  means
  \[\mname{type}((\ClassAbsApp p \mcolon \mname{V} \times \mname{V} \mdot
  \ForsomeApp x \mcolon \alpha \mdot  \ForsomeApp y \mcolon \beta \mdot
  p = \seq{x,y}))\] 
  where $p$ is some member of $\sS$ such that \[T_{\rm ker}^{L} \models
  \Neg\mname{free-in}(\synbrack{p},\synbrack{\alpha})\] and
  \[T_{\rm ker}^{L} \models
  \Neg\mname{free-in}(\synbrack{p},\synbrack{\beta}).\]

  \item \textbf{Unique Existential Quantification}

  $(\ForsomeUniqueApp x \mcolon \alpha \mdot A)$ 
  means 
  \[(\ForsomeApp x \mcolon \alpha \mdot (A \And 
  (\ForallApp y \mcolon \alpha \mdot 
  \sembrack{\mname{sub}(\synbrack{(y \mcolon \alpha)},
  \synbrack{x},\synbrack{A})}_{\rm fo}
  \Implies y = x)))\]
  where $y$ is some member of $\sS$ such that \[T_{\rm ker}^{L} \models 
  \Neg\mname{free-in}(\synbrack{y},\synbrack{A}),\] 
  \[T_{\rm ker}^{L} \models
  \mname{sub}(\synbrack{(y \mcolon \alpha)},
  \synbrack{x},\synbrack{A})\IsDefApp,\] and \[T_{\rm ker}^{L}
  \models \mname{free-for}(\synbrack{(y \mcolon \alpha)},
  \synbrack{x},\synbrack{A}).\]

\ee

\section{Examples}\label{sec:examples}

\subsection{Law of Excluded Middle}\label{subsec:lem}

In many traditional logics, e.g., first-order logic, the law of
excluded middle (LEM) is expressed as a formula schema \[A \Or \Neg
A\] where $A$ can be any formula.  In Chiron LEM can be expressed as a
single formula: \[\ForallApp e \mcolon \mname{E}_{\rm fo} \mdot
\sembrack{e} \Or \Neg\sembrack{e}.\] Using quasiquotation, LEM can
alternately be expressed as the following formula:\[\ForallApp e
\mcolon \mname{E}_{\rm fo} \mdot \mname{is-eval-free}(e) \Implies
\sembrack{\synbrack{\commabrack{e} \Or \Neg\commabrack{e}}}_{\rm
  fo}.\]

\subsection{Infinite Dependency 1}\label{subsec:inf-dep-1}

The value of a proper expression may depend on the values that are
assigned to infinitely many symbols.  For example, let $A$ is the
formula \[\Forall z : \mname{E}_{\rm sy} \mdot
\sembrack{\synbrack{(\commabrack{(z : \mname{E}_{\rm sy})} :
    \mname{E}_{\rm sy})}}_{{\sf E}_{\rm sy}} = \synbrack{z}.\] Let
$M=(S,V)$ be a standard model of $T$ and $\phi \in \mname{assign}(S)$.
Then $V_\phi(A) = \TRUE$ iff $\phi(x) = H(z)$ for all $x \in \sS$ with
$x \not= z$.  Moreover, neither $T_{\rm ker}^{L} \models
\mname{free-in} (\synbrack{x},\synbrack{A})$ nor $T_{\rm ker}^{L}
\models \Neg \mname{free-in} (\synbrack{x},\synbrack{A})$ when $x
\not= z$.  Hence $A$ is not syntactically closed.

A more inclusive definition of \mname{free-in} could be defined by
adding an argument to \mname{free-in} that represents the local
context~\cite{Monk88} of its second argument.

\subsection{Infinite Dependency 2}\label{subsec:inf-dep-2}

Here is another example of a proper expression that depends on the
values that are assigned to infinitely many symbols.  Let $A$ is
the formula
\[\mname{infinite}(\ClassAbsApp e \mcolon \mname{E}_{\rm sy} \mdot
\sembrack{\synbrack{(\commabrack{e} \mcolon \mname{C})}}_{\rm te} =
\emptyset).\] Let $M=(S,V)$ be a standard model of $T$ and $\phi \in
\mname{assign}(S)$.  Then $V_\phi(A) = \TRUE$ iff $\phi(x)$ is the
empty set for infinitely many $x \in \sS$.  This implies $V_\phi(A) =
V_{\phi[x \mapsto d]}(A)$ for all $\phi \in \mname{assign}(S)$, $x \in
\sS$, and $d \in \Dc$.  Neither $T_{\rm ker}^{L} \models
\mname{free-in} (\synbrack{x},\synbrack{A})$ nor $T_{\rm ker}^{L}
\models \Neg \mname{free-in}(\synbrack{x},\synbrack{A})$ when $x \in
\sS$.  Hence $A$ is not syntactically closed.

\subsection{Conjunction}\label{subsec:conj}

\bsp In subsection~\ref{subsec:logop} we defined the operator
\mname{and} for conjunction using an infinite set of syntactically
closed eval-free formulas.  We can alternately define conjunction as a
single formula $A$: 
\begin{eqnarray*} 
\lefteqn{\ForallApp p \mcolon 
\mname{E}_{\rm fo} \times \mname{E}_{\rm fo} \mdot
(\mname{and}::\mname{formula},\mname{formula},\mname{formula})
(\sembrack{\mname{hd}(p)}_{\rm fo},
\sembrack{\mname{tl}(p)}_{\rm fo}) \Iff} \\ & & 
\Neg(\Neg \sembrack{\mname{hd}(p)}_{\rm fo} \Or 
\Neg \sembrack{\mname{tl}(p)}_{\rm fo}).  
\end{eqnarray*}
$A$ is obviously not eval-free, and neither $T_{\rm ker}^{L} \models
\mname{free-in} (\synbrack{x},\synbrack{A})$ nor $T_{\rm ker}^{L}
\models \Neg \mname{free-in} (\synbrack{x},\synbrack{A})$ when $x
\not= p$.  Hence $A$ is not syntactically closed.\esp

\bsp The following is an alternative form for this formula with two
universally quantified variables instead of one:
\begin{eqnarray*} 
\lefteqn{\ForallApp e_1,e_2 \mcolon \mname{E}_{\rm fo} \mdot
(\mname{and}::\mname{formula},\mname{formula},\mname{formula})
(\sembrack{e_1},\sembrack{e_2}) \Iff} \\ &
& \Neg(\Neg \sembrack{e_1} \Or \Neg \sembrack{e_2}).  
\end{eqnarray*}
This latter form is more natural, but it cannot be instantiated using
Lemma~\ref{lem:sub-b} like the former form. \esp

\subsection{Modus Ponens}

Like the law of excluded middle, other laws of logic are usually
expressed as formula schemas.  In Chiron these laws can be expressed
as single formulas.  For example, the following Chiron formula
expresses the law of modus ponens:
\begin{eqnarray*}
\lefteqn{\ForallApp p \mcolon 
   \mname{E}_{\rm fo} \times \mname{E}_{\rm fo} \mdot} \\ & &
   (\sembrack{\mname{hd}(p)}_{\rm fo} \And 
   \sembrack{\mname{tl}(p)}_{\rm fo} \And
   \mname{is-impl}(\mname{tl}(p)) \And 
   \mname{hd}(p) = \mname{1st-arg}(\mname{tl}(p)))
   \Implies \\ & &
   \sembrack{\mname{2nd-arg}(\mname{tl}(p))}_{\rm fo}.
\end{eqnarray*}

Deduction and computation rules are naturally represented as
\emph{transformers}~\cite{FarmerMohrenschildt03}, algorithms that map
expressions to expressions.  Transformers can be directly formalized
in Chiron.  For example, the following function abstraction, which
maps a pair of formulas to a formula, formalizes the modus ponens rule
of inference:
\begin{eqnarray*}
\lefteqn{\LambdaApp p \mcolon 
   \mname{E}_{\rm fo} \times \mname{E}_{\rm fo} \mdot} \\ & &
   \mname{if}(\mname{is-impl}(\mname{tl}(p)) \And 
   \mname{hd}(p) = \mname{1st-arg}(\mname{tl}(p)),
   \mname{2nd-arg}(\mname{tl}(p)),
   \Undefined_{\sf C}).
\end{eqnarray*}
Let $(\mname{modus-ponens} :: (\mname{E}_{\rm fo} \times
\mname{E}_{\rm fo}) \tarrow \mname{E}_{\rm fo})(\,)$ be a constant
that is defined to be this function abstraction, and let
$\mname{mp}(a)$ mean \[(\mname{modus-ponens} :: ( \mname{E}_{\rm fo}
\times \mname{E}_{\rm fo}) \tarrow \mname{E}_{\rm fo})(\,)(a).\] Then
the following formula is an alternate expression of the law of modus
ponens:
\begin{eqnarray*}
\ForallApp p \mcolon \mname{E}_{\rm fo} \times \mname{E}_{\rm fo} \mdot
   (\sembrack{\mname{hd}(p)}_{\rm fo} \And 
   \sembrack{\mname{tl}(p)}_{\rm fo} \And 
   \mname{mp}(p)\IsDefApp) \Implies \sembrack{\mname{mp}(p)}_{\rm fo}.
\end{eqnarray*}

\subsection{Beta Reduction}

There are two laws of beta reduction in Chiron, one for the
application of a dependent function type and one for the application
of a function abstraction.  Without quotation and evaluation, the
latter beta reduction law would be informally expressed as the formula
schema
\[(\LambdaApp x \mcolon \alpha \mdot b)(a) \QuasiEqual b[x \mapsto a]\] 
provided:
\be

  \item $a \IsDef (\alpha \cap \mname{V})$.

  \item $b[x \mapsto a] \IsDef \mname{V}$.

  \item $a$ is free for $(x \mcolon \alpha)$ in $b$.  

\ee
Notice that this schema includes four schema variables
($x,\alpha,b,a$), a substitution instruction, two semantic side
conditions (conditions (1) and (2)), and a syntactic side condition
(condition (3)).  Moreover, if either of the two semantic side
conditions is false, $(\LambdaApp x \mcolon \alpha \mdot b)(a)$ is
undefined.

Using constructions, quotation, and evaluation, both laws of beta
reduction can be formalized in Chiron as rules of inference in which
the substitution instructions and syntactic side conditions are
expressed in Chiron.  For example, the law of beta reduction for
function abstractions would be:

\bigskip
\newpage

\noindent
\emph{From}
\be

  \item $\mname{is-fun-redex}(a)$,

  \item $\mname{sub}(\mname{redex-arg}(a),
    \mname{1st-comp}(\mname{redex-var}(a)),\mname{redex-body}(a))
    \IsDefApp$,

  \item $\mname{free-for}(\mname{redex-arg}(a),
     \mname{1st-comp}(\mname{redex-var}(a)),\mname{redex-body}(a)))$

\ee
\emph{infer}
\begin{eqnarray*}
  \lefteqn{\If(\sembrack{\mname{redex-arg}(a)}_{\rm te} \IsDef
  (\sembrack{\mname{2nd-comp}(\mname{redex-var}(a))}_{\rm ty} 
  \cap \mname{V}) \And {}} \\ & &
  \sembrack{\mname{sub}(\mname{redex-arg}(a),
  \mname{1st-comp}(\mname{redex-var}(a)),\mname{redex-body}(a))}_{\rm te}
  \IsDef \mname{V}, \\ & &
  \hspace*{-2ex} \sembrack{a}_{\rm te} \QuasiEqual
  \sembrack{\mname{sub}(\mname{redex-arg}(a),
  \mname{1st-comp}(\mname{redex-var}(a)),\mname{redex-body}(a))}_{\rm te},
  \\ & &
  \hspace*{-2ex} \sembrack{a}_{\rm te} \IsUndefApp)
\end{eqnarray*}

\noindent
The beta reduction law for function abstractions is applied to an
application $b$ of a function abstraction by letting $a$ be the
quotation $\synbrack{b}$.

\iffalse
\begin{eqnarray*}
\lefteqn{\ForallApp e \mcolon \mname{E}_{\rm te} \mdot } \\ & &
   (\mname{is-redex}(e) 
   \And \\ & & \hspace{.9ex}
   \mname{free-for}(\mname{redex-arg}(e),
     \mname{1st-comp}(\mname{redex-var}(e)),\mname{redex-body}(e))) 
   \\ & & 
   \hspace{2ex} \Implies \\ & &
   \sembrack{e} \QuasiEqual
   \sembrack{\mname{sub}(\mname{redex-arg}(e),
   \mname{1st-comp}(\mname{redex-var}(e)),\mname{redex-body}(e))}_{\rm te}.
\end{eqnarray*}
\fi

\subsection{Liar Paradox}\label{subsec:liar.b}

We will formalize in Chiron the liar paradox mentioned in
subsection~\ref{subsec:liar.a}.  Assume that $\mname{nat}$ is a type
and $0,1,2,\ldots$ denote terms such that $\mname{nat}$ denotes an
infinite set $\set{0,1,2,\ldots}$.  Assume also that $\mname{num}$ is
a type that denotes the set
$\set{\synbrack{0},\synbrack{1},\synbrack{2},\ldots}$.  And finally
assume that \mname{enum} is a term that denotes a function which is a
bijection from
$\set{\synbrack{0},\synbrack{1},\synbrack{2},\ldots}$ to the set
$F$ of all terms of type $(\LAMBDAapp e \mcolon \mname{E} \mdot
\mname{E})$---which is the same as $(\mname{E} \tarrow \mname{E})$---that 
are definable by a syntactically closed function abstraction of the
form $(\LambdaApp e \mcolon \mname{E} \mdot b)$.

The following function abstraction denotes some $f \in F$:
\begin{eqnarray*}
\lefteqn{\LambdaApp e \mcolon \mname{E} \mdot} \\  & &
  [ \qmname{op-app}, \\ & & \hspace{2ex}
  \mlist{\qmname{op},\qmname{not},\qmname{formula},
    \qmname{formula}}, \\ & & \hspace{2ex}
  [ \qmname{eval}, \\ & & \hspace{4ex}
  \mlist{\qmname{fun-app}, \mlist{\qmname{fun-app},\qmname{enum},e}, e},
  \\ & & \hspace{4ex}
  \qmname{formula} ] ].
\end{eqnarray*}
Using quasiquotation, $f$ can be expressed much more succinctly as
\[\LambdaApp e
\mcolon \mname{E} \mdot \synbrack{ \Neg
  \sembrack{\mname{enum}(\commabrack{e})(\commabrack{e})}_{\rm fo}}.\]
For some $i$ of type \mname{nat}, $\mname{enum}(\synbrack{i}) = f$.
Then
\begin{eqnarray*}
\mname{enum}(\synbrack{i})(\synbrack{i}) & = & f(\synbrack{i}) \\ &
  = &
  \synbrack{\Neg\sembrack{\mname{enum}(\synbrack{i})(\synbrack{i})}_{\rm
  fo}}.
\end{eqnarray*}
Therefore, if \mname{LIAR} is the term
$\mname{enum}(\synbrack{i})(\synbrack{i})$, then
\[\mname{LIAR} = \synbrack{\Neg\sembrack{\mname{LIAR}}_{\rm fo}}.\]

\section{A Proof System} \label{sec:ps}

We present in this section a proof system for Chiron called
$\textbf{C}_L$ (parameterized by a normal language $L$).  We will
state and prove a soundness theorem and a completeness theorem for
$\textbf{C}_L$.  $\textbf{C}_L$ is intended to be neither a practical
nor implemented proof system for Chiron.  Its purpose is to serve as
(1)~a system test of Chiron's definition and (2) a reference system
for other, possibly implemented, proof systems for Chiron.

\subsection{Definitions}

Fix a normal language $L$ of Chiron for the rest of this section.
$\textbf{C}_L$ is a Hilbert-style proof system for Chiron consisting
of 10 rules of inference and an infinite set of axioms.  The
\emph{rules of inference} are given below in
subsection~\ref{subsec:rules}.  The \emph{axioms} are the instances of
the 68 axiom schemas given below in subsection~\ref{subsec:schemas}.
The set of axioms depends on the language $L$.

Let $T=(L,\Gamma)$ be a theory of Chiron and $A$ be a formula.  A
\emph{proof} of $A$ from $T$ in $\textbf{C}_L$ is a finite sequence of
formulas of $L$ such that $A$ is the last member of the sequence and
every member of the sequence is an axiom of $\textbf{C}_L$, a member
of $\Gamma$, or is inferred from previous members of the sequence by
one of the rules of inference of $\textbf{C}_L$.  Let $\proves{T}{A}$
mean there is a proof of $A$ from $T$ in $\textbf{C}_L$.  $T$ is
\emph{consistent} if there is some formula $A$ of $L$ such that
$\proves{T}{A}$ does not hold.

Let $\sT$ be a set of theories over $L$ and $\sF$ be a set of formulas
of $L$.  $\textbf{C}_L$ is \emph{sound with respect to $\sT$ and
  $\sF$} if, for every $T \in \sT$ and formula $A \in
\sF$, \[\proves{T}{A} \dblsp\mbox{implies} \dblsp T \models A.\]
$\textbf{C}_L$ is \emph{complete with respect to $\sT$ and $\sF$} if,
for every $T \in \sT$ and formula $A \in \sF$,
\[T \models A \dblsp \mbox{implies} \dblsp \proves{T}{A}.\]

After presenting the rules of inference and axioms of $\textbf{C}_L$,
we will prove that $\textbf{C}_L$ is sound with respect to the set of
all normal theories over $L$ and the set of all formulas of $L$ and
complete with respect to the set of all eval-free normal theories
over $L$ and the set of all eval-free formulas of $L$.

\subsection{Rules of Inference} \label{subsec:rules}

$\textbf{C}_L$ has the 10 rules of inference stated below.  The first
two rules of inference are often employed in Hilbert-style proof
systems for first-order logic.  The last eight would normally be
expressed as axiom schemas in a Hilbert-style proof systems for a
traditional logic.  They are expressed as rules of inference in
$\textbf{C}_L$ because the syntactic side conditions, which are
expressed directly in Chiron, must be hypotheses in order for the
rules to a preserve validity in a standard model of a normal theory.

%\newpage

\begin{infrule} [Modus Ponens] 
\bi \item[] \ei
\noindent From 
\be

  \item $A$,

  \item $A \Implies B$

\ee
 infer \[B.\]
\end{infrule}

%\newpage

\begin{infrule} [Universal Generalization]
\bi \item[] \ei
\noindent From \[A\] infer \[\ForallApp x \mcolon \alpha \mdot A.\]
\end{infrule}

\begin{infrule} [Universal Quantifier Shifting]
\bi \item[] \ei
\noindent From 
\[\Neg\mname{free-in}(\synbrack{x},\synbrack{A})\]
infer 
\[(\ForallApp x \mcolon \alpha \mdot (A \Or B))
\Implies (A \Or (\ForallApp x \mcolon \alpha \mdot B)).\]
\end{infrule}

\begin{infrule} [Universal Instantiation]
\bi \item[] \ei
\noindent From 
\be

  \item $\mname{sub}(\synbrack{a}, \synbrack{x}, \synbrack{A})
    \IsDefApp$,

  \item $\mname{free-for}(\synbrack{a}, \synbrack{x},
    \synbrack{A})$

\ee
infer 
\[((\ForallApp x \mcolon \alpha \mdot A) \And {a \IsDef \alpha})
\Implies \sembrack{\mname{sub}(\synbrack{a},
\synbrack{x}, \synbrack{A})}_{\rm fo}.\]
\end{infrule}

%\newpage

\begin{infrule} [Definite Description]
\bi \item[] \ei
\noindent From 
\be

  \item $\mname{sub}(\synbrack{(\iotaApp x \mcolon \alpha \mdot A)},
    \synbrack{x}, \synbrack{A}) \IsDefApp$,

  \item $\mname{free-for}(\synbrack{(\iotaApp x \mcolon \alpha \mdot
    A)}, \synbrack{x}, \synbrack{A})$

\ee
infer 
\[(\ForsomeUniqueApp x \mcolon \alpha \mdot A) \Implies
\sembrack{\mname{sub}(\synbrack{(\iotaApp x \mcolon \alpha
\mdot A)}, \synbrack{x}, \synbrack{A})}_{\rm fo}.\]
\end{infrule}

\begin{infrule} [Indefinite Description]
\bi \item[] \ei
\noindent From
\be

  \item $\mname{sub}(\synbrack{(\epsilonApp x \mcolon \alpha \mdot
    A)}, \synbrack{x}, \synbrack{A}) \IsDefApp$,

  \item $\mname{free-for}(\synbrack{(\epsilonApp x \mcolon \alpha
    \mdot A)}, \synbrack{x}, \synbrack{A})$

\ee
infer 
\[(\ForsomeApp x \mcolon \alpha \mdot A) \Implies
\sembrack{\mname{sub}(\synbrack{(\epsilonApp x \mcolon \alpha
\mdot A)}, \synbrack{x}, \synbrack{A})}_{\rm fo}.\]
\end{infrule}

The function machinery of Chiron---type application, dependent
function types, function application, and function abstraction---is
specified by the following four rules of inference.

\begin{infrule} [Type Application] 
\bi \item[] \ei
\noindent From 
\be

  \item $\Neg \mname{free-in}(\synbrack{y},\synbrack{\alpha})$,

  \item $\Neg \mname{free-in}(\synbrack{y},\synbrack{a})$,

  \item $\Neg \mname{free-in}(\synbrack{f},\synbrack{\alpha})$,
 
  \item $\Neg \mname{free-in}(\synbrack{f},\synbrack{a})$

\ee
 infer 
\be

  \item[] $\If({a \IsDefApp},\ 
  \ForallApp
  y \mcolon \mname{C} \mdot y \IsDef \alpha(a) \Iff (\ForsomeApp
  f \mcolon \alpha \mdot \mname{fun}(f) \And \seq{a,y} \in f),\ 
  \alpha(a) \TypeEqual \mname{C})$.

\iffalse
  \item[] $(({a \IsDefApp} \Implies (\ForallApp
  y \mcolon \mname{C} \mdot y \IsDef \alpha(a) \Iff (\ForsomeApp
  f \mcolon \alpha \mdot \mname{fun}(f) \And \seq{a,y} \in f))) \And \\ 
  \hspace*{.9ex}({a \IsUndefApp} \Implies \alpha(a) \TypeEqual \mname{C}))$.
\fi

\ee
\end{infrule}

\newpage

\begin{infrule} [Dependent Function Types] 
\bi \item[] \ei
\noindent From 
\be

  \item $\Neg \mname{free-in}(\synbrack{f},\synbrack{\alpha})$,

  \item $\Neg \mname{free-in}(\synbrack{f},\synbrack{\beta})$,

  \item $\Neg \mname{free-in}(\synbrack{x},\synbrack{\alpha})$,

  \item $\Neg \mname{free-in}(\synbrack{x},\synbrack{\beta})$

\ee
 infer 
\be

  \item[] $\ForallApp f\mcolon \mname{C} \mdot
  f \IsDef (\LAMBDAapp x \mcolon \alpha \mdot \beta) \Iff \\
  \hspace*{4ex}(\mname{fun}(f) \And 
  (\ForallApp x \mcolon V \mdot {f(x)\IsDefApp}
  \Implies (x \IsDef \alpha \And f(x) \IsDef \beta))).$

\ee
\end{infrule}

\begin{infrule} [Function Application] 
\bi \item[] \ei
\noindent From 
\be

  \item $\Neg \mname{free-in}(\synbrack{y},\synbrack{\alpha})$,

  \item $\Neg \mname{free-in}(\synbrack{y},\synbrack{a})$,

  \item $\Neg \mname{free-in}(\synbrack{y},\synbrack{f})$

\ee
 infer 
\be

  \item[] $f(a) \QuasiEqual (\iotaApp y \mcolon \alpha(a) \mdot
  \mname{fun}(f) \And \seq{a,y} \in f)$.

\ee
where $f$ is of type $\alpha$.
\end{infrule}

\begin{infrule} [Function Abstraction] 
\bi \item[] \ei
\noindent From 
\be

  \item $\Neg \mname{free-in}(\synbrack{f},\synbrack{\alpha})$,

  \item $\Neg \mname{free-in}(\synbrack{f},\synbrack{\beta})$,

  \item $\Neg \mname{free-in}(\synbrack{f},\synbrack{b})$,

  \item $\Neg \mname{free-in}(\synbrack{x},\synbrack{\alpha})$,

  \item $\Neg \mname{free-in}(\synbrack{x},\synbrack{b})$,

  \item $\Neg \mname{free-in}(\synbrack{y},\synbrack{\beta})$,

  \item $\Neg \mname{free-in}(\synbrack{y},\synbrack{b})$

\ee
infer
\be

  \item[] $(\LambdaApp x \mcolon \alpha \mdot b) \QuasiEqual\\
  \hspace*{4ex}(\iotaApp f \mcolon (\LAMBDAapp x \mcolon \alpha \mdot \beta) 
  \mdot \\
  \hspace*{8ex}(\ForallApp x \mcolon \alpha,\, y \mcolon \beta \mdot 
  {f(x) = y} \Iff (x  \IsDef \mname{V} \And b \IsDef \mname{V} \And y = b)))$.

\ee
where $b$ is of type $\beta$.
\end{infrule}

\subsection{Axiom Schemas} \label{subsec:schemas}

The axioms of $\textbf{C}_L$ are presented by 68 axiom schemas,
organized into 13 groups.  Each axiom schema is specified by a formula
schema.  An instance of an axiom schema of $\textbf{C}_L$ is any
formula of $L$ that is obtained by replacing the schema variables in
the schema with appropriate expressions.

The first group of axiom schemas define conjunction and implication in
terms of negation and disjunction in the usual way and give a complete
set of axioms for propositional logic in terms of disjunction and
implication.

\begin{axschemas} [Propositional Logic] \label{axschemas:prop}
\be

  \item[]

  \item $(A \And B) \Iff \Neg(\Neg A \Or \Neg B)$.

  \item $(A \Implies B) \Iff (\Neg A \Or B)$.

  \item $(A \Or A) \Implies A.$

  \item $A \Implies (B \Or A).$

  \item $(A \Implies B) \Implies ((C \Or A) \Implies (B \Or C))).$

\ee
\end{axschemas}

\bsp The next group of axiom schemas specify that the three equalities
\mname{type-equal}, \mname{quasi-equal}, and \mname{formula-equal} are
equivalence relations as well as congruences with respect to
formulas. \esp

\begin{axschemas} [Equality]  \label{axschemas:equality}
\be

  \item[]

  \item $\alpha \TypeEqual \alpha.$

  \item $\alpha \TypeEqual \beta \subseteq C \Iff D$ \\[1.5ex] where
    $D$ is the result of replacing one occurrence of $\alpha$ in $C$
    by an occurrence of $\beta$, provided that the occurrence of
    $\alpha$ in $C$ is not within a quotation.

  \item $a \QuasiEqual a.$

  \item $a \QuasiEqual b \subseteq C \Iff D$ \\[1.5ex] where $D$ is
    the result of replacing one occurrence of $a$ in $C$ by an
    occurrence of $b$, provided that the occurrence of $a$ in $C$ is
    not within a quotation and is not a variable $(x\mcolon\alpha)$
    immediately preceded by $\Lambda$, $\lambda$, $\Forsome$, $\iota$,
    or $\epsilon$.

  \item $A \Iff A.$

  \item $A \Iff B \subseteq C \Iff D$ \\[1.5ex] where $D$ is the
    result of replacing one occurrence of $A$ in $C$ by an occurrence
    of $B$, provided that the occurrence of $A$ in $C$ is not within a
    quotation.

\ee
\end{axschemas}

The following axiom schemas specify the general definedness laws for
operators.  For a kind $k$, let 
\[\overline{k} =
\left\{\begin{array}{ll}
         k & \mbox{if } k = \mname{type} \\
         \mname{C} & \mbox{if } \ctype{L}{k} \\
         k & \mbox{if } k = \mname{formula}. 
       \end{array}
\right.\]

%\newpage

\begin{axschemas} [General Operator Properties] \label{axschemas:op}

\be

  \item[]

  \item $e_{i_1} \IsDef k_{i_1} \And \cdots \And e_{i_m} \IsDef
    k_{i_m} \Implies \\
  \hspace*{4ex}
  (o::k_1,\ldots,k_n,\mname{type})(e_1,\ldots,e_n) \TypeEqual \\
  \hspace*{4ex} (o::\overline{k_1},\ldots,\overline{k_n},
  \mname{type})(e_1,\ldots,e_n)$ \\[1.5ex] 
  where $n \ge 1$ and $k_{i_1},\ldots,k_{i_m}$ is the
  subsequence of types in the sequence $k_1,\ldots,k_n$ of kinds.

  \item $e_{i_1} \IsDef k_{i_1} \And \cdots \And e_{i_m} \IsDef
    k_{i_m} \Implies \\
  \hspace*{4ex}
  (o::k_1,\ldots,k_n,\beta)(e_1,\ldots,e_n) \QuasiEqual \\
  \hspace*{4ex} (o::\overline{k_1},\ldots,\overline{k_n},
  \beta)(e_1,\ldots,e_n)$ \\[1.5ex]
  where $n \ge 1$ and $k_{i_1},\ldots,k_{i_m}$ is the 
  subsequence of types in the sequence $k_1,\ldots,k_n$ of kinds.

  \item $e_{i_1} \IsDef k_{i_1} \And \cdots \And e_{i_m} \IsDef
    k_{i_m} \Implies \\
  \hspace*{4ex}
  (o::k_1,\ldots,k_n,\mname{formula})(e_1,\ldots,e_n) \Iff \\
  \hspace*{4ex} (o::\overline{k_1},\ldots,\overline{k_n},
  \mname{formula})(e_1,\ldots,e_n)$ \\[1.5ex] 
  where $n \ge 1$ and $k_{i_1},\ldots,k_{i_m}$ is the
  subsequence of types in the sequence $k_1,\ldots,k_n$ of kinds.

  \item ${(o::k_1,\ldots,k_n,\beta)(e_1,\ldots,e_n) \IsDefApp} \Implies {}\\
  \hspace*{4ex} 
  (o::k_1,\ldots,k_n,\beta)(e_1,\ldots,e_n) = \\
  \hspace*{4ex} (o::k_1,\ldots,k_n,\mname{C})(e_1,\ldots,e_n)$
  \\[1.5ex] where $n \ge 0$.

  \item $({a \IsDefApp} \And a \IsUndef \alpha) \Implies \\
  \hspace*{4ex} 
  (o::k_1,\ldots,k_{i-1},\alpha,k_{i+1},\ldots,k_n,\mname{type})\\
  \hspace*{4ex}
  (e_1,\ldots,e_{i-1},a,e_{i+1},\ldots,e_n) \TypeEqual \mname{C}$
  \\[1.5ex] where $n \ge 1$.

  \item $({a \IsDefApp} \And a \IsUndef \alpha) \Implies \\
  \hspace*{4ex}
  (o::k_1,\ldots,k_{i-1},\alpha,k_{i+1},\ldots,k_n,\beta)\\
  \hspace*{4ex}
  (e_1,\ldots,e_{i-1},a,e_{i+1},\ldots,e_n) \IsUndefApp$
  \\[1.5ex] where $n \ge 1$.

  \item $({a \IsDefApp} \And a \IsUndef \alpha) \Implies \\
  \hspace*{4ex} 
  (o::k_1,\ldots,k_{i-1},\alpha,k_{i+1},\ldots,k_n,\mname{formula})\\
  \hspace*{4ex}
  (e_1,\ldots,e_{i-1},a,e_{i+1},\ldots,e_n) \Iff \mname{F}$
  \\[1.5ex] where $n \ge 1$.

  \item ${(o::k_1,\ldots,k_n,\beta)(e_1,\ldots,e_n) \IsDefApp} \Implies {}\\
  \hspace*{4ex} 
  (o::k_1,\ldots,k_n,\beta)(e_1,\ldots,e_n) \IsDef \beta$
  \\[1.5ex] where $n \ge 0$.

\ee
\end{axschemas}

The following axiom schemas specify specific definedness laws for the
built-in operators.

\begin{axschemas} [Built-In Operator Definedness] \label{axschemas:bi-op-def}
\be

  \item[]

  \item ${\ell \IsDefApp}.$

  \item ${\mname{E}_{{\rm on},a} \not\TypeEqual \mname{C}} \Implies {a
    \IsDefApp}$.

  \item ${\mname{E}_a \not\TypeEqual \mname{C}} \Implies {a
    \IsDefApp}$.

  \item ${\mname{E}_{{\rm op},a} \not\TypeEqual \mname{C}} \Implies {a
    \IsDefApp}$.

  \item ${\mname{E}_{{\rm ty},a} \not\TypeEqual \mname{C}} \Implies {a
    \IsDefApp}$.

  \item ${\mname{E}_{{\rm te},a} \not\TypeEqual \mname{C}} \Implies {a
    \IsDefApp}$.

  \item ${\mname{E}^{b}_{{\rm te},a} \not\TypeEqual \mname{C}}
    \Implies ({a \IsDefApp} \And {b \IsDefApp})$.

  \item ${\mname{E}_{{\rm fo},a} \not\TypeEqual \mname{C}} \Implies {a
    \IsDefApp}$.

  \item $a \in b \Implies ({a \IsDefApp} \And {b \IsDefApp}).$

  \item $a =_\alpha b \Implies ({a \IsDef \alpha} \And {b \IsDef
  \alpha}).$

\ee
\end{axschemas}

The general definedness law for variables is specified by the
following axiom schema.

\begin{axschemas} [Variable Definedness]
\be

  \item[]

  \item ${(x \mcolon \alpha)\IsDefApp} \Implies (x \mcolon \alpha)
  \IsDef \alpha.$

\ee
\end{axschemas}

The following axiom schemas specify the extensionality of types and
universal quantification over the canonical empty type.

%\newpage

\begin{axschemas} [Types] \label{axschemas:types}
\be

  \item[]

  \item $\alpha \TypeEqual \beta \Iff (\Forall x \mcolon \mname{C}
    \mdot x \IsDef \alpha \Iff x \IsDef \beta).$

  \item $\ForallApp x \mcolon \nabla \mdot A.$

\ee
\end{axschemas}

\bsp The following 18 axiom schemas correspond to the 18 axiom schemas
that constitute the axiomatization of {\nbg} set theory given by
K. G\"{o}del in~\cite{Goedel40}. \esp

\begin{axschemas} [NBG Set Theory] \label{axschemas:set-theory}
\be

  \item[]

  \item $(x \mcolon \mname{C})\IsDefApp.$

  %%\item $\ForallApp x \mcolon \alpha \mdot x \IsDef \mname{C}.$

  \item $\ForallApp x \mcolon \mname{C} \mdot x \IsDef \mname{V} \Iff 
  (\ForsomeApp y \mcolon \mname{C} \mdot x \in y).$

  \item $\ForallApp x,y \mcolon \mname{C} \mdot (\ForallApp u \mcolon
  \mname{V} \mdot u \in x \Iff u \in y) \Implies x = y.$

  \item $\ForallApp u,v \mcolon \mname{V} \mdot \set{u,v}\IsDefApp.$

  \item $\ForsomeApp x \mcolon \mname{C} \mdot \ForallApp
  u,v \mcolon \mname{V} \mdot \seq{u,v} \in x \Iff u \in v.$

  \item $\ForallApp x,y \mcolon \mname{C} \mdot (x \cap y) \IsDefApp.$

  \item $\ForallApp x \mcolon \mname{C} \mdot \overline{x}\IsDefApp.$

  \item $\ForallApp x \mcolon \mname{C} \mdot \mname{dom}(x) \IsDefApp.$

  \item $\ForallApp x\mcolon \mname{C} \mdot \ForsomeApp y \mcolon
    \mname{C} \mdot \ForallApp u,v \mcolon \mname{V} \mdot \seq{u,v}
    \in y \Iff u \in x.$

  \item $\ForallApp x \mcolon \mname{C} \mdot \ForsomeApp y \mcolon
    \mname{C} \mdot \ForallApp u,v \mcolon \mname{V} \mdot \seq{u,v}
    \in y \Iff \seq{v,u} \in x.$

  \item $\ForallApp x \mcolon \mname{C} \mdot \ForsomeApp y \mcolon
    \mname{C} \mdot \ForallApp u,v,w \mcolon \mname{V} \mdot
    \seq{u,v,w} \in y \Iff \seq{v,w,u} \in x.$

  \item $\ForallApp x \mcolon \mname{C} \mdot \ForsomeApp y \mcolon
    \mname{C} \mdot \ForallApp u,v,w \mcolon \mname{V} \mdot
    \seq{u,v,w} \in y \Iff \seq{u,w,v} \in x.$

  \item $\ForsomeApp u \mcolon \mname{V} \mdot u \not= \emptyset \And
    (\ForallApp v \mcolon \mname{V} \mdot v \in u \Implies (\ForsomeApp w
    \mcolon \mname{V} \mdot w \in u \And v \subset w)).$

  \item $\ForallApp u \mcolon \mname{V} \mdot \mname{sum}(u)\IsDef
    \mname{V}.$

  \item $\ForallApp u \mcolon \mname{V} \mdot \mname{power}(u) \IsDef
    \mname{V}.$

  \item $\ForallApp x \mcolon \mname{C} \mdot \mname{univocal}(x) \Implies\\ 
    \hspace*{4ex} 
    (\ForallApp u \mcolon \mname{V} \mdot \ForsomeApp v
    \mcolon \mname{V} \mdot \ForallApp t \mcolon \mname{V} \mdot t \in
    v \Iff (\ForsomeApp s \mcolon \mname{V} \mdot s \in u \And
    \seq{s,t} \in x)).$

  \item $\ForallApp x \mcolon \mname{C} \mdot x \not= \emptyset
    \Implies (\ForsomeApp u \mcolon \mname{V} \mdot {u \in x} \And
    {\ForallApp v \mcolon \mname{V} \mdot \Neg(v \in u \And v \in
    x)}).$

  \item $\ForsomeApp f \mcolon \mname{C} \mdot \mname{fun}(f) \And
    (\ForallApp u \mcolon \mname{V} \mdot {u \not= \emptyset}
    \Implies {f(u) \in u}).$

\ee
\end{axschemas}

Conditional terms are specified by two axiom schemas.

\begin{axschemas} [Conditional Terms] \label{axschemas:if}
\be

  \item[]

  \item $A \Implies \mname{if}(A,b,c) \QuasiEqual b.$

  \item $\Neg A \Implies \mname{if}(A,b,c) \QuasiEqual c.$

\ee
\end{axschemas}

Definite description is specified by the Definite Description rule of
inference given above and the following two axiom schemas, one for
proper definite descriptions and one for improper.

\begin{axschemas} [Definite Description] \label{axschemas:def-desc}
\be

  \item[]

  \item $(\ForsomeUniqueApp x \mcolon \alpha \mdot A) \Implies
  (\iotaApp x \mcolon \alpha \mdot A) \IsDef \alpha)$.

  \item $\Neg(\ForsomeUniqueApp x \mcolon \alpha \mdot A) \Implies
  (\iotaApp x \mcolon \alpha \mdot A)\IsUndefApp.$

\ee
\end{axschemas}

Indefinite description is specified by the Indefinite Description rule
of inference given above and the following three axiom schemas, one for
proper indefinite descriptions, one for improper indefinite
descriptions, and one for formulas that are satisfied by the same
classes.

\newpage

\begin{axschemas} [Indefinite Description] \label{axschemas:indef-desc}
\be

  \item[]

  \item $(\ForsomeApp x \mcolon \alpha \mdot A) \Implies 
  (\epsilonApp x \mcolon \alpha \mdot A) \IsDef \alpha)$.

  \item $\Neg(\ForsomeApp x \mcolon \alpha \mdot A) \Implies
  (\epsilonApp x \mcolon \alpha \mdot A)\IsUndefApp.$

  \item $(\ForallApp x \mcolon \mname{C} \mdot A \Iff B) \Implies
  (\epsilonApp x \mcolon \mname{C} \mdot A) = 
  (\epsilonApp x \mcolon \mname{C} \mdot B).$

\ee
\end{axschemas}

Quotation is specified by the following three axiom schemas.  Note
that the quotation of a symbol and the quotation of an operator name
are not fully specified.

\begin{axschemas}[Quotation] \label{axschemas:quote}
\be

  \item[]

  \item $\synbrack{s} \IsDef \mname{E}_{\rm sy}$ \\[1.5ex] 
  where $s \in \sS$.

  \item $\synbrack{o} \IsDef \mname{E}_{\rm on}$ \\[1.5ex]
  where $o \in \sO$.

  \item $\synbrack{s_1} \not= \synbrack{s_2}$ \\[1.5ex]
  where $s_1,s_2 \in \sS \cup \sO$ with $s_1 \not= s_2$.

  \item $\synbrack{(e_1,\ldots,e_n)} =_{\sf E}
  [\synbrack{e_1},\ldots,\synbrack{e_n}]$ \\[1.5ex] where
  $e_1,\ldots,e_n \in \sE_L$ and $n \ge 0$.

\ee
\end{axschemas}

Evaluation is specified by the following six axiom schemas.  This set
of axiom schemas expresses the same information as
parts 1--6 of Lemma~\ref{lem:gea}.

\begin{axschemas} [Evaluation] \label{axschemas:eval}
\be

  \item[]

  \item $\mname{is-eval-free}(\synbrack{\alpha}) \Implies
    \sembrack{\synbrack{\alpha}}_{\rm ty} \TypeEqual \alpha$.

  \item $\Neg\mname{gea} (b,\synbrack{\mname{type}})
    \Implies \sembrack{b}_{\rm ty} \TypeEqual \mname{C}.$

  \item $\mname{is-eval-free}(\synbrack{a}) \Implies
    \sembrack{\synbrack{a}}_\alpha = \mname{if}(a \IsDef \alpha, a,
    \Undefined_{\sf C})$.

  \item $\Neg\mname{gea}(b,\synbrack{\alpha}) \Implies
    \sembrack{b}_\alpha \QuasiEqual \Undefined_{\sf C}.$

  \item $\mname{is-eval-free}(\synbrack{A}) \Implies
    \sembrack{\synbrack{A}}_{\rm fo} \Iff A$.

  \item $\Neg\mname{gea}
  (b,\synbrack{\mname{formula}}) \Implies
  \sembrack{b}_{\rm fo} \Iff \mname{F}.$

\ee
\end{axschemas}

The following eight axiom schemas specify the eight kinds of
construction types.  The type operators named \mname{expr-sym} and
\mname{expr-op-name} are specified by their properties. \mname{expr}
is specified from \mname{expr-sym} and \mname{expr-op-name} by
induction.  And those named \mname{expr-op}, \mname{expr-type},
\mname{expr-term-type}, \mname{expr-term}, and \mname{expr-formula}
are defined by mutual recursion.

\begin{axschemas}[Construction Types] \label{axschemas:con-types}
\be

  \item[]

  \item $\ell = \mname{term}(\mname{E}_{\rm on})$.

  \item $\mname{E}_{\rm sy} \cup \mname{E}_{\rm on} \TypeLE
    \mname{V}$.

  \item $\mname{E}_{\rm sy} \cap \mname{E}_{\rm on} \TypeEqual
    \nabla$.

  \item $\mname{countably-infinite}(\mname{term}(\mname{E}_{\rm sy}))$.

  \item $\ForsomeApp u \mcolon \mname{V} \mdot \mname{countably-infinite}(u)
    \And \mname{term}(\mname{E}_{\rm on}) \subseteq u$.

  \item $\ForallApp x \mcolon \mname{E}_{\rm sy} \cup \mname{E}_{\rm
    on} \mdot x \not= \emptyset \And \Neg \, \mname{is-ord-pair}(x).$

  \item $\ForallApp u \mcolon \mname{L}, u =
    \mname{term}(\mname{E}_{{\rm on},u}).$

  \item $\ForallApp u \mcolon \mname{L}, v \mcolon \mname{V} \mdot \\
  \hspace*{4ex} ((\ForallApp u_0 \mcolon \mname{E}_{\rm sy} 
  \cup \mname{E}_{{\rm on},u} \mdot u_0 \in v)  \And \\
  \hspace*{5ex} \emptyset \in v \And \\
  \hspace*{5ex} (\ForallApp u_1,u_2 \mcolon \mname{V} \mdot 
  (u_1 \in v \And u_2 \in v \And u_2 \IsUndef {\mname{E}_{\rm sy}}
  \cup {\mname{E}_{{\rm on},u}}) 
  \Implies \seq{u_1,u_2} \in v)) \\
  \hspace*{7ex}\Implies \mname{term}(\mname{E}_u) \subseteq v.$

  \item $\ForallApp u \mcolon \mname{L}, x \mcolon \mname{C} \mdot 
  x \IsDef \mname{E}_{{\rm op},u} \Iff \\
  \hspace*{4ex} 
  (\ForsomeApp e \mcolon \mname{E}_{{\rm on},u} \mdot
  x = \synbrack{(\commabrack{e} :: \mname{type})} \Or \\
  \hspace*{14.5ex} 
  (\ForsomeApp e' \mcolon \mname{E}_{{\rm ty},u} \mdot
  x = \synbrack{(\commabrack{e} :: \commabrack{e'})}) \Or \\
  \hspace*{14.5ex} 
  x = \synbrack{(\commabrack{e} :: \mname{formula})}) \Or \\
  \hspace*{4ex} 
  (\ForsomeApp e \mcolon \mname{E}_{{\rm op},u} \mdot
  x = e{\verb+^+}[\synbrack{\mname{type}}] \Or \\
  \hspace*{14.5ex} 
  (\ForsomeApp e' \mcolon \mname{E}_{{\rm ty},u} \mdot
  x = e{\verb+^+}[e']) \\
  \hspace*{14.5ex}
  x = e{\verb+^+}[\synbrack{\mname{formula}}])$.

  \item $\ForallApp u \mcolon \mname{L}, x \mcolon \mname{C} \mdot 
  x \IsDef \mname{E}_{{\rm ty},u} \Iff \\
  \hspace*{4ex}
  (\ForsomeApp e \mcolon \mname{E}_{{\rm on},u} \mdot
  x =  [\synbrack{\mname{op-app}},[\synbrack{\mname{op}},e,
  \synbrack{\mname{type}}]]) \Or \\
  \hspace*{4ex}
  (\ForsomeApp e_1 \mcolon \mname{E}_{{\rm on},u}, e_2,e_3 \mcolon 
  \mname{E}_u \mdot \\
  \hspace*{8ex}
  [\synbrack{\mname{op-app}},
  [\synbrack{\mname{op}},e_1]\verb+^+e_2\verb+^+[\synbrack{\mname{type}}]]
  \verb+^+e_3 \in \mname{E}_{{\rm ty},u} \And \\
  \hspace*{8ex}
  ((\ForsomeApp e \mcolon \mname{E}_{{\rm ty},u} \mdot \\
  \hspace*{12ex} x = [\synbrack{\mname{op-app}},
  [\synbrack{\mname{op}},e_1]\verb+^+e_2\verb+^+
  [\synbrack{\mname{type}},\synbrack{\mname{type}}]]  
  \verb+^+e_3\verb+^+[e]) \Or \\
  \hspace*{9ex}
  (\ForsomeApp e \mcolon \mname{E}_{{\rm ty},u} , 
  e' \mcolon \mname{E}_{{\rm te},u} \mdot \\
  \hspace*{12ex} x = [\synbrack{\mname{op-app}},
  [\synbrack{\mname{op}},e_1]\verb+^+e_2\verb+^+
  [e,\synbrack{\mname{type}}]]
  \verb+^+e_3\verb+^+[e']) \\
  \hspace*{9ex}
  (\ForsomeApp e \mcolon \mname{E}_{{\rm fo},u} \mdot \\
  \hspace*{12ex} x = [\synbrack{\mname{op-app}},
  [\synbrack{\mname{op}},e_1]\verb+^+e_2\verb+^+
  [\synbrack{\mname{formula}},\synbrack{\mname{type}}]]
  \verb+^+e_3\verb+^+[e]) \\
  \hspace*{8ex})) \Or \\
  \hspace*{4ex} 
  (\ForsomeApp e_1 \mcolon \mname{E}_{{\rm ty},u}, 
  e_2 \mcolon \mname{E}_{{\rm te},u} \mdot
  x = \synbrack{\commabrack{e_1}(\commabrack{e_2})}) \Or \\
  \hspace*{4ex}
  (\ForsomeApp e_1 \mcolon \mname{E}_{{\rm on},u},
  e_2,e_3 \mcolon \mname{E}_{{\rm ty},u} \mdot
  x = \synbrack{(\LAMBDAapp \commabrack{e_1} : \commabrack{e_2} \mdot
  \commabrack{e_3})}) \Or \\
  \hspace*{4ex}
  (\ForsomeApp e \mcolon \mname{E}_{{\rm te},u} \mdot
  x = \synbrack{\sembrack{\commabrack{e}}_{\rm ty}})$.

  \item $\ForallApp u \mcolon \mname{L}, \hat{e} \mcolon \mname{E}_{{\rm ty},u}, 
  x \mcolon \mname{C} \mdot 
  x \IsDef \mname{E}_{{\rm te},u}^{\hat{e}} \Iff \\
  \hspace*{4ex}
  (\ForsomeApp e \mcolon \mname{E}_{{\rm on},u} \mdot
  x =  [\synbrack{\mname{op-app}},[\synbrack{\mname{op}},e,\hat{e}]]) \Or \\
  \hspace*{4ex}
  (\ForsomeApp e_1 \mcolon \mname{E}_{{\rm on},u}, e_2,e_3 \mcolon 
  \mname{E}_u \mdot \\
  \hspace*{8ex}
  [\synbrack{\mname{op-app}},
  [\synbrack{\mname{op}},e_1]\verb+^+e_2\verb+^+[\hat{e}]]
  \verb+^+e_3 \in \mname{E}_{{\rm te},u}^{\hat{e}} \And \\
  \hspace*{8ex}
  ((\ForsomeApp e \mcolon \mname{E}_{{\rm ty},u} \mdot \\
  \hspace*{12ex} x = [\synbrack{\mname{op-app}},
  [\synbrack{\mname{op}},e_1]\verb+^+e_2\verb+^+
  [\synbrack{\mname{type}},\hat{e}]]  
  \verb+^+e_3\verb+^+[e]) \Or \\
  \hspace*{9ex}
  (\ForsomeApp e \mcolon \mname{E}_{{\rm ty},u} , 
  e' \mcolon \mname{E}_{{\rm te},u} \mdot \\
  \hspace*{12ex} x = [\synbrack{\mname{op-app}},
  [\synbrack{\mname{op}},e_1]\verb+^+e_2\verb+^+
  [e,\hat{e}]]
  \verb+^+e_3\verb+^+[e']) \\
  \hspace*{9ex}
  (\ForsomeApp e \mcolon \mname{E}_{{\rm fo},u} \mdot \\
  \hspace*{12ex} x = [\synbrack{\mname{op-app}},
  [\synbrack{\mname{op}},e_1]\verb+^+e_2\verb+^+
  [\synbrack{\mname{formula}},\hat{e}]]\verb+^+e_3\verb+^+[e]))) \Or \\
  \hspace*{4ex}
  (\ForsomeApp e \mcolon \mname{E}_{{\rm sy},u} \mdot
  x = \synbrack{(\commabrack{e} : \commabrack{\hat{e}})}) \Or \\
  \hspace*{4ex}
  (\ForsomeApp e_1 \mcolon \mname{E}_{{\rm ty},u},
  e_2 \mcolon \mname{E}_{{\rm te},u}^{e_1},
  e_3 \mcolon \mname{E}_{{\rm te},u} \mdot \\
  \hspace*{8ex}
  x = \synbrack{\commabrack{e_2}(\commabrack{e_3})} \And 
  \hat{e} = \synbrack{\commabrack{e_1}(\commabrack{e_3})}) \Or \\
  \hspace*{4ex}
  (\ForsomeApp e_1 \mcolon \mname{E}_{{\rm sy},u},
  e_2,e_3 \mcolon \mname{E}_{{\rm ty},u},
  e_4 \mcolon \mname{E}_{{\rm te},u}^{e_3} \mdot \\
  \hspace*{8ex}
  x = \synbrack{\LambdaApp \commabrack{e_1} : \commabrack{e_2} \mdot
  \commabrack{e_4}} \And
  \hat{e} = \synbrack{\LAMBDAapp \commabrack{e_1} : \commabrack{e_2} \mdot
  \commabrack{e_3}}) \Or \\
  \hspace*{4ex}
  (\ForsomeApp e_1 \mcolon \mname{E}_{{\rm fo},u},
  e_2,e_3 \mcolon \mname{E}_{{\rm ty},u},
  e_4 \mcolon \mname{E}_{{\rm te},u}^{e_2}, e_5 \mcolon 
  \mname{E}_{{\rm te},u}^{e_3} \mdot \\
  \hspace*{8ex}
  x = \synbrack{\mname{if},\commabrack{e_1},
  \commabrack{e_4},\commabrack{e_5}} \And 
  \hat{e} = \mname{if}(e_2 = e_3, e_2, \synbrack{\mname{C}})) \Or \\
  \hspace*{4ex}
  (\ForsomeApp e_1 \mcolon \mname{E}_{{\rm sy},u},
  e_2 \mcolon \mname{E}_{{\rm fo},u} \mdot
  x = \synbrack{(\iotaApp \commabrack{e_1} : \commabrack{\hat{e}} \mdot
  \commabrack{e_2})}) \Or \\\hspace*{4ex}
  (\ForsomeApp e_1 \mcolon \mname{E}_{{\rm sy},u},
  e_2 \mcolon \mname{E}_{{\rm fo},u} \mdot
  x = \synbrack{(\epsilonApp \commabrack{e_1} : \commabrack{\hat{e}} \mdot
  \commabrack{e_2})}) \Or \\
  \hspace*{4ex}
  (\ForsomeApp e \mcolon \mname{E}_u \mdot
  x = \synbrack{\commabrack{e}} \And \hat{e} = \synbrack{\mname{E}_u}) \Or \\
  \hspace*{4ex}
  (\ForsomeApp e \mcolon \mname{E}_{{\rm te},u} \mdot 
  x = \synbrack{\sembrack{\commabrack{e}}_{\commabrack{\hat{e}}}})$.

  \item $\ForallApp u \mcolon \mname{L}, x \mcolon \mname{C} \mdot 
  x \IsDef \mname{E}_{{\rm te},u} \Iff 
  (\ForsomeApp \hat{e} \mcolon \mname{E}_{{\rm ty},u} \mdot 
  x \IsDef \mname{E}_{{\rm te},u}^{\hat{e}})$.

  \item $\ForallApp u \mcolon \mname{L}, x \mcolon \mname{C} \mdot 
  x \IsDef \mname{E}_{{\rm fo},u} \Iff \\
  \hspace*{4ex}
  (\ForsomeApp e \mcolon \mname{E}_{{\rm sy},u} \mdot
  x =  [\synbrack{\mname{op-app}},[\synbrack{\mname{op}},e,
  \synbrack{\mname{formula}}]]) \Or \\
  \hspace*{4ex}
  (\ForsomeApp e_1 \mcolon \mname{E}_{{\rm sy},u}, e_2,e_3 \mcolon 
  \mname{E}_u \mdot \\
  \hspace*{8ex}
  [\synbrack{\mname{op-app}},
  [\synbrack{\mname{op}},e_1]\verb+^+e_2\verb+^+[\synbrack{\mname{formula}}]]
  \verb+^+e_3 \in \mname{E}_{{\rm fo},u} \And \\
  \hspace*{8ex}
  ((\ForsomeApp e \mcolon \mname{E}_{{\rm ty},u} \mdot \\
  \hspace*{12ex} x = [\synbrack{\mname{op-app}},
  [\synbrack{\mname{op}},e_1]\verb+^+e_2\verb+^+
  [\synbrack{\mname{type}},\synbrack{\mname{formula}}]]  
  \verb+^+e_3\verb+^+[e]) \Or \\
  \hspace*{9ex}
  (\ForsomeApp e \mcolon \mname{E}_{{\rm ty},u} , 
  e' \mcolon \mname{E}_{{\rm te},u} \mdot \\
  \hspace*{12ex} x = [\synbrack{\mname{op-app}},
  [\synbrack{\mname{op}},e_1]\verb+^+e_2\verb+^+
  [e,\synbrack{\mname{formula}}]]
  \verb+^+e_3\verb+^+[e']) \\
  \hspace*{9ex}
  (\ForsomeApp e \mcolon \mname{E}_{{\rm fo},u} \mdot \\
  \hspace*{12ex} x = [\synbrack{\mname{op-app}},
  [\synbrack{\mname{op}},e_1]\verb+^+e_2\verb+^+
  [\synbrack{\mname{formula}},\synbrack{\mname{formula}}]] \\
  \hspace*{16ex} \verb+^+e_3\verb+^+[e]))) \Or \\
  \hspace*{4ex}
  (\ForsomeApp e_1 \mcolon \mname{E}_{{\rm sy},u},
  e_2 \mcolon \mname{E}_{{\rm ty},u},
  e_3 \mcolon \mname{E}_{{\rm fo},u} \mdot
  x = \synbrack{(\ForsomeApp \commabrack{e_1} : \commabrack{e_2} \mdot
  \commabrack{e_3})}) \Or \\
  \hspace*{4ex}
  (\ForsomeApp e \mcolon \mname{E}_{{\rm te},u} \mdot
  x = \synbrack{\sembrack{\commabrack{e}}_{\rm fo}})$.

\ee
\end{axschemas}

\subsection{Soundness}

Fix a normal theory $T=(L,\Gamma)$ for the rest of this subsection.
We will prove that $\textbf{C}_L$ is sound (with respect to the set of
all normal theories over $L$ and the set of all formulas of $L$) by
showing that its rules of inference preserve validity in every
standard model of $T$ and its axioms are valid in $T$.

\begin{cprop} [Propositional Connectives] \bsp \label{prop:pc}
The built-in operators and defined operators that represent
propositional connectives in Chiron---the formula operators named
\mname{formula-equal}, \mname{not}, \mname{or}, \mname{true},
\mname{false}, \mname{and}, and \mname{implies}---have their usual
meanings in every standard model of $T$. \esp
\end{cprop}

\begin{proof}
Let $M$ be a standard model of $T$.  The built-in operator names
\mname{formula-equal}, \mname{not}, and \mname{or} are assigned their
usual meanings in $M$ by the definition of a standard model.  The
defined operators \mname{true}, \mname{false}, \mname{and}, and
\mname{implies} are given their usual meanings in $M$ by their
definitions.
\end{proof}

\begin{cprop} [Quantification over Empty Types] \bsp \label{prop:qet}
Let $M = (S,V)$ be a standard model of $T$ and $\phi \in
\mname{assign}(S)$.  If $V_\phi(\alpha)$ is empty, then: \esp

\be

  \item $V_\phi(\ForsomeApp x \mcolon \alpha \mdot A) = \FALSE.$

  \item $V_\phi(\ForallApp x \mcolon \alpha \mdot A) = \TRUE.$

\ee
\end{cprop}

\begin{proof} $V_\phi(\ForsomeApp x \mcolon \alpha \mdot A) = \FALSE$ 
by the definition of $V$ on existential quantifications.  This implies
$V_\phi(\ForallApp x \mcolon \alpha \mdot A) = \TRUE$ by the
notational definition of $\Forall$ and Proposition~\ref{prop:pc}.
\end{proof}

\bigskip

\begin{clem} [Modus Ponens] \label{lem:rule-mp}
The rule Modus Ponens preserves validity in every standard model of
$T$.
\end{clem}

\begin{proof}
Let $M = (S,V)$ be a standard model of $T$. Suppose $M \models A$ and
$M \models A \Implies B$. It follows immediately by
Proposition~\ref{prop:pc} that $M \models B$.  Therefore, Modus Ponens
preserves validity in every standard model of $T$.
\end{proof}

\begin{clem} [Universal Generalization] \label{lem:rule-ug}
\bsp The rule Universal Generalization preserves validity in every standard 
model of $T$. \esp
\end{clem}

\begin{proof}
Let $M = (S,V)$ be a standard model of $T$.  Suppose $M \models A$.
Then (i) $V_\phi(A)= \TRUE$ for all $\phi \in \mname{assign}(S)$.  We
need to show that $M \models \ForallApp x
\mcolon \alpha \mdot A$.  That is, we need to show
(ii) $V_\phi(\ForallApp x \mcolon \alpha \mdot A) = \TRUE$ for all
$\phi \in \mname{assign}(S)$.  Let $\phi \in \mname{assign}(S)$.  If
$V_\phi(\alpha)$ is empty, then $V_\phi(\ForallApp x \mcolon \alpha
\mdot A) = \TRUE$ by Proposition~\ref{prop:qet}.  So we may assume
that $V_\phi(\alpha)$ is nonempty.  Let $d$ be in $V_\phi(\alpha)$.
Then $\phi[x \mapsto d](x)$ is in $V_\phi(\alpha)$, and so by (i),
$V_{\phi[x \mapsto d]}(A) = \TRUE$.  This implies $V_\phi(\ForallApp x
\mcolon \alpha \mdot A) = \TRUE$ by the notational definition of
$\Forall$ and the definition of $V$ on existential quantifications.
Therefore, (ii) holds and thus Universal Generalization preserves
validity in every standard model of $T$.
\end{proof}

\begin{clem} [Universal Quantifier Shifting] \label{lem:rule-uqs}
\bsp The rule Universal Quantifier Shifting preserves validity in every
standard model of $T$. \esp
\end{clem}

\begin{proof}
Let $M = (S,V)$ be a standard model of $T$.  Suppose (i)
$M \models \Neg\mname{free-in}(\synbrack{x},\synbrack{A})$ and (ii)
$V_\phi(\ForallApp x \mcolon \alpha \mdot (A \Or B)) = \TRUE$ for
$\phi \in \mname{assign}(S)$.  We must show $V_\phi(A \Or (\ForallApp
x \mcolon \alpha \mdot B)) = \TRUE$.  Our argument is by cases:

\bi

  \item[] \textbf{Case 1}: $V_\phi(A) = \TRUE$.  Hence, by
    Proposition~\ref{prop:pc}, $V_\phi(A \Or (\ForallApp x \mcolon
    \alpha \mdot B)) = \TRUE$.

  \item[] \textbf{Case 2}: $V_\phi(A) = \FALSE$.  It suffices to show
    (iii) $V_\phi(\ForallApp x \mcolon \alpha \mdot B) = \TRUE$ since
    this implies $V_\phi(A \Or (\ForallApp x \mcolon \alpha \mdot B))
    = \TRUE$ by Proposition~\ref{prop:pc}.  If $V_\phi(\alpha)$ is
    empty, then (iii) holds by Proposition~\ref{prop:qet}.  So let $d$
    be in $V_\phi(\alpha)$.  By the notational definition of
    $\Forall$, the definition of $V$ on existential quantifications,
    and (ii), $V_{\phi[x \mapsto d]}(A \Or B) = \TRUE$.  This implies
    $V_{\phi[x \mapsto d]}(A) = \TRUE$ or $V_{\phi[x \mapsto d]}(B)
    = \TRUE$ by Proposition~\ref{prop:pc}.  By the hypothesis, (i), and
    Lemma~\ref{lem:fv}, $V_{\phi[x \mapsto d]}(A) = \FALSE$.  This
    implies $V_{\phi[x \mapsto d]}(B) = \TRUE$, and hence (iii) holds
    by the notational definition of $\Forall$ and the definition of
    $V$ on existential quantifications.

\ei
\end{proof}

\begin{clem} [Universal Instantiation] \label{lem:rule-ui}
The rule Universal Instantiation preserves validity in every standard
model of $T$.
\end{clem}

\begin{proof}
Let $M = (S,V)$ be a standard model of $T$.  Suppose (i) $M \models
\mname{sub}(\synbrack{a}, \synbrack{x}, \synbrack{A}) \IsDefApp$,
(ii) $M \models \mname{free-for}(\synbrack{a}, \synbrack{x},
\synbrack{A})$, and (iii) $V_\phi((\ForallApp x \mcolon \alpha \mdot
A) \And {a \IsDef \alpha}) =\TRUE$ for $\phi \in \mname{assign}(S)$.
We must show
\[V_\phi(\sembrack{\mname{sub}(\synbrack{a}, \synbrack{x},
  \synbrack{A})}_{\rm fo}) = \TRUE.\] By Proposition~\ref{prop:pc},
(iii) implies (iv) $V_\phi(\ForallApp x \mcolon \alpha \mdot A) = \TRUE$
and (v) $V_\phi(a \IsDef \alpha) = \TRUE$.  (v) implies (vi)
$V_\phi(a) \not= \Undefined$ and (vii) $V_\phi(a)$ is in
$V_\phi(\alpha)$. (iv) and (vii) imply (viii) $V_{\phi[x \mapsto
    V_\phi(a)]}(A) = \TRUE$ by the notational definition of $\Forall$
and the definition of $V$ on existential quantifications.  (i),
(ii), (vi), and (viii)
imply \[V_\phi(\sembrack{\mname{sub}(\synbrack{a}, \synbrack{x},
  \synbrack{A})}_{\rm fo}) = V_{\phi[x \mapsto V_\phi(a)]}(A) =
\TRUE\] by Lemma~\ref{lem:sub-b}.
\end{proof}

\begin{clem} [Definite Description] \label{lem:rule-dd}
\bsp The rule Definite Description preserves validity in every standard
model of $T$.\esp
\end{clem}

\begin{proof}
Let $M = (S,V)$ be a standard model of $T$.  Suppose (i) $M \models
\mname{sub}(\synbrack{(\iotaApp x \mcolon \alpha \mdot A)},
\synbrack{x}, \synbrack{A}) \IsDefApp$, (ii) $M \models
\mname{free-for}(\synbrack{(\iotaApp x \mcolon \alpha \mdot A)},
\synbrack{x}, \synbrack{A})$ and (iii) $V_\phi(\ForsomeUniqueApp x
\mcolon \alpha \mdot A) = \TRUE$ for $\phi \in \mname{assign}(S)$.  We
need to show \[V_\phi(\sembrack{\mname{sub}(\synbrack{(\iotaApp x
    \mcolon \alpha \mdot A)}, \synbrack{x}, \synbrack{A})}_{\rm fo})
= \TRUE.\]  By the notational definition of unique existential
quantification and the definition of $V$ on existential
quantifications, (iii) implies (iv) there is a unique $d$ in
$V_\phi(\alpha)$ such that $V_{\phi[x \mapsto d]}(A) = \TRUE$, and by
the definition of $V$ on definite descriptions, (iv) implies (v)
there is a $d$ in $V_\phi(\alpha)$ such that $V_\phi(\iotaApp x
\mcolon \alpha \mdot A) = d$.  (v) implies (vi) $V_\phi(\iotaApp x
\mcolon \alpha \mdot A) \not= \Undefined$.  (i), (ii), (iv),
(v), and (vi) imply
\begin{eqnarray*}
V_\phi(\sembrack{\mname{sub}(\synbrack{(\iotaApp x \mcolon \alpha
    \mdot A)}, \synbrack{x}, \synbrack{A})}_{\rm fo}) 
& = & V_{\phi [x \mapsto V_{\phi}(\iotaApp x \mcolon \alpha \mdot A)]}(A) \\
& = & V_{\phi [x \mapsto d]}(A) \\
& = & \TRUE
\end{eqnarray*}
by Lemma~\ref{lem:sub-b}.
\end{proof}

\begin{clem} [Indefinite Description] \label{lem:rule-id}
\bsp The rule Indefinite Description preserves validity in every standard
model of $T$. \esp
\end{clem}

\begin{proof}
Similar to the proof of Lemma~\ref{lem:rule-dd}.
\end{proof}

\begin{clem} [Functions] \label{lem:rule-fun}
The four rules of inference for functions preserve validity in every
standard model of $T$.
\end{clem}

\begin{proof}
The rule Type Application
specifies the value of a type application $\alpha(a)$ by
Proposition~\ref{prop:pc}, the definitions of \mname{free-in},
\mname{fun}, and \mname{ord-pair}, and the definition of $V$ on type
applications.  The rule Dependent Function Type specifies the value of
a dependent function type $(\LAMBDAapp x \mcolon \alpha \mdot \beta)$
by Proposition~\ref{prop:pc}, the definition of \mname{free-in} and
\mname{fun}, and the definition of $V$ on dependent function types.
The rule Function Application specifies the value of a function
application $f(a)$ by Proposition~\ref{prop:pc}, the definitions
of \mname{free-in}, and
\mname{fun}, and \mname{ord-pair}, and the definition of $V$ on type
and function applications.  And, finally, the rule Function
Abstraction  specifies the value
of a function abstraction $(\LambdaApp x
\mcolon \alpha \mdot b)$ of a function abstraction by
Proposition~\ref{prop:pc}, the definition of \mname{free-in}, and the
definition of $V$ on dependent function types, function applications,
and function abstractions.  Therefore, each of the four rules of
inference for functions preserves validity in every standard model of
$T$.
\end{proof}

\begin{clem} [Axiom Schemas 1] \label{lem:asa}
\bsp Each instance of the axiom schemas in Axiom Schemas 1 is valid 
in $T$.\esp
\end{clem}

\begin{proof}
The instances of these five axiom schemas are tautologies under the
usual interpretations of the propositional connectives.  Therefore, by
Proposition~\ref{prop:pc}, each such instance is valid in $T$.
\end{proof}

\begin{clem} [Axiom Schemas 2]
\bsp Each instance of the axiom schemas in Axiom Schemas 2 is valid 
in $T$.\esp
\end{clem}

\begin{proof}
The instances of these six axiom schemas are valid in $T$ by clauses
n, o, and p of the definition of $I$ in a structure for $L$, the
notational definition of $\QuasiEqual$, and the compositionality of
the definition of the standard valuation.
\end{proof}

\begin{clem} [Axiom Schemas 3]
\bsp Each instance of the axiom schemas in Axiom Schemas 3 is valid 
in $T$.\esp
\end{clem}

\begin{proof}
The instances of the eight axiom schemas are valid in $T$ by
Proposition~\ref{prop:pc}, the definition of $I$ in a structure for
$L$, and the definition of the standard valuation on operator
applications.
\end{proof}

\begin{clem} [Axiom Schemas 4]
\bsp Each instance of the axiom schemas in Axiom Schemas 4 is valid 
in $T$.\esp
\end{clem}

\begin{proof}
The instances of these ten axiom schemas are valid in $T$ by
Proposition~\ref{prop:pc} and clauses c--d and f--m, respectively, of
the definition of $I$ in a structure for $L$.
\end{proof}

\begin{clem} [Axiom Schemas 5]
\bsp Each instance of the axiom schemas in Axiom Schemas 5 is valid 
in $T$.\esp
\end{clem}

\begin{proof}
The instances of this axiom schema are valid in $T$ by
Proposition~\ref{prop:pc} and the definition of the standard valuation
on variables.
\end{proof}

\begin{clem} [Axiom Schemas 6]
\bsp Each instance of the axiom schemas in Axiom Schemas 6 is valid 
in $T$.\esp
\end{clem}

\begin{proof}
The instances of the first axiom schema are valid in $T$ by the
definition of $\Ds$ and clause n of the definition of $I$ in a
structure for $L$.  The instances of the second axiom schema are valid
in $T$ by the definition of $\nabla$ and Proposition~\ref{prop:qet}.
\end{proof}

\begin{clem} [Axiom Schemas 7]
\bsp Each instance of the axiom schemas in Axiom Schemas 7 is valid 
in $T$.\esp
\end{clem}

\begin{proof}
Let $M = (S,V)$ be a standard model of $T$.  The set-theoretic
built-in operator names \mname{class}, \mname{set}, and \mname{in} are
given their usual meanings in $M$ by the definition of a standard
model.  The defined set-theoretic operators \mname{empty-set},
\mname{pair}, \mname{ord-pair}, \mname{subclass},
\mname{intersection}, \mname{complement}, \mname{fun}, \mname{dom},
\mname{sum}, and \mname{power} are given their usual meanings in $M$
by their definitions.  $S$ is constructed from a prestructure
$(D,\in)$ that satisfies the axioms of {\nbg} set theory.  The
instances of the axiom schemas in Axiom Schema 7 are exactly these
axioms expressed in the language of Chiron.  Therefore, the instances
are valid in $M$ and thus valid in $T$.
\end{proof}

\begin{clem} [Axiom Schemas 8]
\bsp Each instance of the axiom schemas in Axiom Schemas 8 is valid in
$T$.\esp
\end{clem}

\begin{proof}
The instances of these two axiom schemas are valid in $T$ by
Proposition~\ref{prop:pc}, the definition of the standard valuation on
conditional terms, and the notational definition of $\QuasiEqual$.
\end{proof}

\begin{clem} [Axiom Schemas 9] \label{lem:asj}
\bsp Each instance of the axiom schemas in Axiom Schemas 9 is valid 
in $T$.\esp
\end{clem}

\begin{proof} Let $M = (S,V)$ be a standard model of $T$ and 
$\phi \in \mname{assign}(S)$.

\bigskip

\noindent \textbf{Schema 1} \sglsp Follows from the proof of 
Lemma~\ref{lem:rule-dd}.

\bigskip

\noindent \textbf{Schema 2} \sglsp Assume (i)
$V_\phi(\Neg(\ForsomeUniqueApp x \mcolon \alpha \mdot A)) = \TRUE$.
By Proposition~\ref{prop:pc}, we must show (ii) $V_\phi((\iotaApp x
\mcolon \alpha \mdot A) \IsUndefApp) = \TRUE$.  By the notational
definition of unique existential quantification and the definition of
$V$ on existential quantifications, (i) implies (iii) there is no
unique $d$ in $V_\phi(\alpha)$ such that $V_{\phi[x \mapsto d]}(A) =
\TRUE$, and by the definition of $V$ on definite descriptions, (iii)
implies (ii).
\end{proof}

\begin{clem} [Axiom Schemas 10] \label{lem:ask}
\bsp Each instance of the axiom schemas in Axiom Schemas 10 is valid in
$T$.\esp
\end{clem}

\begin{proof}
The proof for the first two axiom schema is similar to the proof of
Lemma~\ref{lem:asj}.  The proof for the third axiom schema follows
immediately from the definition of the standard valuation on
indefinite descriptions.
\end{proof}

\begin{clem} [Axiom Schemas 11]
\bsp Each instance of the axiom schemas in Axiom Schemas 11 is valid 
in $T$.\esp
\end{clem}

\begin{proof}
The instances of these three axiom schemas are valid in $T$ by the
definition of $H$ in a structure for $L$, the definition of the
defined operator named \mname{ord-pair}, and the definition of the
standard valuation on quotations.
\end{proof}

\begin{clem} [Axiom Schemas 12]
\bsp Each instance of the axiom schemas in Axiom Schemas 12 is valid in
$T$. \esp
\end{clem}

\begin{proof}
The instances of these six axiom schemas are valid in $T$ by
Proposition~\ref{prop:pc}, Lemma~\ref{lem:gea}, and the definition of
the defined operator named \mname{gea}.
\end{proof}

\begin{clem} [Axiom Schemas 13] \label{lem:asm}
\bsp Each instance of the axiom schemas in Axiom Schemas 13 is valid in
$T$.\esp
\end{clem}

\bsp
\begin{proof} 
Let $M = (S,V)$ be a standard model of $T$ and $I$ be the last
component of $S$.  The first schema specifies \mname{op-names} to be
$I(\mname{op-names})$.  The second to ninth schemas specify
\mname{expr-sym} and \mname{expr-op-name} to be $I(\mname{expr-sym})$
and $I(\mname{expr-op-name})$, respectively.  Using induction, the
first tenth schema specifies \mname{expr} to be $I(\mname{expr})$.
Using mutual recursion, the remaining schemas define \mname{expr-op},
\mname{expr-type}, \mname{expr-term-type}, \mname{expr-term}, and
\mname{expr-formula} to be $I(\mname{expr-op})$,
$I(\mname{expr-type})$, $I(\mname{expr-term-type})$,
$I(\mname{expr-term})$, and $I(\mname{expr-formula})$, respectively.
Therefore, each instance of these schemas is valid in $T$.
\end{proof}
\esp

\begin{cthm} [Soundness] \label{thm:soundness}
$\textbf{C}_L$ is sound with respect to the set of all normal theories
  over $L$ and the set of all formulas of $L$.
\end{cthm}

\begin{proof}
Let $T=(L,\Gamma)$ be a normal theory of Chiron and $A$ be a formula
of $L$.  By Lemma~\ref{lem:rule-mp}--\ref{lem:rule-fun}, the rules of
inference of $\textbf{C}_L$ preserve validity in every standard model
of $T$.  By Lemmas~\ref{lem:asa}--\ref{lem:asm}, each axiom of
$\textbf{C}_L$ is valid in $T$.  Therefore, if $\proves{T}{A}$, then
$T \models A$, and hence $\textbf{C}_L$ is sound with respect to all
formulas of $L$.
\end{proof}

\subsection{Some Metatheorems}

Fix an eval-free normal theory $T=(L,\Gamma)$ for the rest of this
subsection.

Let $\textbf{C}^{\ast}_{L}$ be $\textbf{C}_L$ where the schema
variables in Rules~3--10 are restricted to eval-free expressions.  Let
$\provesast{T}{A}$ mean there is a proof of $A$ from $T$ in
$\textbf{C}^{\ast}_{L}$.  $T$ is $\emph{consistent}^\ast$ if there is
some formula $A$ of $L$ such that $\provesast{T}{A}$ does not hold.
Obviously $\provesast{T}{A}$ implies $\proves{T}{A}$.

\begin{cthm}[Tautology]\label{thm:tautology}
If $\provesast{T}{A_1}$, \ldots, $\provesast{T}{A_n}$ and $(A_1 \And \cdots
\And A_n) \Implies B$ is a tautology for $n \ge 1$, then
$\provesast{T}{B}$.  Also, if $B$ is tautology, then $\provesast{T}{B}$.
\end{cthm}

\begin{proof}
Follows from Axiom Schemas~1 and the Modus Ponens rule of inference by
a standard argument.
\end{proof}

\begin{clem} \label{lem:eval-free-ded}
Let $a$ be an eval-free term, $x$ be a symbol, and $e,e'$ be eval-free
proper expressions of $L$.
\be

  \item If $x$ is free in $e$, then
    $\provesast{T}{\mname{free-in}(\synbrack{x},\synbrack{e})}$.

  \item If $x$ is not free in $e$,
    $\provesast{T}{\Neg \mname{free-in}(\synbrack{x},\synbrack{e})}$.

  \item If $a$ is free for $x$ in $e$, then
    $\provesast{T}{\mname{free-for}
    (\synbrack{a}, \synbrack{x}, \synbrack{e})}$.

  \item If $a$ is not free for $x$ in $e$, then
    $\provesast{T}{\Neg \mname{free-for}
    (\synbrack{a}, \synbrack{x}, \synbrack{e})}$.

  \item If $e$ is syntactically closed, then
    $\provesast{T}{\mname{syn-closed}(\synbrack{e})}$.

  \item If $T_{\rm ker}^{L} \models \mname{cleanse}(\synbrack{e}) =
    \synbrack{e}$, then $\provesast{T}{\mname{cleanse}(\synbrack{e})
      = \synbrack{e}}$.

  \item If $T_{\rm ker}^{L} \models \mname{sub}
    (\synbrack{a}, \synbrack{x},\synbrack{e}) = \synbrack{e'}$,
    then $\provesast{T}{\mname{sub}
    (\synbrack{a}, \synbrack{x},\synbrack{e}) = \synbrack{e'}}$.

\ee
\end{clem}

\begin{proof}
By induction on the length of $e$ using the Tautology theorem, Axiom
Schemas~\ref{axschemas:equality}, and the definitions of
\mname{free-in}, \mname{free-for}, \mname{syn-closed},
\mname{cleanse}, and \mname{sub}.  The Universal Generalization rule
of inference is also needed for part 5.
\end{proof}

\begin{cthm}[Deduction]\label{thm:deduction}
Let $A$ be a syntactically closed, eval-free formula of $L$ and $T' =
(L,\Gamma \cup \set{A})$.  If $\provesast{T'}{B}$, then
$\provesast{T}{A \Implies B}$.
\end{cthm}

\begin{proof}
Follows from Lemmas~\ref{lem:eval-free} and \ref{lem:eval-free-ded};
Axiom Schemas~1; and the Modus Ponens, Universal Generalization, and
Universal Quantifier Shifting rules of inference by a standard
argument.  (A standard argument fails if $\textbf{C}_L$ is used in
place of $\textbf{C}^{\ast}_{L}$.)
\end{proof}

\begin{clem}\label{lem:incons}
Let $A$ be a syntactically closed, eval-free formula of $L$ and $T' =
(L,\Gamma \cup \set{A})$.  If $T'$ is not $\textrm{consistent}^\ast$,
then $\provesast{T}{\Neg A}$.
\end{clem}

\begin{proof}
It follows from the hypotheses of the lemma and the Tautology and
Deduction theorems by a standard argument that $\provesast{T}{\Neg A}$.
\end{proof}

\begin{clem}\label{lem:var-ins}
If $A$ is an eval-free formula of $L$,
then \[\provesast{T}{(\ForallApp x \mcolon \mname{C} \mdot A) \Implies
  A}.\]
\end{clem}

\begin{proof}
By the Deduction theorem, the Universal Instantiation rule of
inference, the first axiom schema of Axiom Schemas 7,
Lemmas~\ref{lem:eval-free} and \ref{lem:eval-free-ded}, the definition
of \mname{sub}, and the fifth axiom schema of Axiom Schemas~12.
\end{proof}

\subsection{Completeness}

Let $\sT$ be the set of all eval-free normal theories over $L$ and
$\sF$ be the set of all eval-free formulas of $L$.  We will prove that
$\textbf{C}^{\ast}_{L}$ is complete with respect to $\sT$ and $\sF$.
This will then immediately imply that $\textbf{C}_L$ is complete with
respect to $\sT$ and $\sF$ as well.  The key component of the proof is
the following lemma that says a standard model can be constructed for
any $\textrm{consistent}^\ast$ eval-free normal theory.  The proof of
this lemma is long and tedious.

\begin{clem} [Model Construction] \label{lem:construction} \bsp 
Let $T = (L,\Gamma)$ be a $\textrm{consistent}^\ast$ eval-free normal
theory.  Then $T$ has a standard model $M=(S,V)$ such that, for every
class $x$ in $M$, there is a syntactically closed, eval-free term $a$
of $L$ whose value in $M$ (with respect to any $\phi$) is $x$. \esp
\end{clem}

\begin{proof} Let $T'=(L,\Gamma')$ be an extension of $T$ that is maximal
$\textrm{consistent}^\ast$.  We will tacitly use the Tautology
  theorem, Lemma~\ref{lem:eval-free-ded}, and the fact that $T$ is
  maximal $\textrm{consistent}^\ast$ throughout this proof.

\bigskip

\noindent \textbf{Step 1} \sglsp Let $\sC$ be the set of syntactically
closed terms $c$ of $L$ such that $c = c'$ is in $\Gamma'$ for some
syntactically closed, eval-free term $c'$ of $L$.  Let $\sC'$ be the
set of syntactically closed terms $c$ of $L$ such that $c \in \sC$ or
$c\IsUndefApp$ is in $\Gamma'$.  And let $\sT$ be the set of types
$\alpha$ such that (1) $\alpha = \mname{C}$ or (2) for all $c \not\in
\sC$, $c \IsUndef \alpha$ is in $\Gamma'$.  For $c,d \in \sC$, define
$c \sim d$ iff $c = d$ is in $\Gamma'$.  By Axiom Schemas 2, $\sim$ is
an equivalence relation on $\sC$.  For each $c \in \sC$,
let \[\tilde{c} = \set{d \in \sC \;|\; c \sim d}\] be the equivalence
class of $c$.  Note that each $\tilde{c}$ contains an eval-free term.

Define $\Dc = \set{\tilde{c} \;|\; c \in \sC}$.  For $c,d \in \sC$,
define $\tilde{c} \in \tilde{d}$ iff $c \in d$ is in $\Gamma'$.  By
Axiom Schemas 2, the definition of $\in$ is well defined.  Define
$\Dv$, $\Ds$, $\Df$, $\Do$, $\De$, $\TRUE$, $\FALSE$, and $\Undefined$
as in the definition of a structure for $L$.  Let $\xi$ be any choice
function on $\Ds$.  $H$ and $I$ will be defined later.

\bigskip

\noindent \textbf{Step 2} \sglsp For a syntactically closed type
$\alpha$ of $L$, define \[U(\alpha) = \set{\tilde{c} \in \Dc \;|\; c
  \in \sC \mbox{ and } c \IsDef \alpha \mbox{ is in } \Gamma'}\] if
$\alpha \in \sT$ and $U(\alpha) = \mname{C}$ if $\alpha \not\in \sT$.
For a syntactically closed term $a$ of $L$, define $U(a) = \tilde{a}$
if $a \in \sC$ and $U(a) = \Undefined$ if $a \not\in \sC$.  For a
syntactically closed formula $A$ of $L$, define $U(A) = \TRUE$ if $A$
in $\Gamma'$ and $U(A) = \FALSE$ if $\Neg A$ is in $\Gamma'$.

\bigskip
%\newpage

\noindent \emph{Claim 1} 
\be

  \item For all syntactically closed types $\alpha$ and $\beta$ of $L$
    with $\alpha,\beta \in \sT$, $U(\alpha) = U(\beta)$ iff 
    $\alpha \TypeEqual \beta$  is in $\Gamma'$.

  \item For all syntactically closed terms $a$ and $b$ of $L$ with
    $a,b \in \sC'$, $U(a) = U(b)$ iff $a \QuasiEqual b$ is in
    $\Gamma'$.

  \item For all syntactically closed formula $A$ and $B$ of $L$, $U(A)
    = U(B)$ iff $A \Iff B$ is in $\Gamma'$.

\ee

%\newpage

\noindent \emph{Proof} 
\be

  \item \bsp $U(\alpha) = U(\beta)$ iff $\set{\tilde{c} \in \Dc \;|\;
    c \in \sC \mbox{ and } c \IsDef \alpha \mbox{ is in } \Gamma'} =
    \set{\tilde{c} \in \Dc \;|\; c \in \sC \mbox{ and } c \IsDef \beta
      \mbox{ is in } \Gamma'}$ iff for all
    $\tilde{c} \in \Dc$, ($c \IsDef \alpha$ in $\Gamma'$ iff $c \IsDef
    \beta$ in $\Gamma'$) iff for all $\tilde{c} \in \Dc$, $c \IsDef
    \alpha \Iff c \IsDef \beta$ is in $\Gamma'$ by the Tautology
    theorem iff $\ForallApp x \mcolon \mname{C} \mdot x \IsDef \alpha
    \Iff x \IsDef \beta$ is in $\Gamma'$ by the Universal Instantiation
    rule of inference iff $\alpha \TypeEqual \beta$ is in $\Gamma'$ by
    the first axiom schema of Axiom Schemas~\ref{axschemas:types}.
    Therefore, $U(\alpha) = U(\beta)$ iff $\alpha \TypeEqual \beta$ is
    in $\Gamma'$. \esp

  \item $U(a) = U(b) \not= \Undefined$ iff $\tilde{a} = \tilde{b}$ iff
    $a \sim b$ iff (1) $a = b$ is in $\Gamma'$. $U(a) = U(b) =
    \Undefined$ iff (2) $a\IsUndefApp$ and $b\IsUndefApp$ are in
    $\Gamma'$.  (1) and (2) hold iff $a \QuasiEqual b$ is in $\Gamma'$
    by the definition of \mname{quasi-equal}.  Therefore, $U(a) =
    U(b)$ iff $a \QuasiEqual b$ is in $\Gamma'$.

  \item $U(A) = U(B) = \TRUE$ iff (1) $A$ and $B$ are in $\Gamma'$, and
    $U(A) = U(B) = \FALSE$ iff (2) $\Neg A$ and $\Neg B$ are in
    $\Gamma'$.  (1) and (2) hold iff $A \Iff B$ is in $\Gamma'$.
    Therefore, $U(A) = U(B)$ iff $A \Iff B$ is in $\Gamma'$.

\ee
This completes the proof of Claim 1.

\bigskip

\noindent \textbf{Step 3} \sglsp Recall that each quotation is
syntactically closed and eval-free.  Let $G$ be the mapping from $\sS
\cup \sO$ to $\Dc$ such that, for all $s \in \sS \cup \sO$, $G(s) =
U(\synbrack{s})$.  By the first axiom schema of Axiom
Schemas~\ref{axschemas:quote} and the fourth axiom schema of Axiom
Schemas~\ref{axschemas:con-types}, the range of $G$ is a subset of
$\Dv$.  By part 2 of Claim~1 and the third axiom schema of Axiom
Schemas~\ref{axschemas:quote}, $G$ is injective.  Define $H$ from $G$
as in the definition of a structure for $L$.

\bigskip

\noindent \textbf{Step 4} \sglsp Let $o$ be an operator name of $L$
with the signature form $s_1,\ldots,s_{n+1}$ where $n \ge 0$. Then let
$k_1,\ldots,k_{n+1}$ be the signature where $k_i = s_i$ if $s_i$ is
\mname{type} or \mname{formula} and $k_i = \mname{C}$ if $s_i$ is
\mname{term} for all $i$ with $1 \le i \le n + 1$, and let $D_i = \Ds$
if $s_i$ is \mname{type}, $D_i = \Dc \cup \set{\Undefined}$ if $s_i$
is \mname{term}, and $D_i = \set{\TRUE,\FALSE}$ if $s_i$ is
\mname{formula} for all $i$ with $1 \le i \le n+1$.  Define $U(o)$ to
be any operation $\sigma$ from $D_1 \times \cdots \times D_n$ into
$D_{n+1}$ such that, for all syntactically closed expressions
$e_1,\ldots,e_n$ where $e_i$ is a type in $\sT$ if $s_i =
\mname{type}$, a term in $\sC'$ if $s_i = \mname{term}$, and a formula
if $s_i = \mname{formula}$ for all $i$ with $1 \le i \le n$,
\[\sigma(U(e_1),\ldots,U(e_n)) = 
U((o :: k_1,\ldots,k_{n+1})(e_1,\ldots,e_n)).\] By Claim 1 and Axiom
Schemas~\ref{axschemas:equality}, $\sigma(U(e_1),\ldots,U(e_n))$ is
well defined by this equation.

\bigskip

\noindent \emph{Claim 2} \sglsp $U$ on operator names satisfies the
specification for $I$ in the definition of a structure for $L$.

\bigskip

\noindent \emph{Proof} \sglsp 
Let $o$ be an operator name of $L$ with the signature form
$s_1,\ldots,s_{n+1}$ where $n \ge 0$.  We must show that $U(o)$
satisfies the specification for $I(o)$ in the definition of a
structure for $L$.  First, we have to show that $U(o)$ is an
operation from $D_1 \times \cdots \times D_n$ into $D_{n+1}$ where the
$D_i$ are defined as above.  This is true immediately by the
definition of $U(o)$.  Second, if $o$ is a built-in operator name of
Chiron, we have to verify that the clause of part 5 of the definition
of $I$ corresponding to $o$ is satisfied.
\be

  \item[] Clause a: \mname{set}.  
  \begin{eqnarray*}
  U(\mname{set})(\,) & = & U(\mname{V}) \\
  & = & \set{\tilde{c} \in \Dc \;|\; c \IsDef 
  \mname{V} \mbox{ is in } \Gamma'} \\
  & = & \set{\tilde{c} \in \Dc \;|\; 
  \ForsomeApp y \mcolon \mname{C} \mdot c \in y \mbox{ is in } \Gamma'} \\
  & = & \set{\tilde{c} \in \Dc \;|\; 
  \mbox{for some } \tilde{d} \in \Dc, c \in d \mbox{ is in } \Gamma'} \\
  & = & \set{\tilde{c} \in \Dc \;|\; \mbox{for some } 
  \tilde{d} \in \Dc, \tilde{c} \in \tilde{d}} \\
  & = & \Dv.
  \end{eqnarray*}
  The second line is by the definition of $U$ on syntactically closed
  types; the third line is by the second axiom schema of Axiom
  Schemas~\ref{axschemas:set-theory} and the Universal Instantiation
  rule of inference; fourth line is by the Indefinite Description rule
  of inference; the fifth line is by the definition of $\in$ on $\Dc$;
  and sixth line is by the definition of $\Dv$. Therefore,
  $U(\mname{set})$ satisfies clause a.

  \item[] Clause b: \mname{class}.
  \begin{eqnarray*}
  U(\mname{class})(\,) & = & U(\mname{C}) \\
  & = & \set{\tilde{c} \in \Dc \;|\; c \IsDef 
  \mname{C} \mbox{ is in } \Gamma'} \\
  & = & \set{\tilde{c} \in \Dc \;|\; c \IsDefApp 
  \mbox{ is in } \Gamma'} \\
  & = & \Dc.
  \end{eqnarray*}
  The second line is by the definition of $U$ on syntactically closed
  types; the third line is by the definition of the notation $c
  \IsDefApp$; and the fourth line is by the definition of $\Dc$.
  Therefore, $U(\mname{class})$ satisfies clause b.

  \item[] \bsp Clauses c--l: \mname{op-names}, \mname{lang},
    \mname{expr-sym}, \mname{expr-op-name}, \mname{expr},
    \mname{expr-op}, \mname{expr-type}, \mname{expr-term-type},
    \mname{expr-formula}. Follows from Axiom
    Schemas~\ref{axschemas:bi-op-def}~and~\ref{axschemas:con-types}. 
    \esp

  \item[] Clause m: \mname{in}.  Let $x$ and $y$ be in $\Dc \cup
    \set{\Undefined}$.  Then $x = U(c)$ and $y = U(d)$ for some
    syntactically closed terms $c$ and $d$.  $U(\mname{in})(x,y) =
    U(\mname{in})(U(c),U(d)) = U(c \in d)$.  $U(c \in d) = \TRUE$ iff
    $c \in d$ is in $\Gamma'$ iff $U(c), U(d) \in \Dc$ and $U(c) \in
    U(d)$ by the definition of $\in$ on $\sC$.  $U(c \in d) = \FALSE$
    iff $c \not\in d$ is in $\Gamma'$ iff either (1) $U(c), U(d) \in
    \Dc$ and $U(c) \not\in U(d)$ by the definition of $\in$ on $\sC$
    or (2) $U(c) = \Undefined$ or $U(d) = \Undefined$ by the second
    axiom schema of Axiom Schema~\ref{axschemas:bi-op-def}.
    Therefore, $U(\mname{in})$ satisfies clause m.

  \item[] Clauses n--p: \mname{type-equal}, \mname{term-equal},
  \mname{formula-equal}.  Follow from Axiom 
  Schemas~\ref{axschemas:equality}.  The third axiom schema of 
  Axiom Schema~\ref{axschemas:bi-op-def} is needed for 
  \mname{term-equal}. 

  \item[] Clauses q--r: \mname{not}, \mname{or}. Follow from Axiom 
  Schemas~\ref{axschemas:prop}.

\ee  
This completes the proof of Claim 2.

\bigskip

If we define $I(o) = U(o)$ for all $o \in \sO$, then \[S =
(\Dv,\Dc,\Ds,\Df,\Do,\De,\in,\TRUE,\FALSE,\Undefined,\xi,H,I)\] is a
structure for $L$ \emph{provided $(\Dc,\in)$ is a prestructure}.  We
will show that $(\Dc,\in)$ is a prestructure in Step 6 of the proof.

\bigskip

\noindent \textbf{Step 5} \sglsp Let $e$ be an eval-free type, term,
or formula, and let $\phi \in \mname{assign}(S)$.  Assume that the
members of $\sS \cup \sO$ are linearly ordered by $<_{{\cal S} \cup
  {\cal O}}$.  For $n \ge 0$, let
\[\mname{S}^{a_1 \cdots a_n}_{x_1 \cdots x_n} e = 
\left\{\begin{array}{ll}
         e & \mbox{if } n = 0 \\
         \sembrack{\mname{sub}
         (\synbrack{a_1},\synbrack{x_1},
         \synbrack{\mname{S}^{a_2 \cdots a_{n-1}}_{x_2 \cdots x_{n-1}}e})}_{k[e]}
         & \mbox{if } n > 0
       \end{array}
\right.\] 
where
\[\set{x_1,\ldots,x_n}= \set{x \in \sS \cup \sO\;|\;
\provesast{T}{\mname{free-in}(\synbrack{x},\synbrack{e})}},\] $x_1
<_{{\cal S} \cup {\cal O}} \cdots <_{{\cal S} \cup {\cal O}} x_n$, and
$a_i$ is the first eval-free member (in some fixed enumeration) of
$\sC$ such that $\phi(x_i) = U(a_i)$ for all $i$ with $1 \le i \le n$.
Then define
\[e^\phi = \mname{S}^{a_1 \cdots a_n}_{x_1 \cdots x_n} e\]
and 
\[e^{\phi - x_i} = \mname{S}^{a_1 a_2 \cdots a_{i-1} a_{i+1} \cdots a_n}_{x_1 x_2 \cdots x_{i-1} x_{i+1} \cdots x_n} e\]
where $1 \le i \le n$.

\bigskip

\noindent \emph{Claim 3} \sglsp $\mname{S}^{a_i}_{x_i}(e^{\phi - x_i})
= e^\phi$.

\bigskip

\noindent \emph{Proof} \sglsp Follows from $e$ being eval-free and
$a_1,\ldots,a_n$ being syntactically closed and eval-free.  This
completes the proof of Claim 3.

\bigskip

Let $V$ be defined by (1)~$V_\phi(e)$ is undefined when $e$ is
improper, (2)~$V_\phi(e)$ is $I(o)$ when $e = (o ::
k_1,\ldots,k_{n+1})$ is proper, (3)~$V_\phi(e) = U(e^\phi)$ when $e$
is an eval-free type, term, or formula, and (4)~$V_\phi(e)$ is the
standard valuation for $S$ (provided $(\Dc,\in)$ is a prestructure)
applied to $e$ and $\phi$ when $e$ is a non-eval-free type, term, or
formula.  When $e$ is an eval-free type, term, or formula, $e^\phi$ is
clearly a syntactically closed, eval-free type, term, or formula, and
so $V$ is a valuation for $S$ (provided $(\Dc,\in)$ is a
prestructure).

\bigskip

\noindent \emph{Claim 4} \sglsp $V$ is the standard valuation for $S$
(provided $(\Dc,\in)$ is a prestructure).

\bigskip

\noindent \emph{Proof} \sglsp We must show that, for all $e \in \sE_L$
and $\phi \in \mname{assign}(S)$, $V$ is the standard evaluation
applied to $e$ and $\phi$.  This is true by the definition of $V$ when
$e$ is improper, an operator, or a non-eval-free type, term, or
formula.  So let $e$ be an eval-free type, term, or formula and $\phi
\in \mname{assign}(S)$.  Our proof will be by induction on the length
of $e$.  There are 14 cases corresponding to the 14 clauses of the
definition of a standard valuation.  The proofs for cases 1, 2, and 14
are trivial.  The proofs for cases 4 and 10 are given in detail.  And
the proofs for the remaining cases are briefly sketched.

\bi

  \item[] Case 1: $e$ is improper.  This case does not occur because
    $e$ is proper.

  \item[] Case 2: $e = (o :: k_1,\ldots,k_{n+1})$.  This case does not
    occur because $e$ is not an operator.

  \item[] Case 3: $e = O(e_1,\ldots,e_n)$ and $O = (o ::
    k_1,\ldots,k_{n+1})$.  Follows from Axiom
    Schemas~\ref{axschemas:op}.

  \item[] Case 4: $e = (x : \alpha)$.  Let $a$ be some eval-free
    member of $\sC$ such that $\phi(x) = U(a)$.  By the induction
    hypothesis, $V_\phi(\alpha) = U(\alpha^\phi)$ is the standard
    valuation applied to $\alpha$ and $\phi$.

    If $\phi(x) \in V_\phi(\alpha)$, then $U(a) \in U(\alpha^\phi)$
    and
    \begin{eqnarray*}
    V_\phi((x : \alpha)) & = & U((x : \alpha)^\phi) \\
    & = & U(\sembrack{\mname{sub}
      (\synbrack{a},\synbrack{x},
      \synbrack{(x : \alpha)^{\phi - x}})}_{\rm te}) \\
    & = & U(\sembrack{\synbrack{\If(a \IsDef 
      \commabrack{\mname{sub}(\synbrack{a},\synbrack{x},
      \sembrack{\alpha^{\phi-x}})}, \\
    & & \hspace{6.5ex} \commabrack{\mname{cleanse}(\synbrack{a})}, \\
    & & \hspace{6.5ex} \Undefined_{\sf C})}}_{\rm te}) \\
    & = & U(\If(a \IsDef 
      \sembrack{\mname{sub}(\synbrack{a},\synbrack{x},
      \sembrack{\alpha^{\phi-x}})}_{\rm ty}, \\
    & & \hspace{6.5ex} \sembrack{\mname{cleanse}(\synbrack{a})}_{\rm te}, \\
    & & \hspace{6.5ex} \Undefined_{\sf C})) \\
    & = & U(\If(a \IsDef \alpha^\phi, 
      \sembrack{\synbrack{a}}_{\rm te}, 
      \Undefined_{\sf C}) \\
    & = & U(\If(a \IsDef \alpha^\phi, a, \Undefined_{\sf C})) \\
    & = & U(a) \\
    & = & \phi(x)
    \end{eqnarray*}
    The first line is by the definition of $V$; the second line is by
    Claim 3; the third line is by the definition of \mname{sub}; the
    fourth line is by Lemmas~\ref{lem:gea} and \ref{lem:eval-qq}; the
    fifth is by Claim 3 and part 6 of Lemma~\ref{lem:eval-free-ded}; the
    sixth is by the first schema of Axiom
    Schemas~\ref{axschemas:set-theory}, by Axiom
    Schemas~\ref{axschemas:if}, and by the third schema of Axiom
    Schemas~\ref{axschemas:eval}; and the seventh line is by Axiom
    Schemas~\ref{axschemas:if} and the fact that $U(a) \in
    U(\alpha^\phi)$.  Therefore, if $\phi(x) \in V_\phi(\alpha)$, then
    $V_\phi((x : \alpha)) = \phi(x)$, and so in this case $V_\phi(x :
    \alpha)$ is the standard valuation applied to $(x : \alpha)$ and
    $\phi$.

    Similarly, if $\phi(x) \not\in V_\phi(\alpha)$, then $V_\phi(e) =
    \Undefined$.

  \item[] Case 5: $e = \alpha(a)$.  Follows from the Type Application
    rule of inference.

  \item[] Case 6: $e = (\LAMBDAapp x \mcolon \alpha \mdot \beta)$.
    Follows from the Dependent Function Types rule of inference.

  \item[] Case 7: $e = f(a)$. Follows from the Function Application
    rule of inference.

  \item[] Case 8: $e = (\LambdaApp x \mcolon \alpha \mdot b)$.
    Follows from the Function Abstraction rule of inference.

  \item[] Case 9: $e = \If(A,b,c)$. Follows from Axiom
    Schemas~\ref{axschemas:if}.

  \item[] Case 10: $e = (\ForsomeApp x \mcolon \alpha \mdot B)$.  By
    the induction hypothesis, $V_\phi(\alpha) = U(\alpha^\phi)$ is the
    standard valuation applied to $\alpha$ and $\phi$ and $V_\psi(B) =
    U(B^\psi)$ is the standard valuation applied to $B$ and $\psi$ for
    all $\psi \in \mname{assign}(S)$.  Then
  \begin{eqnarray*}
  & &
    V_\phi((\ForsomeApp x \mcolon \alpha \mdot B)) = \TRUE \\
  & \mbox{iff} & 
    U((\ForsomeApp x \mcolon \alpha \mdot B)^\phi) = \TRUE \\
  & \mbox{iff} & 
    U(\ForsomeApp x \mcolon \alpha^\phi \mdot B^{\phi - x}) = \TRUE \\
  & \mbox{iff} & 
    U(\sembrack{\mname{sub}(\synbrack{c}, \synbrack{x}, 
    \synbrack{B^{\phi - x}})}_{\rm fo}) = \TRUE \mbox{ and }
    U(c \IsDef \alpha^\phi) = \TRUE \\
  & \mbox{iff} &
    U(B^{\phi[x \mapsto U(c)]}) = \TRUE \mbox{ and }
    U(c) \in U(\alpha^\phi) \\
  & \mbox{iff} &
    V_{\phi[x \mapsto V_\phi(c)]}(B) = \TRUE \mbox{ and }
    V_\phi(c) \in V_\phi(\alpha)
  \end{eqnarray*}
  where $c$ is $(\epsilonApp x \mcolon \alpha^\phi \mdot B^{\phi -
    x})$.  The second line is by the definition of $V$; the third line
  is by the definitions of $e^\phi$ and \mname{sub}; the fourth line
  is by Lemma~\ref{lem:eval-free-ded}, the Indefinite Description rule
  of inference and the first schema of Axiom
  Schemas~\ref{axschemas:indef-desc} since $c$ is both syntactically
  closed and eval-free; the fifth line is by Claim 3 and the
  definition of \mname{defined-in}; and the sixth line is by the
  induction hypothesis.  Therefore, $V_\phi((\ForsomeApp x \mcolon
  \alpha \mdot B)) = \TRUE$ iff there is some $d$ in $V_\phi(\alpha)$
  such that $V_{\phi[x \mapsto d]}(B) = \TRUE$, and so in this case
  $V_\phi((\ForsomeApp x \mcolon \alpha \mdot B))$ is the standard
  valuation applied to $(\ForsomeApp x \mcolon \alpha \mdot B)$ and
  $\phi$.

  Similarly, $V_\phi((\ForsomeApp x \mcolon \alpha \mdot B)) =
  \FALSE$ iff there is no $d$ in $V_\phi(\alpha)$ such that
  $V_{\phi[x \mapsto d]}(B) = \TRUE$.

  \item[] Case 11: $e = (\iotaApp x \mcolon \alpha \mdot B)$.  Follows
    from the Definite Description rule of inference and Axiom
    Schemas~\ref{axschemas:def-desc}.

  \item[] Case 12: $e = (\epsilonApp x \mcolon \alpha \mdot B)$.
    Follows from the Indefinite Description rule of inference and
    Axiom Schemas~\ref{axschemas:indef-desc}.

  \item[] Case 13: $e = \synbrack{e'}$. Follows from Axiom
    Schemas~\ref{axschemas:quote}.

  \item[] Case 14: $e = \sembrack{a}_k$.  This case does not occur
    because $e$ is eval-free.

\ei
This completes the proof of Claim 4.

\bigskip

\noindent \textbf{Step 6} \sglsp We have only now to show that
$(\Dc,\in)$ is a prestructure and that $M$ is a standard model for
$T$.  Let $A$ be an eval-free formula in $\Gamma'$.  Then $U(A) =
\TRUE$, and so by Claim~4, $V_\phi(A) = U(A^\phi) = U(A)= \TRUE$ for
all $\phi \in \mname{assign}(S)$.  Each instance of Axiom
Schemas~\ref{axschemas:set-theory} is eval-closed and a member of
$\Gamma'$.  Thus, by Axiom Schemas~\ref{axschemas:set-theory},
$(\Dc,\in)$ satisfies the axioms of {\nbg} set theory and hence $M$ is
a model for $L$.  And so $M \models A$ for each eval-free formula $A$
in $\Gamma'$.  Since $\Gamma \subset \Gamma'$ and each member of
$\Gamma$ is eval-free, $M \models A$ for each $A$ in $\Gamma$.
Therefore, $M$ is a standard model for $T$.
\end{proof}

\begin{cthm} [Consistency and Satisfiability] \label{thm:cons-sat}
Let $T=(L,\Gamma)$ be an eval-free normal theory of Chiron.  Then $T$
is $\textrm{consistent}^\ast$ iff\/ $T$ is satisfiable.
\end{cthm}

\begin{proof}
Let $T$ be $\textrm{consistent}^\ast$. Then there is a standard model
$M$ of $T$ by the Model Construction lemma.  Therefore, $T$ is
satisfiable.

Now let $T$ be satisfiable.  Then there is a standard model $M$ of
$T$.  Assume $T$ is not $\textrm{consistent}^\ast$.  Then
$\provesast{T}{\mname{F}}$. By the Soundness theorem, $T \models
\mname{F}$, which contradicts the definition of a standard model.
Therefore, $T$ is $\textrm{consistent}^\ast$.
\end{proof}  

\begin{cthm} [Completeness]
$\textbf{C}_L$ is complete with respect to the set of all eval-free
  normal theories over $L$ and the set of all eval-free formulas of
  $L$.
\end{cthm}

\begin{proof}
Let $T=(L,\Gamma)$ be an eval-free normal theory and $A$ be an
eval-free formula of $L$ such that $T \models A$.  We need to show
$\proves{T}{A}$.  Obviously it suffices to show $\provesast{T}{A}$.
Let $A'$ be a universal closure of $A$.  ($A'$ exists by
Lemma~\ref{lem:eval-free}.)  $A'$ is obviously eval-free and is
syntactically closed by Lemma~\ref{lem:univ-closure}.  The hypothesis
implies $T \models A'$, and this implies that there is no standard
model of $T'=(L,\Gamma \cup \set{\Neg A'})$.  Hence $T'$ is not
$\textrm{consistent}^\ast$ by the Consistency and Satisfiability
theorem.  Therefore, $\provesast{T}{\Neg\Neg A'}$ by
Lemma~\ref{lem:incons} and hence $\provesast{T}{A'}$ by the Tautology
theorem.  Finally, $\provesast{T}{A}$ by Lemma~\ref{lem:var-ins}.
\end{proof}

\section{Interpretations}\label{sec:interp}

A theory interpretation~\cite{Enderton72,Farmer94,Shoenfield67} is a
meaning-preserving mapping from the expressions of one theory to the
expressions of another theory.  An interpretation serves as a conduit
for passing information (in the form of formulas) from an abstract
theory to a more concrete theory, or an equally abstract theory.
Interpretations are the basis for the \emph{little theories
  method}~\cite{FarmerEtAl92b} for organizing mathematics where
mathematical knowledge and reasoning is distributed across a network
of theories.

\subsection{Translations}

Let $L_i = (\sO_i,\theta_i)$ be a language of Chiron for $i = 1,2$.  A
\emph{translation from $L_1$ to $L_2$} is an injective, total mapping
$\Phi : \sO_1 \tarrow \sO_2$ such that, for each $o \in \sO_1$,
$\theta_1(o) = \theta_2(\Phi(o))$.  $\Phi$ is extended to an
injective, total mapping \[\widehat{\Phi}: \sE_{L_1} \tarrow
\sE_{L_2}\] by the following rules:

\be

  \item Let $e \in \sS$.  Then $\widehat{\Phi}(e) = e$.

  \item Let $e \in \sO_1$. Then $\widehat{\Phi}(e) = \Phi(e)$.

  \item Let $e = (e_1,\ldots,e_n) \in \sE_{L_1}$.  Then
    $\widehat{\Phi}(e) = (\widehat{\Phi}(e_1), \ldots,
    \widehat{\Phi}(e_n))$.

\ee 
Hence, for every $e \in \sE_{L_1}$, $\widehat{\Phi}(e)$ is exactly the
same as $e$ except that each operator name $o$ in $e$ has been
replaced by the operator name $\Phi(o)$.

\begin{cprop}
Let $\Phi$ be a translation from $L_1$ to $L_2$.  If $e$ is a symbol,
an operator name, an operator, a type, a term, a term of type
$\alpha$, or a formula of $L_1$, then $\widehat{\Phi}(e)$ is a symbol,
an operator name, an operator, a type, a term, a term of type
$\widehat{\Phi}(\alpha)$, or a formula of $L_2$, respectively.
\end{cprop}

\subsection{Interpretations}

Let $T_i = (L_i,\Gamma_i)$ be a normal theory of Chiron for $i = 1,2$
and $\Phi$ be a translation from $L_1$ to $L_2$.  $\Phi$ \emph{fixes}
a language $L=(\sO,\theta) \leq L_1$ if $\Phi(o) = o$ for all $o \in
\sO \setminus \set{\mname{op-names},\mname{lang}}$.

\begin{cprop}\label{prop:fixes}
Suppose $\Phi$ fixes $L \leq L_1$.  Then $\widehat{\Phi}(e) = e$ for
all expressions $e$ of $L$ that do not contain the operator names
\mname{op-names} or \mname{lang}.
\end{cprop}

\bsp Recall from subsection~\ref{subsec:rel-eval} that $(\mname{eval},
a, k, b)$ is the relativization of $(\mname{eval}, a, k)$ to the
language denoted by $b$ and that $(\mname{eval}, a, k, \ell)$ is
logically equivalent to $(\mname{eval}, a, k)$.  The
\emph{relativization} of an expression $e$, written $\crel{e}$, is the
expression obtained from $e$ by repeatedly replacing each occurrence
in $e$ of an expression of the form $(\mname{eval}, a, k)$ that is not
within a quotation with the expression $(\mname{eval}, a, k, \ell)$
until every original occurrence of this kind has been replaced.
Clearly, $\crel{e} = e$ if $e$ is eval-free.  The next proposition
follows immediately from the fact that $(\mname{eval}, a, k, \ell)$ is
logically equivalent to $(\mname{eval}, a, k)$. \esp

\begin{cprop} \label{prop:rel}
For all $e \in \sE_L$, $e$ and $\crel{e}$ are logically equivalent.
\end{cprop}

Let $\Phi$ be \emph{normal for} $T_2$ if: 
\be

  \item $\sO_1$ is finite, i.e., $\sO_1 = \set{o_1,\ldots,o_n}$ for
    some $n \ge 1$.

  \item $\Phi$ fixes $L_{\rm ker}$.

  \item $T_2 \models \widehat{\Phi}(\ell) =
    \set{\synbrack{\Phi(o_1)},\ldots,\synbrack{\Phi(o_n)}}$.

  \item $T_2 \models \widehat{\Phi}(\mname{L}) \TypeEqual
    \mname{type}(\mname{power}
    (\set{\synbrack{\Phi(o_1)},\ldots,\synbrack{\Phi(o_n)}}))$.

\ee

\begin{clem} \label{lem:normal-trans}
Let $\Phi$ be normal for $T_2$.  Then $T_2 \models
\widehat{\Phi}(\crel{A})$ for each $A \in \Gamma_{\rm ker}^{L_1}$.
\end{clem}

\begin{proof}
Let $A \in \Gamma_{\rm ker}^{L_1}$.  By Proposition~\ref{prop:rel} and
the fact that $T_2$ is normal, (a) $T_2 \models \crel{A}$.  Since
$\Phi$ is normal for $T_2$, (b) $\Phi$ fixes $L_{\rm ker}$, (c) $T_2
\models \widehat{\Phi}(\ell) \subseteq \ell$, and (d) $T_2 \models
\mname{term}(\widehat{\Phi}(\mname{L})) \subseteq
\mname{term}(\mname{L})$.  (a), (b), (c), and (d) imply $T_2 \models
\widehat{\Phi}(\crel{A})$.
\end{proof}

\bigskip

Suppose $M_2 = (S_2,V_2)$ is a standard model of $T_2$ where \[S_2 =
(\Dv,\Dc,\Ds,\Df,\Do,\Deb,\in,\TRUE,\FALSE,\Undefined,\xi,H_2,I_2).\]
Let \[S_1 =
(\Dv,\Dc,\Ds,\Df,\Do,\Dea,\in,\TRUE,\FALSE,\Undefined,\xi,H_1,I_1)\]
where:

\be

  \item $\Dea = \set{H_1(e) \;|\; e \in \sE_{L_1}}$.

  \item $H_1(e) = H_2(\widehat{\Phi}(e))$ for each $e \in
    \sE_{L_1}$.

  \item $I_1(o) = I_2(\Phi(o))$ for each $o \in \sO_1$.

\ee

\begin{clem}\label{lem:modela}
Suppose $\Phi$ is normal for $T_2$.
\be

  \item $S_1$ is a structure for $L_1$.

  \item $M_1=(S_1,V_1)$, where $V_1$ is the standard valuation for
    $S_1$, is a standard model for $L_1$.

  \item For all proper expressions $e$ of $L_1$ and $\phi \in
    \mname{assign}(S_1)$, \[V_{1,\phi}(e) =
    V_{2,\phi}(\widehat{\Phi}(\crel{e})).\]

  \item For all formulas of $L_1$, \[M_1 \models A \dblsp \mbox{iff}
  \dblsp M_2 \models \widehat{\Phi}(\crel{A}).\]

\ee
\end{clem}

\begin{proof}

\bigskip

\noindent \textbf{Part 1} \sglsp Since $M_2$ is a standard model for
$L_2$, we need to only show that $\Dea$, $H_1$, and $I_1$ are defined
correctly.

\be

  \item $\Dea$ is defined correctly provided $H_1$ is defined
    correctly.

  \item Let $e \in \sE_{L_1}$.  We will show that $H_1(e)$ is a
    correct value by induction on the length of $e$.  If $e \in \sS$,
    then $H_1(e) = H_2(\widehat{\Phi}(e)) = H_2(e)$, which is a member
    of $\Dv$ that is neither the empty set nor an ordered pair.  If $e
    \in \sO_1$, then $\Phi(e) \in \sO_2$ and $H_1(e) =
    H_2(\widehat{\Phi}(e)) = H_2(\Phi(e))$, which is a member of $\Dv$
    that is neither the empty set nor an ordered pair.  If $e = (\,)$,
    then $H_1(e) = H_2(\widehat{\Phi}(e)) = H_2(e)$, which is
    $\emptyset$.  Thus, in each of these three cases, $H_1(e)$ is a
    correct value.  Now if $e = (e_1,\ldots,e_n) \in \sE_{L_1}$ with
    $n \ge 1$, then
    \begin{eqnarray*}
    H_1(e) 
    & = & H_2(\widehat{\Phi}(e)) \\
    & = & H_2(\widehat{\Phi}((e_1,\ldots,e_n))) \\
    & = & H_2((\widehat{\Phi}(e_1), \ldots, \widehat{\Phi}(e_n))) \\
    & = & \seq{H_2(\widehat{\Phi}(e_1)),
      H_2((\widehat{\Phi}(e_2),\ldots,\widehat{\Phi}(e_n)))} \\
    & = & \seq{H_1(e_1),H_1((e_2,\ldots,e_n))},
    \end{eqnarray*}
    which is a correct value by the induction hypothesis.

   \item We will show that $I_1$ satisfies the fifth condition of the
     definition of a structure for $L_1$.  Let $o \in \sO_1$.  Notice
     that $I_2(\Phi(o))$ is an operation of the signature form
     $\theta_2(\Phi(o))$ since $S_2$ is a structure for $L_2$.
     $\theta_1(o) = \theta_2(\Phi(o))$ since $\Phi$ is a translation.
     Hence $I_2(\Phi(o))$ is an operation of the signature form
     $\theta_1(o)$.  It remains only to show that clauses a--q of this
     condition (for built-in operator names) are satisfied.  Assume
     $\sO_1 = \set{o_1,\ldots,o_n}$.  Let $D_{{\rm on},1} =
     \set{H_1(o_1),\ldots,H_1(o_n)}$ and $D_{{\rm on},2} = \set{H_2(o)
       \;|\; o \in \sO_2}$, the sets of representations of operator
     names in $\Dea$ and $\Deb$, respectively.  $D_{{\rm on},1}
     \subseteq D_{{\rm on},2}$ since $H_1(o_i) =
     H_2(\widehat{\Phi}(o_i)) = H_2(\Phi(o_i))$ and $\Phi(o_i) \in
     \sO_2$ for each $i$ with $1 \le i \le n$.

     \be

       \item[] \textbf{Clause a}: $o = \mname{set}$.   Hence
         \begin{eqnarray*}
         I_1(\mname{set})(\,)
         & = & I_2(\Phi(\mname{set}))(\,) \\
         & = & I_2(\mname{set})(\,) \\
         & = & \Dv.
         \end{eqnarray*}
         The first line is by the definition of $I_1$.  The second
         line is by the fact $\Phi$ fixes $L_{\rm ker}$.  And the
         third line is by the fact $S_2$ is a structure for $L_2$.

       \item[] \textbf{Clauses b, e, m--r}: Similar to clause a.

       \item[] \textbf{Clause c}: $o = \mname{op-names}$.  Hence
         \begin{eqnarray*}
         I_1(\mname{op-names})(\,) 
         & = & I_2(\Phi(\mname{op-names}))(\,) \\
         & = & V_{2,\phi}(\widehat{\Phi}((\mname{op-names} ::
             \mname{term})))(\,) \\
         & = & V_{2,\phi}(\widehat{\Phi}(\ell)) \\
         & = & V_{2,\phi}(\set{\synbrack{\Phi(o_1)},\ldots,
             \synbrack{\Phi(o_n)}}) \\
         & = & \set{H_2(\widehat{\Phi}(o_1)),\ldots,H_2(\widehat{\Phi}(o_n))} \\
         & = & \set{H_1(o_1),\ldots,H_1(o_n)} \\
         & = & D_{{\rm on},1}.
         \end{eqnarray*} 
         The first line is by the definition of $I_1$.  The second and
         third lines are by the definitions of $\widehat{\Phi}$ and
         the standard valuation function on operator and operator
         applications, respectively.  The fourth line is by the facts
         $M_2$ is a standard model of $T_2$ and $\Phi$ is normal for
         $T_2$.  The fifth line is by the definition of the standard
         valuation function on quotations and the definition of
         $\widehat{\Phi}$.  The sixth is by the definition of $H_1$.
         And the seventh is by the definition of $D_{{\rm on},1}$.

       \item[] \textbf{Clauses d}: Similar to clause c.

       \item[] \textbf{Clause f}: $o = \mname{expr-op-name}$.  Let $x
         \in \Dc \cup \set{\Undefined}$.
         \begin{eqnarray*}
         &   & I_1(\mname{expr-op-name})(x) \\
         & = & I_2(\Phi(\mname{expr-op-name}))(x) \\
         & = & I_2(\mname{expr-op-name})(x) \\
         & = & \left\{\begin{array}{ll}
                        x & \mbox{if } x \subseteq D_{{\rm on},2} \\
                        \Dc & \mbox{if } x = \Undefined \\
                        \mbox{not } \Dc & \mbox{otherwise.}
                        \end{array}
               \right. \\
         & = & \left\{\begin{array}{ll}
                        x & \mbox{if } x \subseteq D_{{\rm on},1} \\
                        \Dc & \mbox{if } x = \Undefined \\
                        \mbox{not } \Dc & \mbox{otherwise.}
                        \end{array}
               \right.
         \end{eqnarray*} 
         The second line is by the definition of $I_1$.  The third
         line is by the fact $\Phi$ fixes $L_{\rm ker}$.  The fourth
         line is by the fact $S_2$ is a structure for $L_2$.  The
         fifth line is by $D_{{\rm on},1} \subseteq D_{{\rm on},2}$.

       \item[] \textbf{Clauses g--l}: Similar to clause f.  
  
     \ee 
     Therefore, $S_1$ is a structure for $L_1$.

\ee

\bigskip

\noindent \textbf{Part 2} \sglsp This part follows immediately from
part 1 of this lemma.

\bigskip

\noindent \textbf{Part 3} \sglsp Our proof is by induction on the
length of $e$.  There are 13 cases corresponding to the 13 clauses
of the definition of $V$ on proper expressions.

\bi

  \item[] \bsp \textbf{Case 1}: $O = (\mname{op}, o,
    k_1,\ldots,k_{n+1})$ is proper where $o \in \sO_1$.  Then
    \begin{eqnarray*}
    & & V_{1,\phi}(O) \\
    & = & V_{1,\phi}((\mname{op}, o, k_1,\ldots,k_{n+1})) \\
    & = & \mbox{``}I_1(o) \mbox{ restricted by } 
      V_{1,\phi}(k_1),\ldots,V_{1,\phi}(k_{n+1}) \mbox{''} \\
    & = & \mbox{``}I_2(\Phi(o)) \mbox{ restricted by } 
      V_{1,\phi}(k_1),\ldots,V_{1,\phi}(k_{n+1}) \mbox{''} \\
    & = & \mbox{``}I_2(\Phi(o)) \mbox{ restricted by } 
      V_{2,\phi}(\widehat{\Phi}(\crel{k_1})),\ldots,
      V_{2,\phi}(\widehat{\Phi}(\crel{k_{n+1}})) \mbox{''} \\
    & = & V_{2,\phi}((\mname{op}, \Phi(o), \widehat{\Phi}(\crel{k_1}),\ldots,
      \widehat{\Phi}(\crel{k_{n+1}}))) \\
    & = & V_{2,\phi}(\widehat{\Phi}((\mname{op}, o, 
      \crel{k_1},\ldots,\crel{k_{n+1}}))) \\
    & = & V_{2,\phi}(\widehat{\Phi}(\crel{(\mname{op}, o, 
      k_1,\ldots,k_{n+1})})) \\
    & = & V_{2,\phi}(\widehat{\Phi}(\crel{O})).
    \end{eqnarray*}
    The third and sixth lines are by the definition of the standard
    valuation function on operators.  The fourth line is by the
    definitions of $I_1$ and $\widehat{\Phi}$.  The fifth line is by
    the induction hypothesis.  The seventh line is by the definition
    of $\hat{\Phi}$.  And the eighth line is by the definition of a
    relativization. \esp

  \item[] \textbf{Case 2}: $e = (\mname{op-app}, O, e_1,\ldots,e_n)$
    is proper.  Then
    \begin{eqnarray*}
    & & V_{1,\phi}(e) \\
    & = & V_{1,\phi}((\mname{op-app}, O, e_1,\ldots,e_n)) \\
    & = & V_{1,\phi}(O)(V_{1,\phi}(e_1),\ldots,V_{1,\phi}(e_n)) \\
    & = & V_{2,\phi}(\widehat{\Phi}(\crel{O}))
      (V_{2,\phi}(\widehat{\Phi}(\crel{e_1})),\ldots,
      V_{2,\phi}(\widehat{\Phi}(\crel{e_n}))) \\
    & = & V_{2,\phi}(\mname{op-app}, \widehat{\Phi}(\crel{O}), 
      \widehat{\Phi}(\crel{e_1}),\ldots,\widehat{\Phi}(\crel{e_n})) \\
    & = & V_{2,\phi}(\widehat{\Phi}((\mname{op-app}, \crel{O}, 
      \crel{e_1},\ldots,\crel{e_n}))) \\
    & = & V_{2,\phi}(\widehat{\Phi}(\crel{(\mname{op-app}, O, 
      e_1,\ldots,e_n)})) \\
    & = & V_{2,\phi}(\widehat{\Phi}(\crel{e})).
    \end{eqnarray*}
    The third and fifth lines are by the definition of the standard
    valuation function on operator applications.  The fourth line is
    by the induction hypothesis.  The sixth line is by the definition
    of $\hat{\Phi}$.  And the seventh line is by the definition of a
    relativization.

  \item[] \textbf{Cases 3--12}. Similar to case 2.

  \item[] \textbf{Case 13a}: $e = (\mname{eval}, a, \mname{type})$.
    Assume (a) $V_{1,\phi}(a)$ is a member of $\Dea$ that represents a
    type and $H_{1}^{-1}(V_{1,\phi}(a))$ is eval-free.
    Let \[\alpha = (\mname{E}_{{\rm ty},\ell} \cup \mname{E}_{{\rm
        te},\ell} \cup \mname{E}_{{\rm fo},\ell}).\] Then
    \begin{eqnarray*}
    & & V_{1,\phi}(e) \\
    & = & V_{1,\phi}((\mname{eval}, a, \mname{type})) \\
    & = & V_{1,\phi}((\mname{eval}, a, \mname{type}, \ell)) \\
    & = & V_{1,\phi}(\If(a \IsDefApp \alpha, 
      \sembrack{a}_{\rm ty}, \sembrack{\Undefined_{\sf C}}_{\rm ty})) \\
    & = & \mbox{if } V_{1,\phi}(a \IsDefApp \alpha) = \TRUE \\
    & &   \mbox{then } V_{1,\phi}(H_{1}^{-1}(V_{1,\phi}(a))) \\
    & &   \mbox{else } \mname{C} \\
    & = & \mbox{if } V_{2,\phi}(\widehat{\Phi}(\crel{a \IsDefApp \alpha}) 
          = \TRUE \\
    & &   \mbox{then } V_{2,\phi}(\widehat{\Phi}(H_{1}^{-1}
          (V_{2,\phi}(\widehat{\Phi}(\crel{a}))))) \\
    & &   \mbox{else } \mname{C} \\
    & = & \mbox{if } V_{2,\phi}(\widehat{\Phi}(\crel{a \IsDefApp \alpha}) 
          = \TRUE \\
    & &   \mbox{then } V_{2,\phi}(\widehat{\Phi}(\widehat{\Phi}^{-1}(H_{2}^{-1}
          (V_{2,\phi}(\widehat{\Phi}(\crel{a})))))) \\
    & &   \mbox{else } \mname{C} \\
    & = & \mbox{if } V_{2,\phi}(\widehat{\Phi}(\crel{a \IsDefApp \alpha})
          = \TRUE \\
    & &   \mbox{then } V_{2,\phi}(H_{2}^{-1}
          (V_{2,\phi}(\widehat{\Phi}(\crel{a})))) \\
    & &   \mbox{else } \mname{C} \\
    & = & V_{2,\phi}(\If(\widehat{\Phi}(\crel{a \IsDefApp \alpha}),
      \sembrack{\widehat{\Phi}(\crel{a})}_{\rm ty}, 
      \sembrack{\Undefined_{\sf C}}_{\rm ty})) \\
    & = & V_{2,\phi}(\widehat{\Phi}(\If(\crel{a \IsDefApp \alpha},
      \sembrack{\crel{a}}_{\rm ty}, 
      \sembrack{\Undefined_{\sf C}}_{\rm ty})) \\
    & = & V_{2,\phi}(\widehat{\Phi}(\If(\crel{a} \IsDefApp \alpha,
      \sembrack{\crel{a}}_{\rm ty}, 
      \sembrack{\Undefined_{\sf C}}_{\rm ty})) \\
    & = & V_{2,\phi}(\widehat{\Phi}((\mname{eval}, \crel{a}, 
       \mname{type}, \ell))) \\ 
    & = & V_{2,\phi}(\widehat{\Phi}(\crel{(\mname{eval}, a, 
       \mname{type})})) \\ 
    & = & V_{2,\phi}(\widehat{\Phi}(\crel{e}))
    \end{eqnarray*}
    The third line is by Proposition~\ref{prop:rel}.  The fourth and
    twelveth lines are by the definition of relativized evaluation.
    The fifth and ninth lines are by assumption (a) and the definition
    of the standard evaluation function on conditional terms and
    evaluations.  The sixth line is by assumption (a) and the
    inductive hypothesis.  The seventh line is by the definition of
    $H_1$ and the condition $V_{2,\phi}(\widehat{\Phi}(\crel{a
      \IsDefApp \alpha}) = \TRUE$.  The eighth line is a simple
    simplification. The tenth line is by the definition of
    $\hat{\Phi}$.  The eleventh is by the fact that $\alpha$ is
    eval-free.  And the thirteenth line is by the definition of a
    relativization.

    A similar argument works when assumption (a) is false.

  \item[] \textbf{Case 13b}: $e = (\mname{eval}, a, \alpha)$.  Similar
    to case 13a.
 
  \item[] \textbf{Case 13c}: $e = (\mname{eval}, a, \mname{formula})$.
    Similar to case 13a.

\ei

\noindent \textbf{Part 4} \sglsp This part follows immediately from
part 3 of this lemma.
\end{proof}

\begin{clem}\label{lem:modelb}
Suppose $\Phi$ is normal for $T_2$, $M_2$ is a standard model of
$T_2$, and $T_2 \models \widehat{\Phi}(\crel{A})$ for each $A \in
\Gamma_1$.  Then $M_1$ is a standard model of $T_1$.
\end{clem}

\begin{proof}
By part~2 of Lemma~\ref{lem:modela}, $M_1$ is a standard model for
$L_1$.  Let $A \in \Gamma_1$.  Then, by the hypothesis, $T_2 \models
\widehat{\Phi}(\crel{A})$, and thus $M_2 \models
\widehat{\Phi}(\crel{A})$.  By part~4 of Lemma~\ref{lem:modela}, this
implies $M_1 \models A$.  Therefore, $M_1$ is a standard model of
$T_1$.
\end{proof}

\bigskip

$\Phi$ is a \emph{(semantic) interpretation of $T_1$ in $T_2$} if:

\be

  \item $\Phi$ is normal for $T_2$.

  \item $T_1 \models A$ implies $T_2 \models \widehat{\Phi}(\crel{A})$
    for all formulas $A$ of $L_1$.

\ee

\begin{cprop} Suppose $\Phi$ is an interpretation of $T_1$ in $T_2$.  
Then $T_1 \models A$ implies $T_2 \models \widehat{\Phi}(A)$ for all
eval-free formulas $A$ of $L_1$.
\end{cprop}

\begin{cthm}[Relative Satisfiability]  Suppose $\Phi$ is an 
interpretation of $T_1$ in $T_2$.  If there is a standard model of
$T_2$, then there is a standard model of $T_1$.
\end{cthm}

\begin{proof} 
Let $M_2$ be a standard model of $T_2$ and $M_1$ be defined as above.
Suppose $A \in \Gamma_1$.  Then $T_1 \models A$ and so by the
hypothesis $T_2 \models \widehat{\Phi}(\crel{A})$.  Therefore, $M_1$
is a standard model of $T_1$ by Lemma~\ref{lem:modelb}.
\end{proof}

\begin{cthm}[Interpretation] \label{thm:interpretation}
\bsp Suppose $\Phi$ is normal for $T_2$ and $T_2 \models
\widehat{\Phi}(\crel{A})$ for each $A \in \Gamma_1$.  Then $\Phi$ is
an interpretation of $T_1$ in $T_2$. \esp
\end{cthm}

\begin{proof}
Let $A$ be a formula of $L_1$ and suppose $T_1 \models A$.  We need to
only show $T_2 \models \widehat{\Phi}(\crel{A})$ since $\Phi$ is
normal for $T_2$ by hypothesis.  This holds if $T_2$ is unsatisfiable,
so without loss of generality we may assume $M_2$ is a standard model
of $T_2$.  We are done if we show $M_2 \models
\widehat{\Phi}(\crel{A})$.  Let $M_1$ be defined as above. By the
hypothesis and Lemma~\ref{lem:modelb}, $M_1$ is a standard model of
$T_1$.  Hence $M_1 \models A$.  By the hypothesis and part 4 of
Lemma~\ref{lem:modela}, $M_2 \models \widehat{\Phi}(\crel{A})$.
\end{proof}

\bigskip

$\Phi$ is an \emph{anti-interpretation of $T_1$ in $T_2$} if $T_2
\models \widehat{\Phi}(\crel{A})$ implies $T_1 \models A$ for all
formulas $A$ of $L_1$.  $\Phi$ is an \emph{faithful interpretation of
  $T_1$ in $T_2$} if $\Phi$ is both an interpretation and
anti-interpretation of $T_1$ in $T_2$.  $T_2$ is \emph{conservative
  over} $T_1$ if there is a faithful interpretation of $T_1$ in $T_2$.

The following lemma shows that there are theories in Chiron that
cannot be extended ``conservatively''.

\begin{clem}
Let $L_i = (\sO_i,\theta_i)$ be a language of Chiron and $T_i =
(L_i,\Gamma_i)$ be a normal theory of Chiron for $i = 1,2$.  Suppose
$T_1 \le T_2$, $\sO_1 = \set{o_1,\ldots,o_n}$, $\sO_1 \subset \sO_2$,
and $\Gamma_1$ contains the following formula $A$:
\[\ell = \set{\synbrack{o_1},\ldots,\synbrack{o_n}}.\] 
Then $T_2$ is unsatisfiable, i.e., $T_2$ does not have a standard
model.
\end{clem}

\begin{proof}
Since $T_1 \le T_2$, $A \in \Gamma_2$ and so $T_2 \models A$.  Since
$\sO_1 = \set{o_1,\ldots,o_n}$ and $\sO_1 \subset \sO_2$, $T_2 \models
\Neg A$.  Therefore, $T_2$ is unsatisfiable.
\end{proof}

\subsection{Pseudotranslations}

We will define an alternate notion of a translation in which constants
may be translated to expressions other than constants.  Translations
of this kind are often more convenient than regular translations, but
they are not applicable to all expressions.

Let $L_i = (\sO_i,\theta_i)$ be a language of Chiron for $i = 1,2$ and 
\[\sO^{-}_{1} = \sO_1 \setminus
\set{\mname{op-names},\mname{lang}}.\] A \emph{pseudotranslation from
  $L_1$ to $L_2$} is a total mapping \[\Psi : \sO^{-}_{1} \tarrow
\sE_{L_2}\] such that:

\be

  \item For each 0-ary $o \in \sO^{-}_{1}$, either $\Psi(o) \in \sO_2$
    and $\theta_1(o) = \theta_2(\Psi(o))$ or $\Psi(o)$ is a
    semantically closed type, term, or formula of $L_2$ if
    $\theta_1(o)$ is \mname{type}, \mname{term}, or \mname{formula},
    respectively.

  \item For each $n$-ary $o \in O^{-}_{1}$ with $n > 0$, $\Psi(o) \in
    \sO_2$ and $\theta_1(o) = \theta_2(\Psi(o))$.

  \item $\Psi$ is injective on the $\set{o \in \sO^{-}_{1} \;|\;
    \Psi(o) \in \sO_2}$.

\ee
Let $\Delta(\Psi) = \set{o \in \sO^{-}_{1} \;|\; \Psi(o) \not\in \sO_2}$.
(Note that $\Delta(\Psi)$ is a set of $0$-ary operator names that does
not contain \mname{op-names} or \mname{lang}.)  $\Psi$ is extended to
a partial mapping \[\widehat{\Psi}: \sE_{L_1} \tarrow \sE_{L_2}\] by
the following rules:

\be

  \item Let $e \in \sS$.  Then $\widehat{\Psi}(e) = e$.

  \item Let $e \in \sO^{-}_{1} \setminus \Delta(\Psi)$.  Then
    $\widehat{\Psi}(e) = \Psi(e)$.

  \item Let $e \in \Delta(\Psi) \cup
    \set{\mname{op-names},\mname{lang}}$.  Then $\widehat{\Psi}(e)$ is
    undefined.

  \item Let $e \in \sE_{L_1}$ be a constant of the form $(o ::
    \mname{type})(\,)$, $(o :: \mname{C})(\,)$, or $(o ::
    \mname{formula})(\,)$ where $o \in \Delta(\Psi)$.  Then
    $\widehat{\Psi}(e) = \Psi(o)$.

  \item \bsp Let $e$ be a quotation that contains an operator name in
    $\Delta(\Psi) \cup \set{\mname{op-names},\mname{lang}}$. Then
    $\widehat{\Psi}(e)$ is undefined. \esp

  \item Let $e = (e_1,\ldots,e_n) \in \sE_{L_1}$ such that $e$ is
    neither a constant of the form $(o :: \mname{type})(\,)$, $(o ::
    \mname{C})(\,)$, or $(o :: \mname{formula})(\,)$ where $o \in
    \Delta(\Psi)$ nor a quotation that contains an operator name in
    $\Delta(\Psi) \cup \set{\mname{op-names},\mname{lang}}$.  If
    $\widehat{\Psi}(e_1), \ldots, \widehat{\Psi}(e_n)$ are defined,
    then \[\widehat{\Psi}(e) = (\widehat{\Psi}(e_1), \ldots,
    \widehat{\Psi}(e_n)).\] Otherwise $\widehat{\Psi}(e)$ is
    undefined.

\ee 

\begin{cprop}
Let $\Psi$ be a pseudotranslation from $L_1$ to $L_2$.  
\be

  \item If $e$ is a symbol, an operator name, an operator, a type, a
    term, or a formula of $L_1$, then $\widehat{\Psi}(e)$ is a symbol,
    an operator name, an operator, a type, a term, or a formula of
    $L_2$, respectively, provided $\widehat{\Psi}(e)$ is defined.

  \item If $e$ is an expression of $L_1$ such that $\widehat{\Psi}(e)$
    is defined, then the operator names \mname{op-names} and
    \mname{lang} do not occur in $e$ and an operator name in
    $\Delta(\Psi)$ occurs in $e$, if at all, only as the name of a
    constant of the form $(o :: \mname{type})(\,)$, $(o ::
    \mname{C})(\,)$, or $(o :: \mname{formula})(\,)$ that is in not
    within a quotation.

\ee
\end{cprop}

\subsection{Pseudointerpretations}

Let $T_i = (L_i,\Gamma_i)$ be a normal theory of Chiron for $i = 1,2$
and $\Psi$ be a pseudotranslation from $L_1$ to $L_2$.  $\Psi$
\emph{fixes} a language $L=(\sO,\theta) \leq L_1$ if $\Psi(o) = o$ for
all $o \in \sO \setminus \set{\mname{op-names},\mname{lang}}$.  $\Psi$
is \emph{normal} if $\sO_1$ is finite and $\Psi$ fixes $L_{\rm ker}$.
Assume $\Psi$ is normal with $\sO_1 = \set{o_1,\ldots,o_n}$.

An \emph{associate} of $\Psi$ is a pair $(T,\Phi)$ where
$T=(L,\Gamma)$ is a theory of Chiron, $L = (\sO,\theta)$, and $\Phi$
is a (regular) translation from $L_1$ to $L$ such that:
\be

  \item $\Phi(o) = \Psi(o)$ for all $o \in \sO^{-}_{1} \setminus
    \Delta(\Psi)$.

  \item $\sO_2 \cap \set{\Phi(o) \;|\; o \in \Delta(\Psi) \cup
    \set{\mname{op-names}, \mname{lang}}} = \emptyset$.

  \item $\sO = \sO_2 \cup \set{\Phi(o) \;|\; o \in \Delta(\Psi) \cup
    \set{\mname{op-names}, \mname{lang}}}$.

  \item $\theta = \theta_2 \cup \set{\seq{\Phi(o),\theta_1(o)} \;|\; o
    \in \Delta(\Psi) \cup \set{\mname{op-names}, \mname{lang}}}$.

  \item $\Gamma_{\rm ty} = \set{\widehat{\Phi}((o :: \mname{type})(\,))
    \TypeEqual \Psi(o) \;|\; o \in \Delta(\Psi) \mbox{ and }
    \theta(o) = \mname{type}}$.

  \item $\Gamma_{\rm te} = \set{\widehat{\Phi}((o :: \mname{C})(\,))
    \QuasiEqual \Psi(o) \;|\; o \in \Delta(\Psi) \mbox{ and }
    \theta(o) = \mname{term}}$.

  \item $\Gamma_{\rm fo} = \set{\widehat{\Phi}((o ::
    \mname{formula})(\,)) \Iff \Psi(o) \;|\; o \in \Delta(\Psi) \mbox{
      and } \theta(o) = \mname{formula}}$.

  \item $A_1$ is $\widehat{\Phi}(\ell) =
    \set{\synbrack{\Phi(o_1)},\ldots,\synbrack{\Phi(o_n)}}$.

  \item $A_2$ is $\widehat{\Phi}(\mname{L}) \TypeEqual
    \mname{type}(\mname{power}
    (\set{\synbrack{\Phi(o_1)},\ldots,\synbrack{\Phi(o_n)}}))$

  \item $\Gamma = \Gamma_2 \cup \Gamma_{\rm ker}^{L} \cup \Gamma_{\rm
    ty} \cup \Gamma_{\rm te} \cup \Gamma_{\rm fo} \cup \set{A_1,A_2}$.

\ee

\begin{crem} \em
If $T_i = (L_i,\Gamma_i)$ is a normal theory of Chiron for $i = 1,2$
and $\Psi$ is a normal pseudotranslation from $L_1$ to $L_2$, an
associate of $\Psi$ can be easily constructed after an appropriate set
of ``new'' $0$-ary operator names are added to $L_2$ to obtain $L$.
Moreover, two associates $(T,\Phi)$ and $(T',\Phi')$ of $\Psi$ are
identical except that $\Phi$ and $\Phi'$ may map $\Delta(\Psi) \cup
\set{\mname{op-names}, \mname{lang}}$ to different sets of new $0$-ary
operator names.
\end{crem}

Let $L$ be a language of Chiron and $T = (L,\Gamma)$ be a theory of
Chiron.  A formula $A$ of $L$ is \emph{independent of $L$ in $T$} if
$(L',\Gamma) \models A$ for some language $L'$ with $L \le L'$ implies
$(L',\Gamma) \models A$ for all languages $L'$ with $L \le L'$.

\begin{clem} \label{lem:associates}
Let $T_i = (L_i,\Gamma_i)$ be a normal theory of Chiron for $i = 1,2$;
$\Psi$ be a normal pseudotranslation from $L_1$ to $L_2$; and
$(T,\Phi)$ where $T = (L,\Gamma)$ be an associate of $\Psi$.
\be

  \item $T_2 \le T$.

  \item $T$ is normal.

  \item For all formulas $A$ of $L_1$ such that $\widehat{\Psi}(A)$ is
    defined, \[T \models \widehat{\Psi}(A) \Iff \widehat{\Phi}(A).\]

  \item For all formulas $A$ of $L_1$ such that $\widehat{\Psi}(A)$ is
    defined and independent of $L_2$ in $T_2$, \[T_2 \models
    \widehat{\Psi}(A) \mbox{ implies } T \models \widehat{\Phi}(A).\]

\iffalse
  \item For all formulas $A$ of $L_1$ such that $\widehat{\Psi}(A)$ is
    defined, \[T \models \widehat{\Phi}(A) \mbox{ implies } T_2
    \models \widehat{\Psi}(A).\]
\fi

\ee
\end{clem}

\begin{proof}

\bigskip

\noindent \textbf{Part 1} \sglsp Obvious.

\bigskip

\noindent \textbf{Part 2} \sglsp Follows immediately from part 1 and
the fact that $T_2$ is normal.

\bigskip

\noindent \textbf{Part 3} \sglsp By the definition of $\widehat{\Psi}$
and the construction of $\Gamma$.

\bigskip

\noindent \textbf{Part 4} \sglsp Let $A$ be a formula of $L_1$ such
that $\widehat{\Psi}(A)$ is defined and independent of $L_2$ in $T_2$
and $T_2 \models \widehat{\Psi}(A)$.  We must show $T \models
\widehat{\Phi}(A)$.  Let $T'_2 = (L,\Gamma_2)$.  $T'_2 \models
\widehat{\Psi}(A)$ since $\widehat{\Psi}(A)$ is independent of $L_2$
in $T_2$.  By part 1, $T_2 \le T'_2 \le T$, and so, by
Lemma~\ref{lem:theory-ext}, $T \models \widehat{\Psi}(A)$.  By part 3,
$T \models \widehat{\Psi}(A) \Iff \widehat{\Phi}(A)$.  Therefore, $T
\models \widehat{\Phi}(A)$.
\end{proof}

\bigskip

A pseudotranslation $\Psi$ from $L_1$ to $L_2$ is a \emph{(semantic)
  pseudointerpretation of $T_1$ in $T_2$} if:

\be

  \item $\Psi$ is normal.

  \item For each $A \in \Gamma_1 \setminus \Gamma_{\rm ker}^{L_1}$,
    $A$ is eval-free, $\widehat{\Psi}(A)$ is defined and independent
    of $L_2$ in $T_2$, and $T_2 \models \widehat{\Psi}(A)$.

\ee

\begin{cthm} \label{thm:associates}
Let $T_i = (L_i,\Gamma_i)$ be a normal theory of Chiron for $i = 1,2$;
$\Psi$ be a pseudointerpretation of $T_1$ in $T_2$; and $(T,\Phi)$ be
an associate of $\Psi$.  
\be

  \item $\Phi$ is an interpretation of $T_1$ in $T$.

  \item For all formulas $A$ of $L_1$ such that $\widehat{\Psi}(A)$ is
    defined and independent of $L_2$ in $T_2$, \[T_1 \models A \mbox{
      implies } T_2 \models \widehat{\Psi}(A).\]

\ee
\end{cthm} 

\begin{proof} 

\bigskip

\noindent \textbf{Part 1} \sglsp First, we must show that $T$ is a
normal theory.  $T$ is normal since $\Gamma_{\rm ker}^{L} \le \Gamma$
by the construction of $T$.

Second, we must show that $\Phi$ is normal for $T$.  Since $\Psi$ is a
pseudointerpretation, $\Psi$ is normal and hence $\sO_1$ is finite and
$\Psi$ fixes $L_{\rm ker}$.  The latter implies $\Phi$ fixes $L_{\rm
  ker}$.  The last conditions required for $\Phi$ to be normal for $T$
are met since the axioms of $T$ include the formulas $A_1$ and $A_2$
from the definition of an associate of a pseudotranslation.

By Theorem~\ref{thm:interpretation}, it remains for us to show that $T
\models \widehat{\Phi}(\crel{A})$ for each $A \in \Gamma_1$.  Let $A
\in \Gamma_1 \setminus \Gamma_{\rm ker}^{L_1}$.  Then $T \models
\widehat{\Phi}(\crel{A})$ by Lemma~\ref{lem:normal-trans} and the fact
that $\Phi$ is normal for $T$.  Now let $A \in \Gamma_1 \setminus
\Gamma_{\rm ker}^{L_1}$.  Then, since $\Psi$ is a
pseudointerpretation, (a) $A$ is eval-free, (b) $\widehat{\Psi}(A)$ is
defined and independent of $L_2$ in $T_2$, and (c) $T_2 \models
\widehat{\Psi}(A)$.  (a) implies (d) $A = \crel{A}$.  By part 4 of
Lemma~\ref{lem:associates}, (b) and (c) imply (e) $T \models
\widehat{\Phi}(A)$.  And (d) and (e) implies $T \models
\widehat{\Phi}(\crel{A})$.

\bigskip

\noindent \textbf{Part 2} \sglsp Let $A$ be a formula of $L_1$ such
that (a) $\widehat{\Psi}(A)$ is defined, (b) $\widehat{\Psi}(A)$ is
independent of $L_2$ in $T_2$, and (c)~$T_1 \models A$.  We must show
$T_2 \models \widehat{\Psi}(A)$.  (c)~implies (d) $T \models
\widehat{\Phi}(A)$ since $\Phi$ is an interpretation.  By part 3 of
Lemma~\ref{lem:associates}, (d) implies $T \models \widehat{\Psi}(A)$.
(a) and (d) imply (e) $T'_2 \models \widehat{\Psi}(A)$ where $T'_2 =
(L,\Gamma_2)$.  Finally, (b) and (e) imply $T_2 \models
\widehat{\Psi}(A)$.
\end{proof}

\begin{crem} \em
By virtue of Theorem~\ref{thm:associates} a pseudointerpretation can
be viewed as a regular interpretation in a more convenient form.
\end{crem}

\section{Conclusion}\label{sec:conclusion}

In this paper we have presented the syntax and semantics of a set
theory named Chiron that is intended to be a practical,
general-purpose logic for mechanizing mathematics.  Several operator
definitions and simple examples are given that illustrate Chiron's
practical expressivity, especially its facility for reasoning about
the syntax of expressions.  A proof system for Chiron is presented
that is intended to be a system test of Chiron's definition and a
reference system for other, more practical, proof systems for Chiron.
The proof system is proved to be sound and also complete in a
restricted sense.  And a notion of an interpretation of one theory in
another is defined.

This paper is a first step in a long-range research program to design,
analyze, and implement Chiron.  In the future we plan to:
\be

  \item Design a practical proof system for Chiron.

  \item Implement Chiron and its proof system.

  \item Develop a series of applications to demonstrate Chiron's reach
  and level of effectiveness.  As a first step, we have shown how
  biform theories can be formalized in Chiron~\cite{Farmer07b}.  A
  \emph{biform theory} is a theory in which both formulas and
  algorithms can serve as
  axioms~\cite{Farmer07b,FarmerMohrenschildt03}.

\ee

\addcontentsline{toc}{section}{Acknowledgments}
\section*{Acknowledgments}

The author is grateful to Marc Bender and Jacques Carette for many
valuable discussions on the design and use of Chiron.  Over the course
of these discussions, Dr.~Carette convinced the author that Chiron
needs to include a powerful facility for reasoning about the syntax of
expressions.  The author is also grateful to Russell O'Connor for
reading over the paper carefully and pointing out several mistakes.

%\addcontentsline{toc}{section}{Appendix: Alternate Semantics}

\appendix

\section{Appendix: Alternate Semantics}\label{sec:alt-semantics}

\bsp This appendix presents two alternate semantics for Chiron based
on S.~Kripke's framework for defining semantics with \emph{truth-value
gaps} which is described in his famous paper \emph{Outline of a Theory
of Truth}~\cite{Kripke75}.  Both semantics use \emph{value gaps} for
types and terms as well as for formulas.  The first defines the value
gaps according to weak Kleene logic~\cite{Kleene52}, while the second
defines the values gaps according to a valuation scheme based on
B.~van~Fraassen's notion of a
\emph{supervaluation}~\cite{vanFraassen66} that Kripke describes
in~\cite[p.~711]{Kripke75}. \esp

%\addcontentsline{toc}{subsection}{Valuations}

\subsection{Valuations}

The notion of a valuation for a structure was defined in
subsection~\ref{subsec:val}.  Fix a structure $S$ for $L$.  Let
$\mname{val}(S)$ be the collection of valuations for $S$.  Given $U,V
\in \mname{val}(S)$, $U$ is a \emph{subvaluation} of $V$, written $U
\subvaluation V$, if, for all $e \in \sE_L$ and $\phi \in
\mname{assign}(S)$, $U_\phi(e)$ is defined implies $U_\phi(e) =
V_\phi(e)$.  A \emph{valuation functional} for $S$ is a mapping from
$\mname{val}(S)$ into $\mname{val}(S)$.  Let $\Psi$ be a valuation
functional for $S$.  A \emph{fixed point} of $\Psi$ is a $V \in
\mname{val}(S)$ such that $\Psi(V) = V$.  $\Psi$ is \emph{monotone} if
$U \subvaluation V$ implies $\Psi(U) \subvaluation \Psi(V)$ for all
$U,V \in \mname{val}(S)$.

\begin{cthm}\label{thm:fixed-point}
Let $\Psi$ be a monotone valuation functional for $S$.  Then $\Psi$
has a fixed point.
\end{cthm}

\begin{proof}
The construction of a fixed point of $\Psi$ is similar to the
construction of the fixed point Kripke gives
in~\cite[pp.~703--705]{Kripke75}.
\end{proof}

\bigskip

$\Psi_{1}^{S}$ is the valuation functional for $S$ defined by the
following rules where $V \in \mname{val}(S)$ and $V' =
\Psi_{1}^{S}(V)$.  There is a rule for the category of improper
expressions and a rule for each of the 13 categories of proper
expressions.  Note that only part (a) of the rule 14 (the rule for
evaluation) makes use of $V$.  $\Psi_{1}^{S}$ defines value gaps
according to the weak Kleene logic valuation scheme in which a proper
expression is denoting only if all of its proper subexpressions are
also denoting.

\be

  \item Let $e \in \sE_L$ be improper.  Then $V'_{\phi}(e)$ is
    undefined.

  \item Let $O = (\mname{op}, s, k_1,\ldots,k_n,k_{n+1})$ be proper.

  \be

    \item Let $V'_{\phi}(k_i)$ be defined for all $i$ with $1 \le i
      \le n+1$ and $\ctype{L}{k_i}$.  Then $I(o)$ is an $n$-ary
      operation in $\Do$ from $D_1 \times \cdots \times D_n$ into
      $D_{n+1}$.  $V'_{\phi}(O)$ is the $n$-ary operation in $\Do$
      from $D_1 \times \cdots \times D_n$ into $D_{n+1}$ defined as
      follows.  Let $(d_1,\ldots,d_n) \in D_1 \times \cdots \times
      D_n$ and $d = I(o)(d_1,\ldots,d_n)$.  If $d_i$ is in
      $V'_{\phi}(k_i)$ or $d_i = \Undefined$ for all $i$ such that $1
      \le i \le n$ and $\ctype{L}{k_i}$ and $d$ is in
      $V'_{\phi}(k_{n+1})$ or $d = \Undefined$ when
      $\ctype{L}{k_{n+1}}$, then $V'_{\phi}(O)(d_1,\ldots,d_n) = d$.
      Otherwise, $V'_{\phi}(O)(d_1,\ldots,d_n)$ is $\Dc$ if $k_{n+1} =
      \mname{type}$, $\Undefined$ if $\ctype{L}{k_{n+1}}$, and
      $\FALSE$ if $k_{n+1} = \mname{formula}$.

    \item Let $V'_{\phi}(k_i)$ be undefined for some $i$ such that $1
      \le i \le n+1$ and $\ctype{L}{k_i}$. Then $V'_{\phi}(e)$ is
      undefined.

  \ee

  \item Let $e = (\mname{op-app}, O, e_1,\ldots,e_n)$ be proper.

  \be

    \item Let $V'_{\phi}(O), V'_{\phi}(e_1), \ldots, V'_{\phi}(e_n)$
      be defined.  Then \[V'_{\phi}(e) = V'_{\phi}(O)(V'_{\phi}(e_1),
      \ldots, V'_{\phi}(e_n)).\]

    \item Let one of $V'(O), V'_{\phi}(e_1), \ldots, V'_{\phi}(e_n)$
      be undefined.  Then $V'_{\phi}(e)$ is undefined.

  \ee

\iffalse

  \item Let $O = (\mname{op}, s, k_1,\ldots,k_n,k_{n+1})$ be proper.
  Then $V'_{\phi}(O) = I(O)$.

  \item Let $e = (\mname{op-app}, O, e_1,\ldots,e_n)$ be proper where
  \[O = (\mname{op}, s, k_1,\ldots,k_n,k_{n+1}).\]

  \be

    \item Let $V'_{\phi}(k_i)$ be defined for all $i$ with $1 \le i
    \le n$ and $\ctype{L}{k_i}$, and let $V'_{\phi}(e_1), \ldots,
    V'_{\phi}(e_n)$ be defined.  If $V'_{\phi}(e_i)$ is in
    $V'_{\phi}(k_i)$ or $V'_{\phi}(e_i) = \Undefined$ for all $i$ such
    that $1 \le i \le n$ and $\ctype{L}{k_i}$, then
    \[V'_{\phi}(e) = V'_{\phi}(O)(V'_{\phi}(e_1),\ldots,V'_{\phi}(e_n)).\] 
    Otherwise $V'_{\phi}(e)$ is $\Dc$ if $k = \mname{type}$,
    $\Undefined$ if $\ctype{L}{k}$, and $\FALSE$ if $k =
    \mname{formula}$.

    \item Let $V'_{\phi}(k_i)$ be undefined for some $i$ with $1 \le i
    \le n$ and $\ctype{L}{k_i}$, or let $V'_{\phi}(e_i)$ be undefined for
    some $i$ with $1 \le i \le n$. Then $V'_{\phi}(e)$ is undefined.

  \ee

\fi

  \item Let $a = (\mname{var}, x, \alpha)$ be proper.
  
  \be

    \item Let $V'_{\phi}(\alpha)$ be defined.  If $\phi(x)$ is in
    $V'_{\phi}(\alpha)$, then $V'_{\phi}(a) = \phi(x)$.  Otherwise
    $V'_{\phi}(a) = \Undefined$.

    \item Let $V'_{\phi}(\alpha)$ be undefined.  Then
    $V'_{\phi}(a)$ is undefined.

  \ee

  \item Let $\beta = (\mname{type-app},\alpha,a)$ be proper.

  \be

    \item Let $V'_{\phi}(\alpha)$ and $V'_{\phi}(a)$ be defined.  If
    $V'_{\phi}(a)\not= \Undefined$, then $V'_{\phi}(\beta) =
    V'_{\phi}(\alpha)[V'_{\phi}(a)]$.  Otherwise $V'_{\phi}(\beta) =
    \Dc$.

    \item Let $V'_{\phi}(\alpha)$ or $V'_{\phi}(a)$ be undefined.  Then
    $V'_{\phi}(\beta)$ is undefined.

  \ee

  \item Let $\gamma =
  (\mname{dep-fun-type},(\mname{var},x,\alpha),\beta)$ be proper.

  \be

    \item Let $V'_{\phi}(\alpha)$ be defined, and let $V'_{\phi[x
        \mapsto d]}(\beta)$ be defined for all sets $d$ in
      $V'_{\phi}(\alpha)$.  Then $V'_{\phi}(\gamma)$ is the superclass
      of all $g$ in $\Df$ such that, for all $d$ in $\Dv$, if $g(d)$
      is defined, then $d$ is in $V'_{\phi}(\alpha)$ and $g(d)$ is in
      $V'_{\phi[x \mapsto d]}(\beta)$.

    \item Let $V'_{\phi}(\alpha)$ be undefined, or let $V'_{\phi[x
    \mapsto d]}(\beta)$ be undefined for some set $d$ in
    $V'_{\phi}(\alpha)$.  Then $V'_{\phi}(\gamma)$ is undefined.

  \ee

  \item Let $b =(\mname{fun-app},f,a)$ be proper.

  \be

    \item Let $V'_{\phi}(f)$ and $V'_{\phi}(a)$ be defined.  If
    $V'_{\phi}(f) \not= \Undefined$ and $V'_{\phi}(a) \not= \Undefined$,
    then $V'_{\phi}(b) = V'_{\phi}(f)(V'_{\phi}(a))$.  Otherwise
    $V'_{\phi}(b) = \Undefined$.

    \item Let $V'_{\phi}(f)$ or $V'_{\phi}(a)$ be undefined.  Then
    $V'_{\phi}(b)$ is undefined.

  \ee

  \item Let $f = (\mname{fun-abs},(\mname{var},x,\alpha),b)$ be
  proper.

  \be

    \item Let $V'_{\phi}(\alpha)$ be defined, and let $V'_{\phi[x
    \mapsto d]}(b)$ be defined for all sets $d$ in
    $V'_{\phi}(\alpha)$.  If
    \begin{eqnarray*}
    g & = & \{\seq{d,d'} \; | \; d \mbox{ is a set in }
    V'_{\phi}(\alpha) \mbox{ and } \\ & & \hspace{9ex} d' = V'_{\phi[x
    \mapsto d]}(b) \mbox{ is a set}\}
    \end{eqnarray*}
    is in $\Df$, then $V'_{\phi}(f) = g$.  Otherwise $V'_{\phi}(f) =
    \Undefined$.

    \item Let $V'_{\phi}(\alpha)$ be undefined, or let $V'_{\phi[x
    \mapsto d]}(b)$ be undefined for some set $d$ in
    $V'_{\phi}(\alpha)$.  Then $V'_{\phi}(f)$ is undefined.

  \ee

  \item Let $a = (\mname{if},A,b,c)$ be proper.

  \be

    \item Let $V'_{\phi}(A),V'_{\phi}(b),V'_{\phi}(c)$ be defined.  If
    $V'_{\phi}(A) = \TRUE$, then $V'_{\phi}(a) = V'_{\phi}(b)$.  Otherwise
    $V'_{\phi}(a) = V'_{\phi}(c)$.

    \item Let one of $V'_{\phi}(A),V'_{\phi}(b),V'_{\phi}(c)$ be
    undefined.  Then $V'_{\phi}(a)$ is undefined.

  \ee

  \item Let $A = (\mname{exists},(\mname{var}, x, \alpha),B)$ be
  proper.

  \be

    \item Let $V'_{\phi}(\alpha)$ be defined, and let $V'_{\phi[x
    \mapsto d]}(B)$ be defined for all $d$ in $V'_{\phi}(\alpha)$.  If
    there is some $d$ in $V'_{\phi}(\alpha)$ such that $V'_{\phi[x
    \mapsto d]}(B) = \TRUE$, then $V'_{\phi}(A) = \TRUE$.  Otherwise,
    $V'_{\phi}(A) = \FALSE$.

    \item Let $V'_{\phi}(\alpha)$ be undefined, or let $V'_{\phi[x
    \mapsto d]}(B)$ be undefined for some $d$ in $V'_{\phi}(\alpha)$.
    Then $V'_{\phi}(A)$ is undefined.

  \ee

  \item Let $a = (\mname{def-des},(\mname{var}, x, \alpha),B)$ be
  proper.

  \be

    \item \bsp Let $V'_{\phi}(\alpha)$ be defined, and let $V'_{\phi[x
    \mapsto d]}(B)$ be defined for all $d$ in $V'_{\phi}(\alpha)$.  If
    there is a unique $d$ in $V'_{\phi}(\alpha)$ such that $V'_{\phi[x
    \mapsto d]}(B) = \TRUE$, then $V'_{\phi}(a) = d$.  Otherwise,
    $V'_{\phi}(a) = \Undefined$. \esp

    \item Let $V'_{\phi}(\alpha)$ be undefined, or let $V'_{\phi[x
    \mapsto d]}(B)$ be undefined for some $d$ in $V'_{\phi}(\alpha)$.
    Then $V'_{\phi}(a)$ is undefined.

  \ee

  \item Let $a = (\mname{indef-des},(\mname{var}, x, \alpha),B)$ be
  proper.

  \be

    \item \bsp Let $V'_{\phi}(\alpha)$ be defined, and let $V'_{\phi[x
    \mapsto d]}(B)$ be defined for all $d$ in $V'_{\phi}(\alpha)$.  If
    there is some $d$ in $V'_{\phi}(\alpha)$ such that $V'_{\phi[x
    \mapsto d]}(B) = \TRUE$, then $V'_{\phi}(a) = \xi(\Sigma)$ where
    $\Sigma$ is the superclass of all $d$ in $V'_{\phi}(\alpha)$ such
    that $V'_{\phi[x \mapsto d]}(B) = \TRUE$.  Otherwise,
    $V'_{\phi}(a) = \Undefined$. \esp

    \item Let $V'_{\phi}(\alpha)$ be undefined, or let $V'_{\phi[x
    \mapsto d]}(B)$ be undefined for some $d$ in $V'_{\phi}(\alpha)$.
    Then $V'_{\phi}(a)$ is undefined.

  \ee

  \item Let $a = (\mname{quote},e)$ be proper.  Then $V'_{\phi}(a) =
  H(e)$.

  \item Let $b = (\mname{eval},a, k)$ be proper.

  \be

    \item Let $V'_{\phi}(a)$ be defined and $V'_{\phi}(k)$ be defined
    if $\ctype{L}{k}$.

    \be

      \item Let $V'_{\phi}(a)$ be in $\Dty$ and $k = \mname{type}$,
      $V'_{\phi}(a)$ be in $\Dte$ and $\ctype{L}{k}$, or $V'_{\phi}(a)$
      be in $\Dfo$ and $k = \mname{formula}$.

      \be

        \item \bsp Let $V_{\phi}(H^{-1}(V'_{\phi}(a)))$ be defined.
        If $k \in \set{\mname{type},\mname{formula}}$ or $\ctype{L}{k}$
        and $V_{\phi}(H^{-1}(V'_{\phi}(a)))$ is in $V'_{\phi}(k)$,
        then $V'_{\phi}(b) = V_{\phi}(H^{-1}(V'_{\phi}(a)))$.
        Otherwise $V'_{\phi}(b)$ is $\Undefined$. \esp

        \item Let $V_{\phi}(H^{-1}(V'_{\phi}(a)))$ be undefined.  Then
        $V'_{\phi}(b)$ is undefined.

      \ee

      \item Let $V'_{\phi}(a)$ not be in $\Dty$ or $k \not=
      \mname{type}$, $V'_{\phi}(a)$ not be in $\Dte$ or not
      $\ctype{L}{k}$, and $V'_{\phi}(a)$ not be in $\Dfo$ or $k \not=
      \mname{formula}$.  Then $V'_{\phi}(b)$ is $\Dc$ if $k =
      \mname{type}$, $\Undefined$ if $\ctype{L}{k}$, and $\FALSE$ if $k =
      \mname{formula}$.

    \ee

    \item Let $V'_{\phi}(a)$ be undefined or $V'_{\phi}(k)$ be
    undefined if $\ctype{L}{k}$.  Then $V'_{\phi}(b)$ is undefined.

  \ee

\ee

\begin{clem}\label{lem:monotone-1}
$\Psi_{1}^{S}$ is monotone.
\end{clem}

\begin{proof}
Let $U,V \in \mname{val}(S)$ such that $U \subvaluation V$.  Assume
$U'_\phi$ and $V'_\phi$ mean $(\Psi_{1}^{S}(U))_\phi$ and
$(\Psi_{1}^{S}(V))_\phi$, respectively.  We must show that, for all $e
\in \sE_L$ and $\phi \in \mname{assign}(S)$, if $U'_\phi(e)$ is
defined, then $U'_\phi(e) = V'_\phi(e)$.  Our proof will be by
induction on the number of symbols in $e$.

There are three cases:
\be

  \item \emph{$e$ is improper.}  Then $U'_\phi(e)$ is undefined by the
  definition of $\Psi_{1}^{S}$.

  \item \emph{$e = (\mname{eval},a, k)$ is proper.} If either
  $U'_\phi(a)$ or $U'_\phi(k)$ is undefined, then $U'_\phi(e)$ is
  undefined.  So assume $U'_\phi(a)$ and $U'_\phi(k)$ are defined.  By
  the induction hypothesis, $U'_\phi(a) = V'_\phi(a)$ and $U'_\phi(k)
  = V'_\phi(k)$.  Assume $U'_\phi(e)$ is defined.  By the definition
  of $\Psi_{1}^{S}$, there are two subcases:

  \be

    \item \emph{For some $e_1,e_2 \in \sE_L$, $U'_\phi(e) =
      U_\phi(e_1)$ and $V'_\phi(e) = V_\phi(e_2)$.}  Since $U'_\phi(a)
      = V'_\phi(a)$, $e_1 = e_2$, and since $U \subvaluation V$,
      $U_\phi(e_1) = V_\phi(e_2)$.  Hence, $U'_\phi(e) = V'_\phi(e)$.

    \item $U'_\phi(e)$ \emph{and} $V'_\phi(e)$ \emph{both equal} $\Dc$
    \emph{if} $k = \mname{type}$, $\Undefined$ \emph{if} $\ctype{L}{k}$,
    and $\FALSE$ \emph{if} $k = \mname{formula}$.  Hence, $U'_\phi(e)
    = V'_\phi(e)$.

  \ee

  \item \emph{$e$ is proper but not an evaluation.}  Assume
  $U'_\phi(e)$ is defined.  Then $U'_\phi(e')$ is defined for each
  subexpression $e'$ of $e$.  By the induction hypothesis,
  $U'_\phi(e') = V'_\phi(e')$ for each such subexpression $e'$ of $e$.
  Hence, $U'_\phi(e) = V'_\phi(e)$.

\ee
\end{proof}

\begin{ccor}\label{cor:fixed-point-1}
$\Psi_{1}^{S}$ has a fixed point.
\end{ccor}

\begin{proof}
By Lemma~\ref{lem:monotone-1}, $\Psi_{1}^{S}$ is monotone.  Therefore,
by Theorem~\ref{thm:fixed-point}, $\Psi_{1}^{S}$ has a fixed point.
\end{proof}

\bigskip

$\Psi_{2}^{S}$ is the valuation functional for $S$ defined by the
following three rules where $V \in \mname{val}(S)$ and $V' =
\Psi_{2}^{S}(V)$.  $\Psi_{2}^{S}$ defines value gaps according to the
supervaluation scheme.

\be

  \item Let $e \in \sE_L$ be improper.  Then $V'_{\phi}(e)$ is
    undefined.

  \item Let $e \in \sE_L$ be proper but not an evaluation.  If there is
    a value $d$ such that, for all total valuations $V^{\ast}$ with $V
    \subvaluation V^{\ast}$, $(\Psi_{1}^{S}(V^{\ast}))_{\phi}(e) = d$,
    then $V'_{\phi}(e) = d$.  Otherwise $V'_{\phi}(e)$ is undefined.

  \item Let $e \in \sE_L$ be proper with $e = (\mname{eval},a, k)$.
    This rule is exactly the same as the $\Psi_{1}^{S}$ rule for
    evaluations.

\ee

\begin{clem}\label{lem:monotone-2}
$\Psi_{2}^{S}$ is monotone.
\end{clem}

\begin{proof}
The proof is exactly the same as the proof of
Lemma~\ref{lem:monotone-1} except for the argument for the third case:
\bi

  \item[3.] \emph{$e$ is proper but not an evaluation.}  Assume
  $U'_\phi(e)$ is defined.  Then there is a value $d$ such that, for
  all total valuations $U^{\ast}$ with $U \subvaluation U^{\ast}$,
  $U^{\ast}_{\phi}(e) = d$.  Since $U \subvaluation V$, it follows
  that, for all total valuations $V^{\ast}$ with $V \subvaluation
  V^{\ast}$, $V^{\ast}_{\phi}(e) = d$.  Hence, $U'_\phi(e) =
  V'_\phi(e)$.

\ei
\end{proof}

\begin{ccor}\label{cor:fixed-point-2}
$\Psi_{2}^{S}$ has a fixed point.
\end{ccor}

\begin{proof}
By Lemma~\ref{lem:monotone-2}, $\Psi_{2}^{S}$ is monotone.  Therefore,
by Theorem~\ref{thm:fixed-point}, $\Psi_{2}^{S}$ has a fixed point.
\end{proof}

%\addcontentsline{toc}{subsection}{Models}

\subsection{Models}

The valuation functionals $\Psi_{1}^{S}$ and $\Psi_{2}^{S}$ define two
semantics, which we will refer to as the \emph{weak Kleene semantics}
and the \emph{supervaluation semantics}, respectively.  Clearly, the
supervaluation semantics allows more expressions to be denoting than
the weak Kleene semantics.

A \emph{weak Kleene model} for $L$ is a model $M=(S,V)$ where $S$ is a
structure for $L$ and $V$ is a valuation for $S$ that is a fixed point
of $\Psi_{1}^{S}$.  A \emph{supervaluation model} for $L$ is a model
$M=(S,V)$ where $S$ is a structure for $L$ and $V$ is a valuation for
$S$ that is a fixed point of $\Psi_{2}^{S}$.

\begin{cthm}
Let $L$ be a language of Chiron.  For each structure $S$ for $L$ there
exists a weak Kleene model and a supervaluation model for $L$.
\end{cthm}

\begin{proof}
Let $L$ be a language of Chiron and $S$ be a structure for $L$.  By
Corollary~\ref{cor:fixed-point-1}, $\Psi_{1}^{S}$ has a fixed point
$V_1$.  Similarly, by Corollary~\ref{cor:fixed-point-2},
$\Psi_{2}^{S}$ has a fixed point $V_2$.  Therefore, $M=(S,V_1)$ is a
weak Kleene model for $L$, and $M=(S,V_2)$ is a supervaluation model
for $L$.
\end{proof}

\bigskip

The weak Kleene semantics defined by $\Psi_{1}^{S}$ is ``strict'' in
the sense that, if any proper subexpression $e$ of a proper expression
$e'$ is nondenoting, then $e'$ itself is nondenoting.  The
supervaluation semantics defined by $\Psi_{2}^{S}$ is not strict in
this sense.  For example, the value of an application of the operator
\[(\mname{op}, \mname{or}, \mname{formula}, \mname{forumla}, 
\mname{formula})\] to a pair of formulas $(A,B)$ is $\TRUE$ if the value 
of $A$ is $\TRUE$ and $B$ is nondenoting or vice versa.

%\addcontentsline{toc}{subsection}{Discussion}

\subsection{Discussion}

There are various Kripke-style value-gap semantics for Chiron; the
weak Kleene and supervaluation semantics are just two examples.  The
weak Kleene semantics is a conservative example: every expression that
could be nondenoting is nondenoting.  On the other hand, the
supervaluation semantics is much more liberal: many expressions that
are nondenoting in the weak Kleene semantics are denoting in the
supervaluation semantics.

It is not possible to define a denoting formula checker in any
Kripke-style value-gap semantics.  If the operator $O =
(o::\mname{formula},\mname{formula})$ were a denoting formula checker,
then $O(e)$ would be true whenever $e$ is denoting and false whenever
$e$ is nondenoting.  However, such an operator breaks the monotonicity
lemmas proved above because, if $U'_\phi(e)$ is undefined but
$V'_\phi(e)$ is defined, then $U'_\phi(O(e)) = \FALSE \not= \TRUE =
V'_\phi(O(e))$.  Similarly, it is not possible to define denoting type
and term checkers in a Kripke-style value-gap semantics.

The lack of available checkers for denoting types, terms, and formulas
makes reasoning in Kripke-style value-gap semantics very difficult.
For example, consider the formalization of the law of excluded middle
given in subsection~\ref{subsec:lem}: \[\ForallApp e \mcolon
\mname{E}_{\rm fo} \mdot \sembrack{e} \Or \Neg\sembrack{e}.\]

This formula is nondenoting in the weak Kleene semantics because, if
$e$ represents a nondenoting formula, then $\sembrack{e} \Or
\Neg\sembrack{e}$ is nondenoting.  Since this formula is nondenoting,
we cannot use it as a basis for proof by cases.

This formula is true in the supervaluations semantics because, if $e$
represents a nondenoting formula, then $\sembrack{e} \Or
\Neg\sembrack{e}$ is true because $\sembrack{e} \Or \Neg\sembrack{e}$
is true no matter what value is assigned to $\sembrack{e}$.  Even
though this formula is true, we cannot use it as a basis for proof by
cases because, if $e$ is the liar paradox, we can derive a
contradiction from either $\sembrack{e}$ or $\Neg\sembrack{e}$.

We expect that reasoning in the official semantics for Chiron will be
much easier than in any Kripke-style value-gap semantics for Chiron.

\section{Appendix: An Expanded Definition of a Proper Expression}
\label{sec:expaned-p-expr}

We give in this appendix an expanded definition of a proper expression
with 25 proper expression categories.  There are modified categories
for operator applications, conditional terms, and quotations and new
categories for constants (applications of 0-ary operators), four
variable binders given above as notational definitions, finite
sets of sets, finite lists of sets, and dependent ordered pairs

\bsp The following symbols are added to the the set $\sK$ of key
words: \mname{con}, \mname{uni-exists}, \mname{forall},
\mname{set-cons}, \mname{list-cons}, \mname{class-abs},
\mname{dep-type-prod}, \mname{left-type}, \mname{right-type},
\mname{dep-ord-pair}, \mname{dep-head}, and
\mname{dep-tail}.  The following formation rules define the
expanded set of proper expression categories: \esp

\bi

  \item[]\textbf{P-Expr-1 (Operator)}\vspace*{-2mm}
  \[\frac{o \in \sO, \ckind{L}{k_1},\ldots,\ckind{L}{k_{n+1}}}
  {\cop{L}{(\mname{op}, o, k_1,\ldots,k_{n+1})}}\] where $n \ge 0$;
  $\theta(o) = s_1,\ldots,s_{n+1}$; and $k_i=s_i=\mname{type}$,
  $\ctype{L}{k_i}$ and $s_i = \mname{term}$, or $k_i=
  s_i=\mname{formula}$ for all $i$ with $1 \le i \le n+1$.

  \item[]\textbf{P-Expr-2 (Operator application)}\vspace*{-2mm}
  \[\frac{\cop{L}{(\mname{op}, o, k_1,\ldots,k_{n+1})},
  \cexpr{L}{e_1},\ldots,\cexpr{L}{e_n}}
  {\cpexpr{L}{(\mname{op-app}, (\mname{op}, o, k_1,\ldots,k_{n+1}), 
  e_1,\ldots,e_n)}{k_{n+1}}}\] 
  \bsp where $n \ge 1$ and ($k_i=\mname{type}$ and $\ctype{L}{e_i}$), 
  ($\ctype{L}{k_i}$ and $\cterm{L}{e_i}$), or
  ($k_i=\mname{formula}$ and $\cform{L}{e_i}$)
  for all $i$ with $1 \le i \le n$. \esp

  \item[]\textbf{P-Expr-3 (Constant)} \vspace*{-2mm} 
  \[\frac{\cop{L}{(\mname{op}, o, k)}}
  {\cpexpr{L}{(\mname{con}, o, k)}{k}}\]

  \item[]\textbf{P-Expr-4 (Variable)} \vspace*{-2mm} 
  \[\frac{x \in \sS, \ctype{L}{\alpha}}
  {\cterma{L}{(\mname{var}, x, \alpha)}{\alpha}}\]

  \item[]\textbf{P-Expr-5 (Type application)}\vspace*{-2mm}
  \[\frac{\ctype{L}{\alpha}, \cterm{L}{a}} 
  {\ctype{L}{(\mname{type-app},\alpha,a)}}\]

  \item[]\textbf{P-Expr-6 (Dependent function type)}\vspace*{-2mm}
  \[\frac{\cterm{L}{(\mname{var},x,\alpha)},\ctype{L}{\beta}}
  {\ctype{L}{(\mname{dep-fun-type},(\mname{var},x,\alpha),\beta)}}\]

  \item[]\textbf{P-Expr-7 (Function application)}\vspace*{-2mm}
  \[\frac{\cterma{L}{f}{\alpha}, \cterm{L}{a}} 
  {\cterma{L}{(\mname{fun-app},f,a)}{(\mname{type-app},\alpha,a)}}\]

  \item[]\textbf{P-Expr-8 (Function abstraction)}\vspace*{-2mm}
  \[\frac{\cterm{L}{(\mname{var},x,\alpha)},\cterma{L}{b}{\beta}}
  {\cterma{L}{(\mname{fun-abs},(\mname{var},x,\alpha),b)}
  {(\mname{dep-fun-type},(\mname{var},x,\alpha),\beta)}}\]

  \item[]\textbf{P-Expr-9 (Conditional term)}\vspace*{-2mm}
  \[\frac{\cform{L}{A},\cterma{L}{b}{\beta},\cterma{L}{c}{\gamma}}
  {\cterma{L}{(\mname{if},A,b,c)} {\beta \cup \gamma}}\]

  \item[]\textbf{P-Expr-10 (Existential quantification)}\vspace*{-2mm}
  \[\frac{\cterm{L}{(\mname{var}, x, \alpha)}, \cform{L}{B}}
  {\cform{L}{(\mname{exists},(\mname{var}, x, \alpha),B)}}\]

  \item[]\textbf{P-Expr-11 (Unique existential quantification)}\vspace*{-2mm}
  \[\frac{\cterm{L}{(\mname{var}, x, \alpha)}, \cform{L}{B}}
  {\cform{L}{(\mname{uni-exists},(\mname{var}, x, \alpha),B)}}\]

  \item[]\textbf{P-Expr-12 (Universal quantification)}\vspace*{-2mm}
  \[\frac{\cterm{L}{(\mname{var}, x, \alpha)}, \cform{L}{B}}
  {\cform{L}{(\mname{forall},(\mname{var}, x, \alpha),B)}}\]

  \item[]\textbf{P-Expr-13 (Definite description)}\vspace*{-2mm}
  \[\frac{\cterm{L}{(\mname{var}, x, \alpha)}, \cform{L}{B}}
  {\cterma{L}{(\mname{def-des},(\mname{var}, x, \alpha),B)} {\alpha}}\]

  \item[]\textbf{P-Expr-14 (Indefinite description)}\vspace*{-2mm}
  \[\frac{\cterm{L}{(\mname{var}, x, \alpha)},  \cform{L}{B}}
  {\cterma{L}{(\mname{indef-des},(\mname{var}, x, \alpha),B)} {\alpha}}\]

  \item[]\textbf{P-Expr-15 (Set construction)}\vspace*{-2mm}
  \[\frac{\cterma{L}{a_1}{\alpha_1},\ldots,\cterma{L}{a_1}{\alpha_1}}
  {\cterma{L}{(\mname{set-cons},a_1,\ldots,a_n)} 
  {\beta}}\]
  where $n \ge 0$ and 
  $\beta =\left\{\begin{array}{ll}
                   \mname{C} & \mbox{if }n = 0\\
                   \alpha_1 \cup \cdots \cup \alpha_n & \mbox{otherwise.}
                 \end{array}
          \right.$

  \item[]\textbf{P-Expr-16 (List construction)}\vspace*{-2mm}
  \[\frac{\cterma{L}{a_1}{\alpha_1},\ldots,\cterma{L}{a_1}{\alpha_1}}
  {\cterma{L}{(\mname{list-cons},a_1,\ldots,a_n)} 
  {\mname{list-type}(\beta)}}\]
  where $n \ge 0$ and 
  $\beta =\left\{\begin{array}{ll}
                   \mname{C} & \mbox{if }n = 0\\
                   \alpha_1 \cup \cdots \cup \alpha_n & \mbox{otherwise.}
                 \end{array}
          \right.$ 

  \item[]\textbf{P-Expr-17 (Class abstraction)}\vspace*{-2mm}
  \[\frac{\cterm{L}{(\mname{var}, x, \alpha)}, \cform{L}{B}}
  {\cterma{L}{(\mname{class-abs},(\mname{var}, x, \alpha),B)} 
  {\mname{power-type}(\alpha)}}\]

  \item[]\textbf{P-Expr-18 (Left type)}\vspace*{-2mm}
  \[\frac{\ctype{L}{\alpha}}
  {\ctype{L}{(\mname{left-type},\alpha)}}\]

  \item[]\textbf{P-Expr-19 (Right type)}\vspace*{-2mm}
  \[\frac{\ctype{L}{\alpha},\cterm{L}{a}}
  {\ctype{L}{(\mname{right-type},\alpha,a)}}\]

  \item[]\textbf{P-Expr-20 (Dependent type product)}\vspace*{-2mm}
  \[\frac{\cterm{L}{(\mname{var}, x, \alpha)}, \ctype{L}{\beta}}
  {\ctype{L}{(\mname{dep-type-prod},(\mname{var}, x, \alpha),\beta)}}\]

  \item[]\textbf{P-Expr-21 (Dependent ordered pair)}\vspace*{-2mm}
  \[\frac{\cterma{L}{a}{\alpha},\cterma{L}{b}{\beta}}
  {\cterma{L}{(\mname{dep-ord-pair},a,b)}
  {(\mname{dep-type-prod},(\mname{var}, x, \alpha),\beta)}}\]

  \item[]\textbf{P-Expr-22 (Dependent head)}\vspace*{-2mm}
  \[\frac{\cterma{L}{a}{\alpha}}
  {\cterma{L}{(\mname{dep-head},a)}
  {(\mname{left-type},\alpha)}}\]

  \item[]\textbf{P-Expr-23 (Dependent tail)}\vspace*{-2mm}
  \[\frac{\cterma{L}{a}{\alpha}}
  {\cterma{L}{(\mname{dep-tail},a)}
  {(\mname{right-type},\alpha,(\mname{dep-head},a))}}\]

  \item[]\textbf{P-Expr-24 (Quotation)}\vspace*{-2mm}
  \[\frac{\cexpr{L}{e}}
  {\cterma{L}{(\mname{quote}, e)}{\alpha}}\]
  where $\alpha$ is:
  
  \be

    \item $\mname{E}_{\rm sy}$ if $e \in \sS$.

    \item $\mname{E}_{\rm on}$ if $e \in \sO$.

    \item $\mname{E}_{\rm op}$ if $e$ is an operator.

    \item $\mname{E}_{\rm ty}$ if $e$ is a type.

    \item $\mname{E}_{\rm te}^{\synbrack{\beta}}$ if $e$ is a term of type $\beta$.

    \item $\mname{E}_{\rm fo}$ if $e$ is a formula.

    \item $\mname{E}$ if none of the above.

  \ee

  \item[]\textbf{P-Expr-25 (Evaluation)}\vspace*{-2mm}
  \[\frac{\cterm{L}{a},\ckind{L}{k}}
  {\cpexpr{L}{(\mname{eval}, a, k)}{k}}\]

\ei

Table~\ref{tab:exp-compacta} defines the compact notation for each of
the 25 proper expression categories.  Note: The definitions for the
compact notations $\set{a_1,\ldots,a_n}$, $\mlist{a_1,\ldots,a_n}$,
$\seq{a,b}$, $\mname{hd}(a)$, and $\mname{tl}(a)$ supersede the
definitions for these notations given in
subsection~\ref{subsec:logop}.

\begin{table}
\bc
\begin{tabular}{|ll|}
\hline
\textbf{Compact Notation} & \textbf{Official Notation} \\
$(o::k_1,\ldots,k_{n+1})$ & $(\mname{op},o, k_1,\ldots,k_{n+1})$ \\
$(o::k_1,\ldots,k_{n+1})(e_1,\ldots,e_n)$ &
  $(\mname{op-app}, (\mname{op},o, k_1,\ldots,k_{n+1}), e_1,\ldots,e_n)$ \\
$[o :: k]$ & $(\mname{con}, o, k)$ \\
$(x \mcolon \alpha)$ & $(\mname{var}, x, \alpha)$ \\
$\alpha(a)$ & $(\mname{type-app}, \alpha, a)$ \\
$(\LAMBDAapp x \mcolon \alpha \mdot \beta)$ &
  $(\mname{dep-fun-type},(\mname{var},x,\alpha),\beta)$ \\
$f(a)$ & $(\mname{fun-app},f,a)$ \\
$(\LambdaApp x \mcolon \alpha \mdot b)$ &
  $(\mname{fun-abs},(\mname{var},x,\alpha),b)$ \\
$\mname{if}(A,b,c)$ & $(\mname{if},A,b,c)$ \\
$(\ForsomeApp x \mcolon \alpha \mdot B)$ &
  $(\mname{exists},(\mname{var},x,\alpha),B)$ \\
$(\ForsomeUniqueApp x \mcolon \alpha \mdot B)$ &
  $(\mname{uni-exists},(\mname{var},x,\alpha),B)$ \\
$(\ForallApp x \mcolon \alpha \mdot B)$ &
  $(\mname{forall},(\mname{var},x,\alpha),B)$ \\
$(\iotaApp x \mcolon \alpha \mdot B)$ &
  $(\mname{def-des},(\mname{var},x,\alpha),B)$ \\
$(\epsilonApp x \mcolon \alpha \mdot B)$ &
  $(\mname{indef-des},(\mname{var},x,\alpha),B)$ \\
$\set{a_1,\ldots,a_n}$ & $(\mname{set-cons}, a_1,\ldots,a_n)$ \\
$\mlist{a_1,\ldots,a_n}$ & $(\mname{list-cons}, a_1,\ldots,a_n)$ \\
$(\ClassAbsApp x \mcolon \alpha \mdot B)$ &
  $(\mname{class-abs},(\mname{var},x,\alpha),B)$ \\
$\mname{left}(\alpha)$ & $(\mname{left-type},\alpha)$ \\
$\mname{right}(\alpha,a)$ & $(\mname{right-type},\alpha,a)$ \\
$(\DepTypeProdApp x \mcolon \alpha \mdot \beta)$ & 
  $(\mname{dep-type-prod},(\mname{var},x,\alpha),\beta)$ \\
$\seq{a,b}$ & $(\mname{dep-ord-pair},a,b)$ \\
$\mname{hd}(a)$ & $(\mname{dep-head},a)$ \\
$\mname{tl}(a)$ & $(\mname{dep-tail},a)$ \\
$\synbrack{e}$ & $(\mname{quote},e)$ \\
$\sembrack{a}_k$ & $(\mname{eval},a,k)$\\
$\sembrack{a}_{\rm ty}$ & $(\mname{eval},a,\mname{type})$\\
$\sembrack{a}_{\rm te}$ & 
  $(\mname{eval},a,(\mname{con},\mname{class},\mname{type}))$\\
$\sembrack{a}_{\rm fo}$ & $(\mname{eval},a,\mname{formula})$\\
\hline
\end{tabular}
\ec
\caption{Expanded Compact Notation}\label{tab:exp-compacta}
\end{table}

The semantics for the new definition of a proper expression is the
same as before except for the definition of the valuation function,
which is left to the reader as an exercise.

\addcontentsline{toc}{section}{References}
\bibliography{$HOME/research/lib/imps}
\bibliographystyle{plain}

\end{document}